\documentclass[reqno,fleqn]{amsart}
\usepackage{amsfonts,amsthm,amssymb}
\usepackage{lipsum}
\usepackage[utf8]{inputenc}
\usepackage[english]{babel}

\usepackage{lmodern,textcomp}
\usepackage{amsmath}
\usepackage{mathtools}
\usepackage{latexsym}
\usepackage{tikz}
\usepackage{fancyvrb}
\usepackage{epsfig}
\usepackage{pstricks,slashbox,multirow}
\usepackage{graphicx}
\usepackage{rotating}
\usetikzlibrary{positioning}
\usepackage{subcaption}
\usepackage{datetime}
\newtheorem{theorem}{Theorem}

\newtheorem{lemma}[theorem]{Lemma}

\newtheorem{definition}[theorem]{Definition}
\newtheorem{con}{Conjecture}

\newtheorem{prm}{Problem}
\newtheorem{oprm}{Open Problem}
\newtheorem{obv}{Observation}
\newtheorem{rem}[theorem]{Remark}

\newtheorem{nota}{Notation}

\title[A study on Type-2 circulant graphs: Part 8: Isomorphism series, digraph and graph of $C_n(R)$]{A study on Type-2 isomorphic circulant graphs. \\ Part 8: Isomorphism series, digraph and graph of $C_n(R)$}

\author{\sc Vilfred Kamalappan} 
\address{Department of Mathematics, Central University of Kerala, Periye, Kasaragod, Kerala, India - 671 316.}
\email{vilfredkamal@gmail.com}  



\date{}

\begin{document} 

\begin{abstract} This article is the $8^{th}$ part of a detailed study on Type-2 isomorphic circulant graphs having ten parts \cite{v2-1}-\cite{v2-10}. In this study, we define {\em isomorphic set}, {\em isomorphism series}, {\em isomorphism digraph} or {\em isomorphism diagram} $\mathcal{D}$ and {\em isomorphism graph} $\mathcal{G}$ of circulant graphs and obtain these  corresponding to $C_{16}(R)$, $C_{27}(S)$ and $C_{54}(1,3,17,19)$ and present the isomorphism digraph and the isomorphism graph of $C_{432}(T)$ which has isomorphic circulant graphs of Type-2 w.r.t. $m$ = 2,3 and of $C_{6750}(U)$ which has isomorphic circulant graphs of Type-2 w.r.t. $m$ = 3,5 where $T$ = $\{16, 27, 48, 54, 128, 160, 189\}$ and $U$ = $\{135, 243, 250, 750, 1107, 1593, 2000, 2457, 2500, 2943\}$. We also define {\em diameter of isomorphic set} of $C_n(R)$ and {\em isomorphic distance} between any two circulant graphs. We find that the order of isomorphic sets of $C_{432}(T)$ and $C_{6750}(U)$ are 36 and 450 and their diameters are 3 and 3, respectively.  
\end{abstract}

\maketitle

	
\section{Introduction}

This study is the $8^{th}$ part of a detailed study on Type-2 isomorphic circulant graphs containing ten parts by Vilfred and Wilson \cite{v2-1}-\cite{v2-10}. While studying isomorphic circulant graphs of Type-1 and Type-2, we come across circulant graphs which are isomorphic but they are neither Type-1 nor Type-2 w.r.t. any particular value of $m$. This motivated us to find how isomorphic circulant graphs are isomorphically related and for that we define {\em isomorphic set} and {\em isomorphism series} of circulant graph(s) and present our study on this topic. It is clear from our study that when the isomorphism series is having more elements, then among the members of the series many possible direct isomorphism connections occur and in order to understand direct isomorphism connections, more clearly, among the members of the isomorphism series, we introduce {\em isomorphism diagram} or {\em isomorphism digraph}  $\mathcal{D}$ and {\em isomorphism graph} $\mathcal{G}$ related to isomorphism series of the circulant graph(s). We present many examples of these digraphs and graphs that include $C_{432}(R)$ which has isomorphic circulant graphs of Type-2 w.r.t. $m$ = 2 as well as $m$ = 3 and also $C_{6750}(S)$ which has isomorphic circulant graphs of Type-2 w.r.t. $m$ = 3 as well as $m$ = 5 where $R$ = $\{16, 27, 48, 54, 128, 160, 189\}$ and $S$ = $\{135, 243, 250, 750, 1107, 1593, 2000, 2457, 2500, 2943\}$. We also define {\em diameter of isomorphic set} of $C_n(T)$ and {\em isomorphic distance} between any two circulant graphs. We find that the order of isomorphic sets of $C_{432}(R)$ and $C_{6750}(S)$ are 36 and 450 and their diameters are 3 and 3, respectively.

This paper contains 8 sections. Section 2 presents preliminaries containing some basic definitions, notations and results that are used in the subsequent sections. In Section 3, we define {\em isomorphic set}, {\em isomorphism series} and {\em Eulerian isomorphism series} of circulant graphs and obtain these corresponding to circulant graphs $C_{16}(1,2,6,7,8)$, $C_{16}(1,4,6,7)$, $C_{24}(2,4,5,6,7,10)$, $C_{24}(4,5,7,10)$, $C_{27}(3,4,5,6,12,13)$, $C_{27}(3,4,5,9,13)$,  $C_{16}(1,2,6,7,8)$ $\cup$ $C_{16}(1,4,6,7)$ $\cup$ $C_{27}(3,4,5,9,13)$. We also prove that the pair of circulant graphs, $C_{54}(1,3,17,19)$, $C_{54}(5,13,21,23)$;  $C_{54}(7, 11, 21, 25)$, $C_{54}(7, 11, 15, 25)$; and $C_{54}(1,3,17,19)$, $C_{54}(7,11,15,25)$ are isomorphic but they are neither of Type-1 nor of Type-2 (w.r.t. $m$ = 3) and obtain isomorphism series and Eulerian isomorphism series of $C_{54}(1,3,17,19)$ and $C_{54}(7,11,15,25)$. In Section 4, we define {\em isomorphism digraph} or {\em isomorphism diagram} $\mathcal{D}$ and {\em isomorphism graph} $\mathcal{G}$ of circulant graphs and obtain these corresponding to $C_{16}(R)$, $C_{27}(S)$ and $C_{54}(1,3,17,19)$. In Section 5, we present the isomorphism digraph and graph of $C_{432}(16, 27, 48, 54, 128, 160, 189)$ which has isomorphic circulant graphs of Type-2 w.r.t. $m$ = 2 as well as $m$ = 3. In Section 6, the isomorphism digraph and graph of $C_{6750}(135, 243, 250, 750$, $1107$, $1593, 2000, 2457, 2500, 2943)$ which has isomorphic circulant graphs of Type-2 w.r.t. $m$ = 3 as well as $m$ = 5 are presented.  In Section 7, we define {\em diameter of isomorphic set} of $C_n(R)$ and {\em isomorphic distance} between circulant graphs and obtained some values. Section 8 is the conclusion and contains some open problems.

\section{Preliminaries }  

In this section, we present a few definitions and results which are needed in this paper.

\begin{definition}{\rm\cite{ad67}} \quad \label{d1} For $R =$ $\{r_1$, $r_2$, $\dots$, $r_k\}$ and $S$ = $\{s_1$, $s_2$, $\dots$, $s_k\}$, circulant graphs $C_n(R)$ and $C_n(S)$ are {\em Adam's isomorphic} or {\em Type-1 isomorphic} \cite{v2-1} if there exists a positive integer $x$ $\ni$ $\gcd(n, x)$ = 1 and $S$ = $\{xr_1$, $xr_2$, $\dots$, $xr_k\}_n^*$ where $<r_i>_n^*$, the {\em reflexive modular reduction} of a sequence $< r_i >$, is the sequence obtained by reducing each $r_i$ under modulo $n$ to yield $r_i'$ and then replacing all resulting terms $r_i'$ which are larger than $\frac{n}{2}$ by $n-r_i'.$  
\end{definition}

\begin{definition} {\rm\cite{v2-1}} \label{d2} Let $Ad_n = \{\varphi_{n,x}: x\in \varphi_n\}$, $Ad_n(S) = \{\varphi_{n,x}(S): x\in \varphi_n\}$ = $\{xS: x\in \varphi_n\}$, $Ad_{n,x}(C_n(R))$ = $T1_{n,x}(C_n(R))$ = $\varphi_{n,x}(C_n(R))$ = $C_n(\varphi_{n,x}(R))$ = $C_n(xR)$, $x\in \varphi_n$ and $Ad_n(C_n(R)) = T1_n(C_n(R)) = \{\varphi_{n,y}(C_n(R)) = C_n(yR): y\in \varphi_n\}$ for sets $R,S \subseteq \mathbb{Z}_n$ where $\varphi_{n,x}(R)$ in $C_n(\varphi_{n,x}(R))$ is calculated under the reflexive modulo $n$. Define $'\circ''$ in $Ad_n(C_n(R))$ such that $\varphi_{n,x} \circ' \varphi_{n,y}$ = $\varphi_{n,xy}$, $C_n(xR) \circ' C_n(yR)$ = $C_n((xy)R)$ and $\varphi_{n,x}(C_n(R)) \circ' \varphi_{n,y}(C_n(R))$ = $\varphi_{n,xy}(C_n(R))$, $\forall$ $x,y\in\varphi_n$.

Here, $(\varphi_{n,x} \circ' \varphi_{n,y})(C_n(R))$ = $\varphi_{n,xy}(C_n(R))$ = $C_n((xy)R)$ = $C_n(xR) \circ' C_n(yR)$  = $\varphi_{n,x}(C_n(R)) \circ' \varphi_{n,y}(C_n(R))$, $\forall$ $x,y \in \varphi_n$, under arithmetic modulo $n$. 
\end{definition}
  	
Clearly, $Ad_n(C_n(R))$ is the set of all circulant graphs which are Type-1 isomorphic to $C_n(R)$ and $(Ad_n(C_n(R)), \circ' )$ = $(T1_n(C_n(R)), \circ' )$ is an Abelian group and we call it as the {\em Adam's group} or {\em Type-1 group of} $C_n(R)$ under $``\circ'"$. Moreover, $(\varphi_{n,x} \circ' \varphi_{n,y})(C_n(R))$ = $\varphi_{n,x}( \varphi_{n,y}(C_n(R)))$ = $(\varphi_{n,x}(C_n(yR)))$ = $C_n(x(yR))$ = $C_n((xy)R)$ = $C_n(xR)$ $\circ'$  $C_n(yR)$ = $\varphi_{n,x}(C_n(R))$ $\circ'$ $\varphi_{n,y}(C_n(R))$, $\forall$ $x,y\in \varphi_n$, $(xy)R$, $xR$ and $yR$ are calculated under reflexive modulo $n$ and $xy$ is calcuated under multiplication modulo $n$.
	
\begin{theorem} \cite{v24} \label{t3} {\rm Let $Ad_n(C_n(R))$ = $\{\varphi_{n,x}(C_n(R)) = C_n(xR): x\in\varphi_n \}$. Then, $C_n(S)\in Ad_n(C_n(R))$ if and only if $Ad_n(C_n(R))$ = $Ad_n(C_n(S))$ if and only if $C_n(R)\in Ad_n(C_n(S))$.   \hfill $\Box$  }
\end{theorem}

In 1996, Vilfred \cite{v96} defined Type-2 isomorphism of $C_n(R)$ w.r.t. $m$ $\ni$ $m$ = $\gcd(n, r) > 1$, $r\in R$ and $r,n\in\mathbb{N}$ and studied Type-2 isomorphic circulant graphs w.r.t. $m$ = 2  in \cite{v13,v20}. And with Wilson \cite{v24} obtained families of isomorphic circulant graphs of Type-2 w.r.t. $m$ = 3,5,7. In \cite{v25}, Vilfred modified the definition of Type-2 isomorphism of circulant graphs $C_n(R)$ w.r.t. $m$ by considering $m > 1$ as a divisor of $\gcd(n, r)$ and $r\in R$ and in \cite{v2-1} and he further modified the definition, as given below, to $r\in R$ and $m > 1$ and $m^3$ are divisors of $\gcd(n, r)$ and $n$, respectively.

\begin{definition} \cite{v2-1}\quad  \label{d4} Let $V(K_n) = \{u_0,u_1,u_2,...,u_{n-1}\}$, $V(C_n(R)) = \{v_0,v_1,v_2,...,$ $v_{n-1}\}$, $r\in R$, $|R| \geq 3$, $m > 1$ and $m$ and $m^3$ be divisors of $\gcd(n, r)$ and $n$, respectively.  Define one-to-one mapping $\theta_{n,m,t} :$ $V(C_n(R)) \rightarrow V(K_n)$ such that $\theta_{n,m,t}(v_x)$ = $u_{x+jtm}$,  $\theta_{n,m,t}((v_x, v_{x+s}))$ = $(\theta_{n,m,t}(v_x), \theta_{n,m,t}(v_{x+s}))$ under subscript arithmetic modulo $n$ and $\theta_{n,m,t}(C_n(R))$ = $C_n(\theta_{n,m,t}(R))$ for every $x$ = $qm+j \in \mathbb{Z}_n$, $s\in R$, $0 \leq j \leq m-1$, $0 \leq q,t \leq \frac{n}{m} -1$ and $\theta_{n,m,t}(R)$ in $C_n(\theta_{n,m,t}(R))$ is calculated under the reflexive modulo $n$. And for a particular value of $t,$ if  $\theta_{n,m,t}(C_n(R))$ = $C_n(S)$ for some $S$  and  $S \neq yR$ for all $y\in \varphi_n$ under reflexive modulo $n$, then $C_n(R)$ and $C_n(S)$ are called {\em isomorphic circulant graphs of Type-2 w.r.t. $m$} and the isomorphism as {\em Type-2 isomorphism w.r.t. $m$.} 

When $C_n(R)$ and $C_n(S)$ are Type-2 isomorphic w.r.t. $m$, then we also say that $C_{kn}(kR)$  and $C_{kn}(kS)$ are Type-2 isomorphic w.r.t. $m$ since $k.C_n(T)$ = $C_{kn}(kT)$, $k\in\mathbb{N}$.
\end{definition}

\begin{rem}\quad \label{r5} Following steps are used to establish isomorphism of Type-2 w.r.t. $m$ between circulant graphs $C_n(R)$ and $C_n(S)$. (i) $R \neq S$ and $|R|$ = $|S| \geq 3$; (ii) $\exists$ $r\in R,S$, $m > 1$ $\ni$ $m$ and $m^3$ are divisors of $\gcd(n, r)$ and $n$, respectively, and for some $t$ $\ni$ $1 \leq t \leq \frac{n}{m} -1$, $\theta_{n,m,t}(C_n(R))$ = $C_n(S)$ and (iii) $S \neq xR$ for all $x\in\varphi_n$ under arithmetic reflexive modulo $n$. 
\end{rem} 

\begin{rem} \label{r6} \quad While searching for possible values of $t$ for which the transformed graph $\theta_{n,m,t}(C_n(R))$ is circulant of the form $C_n(S)$ for some $S \subseteq [1, \frac{n}{2}]$,  calculation on $r_i$s which are integer multiples of $m$ need not be done  under the transformation $\theta_{n,m,t}$ as there is no change in these $r_i$s where $m > 1$, $m$ and $m^3$ are divisors of $\gcd(n, r)$ and $n$, respectively, and $r\in R$. Also, for a given circulant graph $C_n(R)$, w.r.t. different values of $m$, we may get different Type-2 isomorphic circulant graphs.
\end{rem}

\begin{definition}{\rm \cite{vc13}}\quad \label{d7} The {\it symmetric equidistance condition} with respect to $v_i$ in $C_n(R)$ for a set $R = \{r_1,r_2,\dots,r_k\}$ is that $v_{i+j}$ is adjacent to $v_i$ if and only if $v_{n-j+i}$ is adjacent to $v_i$, using subscript arithmetic modulo $n,$ $0 \leq i,j \leq n-1$. 
\end{definition} 

\begin{theorem}{\rm \cite{v2-1} \quad \label{t8} Let $R$ = $\{r_1,r_2,\dots,r_k\}  \subseteq [1, \frac{n}{2}]$, $m_i > 1$ be a divisor of $\gcd(n,r_i)$ for at least one $i$,  $1 \leq i \leq k$. Then, $\theta_{n,m_i,t}(C_n(R))$ = $C_n(S)$ if and only if $\theta_{n,m_i,t}(R \cup (n-R))$ = $S \cup (n-S)$ for some $t$ and $S$ $\ni$ $0 \leq t \leq \frac{n}{m_i}-1$ and $S \subseteq [1, \frac{n}{2}]$ if and only if $\theta_{n,m_i,t}(C_n(R))$ satisfies the symmetric equidistant condition w.r.t. $v_0$.  \hfill $\Box$   }
\end{theorem}

\begin{definition}{\rm \cite{v2-1}}\quad \label{d9} Let $V(C_n(R)) = \{v_0, v_1, \dots, v_{n-1}\}$, $V(K_n)$ = $\{u_0, u_1,\dots, u_{n-1}\}$, $x = qm+j,$ $0 \leq j \leq m-1$, $m > 1$ divide $\gcd(n, r)$, $m^3$ divide $n$, $0 \leq q,t,t' \leq \frac{n}{m}-1$, $m,q,t,t',x\in \mathbb{Z}_n$ and $r \in R$. Define $\theta_{n,m,t}:$ $V(C_n(R))$ $\rightarrow$  $V(K_n)$ $\ni$ $\forall$  $x\in \mathbb{Z}_n$, $\theta_{n,m,t}(v_x)$ = $u_{x+jmt}$ and $\theta_{n,m,t}((v_x, v_{x+s}))$ = $(\theta_{n,m,t}(v_x), \theta_{n,m,t}(v_{x+s}))$  and $s\in R,$ under subscript arithmetic modulo $n$. Let $V_{n,m}$ = $\{\theta_{n,m,t}:$ $t$ = $0,1,\dots,\frac{n}{m}-1\}$, $V_{n,m}(v_s)$ = $\{\theta_{n,m,t}(v_s):$ $t = 0,1,\dots, \frac{n}{m}-1\}$, $s \in \mathbb{Z}_n$ and $V_{n,m}(C_n(R))$ = $\{\theta_{n,m,t}(C_n(R)) = C_n(\theta_{n,m,t}(R)):$ $t$ = $0,1,\dots,\frac{n}{m}-1\}$ where $\theta_{n,m,t}(R)$ in $C_n(\theta_{n,m,t}(R))$  is calculated under reflexive modulo $n$. Define $'\circ'$ in $V_{n,m} \ni \theta_{n,m,t} \circ  \theta_{n,m,t'}$ =  $\theta_{n,m,t+t'}$ and $(\theta_{n,m,t} ~\circ ~ \theta_{n,m,t'})(C_n(R))$ = $\theta_{n,m,t}(C_n(R)) ~\circ ~ \theta_{n,m,t'}(C_n(R))$ = $\theta_{n,m,t+t'}(C_n(R))$, $\forall$ $\theta_{n,m,t},\theta_{n,m,t'}\in V_{n,r}$ where $t+t'$ is calculated under arithmetic modulo ~$\frac{n}{m}$.
\end{definition}

$V_{n,m}(C_n(R))$ = $\{\theta_{n,m,t}(C_n(R)): t = 0,1,\dots, \frac{n}{m}-1\}$ and for $t$ = 0 to $\frac{n}{m}-1$, the $\frac{n}{m}$ graphs $\theta_{n,m,t}(C_n(R))$ are isomorphic and 
$V_{n,m}(C_n(R))$ contains all isomorphic circulant graphs of Type-2 of $C_n(R)$ w.r.t. $m$, if exist, under the transformation $\theta_{n,m,t}$ on $C_n(R)$ where $r\in R$, $m > 1$ divides $\gcd(n, r)$ and $m^3$ divides $n$. Following is an algebraic property of $V_{n,m}(C_n(R))$.  

\begin{theorem}{\rm \cite{v24}  \quad \label{t10}  Under the above definition of $`\circ'$ and $V_{n,m}(C_n(R))$, $(V_{n,m}(C_n(R)), \circ)$ is an Abelian  group.  \hfill $\Box$}
\end{theorem}

\begin{definition}{\rm \cite{v2-1}} \quad \label{d11} Let $T2_{n,m}(C_n(R))$ = $\{C_n(R)\}$ $\cup$ $\{C_n(S):$ $C_n(S)$ is Type-2 isomorphic to $C_n(R)$ w.r.t. $m\}$ where $r\in R$, $m > 1$ divides $\gcd(n,r)$ and $m^3$ divides $n$. We call $T2_{n,m}(C_n(R))$ as {\em the Type-2 set of $C_n(R)$
w.r.t. $m$}.

That is, {\em the Type-2 set of $C_n(R)$ w.r.t. $m$}, denoted by $T2_{n,m}(C_n(R))$, is $\{C_n(R)\}$ $\cup$ $\{\theta_{n,m,t}(C_n(R)):$ $\theta_{n,m,t}(C_n(R))$ = $C_n(S)$ and $C_n(S)$ is Type-2 isomorphic to $C_n(R)$ w.r.t. $m$, $1 \leq t \leq \frac{n}{m}-1\}$ = $\{\theta_{n,m,0}(C_n(R))\}$ $\cup$ $\{\theta_{n,m,t}(C_n(R)):$ $\theta_{n,m,t}(C_n(R))$ = $C_n(S)$ and $C_n(S)$ is Type-2 isomorphic to $C_n(R)$ w.r.t. $m$, $1 \leq t \leq \frac{n}{m}-1\}$ where $r\in R$, $m > 1$ divides $\gcd(n,r)$ and $m^3$ divides $n$.
\end{definition}

$C_n(R)$ has Type-2 isomorphic circulant graph w.r.t. $m$ if and only if $T2_{n,m}(C_n(R))$ $\neq$ $\{C_n(R)\}$ if and only if $T2_{n,m}(C_n(R$ $)) \setminus \{C_n(R)\} \neq \emptyset$ if and only if $|T2_{n,m}(C_n(R))| > 1$.  

\begin{theorem} {\rm \cite{v2-6} \quad \label{t12}  $(T2_{n,m}(C_n(R)),  \circ)$ is a subgroup of $(V_{n,m}(C_n(R)), \circ)$ where $r\in R$, $m > 1$ divides $\gcd(n,r)$ and $m^3$ divides $n$.   \hfill $\Box$}
\end{theorem}

\begin{definition}{\rm \cite{v2-6}}\quad \label{d13} With usual notation, group $(T2_{n,m}(C_n(R)), \circ)$ is called the Type-2 group of $C_n(R)$ w.r.t.  $m$.  
\end{definition}

\begin{theorem}{\rm \cite{v2-6}} \label{t14} \quad {\rm  Let $C_n(R)$ $\cong$ $C_n(S)$, $R \neq S$, $|R| = |S| \geq 3$, $r\in R,S$, $m > 1$, and $m$ and $m^3$ divide $\gcd(n,r)$ and $n$, respectively. Then, $C_n(S)\in$ $T2_{n,m}($ $C_n(R))$ if and only if  $T2_{n,m}(C_n(R))$ = $T2_{n,m}(C_n(S))$ if and only if  $C_n(R)\in T2_{n,m}(C_n(S))$.   \hfill $\Box$}
 \end{theorem}

\begin{nota}{\rm \cite{v2-1} \quad \label{n1} To simplify our work, the following notations are introduced under Type-1 and Type-2 isomorphisms of circulant graphs.
\begin{enumerate}
	\item [\rm (i)] $C_n(R) \cong_{Ad_{n,x}} C_{n}(S)$ or $C_n(R) \cong_{T1_{n,x}} C_{n}(S)$ when $Ad_{n,x}(C_n(R))$ $(= T1_{n,x}(C_n(R))$ = $C_{n}(xR))$ = $C_n(S)$, $x\in\varphi_n$.
	
	\item [\rm (ii)] $C_n(R) \cong_{Ad_{n}} C_{n}(S)$ or $C_n(R) \cong_{T1_{n}} C_{n}(S)$ when $C_n(R)$ and $C_n(S)$ are Adam's isomorphic or Type-1 isomorphic. 
	
	\item [\rm (iii)] In $T2_{n,m,t}(C_{n}(R))$,  either $\theta_{n,m,t}(C_n(R))$ = $C_{n}(R)$ or  $\theta_{n,m,t}(C_n(R))$ and $C_{n}(R)$ are isomorphic circulant graphs of Type-2 w.r.t. $m$ for some $t$ where $m > 1$, $m$ divides $\gcd(n, r)$, $m^3$ divides $n$, $r\in R$ and $0 \leq t \leq \frac{n}{m}-1$.  
			
	\item [\rm (iv)] $C_n(R) \cong_{T2_{n,m,t}} C_{n}(S)$ when $T2_{n,m,t}(C_n(R))$ = $C_n(S)$ for some $t$, $0 \leq t \leq \frac{n}{m} -1$ where $m > 1$ divides $\gcd(n, r)$, $m^3$ divides $n$ and $r\in R,S$. That is, when $\theta_{n,m,t}(C_n(R))$ = $C_n(S)$ for some $t$ and $C_n(R)$ and $C_n(S)$ are  Type-2 isomorphic w.r.t. $m$, $0 \leq t \leq \frac{n}{m} -1$.
			
	\item [\rm (v)]  $C_n(R) \cong_{T2_{n,r}} C_{n}(S)$ when $C_n(R)$ and $C_n(S)$ are Type-2 isomorphic w.r.t. $m$ where $m > 1$ and $m^3$ divide $\gcd(n, r)$ and $n$,
respectively, and $r\in R,S$. 
\end{enumerate}	}
\end{nota}

\begin{definition} \cite{bm82} \quad \label{d15} {\em The cross product} or {\em Cartesian product} of two simple graphs $G(V, E)$ and $H(W, F)$ is the simple graph $G \Box H$ with vertex set $V \times W$ in which two vertices $u$ = $(u_1, u_2)$ and $v$ = $(v_1, v_2)$ are adjacent if and only if either $u_1$ = $v_1$ and $u_2 v_2\in F$ or $u_2$ = $v_2$ and $u_1 v_1\in E$.  
\end{definition} 

\begin{definition}{\rm \cite{bm82}} \quad \label{d16} A {\em directed graph} or {\em digraph} $D$ is a triple $(V(D), A(D), \psi_D)$ consisting of a nonempty set $V(D)$ of vertices, a set $A(D)$, disjoint from $V(D)$, of arcs or edges, and an incidence function $\psi_D$ that associates with each arc of $D$ an ordered pair of (not necessarily distinct) vertices of $D$. If $a$ is an arc and $u$ and $v$ are vertices such that $\psi_D(a)$ = $(u, v)$, then $a$ is said to join $u$ to $v$; $u$ is the tail of $a$, and $v$ is its head.   
\end{definition}

\begin{definition}{\rm \cite{dw02}} \quad \label{d17} A {\em symmetric digraph} is a directed graph where all edges appear twice, one in each direction (that is, for every arrow that belongs to the digraph, the corresponding inverse arrow also belongs to it).  
\end{definition}

\section{Isomorphic set and isomorphism series of circulant graphs}

In this section, we define {\em isomorphic set}, {\em isomorphism series} and {\em Eulerian isomorphism series} of circulant graphs and obtain these corresponding to isomorphic $C_n(R)$ and $C_n(S)$ which are of neither Type-1 nor Type-2 w.r.t. a particular $m$. These include circulant graphs $C_{16}(1,2,6,7,8)$, $C_{16}(1,4,6,7)$, $C_{24}(2,4,5,6,7,10)$, $C_{24}(4,5,7,10)$, $C_{27}(3,4,5,6,12,13)$, $C_{27}(3,4,5,9,13)$,  $C_{16}(1,2,6,7,8)$ $\cup$ $C_{16}(1,4$, $6,7)$ $\cup$ $C_{27}(3,4,5,9,13)$. We show that the pair of circulant graphs, $C_{54}(1,3,17,19)$, $C_{54}(5,13,21,23)$;  $C_{54}(7, 11, 21, 25)$, $C_{54}(7, 11, 15, 25)$; and $C_{54}(1,3,17,19)$, $C_{54}(7,11,15,25)$ are isomorphic but they are of neither Type-1 nor Type-2 isomorphic (w.r.t. $m$ = 3). We also obtain isomorphic sets, some isomorphism series and Eulerian isomorphism series of $C_{54}(1,3,17,19)$ and $C_{54}(7,11,15,25)$. 

\begin{definition} \quad \label{d18} For a given circulant graph $C_n(R)$, the set of all circulant graphs isomorphic to $C_n(R)$ is called {\em the isomorphic set of} $C_n(R)$ and it is denoted by $Isoset(C_n(R))$.

In general, {\em the isomorphic set of} given graphs $G_1, G_2, \dots, G_k$ is denoted by $Isoset(G_1, G_2, \dots, G_k)$ and is defined by $Isoset(G_1, G_2, \dots, G_k)$ = $\{H: H \cong G, ~G = \cup_{i=1}^k  G_i\}$, $k\in\mathbb{N}$.
\end{definition}

\begin{definition} \quad \label{d19} Let $C_n(R)$ be a given circulant graph and $C_n(S)$ be isomorphic to it such that $S \neq R$. Then, a series of isomorphic circulant graphs starting from $C_n(R)$ is called an {\em isomorphism series of $C_n(R)$} if (i) any two consecutive members in the series of isomorphic circulant graphs are either Type-1 isomorphic or Type-2 isomorphic and (ii) Type-1 isomorphism and Type-2 isomorphism occur alternatively in the series.

In general, for a given graph $G$, a series of graphs, each graph isomorphic to $G$, such that consecutive members in the series of isomorphic graphs carry some structural property of these graphs is called an {\em isomorphism series of $G$}. 
\end{definition}

Through out this study, $C_n(R_1) \cong C_n(R_2) \cong C_n(R_3) \cong \dots \cong C_n(R_k) \cong C_n(R_{k+1})$ denotes $C_n(R_1) \cong$ $C_n(R_2)$, $C_n(R_2) \cong C_n(R_3)$, $\dots$, $C_n(R_k)$ $\cong$ $C_n(R_{k+1})$ and each $\cong$ is either $\cong_{T1_n}$ or $\cong_{T2_{n, m}}$ where $C_n(R_i) \cong_{T1_n} C_n(R_j)$ denotes $C_n(R_i)$ and $C_n(R_j)$ are isomorphic of Type-1 and $C_n(R_x) \cong_{T2_{n, m}} C_n(R_y)$ denotes $C_n(R_x)$ and $C_n(R_y)$ are Type-2 isomorphic w.r.t. $m$, $m > 1$ divides $\gcd(n, r)$, $m^3$ divides $n$ and $r\in R_x,R_y$.  

For any circulant graph $C_n(R)$, its isomorphic set $Isoset(C_n(R))$ is unique. Clearly, $Isoset(C_n(R))$ = $Isoset(C_n(S))$ if and only if  $C_n(R)$ $\cong$ $C_n(S)$ if and only if $C_n(S)\in Isoset(C_n(R))$ if and only if $C_n(R)\in Isoset(C_n(S))$ and $Isoset(C_n(R))$ $\cap$ $Isoset(C_n(S))$ = $\emptyset$ if and only if $C_n(R)$ $\ncong$ $C_n(S)$. There may be more than one isomorphism series for a given circulant graph $C_n(R)$.

\begin{definition} \quad \label{d20} An isomorphism series of isomorphic circulant graphs is said to be {\em simple} if no two consecutive elements of the series repeat with the same type of isomorphism.  

Hereafter, we consider simple isomorphism series only unless otherwise it is mentioned in other way.
\end{definition}

\begin{definition} \quad \label{d21} An isomorphism series of isomorphic circulant graphs is said to be {\em an Eulerian isomorphism series} if (i) it is a simple isomorphism series, (ii) the elements of the series covers all the elements of the isomorphic set and (iii) it covers all type of isomorphisms that exist among the isomorphic circulant graphs of the isomorphic series. i.e., it covers all directed edges of the isomorphism diagram of the given circulant graph.  
\end{definition}

\begin{definition} \quad \label{d22} Let $C_n(R)$ = $C_n(R_1) \cong C_n(R_2) \cong C_n(R_3) \cong \dots \cong C_n(R_k) \cong C_n(R_{k+1})$ be an isomorphism series of $Isoset(C_n(R))$. Let $\mathcal{D}$ be a directed graph with $V(\mathcal{D})$ = $\{C_n(R_i): 1 \leq i \leq k+1\}$ and $(C_n(R_i), C_n(R_j))\in E(\mathcal{D})$ if either $C_n(R_i)$ $\cong_{T1_n}$ $C_n(R_j)$, in this case $C_n(R_i)$ is connected to $C_n(R_j)$ by a directed line, or $C_n(R_i)$ $\cong_{T2_{n, m}}$ $C_n(R_j)$, in this case $C_n(R_i)$ is connected to $C_n(R_j)$ by a directed dashed  line, where $m > 1$ divides $\gcd(n, r)$, $m^3$ divides $n$ and $r\in R_i,R_j$, $1 \leq i,j \leq k+1$. In this case, we call the digraph $\mathcal{D}$ as {\em the isomorphism series digraph} or {\em the isomorphism series diagram}. Moreover, when $m$ = $m_1,m_2,\dots,m_c$ are different values taken by $m$ of Type-2 isomorphism w.r.t. $m$ that occur in the isomorphism series, then directed dashed line $(C_n(R_i), C_n(R_j))$ is considered with color $x$ in $\mathcal{D}$ whenever $C_n(R_i)$ $\cong_{T2_{n, m_x}}$ $C_n(R_j)$, $1 \leq x \leq c$ and $1 \leq i,j \leq k+1$. 
\end{definition}

\begin{definition} \quad \label{d23} Let $C_n(R)$ be a given circulant graph, $\mathcal{D}$ be a directed graph with $V(\mathcal{D})$ = $\{C_n(S): C_n(S)\in Isoset(C_n(R))\}$ and $(C_n(S), C_n(T))\in E(\mathcal{D})$ if either $C_n(S)$ $\cong_{T1_n}$ $C_n(T)$, in this case $C_n(S)$ is connected to $C_n(T)$ by a directed line, or $C_n(S)$ $\cong_{T2_{n, m}}$ $C_n(T)$, in this case $C_n(S)$ is connected to $C_n(T)$ by a directed dashed line, where $m > 1$ divides $\gcd(n, r)$, $m^3$ divides $n$, $r\in S,T$ and $C_n(S), C_n(T)\in Isoset(C_n(R))$. In this case, we call the digraph $\mathcal{D}$ as {\em the isomorphism digraph} or {\em the isomorphism diagram} of $C_n(S)$ where $C_n(S)\in V(\mathcal{D})$. Moreover, when $m$ = $m_1,m_2,\dots,m_c$ are different values taken by $m$  of Type-2 isomorphism w.r.t. $m$ that occur in the isomorphism series of $\mathcal{D}$, then directed dashed line $(C_n(S), C_n(T))$ is considered with color $x$ in $\mathcal{D}$ whenever $C_n(S)$ $\cong_{T2_{n, m_x}}$ $C_n(T)$, $C_n(S), C_n(T)\in Isoset(C_n(R))$ and $1 \leq x \leq c$.
		
	When $G_1, G_2, \dots, G_k$ are disjoint graphs, $\mathcal{D}_i$ = the isomorphism digraph, and $\mathcal{G}_i$ = the isomorphism graph of $G_i$, then the isomorphism digraph of $G$ is $\cup^k_{i=1} {\mathcal{D}_i}$, and the isomorphism graph of $G$ is $\cup^k_{i=1} {\mathcal{G}_i}$ where $G$ = $G_1 \cup G_2 \cup \dots \cup G_k$, $1 \leq i \leq k$, and $k\in\mathbb{N}$.  
\end{definition}

\begin{prm} \quad \label{p1} {\rm Find the isomorphic set and an Eulerian isomorphism series, if it exists, for the following graphs.
\begin{enumerate}
\item [\rm (a)] $C_{16}(1,2,6,7,8)$;

\item [\rm (b)] $C_{16}(1,4,6,7)$;

\item [\rm (c)] $C_{24}(2,4,5,6,7,10)$;

\item [\rm (d)] $C_{24}(4,5,7,10)$;

\item [\rm (e)] $C_{27}(3,4,5,6,12,13)$;

\item [\rm (f)] $C_{27}(3,4,5,9,13)$;

\item [\rm (abf)] $C_{16}(1,2,6,7,8)$ $\cup$ $C_{16}(1,4,6,7)$ $\cup$ $C_{27}(3,4,5,9,13)$.
\end{enumerate} }
\end{prm}
\noindent
{\bf Solution.}\quad  In each case, we find circulant graphs which are isomorphic of Type-1 and Type-2 w.r.t. $m$  to each given circulant graph and then again find  isomorphic circulant graphs of Type-1 and Type-2 w.r.t. $m$ to each new isomorphic circulant graphs that we have already obtained. We repeat the process and also with different possible values of $m$ until no more new isomorphic circulant graph exists in this process.

\begin{enumerate}
\item [\rm (a)] In this problem, $n$ = 16 = $2\times 2^3$ and so $m$ = 2 is the only possible value of $m$ such that the given circulant graph is having isomorphic circulant graph of Type-2 w.r.t. $m$ = 2. Also, 
\\
$T1_{16}(C_{16}(1,2,6,7,8))$ = $\{C_{16}(1,2,6,7,8), C_{16}(2,3,5,6,8) = C_{16}(3(1,2,6,7,8))\}$ 

\hfill = $T1_{16}(C_{16}(2,3,5,6,8))$;
\\
$T2_{16,2}(C_{16}(1,2,6,7,8))$ = $\{C_{16}(1,2,6,7,8) = \theta_{16,2,0}(C_{16}(1,2,6,7,8)) = \theta_{16,2,4}(C_{16}(1,2,6,7,8))\}$; 
\\
$T2_{16,2}(C_{16}(2,3,5,6,8))$  = $\{C_{16}(2,3,5,6,8) = \theta_{16,2,0}(C_{16}(2,3,5,6,8)) = \theta_{16,2,4}(C_{16}(2,3,5,6,8))\}$.
\\
Therefore, the isomorphic set of  $C_{16}(1,2,6,7,8)$ is

$Isoset(C_{16}(1,2,6,7,8))$ = $\{C_{16}(1,2,6,7,8), C_{16}(2,3,5,6,8)\}$ = $Isoset(C_{16}(2,3,5,6,8))$ and 
\\
the isomorphism series of $C_{16}(1,2,6,7,8)$ are 

(1) $C_{16}(1,2,6,7,8)$ $\cong_{T1_{16,3}}$ $C_{16}(2,3,5,6,8)$ and 

(2) $C_{16}(2,3,5,6,8)$ $\cong_{T1_{16,3}}$ $C_{16}(1,2,6,7,8)$.

There is no Eulerian isomorphism series in this case.\\

\item [\rm (b)]  We have 

$T2_{16,2, 2}(C_{16}(1,4,6,7))$ = $C_{16}(3,4,5,6)$, $T2_{16,2, 2}(C_{16}(3,4,5,6))$ = $C_{16}(1,4,6,7)$, 

$T1_{16,3}(C_{16}(1,4,6,7))$ = $C_{16}(2,3,4,5)$, $T1_{16,3}(C_{16}(3,4,5,6))$ = $C_{16}(1,2,4,7)$,

$T2_{16,2, 2}(C_{16}(2,3,4,5))$ = $C_{16}(1,2,4,7)$, $T2_{16,2, 2}(C_{16}(1,2,4,7))$ = $C_{16}(2,3,4,5)$,

 $T1_{16, 3}(C_{16}(2,3,4,5))$ = $C_{16}(1,4,6,7)$,  $T1_{16, 3}(C_{16}(1,2,4,7))$ = $C_{16}(3,4,5,6)$.
\\
Therefore, the isomorphic set of  $C_{16}(1,4,6,7)$ is

$Isoset(C_{16}(1,4,6,7))$ = $\{C_{16}(1,4,6,7), C_{16}(3,4,5,6), C_{16}(2,3,4,5), C_{16}(1,2,4,7)\}$ 

\hfill = $Isoset(C_{16}(3,4,5,6))$ = $Isoset(C_{16}(2,3,4,5))$ = $Isoset(C_{16}(1,2,4,7))$ and 
\\
longest isomorphism series of $C_{16}(1,4,6,7)$ are 

(1) $C_{16}(1,4,6,7)$ $\cong_{T1_{16,3}}$ $C_{16}(2,3,5,6)$ $\cong_{T2_{16,2,2}}$  $C_{16}(1,2,4,7)$ 

\hfill $\cong_{T1_{16,3}}$ $C_{16}((3,4,5,6)$  $\cong_{T2_{16,2,2}}$ $C_{16}(1,4,6,7)$; 

(2) $C_{16}(1,4,6,7)$ $\cong_{T2_{16,2,2}}$ $C_{16}(3,4,5,6)$ $\cong_{T1_{16,3}}$ $C_{16}(1,2,4,7)$ 

\hfill $\cong_{T2_{16,2,2}}$ $C_{16}(2,3,4,5)$ $\cong_{T1_{16,3}}$ $C_{16}(1,4,6,7)$.  

There is no Eulerian isomorphism series in this case.\\

\item [\rm (c)]  Here, $n$ = 24 = $3\times 2^3$ and so $m$ = 2 is the only possible value of $m$ such that the given circulant graph is having isomorphic circulant graph of Type-2 w.r.t. $m$ = 2. We have 

$T1_{24}(C_{24}(2,4,5,6,7,10))$ = $\{C_{24}(2,4,5,6,7,10), C_{24}(1,2,4,6,10,11)$  

\hfill = $C_{24}(5(2,4,5,6,7,10)) \}$ = $T1_{24}(C_{24}(1,2,4,6,10,11))$; 

$T2_{24,2}(C_{16}(2,4,5,6,7,10))$  

\hfill = $\{C_{24}(2,4,5,6,7,10) = \theta_{16,2,0}(C_{16}(2,4,5,6,7,10)) = \theta_{16,2,6}(C_{16}(2,4,5,6,7,10))\}$; 

$T2_{24,2}(C_{24}1,2,4,6,10,11))$  

\hfill = $\{C_{24}(1,2,4,6,10,11)  = \theta_{16,2,0}(C_{16}(1,2,4,6,10,11)) = \theta_{16,2,6}(C_{16}(1,2,4,6,10,11))\}$.
\\
Therefore, the isomorphic set of  $C_{24}(2,4,5,6,7,10)$ is

$Isoset(C_{24}(2,4,5,6,7,10))$ = $\{C_{24}(2,4,5,6,7,10), C_{24}(1,2,4,6,10,11)\}$ 

\hfill = $Isoset(C_{24}(1,2,4,6,10,11))$ and 
\\
the isomorphism series of $C_{24}(2,4,5,6,7,10)$ are 

(1) $C_{24}(2,4,5,6,7,10)$ $\cong_{T1_{24,3}}$ $C_{24}(1,2,4,6,10,11)$ and 

(2) $C_{24}(1,2,4,6,10,11)$ $\cong_{T1_{24,3}}$ $C_{24}(2,4,5,6,7,10)$.

There is no Eulerian isomorphism series in this case.\\

\item [\rm (d)]  We have 

$T1_{24,5}(C_{24}(4,5,7,10))$ = $C_{24}(1,2,4,11)$, $T1_{24,5}(C_{24}(1,2,4,11))$ = $C_{24}(4,5,7,10)$,  

$T1_{24}(C_{24}(4,5,7,10))$ = $\{C_{24}(4,5,7,10), C_{24}(1,2,4,11) \}$ = $T1_{24}(C_{24}(1,2,4,11))$,  

$T2_{24,2,3}(C_{24}(4,5,7,10))$ = $C_{24}(1,4,10,11) $, $T2_{24,2,3}(C_{24}(1,4,10,11))$ = $C_{24}(4,5,7,10)$, 

$T2_{24,2,3}(C_{24}(1,2,4,11))$ = $C_{24}(2,4,5,7)$, $T2_{24,2,3}(C_{24}(2,4,5,7))$ = $C_{24}(1,2,4,11)$,

$T1_{24,5}(C_{24}(1,4,10,11))$ = $C_{24}(2,4,5,7)$, $T1_{24}(C_{24}(2,4,5,7))$ = $C_{24}(1,4,10,11)$,

$T1_{24}(C_{24}(1,4,10,11))$ = $\{C_{24}(1,4,10,11), C_{24}(2,4,5,7) \}$ = $T1_{24}(C_{24}(2,4,5,7))$.  

Therefore, the isomorphic set of  $C_{24}(4,5,7,10)$ is

$Isoset(C_{24}(4,5,7,10))$ = $\{C_{24}(4,5,7,10), C_{24}(1,2,4,11), C_{24}(1,4,10,11), C_{24}(2,4,5,7)\}$ 

\hfill = $Isoset(C_{24}(1,2,4,11))$ = $Isoset(C_{24}(1,4,10,11))$ = $Isoset(C_{24}(2,4,5,7))$ and 
\\
longest isomorphism series of $C_{24}(4,5,7,10)$ are 

(1) $C_{24}(4,5,7,10)$ $\cong_{T1_{24,5}}$ $C_{24}(1,2,4,11)$ $\cong_{T2_{24,2,3}}$  $C_{24}(2,4,5,7)$ 

\hfill $\cong_{T1_{24,5}}$ $C_{24}(1,4,10,11)$  $\cong_{T2_{24,2,3}}$ $C_{24}(4,5,7,10)$ and

(2) $C_{24}(4,5,7,10)$ $\cong_{T2_{24,2,3}}$ $C_{24}(1,4,10,11)$ $\cong_{T1_{24,5}}$ $C_{24}(2,4,5,7)$ 

\hfill $\cong_{T2_{24,2,3}}$ $C_{24}(1,2,4,11)$ $\cong_{T1_{24,5}}$ $C_{24}(4,5,7,10)$.  

There is no Eulerian isomorphism series in this case.

\item [\rm (e)] We have $n$ = 27 = $3^3$ and so $m$ = 3 is the only possible value of $m$ such that the given circulant graph may have isomorphic circulant graph of Type-2 w.r.t. $m$ = 3. Also, we have 

$T1_{27}(C_{27}(3,4,5,6,12,13))$ = $\{C_{27}(3,4,5,6,12,13), C_{27}(1,3,6,8,10,12), C_{27}(2,3,6,7,11,12) \}$ 

\hfill = $T1_{27}(C_{27}(1,3,6,8,10,12))$  = $T1_{27}(C_{27}(2,3,6,7,11,12))$,

$T2_{27,3}(C_{27}(3,4,5,6,12,13))$ = $\{C_{27}(3,4,5,6,12,13)\}$, 

$T2_{27,3}(C_{27}(1,2,4,6,10,11))$ = $\{C_{27}(1,3,6,8,10,12) \}$, and

$T2_{27,3}(C_{27}(2,3,6,7,11,12))$ = $\{C_{27}(2,3,6,7,11,12) \}$. 

Therefore, the isomorphic set of  $C_{27}(3,4,5,6,12,13)$ is
\\
$Isoset(C_{27}(3,4,5,6,12,13))$ = $\{C_{27}(3,4,5,6,12,13), C_{27}(1,3,6,8,10,12), C_{27}(2,3,6,7,11,12)\}$ 

\hfill = $Isoset(C_{27}(1,3,6,8,10,12))$ = $Isoset(C_{27}(2,3,6,7,11,12))$ and 
\\
simple isomorphism series of $C_{27}(3,4,5,6,12,13)$ are

(1) $C_{27}(3,4,5,6,12,13)$ $\cong_{T1_{27,2}}$ $C_{27}(1,3,6,8,10,12)$;

(2) $C_{27}(1,3,6,8,10,12)$ $\cong_{T1_{27,2}}$ $C_{27}(2,3,6,7,11,12)$;

(3) $C_{27}(2,3,6,7,11,12)$ $\cong_{T1_{27,2}}$ $C_{27}(3,4,5,6,12,13)$;

(4) $C_{27}(3,4,5,6,12,13)$ $\cong_{T1_{27,4}}$ $C_{27}(2,3,6,7,11,12)$; 
 
(5) $C_{27}(2,3,6,7,11,12)$ $\cong_{T1_{27,4}}$ $C_{27}(1,3,6,8,10,12)$;

(6) $C_{27}(1,3,6,8,10,12)$ $\cong_{T1_{27,4}}$ $C_{27}(3,4,5,6,12,13)$.

There is no Eulerian isomorphism series in this case.

\item [\rm (f)]  We have $n$ = 27 = $3^3$ and so $m$ = 3 is the only possible value such that $C_{27}(R)$ may have Type-2 isomorphic circulant graphs w.r.t. $m$ = 3 and 

$T1_{27}(C_{27}(3,4,5,9,13))$ = $\{C_{27}(3,4,5,9,13), C_{27}(1,6,8,9,10), C_{27}(2,7,9,11,12)\}$ 

\hfill = $T1_{27}(C_{27}(1,6,8,9,10))$ = $T1_{27}(C_{27}(2,7,9,11,12))$;  

$T2_{27,3}(C_{27}(3,4,5,9,13))$ = $\{C_{27}(3,4,5,9,13), C_{27}(2,3,7,9,11),  C_{27}(1,3,8,9,10)\}$ 

\hfill = $T2_{27,3}(C_{27}(2,3,7,9,11))$ = $T2_{27,3}(C_{27}(1,3,8,9,10))$;

$T2_{27,3}(C_{27}(1,6,8,9,10))$ = $\{C_{27}(1,6,8,9,10), C_{27}(4,5,6,9,13),  C_{27}(2,6,7,9,11)\}$ 

\hfill = $T2_{27,3}(C_{27}(4,5,6,9,13))$ = $T2_{27,3}(C_{27}(2,6,7,9,11))$;

$T2_{27,3}(C_{27}(2,7,9,11,12))$ = $\{C_{27}(2,7,9,11,12), C_{27}(1,8,9,10,12),  C_{27}(4,5,9,12,13)\}$ 

\hfill = $T2_{27,3}(C_{27}(1,8,9,10,12))$ = $T2_{27,3}(C_{27}(4,5,9,12,13))$. Also, 

$T1_{27}(C_{27}(2,3,7,9,11))$ = $\{C_{27}(2,3,7,9,11), C_{27}(4,5,6,9,13), C_{27}(1,8,9,10,12)\}$ 

\hfill = $T1_{27}(C_{27}(4,5,6,9,13))$ = $T1_{27}(C_{27}(1,8,9,10,12))$; 

$T1_{27}(C_{27}(1,3,8,9,10))$ = $\{C_{27}(1,3,8,9,10), C_{27}(2,6,7,9,11), C_{27}(4,5,9,12,13)\}$ 

\hfill = $T1_{27}(C_{27}(2,6,7,9,11))$ = $T1_{27}(C_{27}(4,5,9,12,13))$. Let 

$X_1$ = $\{3,4,5,9,13\}$, $X_2$ = $\{1,6,8,9,10\}$, $X_3$ = $\{2,7,9,11,12\}$, 

$Y_1$ = $\{2,3,7,9,11\}$, $Y_2$ = $\{4,5,6,9,13\}$, $Y_3$ = $\{1,8,9,10,12\}$,

$Z_1$ = $\{1,3,8,9,10\}$, $Z_2$ = $\{2,6,7,9,11\}$, $Z_3$ = $\{4,5,9,12,13\}$. Then, 

$C_{27}(3,4,5,9,13)$ = $C_{27}(X_1)$, $C_{27}(1,6,8,9,10)$ = $C_{27}(X_2)$, $C_{27}(2,7,9,11,12)$ = $C_{27}(X_3)$, 

$C_{27}(2,3,7,9,11)$ = $C_{27}(Y_1)$, $C_{27}(4,5,6,9,13)$ = $C_{27}(Y_2)$, $C_{27}(1,8,9,10,12)$ = $C_{27}(Y_3)$,

$C_{27}(1,3,8,9,10)$ = $C_{27}(Z_1)$, $C_{27}(2,6,7,9,11)$ = $C_{27}(Z_2)$, $C_{27}(4,5,9,12,13)$ = $C_{27}(Z_3)$, 

$T2_{27,3, 1}(C_{27}(X_1))$ = $C_{27}(Y_1)$, $T2_{27,3, 2}(C_{27}(X_1))$ = $C_{27}(Z_1)$,

$T2_{27,3, 1}(C_{27}(X_2))$ = $C_{27}(Y_2)$, $T2_{27,3, 2}(C_{27}(X_2))$ = $C_{27}(Z_2)$,

$T2_{27,3, 1}(C_{27}(X_3))$ = $C_{27}(Y_3)$, $T2_{27,3, 2}(C_{27}(X_3))$ = $C_{27}(Z_3)$,

$T2_{27,3, 1}(C_{27}(Y_1))$ = $C_{27}(Z_1)$, $T2_{27,3, 2}(C_{27}(Y_1))$ = $C_{27}(X_1)$,

$T2_{27,3, 1}(C_{27}(Y_2))$ = $C_{27}(Z_2)$, $T2_{27,3, 2}(C_{27}(Y_2))$ = $C_{27}(X_2)$,

$T2_{27,3, 1}(C_{27}(Y_3))$ = $C_{27}(Z_3)$, $T2_{27,3, 2}(C_{27}(Y_3))$ = $C_{27}(X_3)$,

$T2_{27,3, 1}(C_{27}(Z_1))$ = $C_{27}(X_1)$, $T2_{27,3, 2}(C_{27}(Z_1))$ = $C_{27}(Y_1)$,

$T2_{27,3, 1}(C_{27}(Z_2))$ = $C_{27}(X_2)$, $T2_{27,3, 2}(C_{27}(Z_2))$ = $C_{27}(Y_2)$,

$T2_{27,3, 1}(C_{27}(Z_3))$ = $C_{27}(X_3)$, $T2_{27,3, 2}(C_{27}(Z_3))$ = $C_{27}(Y_3)$,

$T1_{27,2}(C_{27}(X_1))$ = $C_{27}(X_2)$, $T1_{27,4}(C_{27}(X_1))$ = $C_{27}(X_3)$,

$T1_{27, 2}(C_{27}(X_2))$ = $C_{27}(X_3)$,	$T1_{27, 4}(C_{27}(X_2))$ = $C_{27}(X_1)$,

$T1_{27, 2}(C_{27}(X_3))$ = $C_{27}(X_1)$,	$T1_{27, 4}(C_{27}(X_3))$ = $C_{27}(X_2)$,

$T1_{27, 2}(C_{27}(Y_1))$ = $C_{27}(Y_2)$,	$T1_{27, 4}(C_{27}(Y_1))$ = $C_{27}(Y_3)$,

$T1_{27, 2}(C_{27}(Y_2))$ = $C_{27}(Y_3)$,	$T1_{27, 4}(C_{27}(Y_2))$ = $C_{27}(Y_1)$,

$T1_{27, 2}(C_{27}(Y_3))$ = $C_{27}(Y_1)$,	$T1_{27, 4}(C_{27}(Y_3))$ = $C_{27}(Y_2)$,

$T1_{27, 2}(C_{27}(Z_1))$ = $C_{27}(Z_2)$, $T1_{27, 4}(C_{27}(Z_1))\}$ = $C_{27}(Z_3)$,

$T1_{27, 2}(C_{27}(Z_2))$ = $C_{27}(Z_3)$,	$T1_{27, 4}(C_{27}(Z_2))$ = $C_{27}(Z_1)$,

$T1_{27, 2}(C_{27}(Z_3))$ = $C_{27}(Z_1)$,	$T1_{27, 4}(C_{27}(Z_3))$ = $C_{27}(Z_2)$.
\\
$\therefore$ The isomorphic set of  $C_{27}(3,4,5,9,13)$ is

$Isoset(C_{27}(3,4,5,9,13))$ = $Isoset(C_{27}(X_1))$ 

\hspace{3.7cm} = $\{C_{27}(3,4,5,9,13), C_{27}(1,6,8,9,10), C_{27}(2,7,9,11,12)$,

\hfill $C_{27}(2,3,7,9,11),  C_{27}(1,3,8,9,10), C_{27}(4,5,6,9,13)$,

\hfill $C_{27}(2,6,7,9,11), C_{27}(1,8,9,10,12),  C_{27}(4,5,9,12,13) \}$ 

\hspace{3.7cm} = $\{C_{27}(X_i), C_{27}(Y_i), C_{27}(Z_i): i = 1,2,3\}$ 

\hfill = $Isoset(C_{27}(X_i))$ = $Isoset(C_{27}(Y_i))$ = $Isoset(C_{27}(Z_i))$, $i$ = 1,2,3.
\\
In this case, the isomorphic set of $C_{27}(3,4,5,9,13)$ = $Isoset(C_{27}(3,4,5,9,13))$ = $Isoset(C_{27}(X_1))$  has the following Eulerian isomorphism series. 

\begin{enumerate}
\item [\rm (1)] $C_{27}(X_1)$ $\cong_{T2_{27,3,1}}$ $C_{27}(Y_1)$ $\cong_{T1_{27,2}}$ $C_{27}(Y_2)$ $\cong_{T2_{27,3,1}}$ $C_{27}(Z_2)$ $\cong_{T1_{27,2}}$ $C_{27}(Z_3)$ 

\hfill $\cong_{T2_{27,3,1}}$ $C_{27}(X_3)$ $\cong_{T1_{27,4}}$ $C_{27}(X_2)$ $\cong_{T2_{27,3,2}}$ $C_{27}(Z_2)$ $\cong_{T1_{27,4}}$ $C_{27}(Z_1)$ $\cong_{T2_{27,3,2}}$ $C_{27}(Y_1)$ 

\hfill $\cong_{T1_{27,4}}$ $C_{27}(Y_3)$ $\cong_{T2_{27,3,2}}$ $C_{27}(X_3)$ $\cong_{T1_{27,2}}$ $C_{27}(X_1)$ $\cong_{T2_{27,3,2}}$ $C_{27}(Z_1)$ $\cong_{T1_{27,4}}$ $C_{27}(Z_3)$ 

\hfill $\cong_{T2_{27,3,2}}$ $C_{27}(Y_3)$ $\cong_{T1_{27,4}}$ $C_{27}(Y_2)$ $\cong_{T2_{27,3,2}}$ $C_{27}(X_2)$ $\cong_{T1_{27,2}}$ $C_{27}(X_3)$ 

\hfill $\cong_{T2_{27,3,1}}$ $C_{27}(Y_3)$ $\cong_{T1_{27,2}}$ $C_{27}(Y_1)$ $\cong_{T2_{27,3,1}}$ $C_{27}(Z_1)$ $\cong_{T1_{27,2}}$ $C_{27}(Z_2)$ 

\hfill $\cong_{T2_{27,3,2}}$ $C_{27}(Y_2)$ $\cong_{T1_{27,4}}$ $C_{27}(Y_1)$ $\cong_{T2_{27,3,2}}$ $C_{27}(X_1)$ $\cong_{T1_{27,2}}$ $C_{27}(X_2)$ 

\hfill $\cong_{T2_{27,3,1}}$ $C_{27}(Y_2)$ $\cong_{T1_{27,2}}$ $C_{27}(Y_3)$ $\cong_{T2_{27,3,1}}$ $C_{27}(Z_3)$
$\cong_{T1_{27,2}}$ $C_{27}(Z_1)$ 

\hfill  $\cong_{T2_{27,3,1}}$ $C_{27}(X_1)$ $\cong_{T1_{27,4}}$ $C_{27}(X_3)$ $\cong_{T2_{27,3,2}}$ $C_{27}(Z_3)$ 

\hfill $\cong_{T1_{27,4}}$ $C_{27}(Z_2)$ $\cong_{T2_{27,3,1}}$ $C_{27}(X_2)$ $\cong_{T1_{27,4}}$ $C_{27}(X_1)$.
\end{enumerate}

Figure 1 presents digraph of Eulerian isomorphism series of $C_{27}(R_i)$ by representing $R_i$ in the place of $C_{27}(R_i)$ in the figure where $R_i$ = $X_i, Y_i, Z_i$ and $1 \leq i \leq 3$. 

\medskip
{\rm 	\begin{center}
		\begin{tikzpicture}  
			[scale=.8,auto=center,every node/.style={draw,circle}]
			
			\node (1) at (-10,-0.5) {\tiny{$X_1$}};
			\node [scale=.0] (2) at (-9,-.93) [label=90: $>$]{};
			\node (3) at (-8,-0.5) {\tiny{$Y_1$}};	
			\node [scale=.0] (4) at (-7,-0.93) [label=91: $>$]{};
			\node (5) at (-6,-0.5)  {\tiny{$Y_2$}};
			
			\draw[ line width=0.2mm] [blue](1)[dashed]  to (3);
			\draw[ line width=0.2mm] [blue](3) to (5);
			
			\node [scale=.0] (6) at (-5,-.93) [label=90: $>$]{};
			\node (7) at (-4,-0.5) {\tiny{$Z_2$}};	
			\node [scale=.0] (8) at (-3,-0.93) [label=91: $>$]{};
			\node (9) at (-2,-0.5)  {\tiny{$Z_3$}};
			
			\draw[ line width=0.2mm] [blue](5)[dashed] to (7);
			\draw[ line width=0.2mm] [blue](7) to (9);
			
			\node [scale=.0] (10) at (-1,-.93) [label=90: $>$]{};
			\node (11) at (0,-0.5) {\tiny{$X_3$}};	
			\node [scale=.0] (12) at (1,-0.93) [label=91: $>$]{};
			\node (13) at (2,-0.5)  {\tiny{$X_2$}};
			
			\draw[ line width=0.2mm] [blue](9)[dashed]  to (11);
			\draw[ line width=0.2mm] [blue](11) to (13);
			
			\node [scale=.0] (14) at (3,-.93) [label=90: $>$]{};
			\node (15) at (4,-0.5) {\tiny{$Z_2$}};	
			\node [scale=.0] (16) at (5,-0.93) [label=91: $>$]{};
			\node (17) at (6,-0.5)  {\tiny{$Z_1$}};
			
			\draw[ line width=0.2mm] [blue](13)[dashed] to (15);
			\draw[ line width=0.2mm] [blue](15) to (17);
			
			\node [scale=.0] (18) at (6,-1.93) [label=91: $\downarrow$]{};
			\node (19) at (6,-2.5) {\tiny{$Y_1$}};	
			\node [scale=.0] (20) at (5,-2.93) [label=90: $<$]{};
			\node (21) at (4,-2.5)  {\tiny{$Y_3$}};
			
			\draw[ line width=0.2mm] [blue](17)[dashed] to (19);
			\draw[ line width=0.2mm] [blue](19) to (21);
			
			\node [scale=.0] (22) at (3,-2.93) [label=90: $<$]{};
			\node (23) at (2,-2.5) {\tiny{$X_3$}};	
			\node [scale=.0] (24) at (1,-2.93) [label=91: $<$]{};
			\node (25) at (0,-2.5)  {\tiny{$X_1$}};
			
			\draw[ line width=0.2mm] [blue](21)[dashed]  to (23);
			\draw[ line width=0.2mm] [blue](23) to (25);
			
			\node [scale=.0] (26) at (-1,-2.93) [label=90: $<$]{};
			\node (27) at (-2,-2.5) {\tiny{$Z_1$}};	
			\node [scale=.0] (28) at (-3,-2.93) [label=91: $<$]{};
			\node (29) at (-4,-2.5)  {\tiny{$Z_3$}};
			
			\draw[ line width=0.2mm] [blue](25)[dashed] to (27);
			\draw[ line width=0.2mm] [blue](27) to (29);
			
			\node [scale=.0] (30) at (-5,-2.93) [label=90: $<$]{};
			\node (31) at (-6,-2.5) {\tiny{$Y_3$}};	
			\node [scale=.0] (32) at (-7,-2.93) [label=91: $<$]{};
			\node (33) at (-8,-2.5)  {\tiny{$Y_2$}};
			
			\draw[ line width=0.2mm] [blue](29)[dashed]  to (31);
			\draw[ line width=0.2mm] [blue](31) to (33);
			
			\node [scale=.0] (34) at (-9,-2.93) [label=90: $<$]{};
			\node (35) at (-10,-2.5) {\tiny{$X_2$}};	
			\node [scale=.0] (36) at (-10,-3.93) [label=91: $\downarrow$]{};
			\node (37) at (-10,-4.5)  {\tiny{$X_3$}};
			
			\draw[ line width=0.2mm] [blue](33)[dashed] to (35);
			\draw[ line width=0.2mm] [blue](35) to (37);
			
			\node [scale=.0] (38) at (-9,-4.93) [label=90: $>$]{};
			\node (39) at (-8,-4.5) {\tiny{$Y_3$}};	
			\node [scale=.0] (40) at (-7,-4.93) [label=91: $>$]{};
			\node (41) at (-6,-4.5)  {\tiny{$Y_1$}};
			
			\draw[ line width=0.2mm] [blue](37)[dashed]  to (39);
			\draw[ line width=0.2mm] [blue](39) to (41);
			
			\node [scale=.0] (42) at (-5,-4.93) [label=90: $>$]{};
			\node (43) at (-4,-4.5) {\tiny{$Z_1$}};	
			\node [scale=.0] (44) at (-3,-4.93) [label=91: $>$]{};
			\node (45) at (-2,-4.5)  {\tiny{$Z_2$}};
			
			\draw[ line width=0.2mm] [blue](41)[dashed] to (43);
			\draw[ line width=0.2mm] [blue](43) to (45);
			
			\node [scale=.0] (46) at (-1,-4.93) [label=90: $>$]{};
			\node (47) at (0,-4.5) {\tiny{$Y_2$}};	
			\node [scale=.0] (48) at (1,-4.93) [label=91: $>$]{};
			\node (49) at (2,-4.5)  {\tiny{$Y_1$}};
			
			\draw[ line width=0.2mm] [blue](45)[dashed]  to (47);
			\draw[ line width=0.2mm] [blue](47) to (49);
			
			\node [scale=.0] (50) at (3,-4.93) [label=90: $>$]{};
			\node (51) at (4,-4.5) {\tiny{$X_1$}};	
			\node [scale=.0] (52) at (5,-4.93) [label=91: $>$]{};
			\node (53) at (6,-4.5)  {\tiny{$X_2$}};
			
			\draw[ line width=0.2mm] [blue](49)[dashed] to (51);
			\draw[ line width=0.2mm] [blue](51) to (53);
			
			\node [scale=.0] (54) at (6,-5.93) [label=91: $\downarrow$]{};
			\node (55) at (6,-6.5) {\tiny{$Y_2$}};	
			\node [scale=.0] (56) at (5,-6.93) [label=90: $<$]{};
			\node (57) at (4,-6.5)  {\tiny{$Y_3$}};
			
			\draw[ line width=0.2mm] [blue](53)[dashed] to (55);
			\draw[ line width=0.2mm] [blue](55) to (57);
			
			\node [scale=.0] (58) at (3,-6.93) [label=90: $<$]{};
			\node (59) at (2,-6.5) {\tiny{$Z_3$}};	
			\node [scale=.0] (60) at (1,-6.93) [label=91: $<$]{};
			\node (61) at (0,-6.5)  {\tiny{$Z_1$}};
			
			\draw[ line width=0.2mm] [blue](57)[dashed]  to (59);
			\draw[ line width=0.2mm] [blue](59) to (61);
			
			\node [scale=.0] (62) at (-1,-6.93) [label=90: $<$]{};
			\node (63) at (-2,-6.5) {\tiny{$X_1$}};	
			\node [scale=.0] (64) at (-3,-6.93) [label=91: $<$]{};
			\node (65) at (-4,-6.5)  {\tiny{$X_3$}};
			
			\draw[ line width=0.2mm] [blue](61)[dashed] to (63);
			\draw[ line width=0.2mm] [blue](63) to (65);
			
			\node [scale=.0] (66) at (-5,-6.93) [label=90: $<$]{};
			\node (67) at (-6,-6.5) {\tiny{$Z_3$}};	
			\node [scale=.0] (68) at (-7,-6.93) [label=91: $<$]{};
			\node (69) at (-8,-6.5)  {\tiny{$Z_2$}};
			
			\draw[ line width=0.2mm] [blue](65)[dashed]  to (67);
			\draw[ line width=0.2mm] [blue](67) to (69);
			
			\node [scale=.0] (70) at (-9,-6.93) [label=90: $<$]{};
			\node (71) at (-10,-6.5) {\tiny{$X_2$}};	
			\node [scale=.0] (72) at (-11,-4.1) [label=91: $\uparrow$]{};
			
			\draw[ line width=0.2mm] [blue](69)[dashed] to (71);
			\draw[-, line width=0.2mm] [blue](71) to [out=120,in=240] (1);
			
		\end{tikzpicture}
		
		\vspace{.2cm}		
		{\small  Fig. 1. Digraph of Eulerian isomorphism series of $C_{27}(R_i)$ with the representation of 
			
			$R_i$ in the place of $C_{27}(R_i)$ where $R_i$ = $X_i, Y_i, Z_i$ for $i$ = 1,2,3.    }
\end{center} }

\item [\rm (abf)] No two circulant graphs among the three given circulant graphs $C_{16}(1,2,6,7,8)$, $C_{16}(1,4,6,7)$ and $C_{27}(3,4,5,9,13)$ are isomorphic and so 

$Isoset(C_{16}(1,2,6,7,8), C_{16}(1,4,6,7), C_{27}(3,4,5,9,13))$ 

\hfill = $Isoset(C_{16}(1,2,6,7,8))$ $\cup$ $Isoset(C_{16}(1,4,6,7))$ $\cup$ $Isoset(C_{27}(3,4,5,9,13))$. 
\\
Using cases (a), (b) and (f), we get

$Isoset(C_{16}(1,2,6,7,8), C_{16}(1,4,6,7), C_{27}(3,4,5,9,13))$ 

\hfill = $\{C_{16}(1,2,6,7,8), C_{16}(2,3,5,6,8)\}$ $\cup$ $\{C_{16}(1,4,6,7), C_{16}(3,4,5,6), C_{16}(2,3,4,5)$, 

\hfill  $C_{16}(1,2,4,7)\}$ $\cup$ $\{C_{27}(3,4,5,9,13), C_{27}(1,6,8,9,10), C_{27}(2,7,9,11,12)$, 

\hfill $C_{27}(2,3,7,9,11),  C_{27}(1,3,8,9,10), C_{27}(4,5,6,9,13)$,

\hfill $C_{27}(2,6,7,9,11), C_{27}(1,8,9,10,12),  C_{27}(4,5,9,12,13) \}$.

Isomorphism series of graph $C_{16}(1,2,6,7,8) \cup C_{16}(1,4,6,7) \cup C_{27}(3,4,5,9,13)$ are independent isomorphism series of graphs $C_{16}(1,2,6,7,8)$, $C_{16}(1,4,6,7)$ and $C_{27}(3,4,5,9,13)$ and can be found as in cases (a), (b) and (f).  \hfill $\Box$
\end{enumerate}

\begin{prm} \quad \label{p2} {\rm Show that $C_{54}(1,3,17,19)$ and $C_{54}(5,13,21,23)$ are isomorphic but they are neither of Type-1 nor of Type-2 w.r.t. $m$ = 3.}
\end{prm}
\noindent
{\bf Solution.}\quad Here, 54 = $2\times 3^3$ and so $m$ = 3 is the only possible value of $m$ such that given two circulant graphs may be Type-2 isomorphic w.r.t. $m$ = 3. We obtain the solution by proving the following steps.

Step-1: $C_{54}(1,3,17,19)$ and $C_{54}(5,13,21,23)$ are not Type-1 isomorphic; 

Step-2: $C_{54}(1,3,17,19)$ and $C_{54}(5,13,21,23)$ are not Type-2 isomorphic w.r.t. $m$ = 3 and 

Step-3: $C_{54}(1,$ $3,17,19) \cong C_{54}(3,7,11,25)$ and $C_{54}(3,7,11,25) \cong C_{54}(5,13,21,23)$.
 
\begin{enumerate}
\item [\rm Step-1:]  To prove $C_{54}(1,3,17,19)$ and $C_{54}(5,13,21,23)$ are not Type-1 isomorphic. 
\\
Consider,
\\
$T1_{54}(C_{54}(1,3,17,19))$ = $\{\varphi_{54,x}(C_{54}(1,3,17,19)): x\in\varphi_{54}\}$ 
	
\hfill = $\{C_{54}(x(1, 3,17,19)): x = 1,5,7,$ $11,13,17,19,23,25,29,31,35,37,41,43,47,49,53\}$ 
	
\hspace{1cm}	= $\{C_{54}(1,3,17, 19),$ $C_{54}(5,13,15,23),$ $C_{54}(7,11,21,25)\}$ 

\hspace{1cm}	 = $\{C_{54}(x(1,3,17,19)) : x = 1,5,7\}$. 
\\
$\Rightarrow$ $C_{54}(5,13,21,23) \notin T1_{54}( C_{54}(1,3,$ $17,19))$ and thereby $C_{54}(1,3,17,19)$ and $C_{54}(5,13,21,23)$ are not Type-1 isomorphic. 
	
\item [\rm Step-2:]  To prove $C_{54}(1,3,17,19)$ and $C_{54}(5,13,21,23)$ are not Type-2 isomorphic w.r.t. $m$ = 3.
\\
$V_{54,3}(\{1,3,17,19,35,37,51,53\})$ = $\{\theta_{54,3,t}(\{1,3,17,19,35,37,51,53\}):$ $t = 0,1,...,\frac{54}{3}-1\}$ 
	
\hfill 	= $\{\{1,3,17,19,35,37,51,53\},$ $\{4,3,23,22,41,40,51,5\}$, $\{7,3, 29,25,47,43,51,11\}$, ~~

\hfill   $\{10,3,35,28,53,46,51,17\}$, $\{13,3,41, 31,5,49,51,23\}$, $\{16,3,47,34,11,52,51,29\}\}$ 
	
\hfill 	= $\{\theta_{54,3,t}(\{1,3,17,19, 35,$ $37,51,53\}):$ $t$ = $0,1,2,3,4,5\}$. See Table 1. 
\\
$\Rightarrow$ $\theta_{54,3,0}(C_{54}(1,3,17,19)$ = $C_{54}(1,3,17,19)$, $\theta_{54,3,2}(C_{54}(1,3,17,19))$ = $C_{54}(3,7,11,25)$ and 

$\theta_{54,3,4}(C_{54}(1,3,17,19))$ = $C_{54}(3,5,13,23)$ are the only graphs of the form $C_{54}(R)$ contained in $V_{54,3}(C_{54}(1,3,17,19))$. See Table 1. 
	
	$\Rightarrow$ $C_{54}(5,13,21,23)\notin V_{54,3}(C_{54}(1,3,17,19))$. 

$\Rightarrow$  $C_{54}(5,13,21,23)\notin T2_{54,3}(C_{54}(1,3,17,19))$ since $T2_{n,m}(C_{n}(R)) \subseteq V_{n,m}(C_{n}(R))$. 

$\Rightarrow$ $C_{54}(1,3,17,19)$ and $C_{54}(5,13,21,23)$ are not Type-2 isomorphic w.r.t. $m$ = 3. 
	
\begin{table} \label{10}
		\caption{{\footnotesize Calculation of $\theta_{54,3,t}(\{1,3,17,19,35,37,51,53\})$ under reflexive modulo 54, $0 \leq t \leq \frac{54}{3}-1$.}}
		\begin{center}
			\scalebox{0.75}{
\begin{tabular}{||c|c|c|c|c|c|c|c|c|c|c||} \hline \hline
				~ \hspace{.1cm} $t$ \hspace{.2cm} & \backslashbox{$\theta_{54,3,t}(x)$}{\\ Jump size $x$}
					& \hspace{.2cm} 1 \hspace{.2cm} & \hspace{.2cm} 3 \hspace{.2cm} & \hspace{.2cm} 17 \hspace{.2cm} & \hspace{.2cm} 19 \hspace{.2cm} & \hspace{.2cm} 35 \hspace{.2cm} & \hspace{.2cm} 37 \hspace{.2cm} & \hspace{.2cm} 51 \hspace{.2cm} & \hspace{.2cm} 53 & Symmetric Test \\\hline \hline
				& & &  &   &  &  & & &  & \\
	t & $\theta_{54,3,t}(x)$ & $1+3t$ & 3 & $17+6t$ & $19+3t$ & $35+6t$ & $37+3t$ & 51 & $53+6t$ & $T1$ or $T2$ or NS  \\\hline \hline
				& & &  &   &  &  & & & & \\
	0 & $\theta_{54,3,0}(x)$ & 1 & 3 & 17 & 19 & 35 & 37 & 51 & 53 & Identity  \\\hline
				& & &  &   &  &  & & & & \\
	1 & $\theta_{54,3,1}(x)$ & 4 & 3 & 23 & 22 & 41 & 40 & 51 & 5 & NS  \\\hline 
				& & &  &   &  &  & & & & \\
	2 & $\theta_{54,3,2}(x)$ & 7 & 3 & 29 & 25 & 47 & 43 & 51 & 11 & Yes (Type-2) \\\hline 
				& & &  &   &  &  & & & & \\
	3 & $\theta_{54,3,3}(x)$ & 10 & 3 & 35 & 28 & 53 & 46 & 51 & 17 & NS  \\\hline 
				& & &  &   &  &  & & & & \\
	4 & $\theta_{54,3,4}(x)$ & 13 & 3 & 41 & 31 & 5 & 49 & 51 & 23 & Yes (Type-2)  \\\hline 
				& & &  &   &  &  & & & & \\
	5 & $\theta_{54,3,5}(x)$ & 16 & 3 & 47 & 34 & 11 & 52 & 51 & 29 & NS \\\hline\hline 
				& & &  &   &  &  & & & & \\
	6 & $\theta_{54,3,6}(x)$ & 19 & 3 & 53 & 37 & 17 & 1 & 51 & 35 & Identity  \\\hline\hline
\end{tabular}}
\end{center}
	
\vspace{.1cm}
\footnotesize{T1: Type-1; T2: Type-2 isomorphic w.r.t. $r$ = 3; NS: Non-symmetric}
\end{table} 
	
\item [\rm Step-3:] To prove $C_{54}(1,3,17,19) \cong C_{54}(3,7,11,25)$ and $C_{54}(3,7,11,25) \cong C_{54}(5,13,21,23)$.
\\	
We have obtained in step-2,  

$\theta_{54,3,2}(C_{54}(1,3,17,19))$ = $C_{54}(3,7,11,25)$.  
\\
This implies, $C_{54}(1,3,17,19)$ $\cong$ $C_{54}(3,7,11,25)$. \hfill (i) 
	\\
	Also, we have $\varphi_{54,5}(C_{54}(5,13,21,23))$ = $\varphi_{54,5}(C_{54}(5,13,21,23,31,33,41,49))$ 

\hfill = $C_{54}(5(5,13,21,23,31,33,41,49))$ = $C_{54}(25,11,51,7,47,3,43,29)$ = $C_{54}(3,$ $7,11,25)$ and 

$\varphi_{54,7}(C_{54}(3,7,11,25))$ = $\varphi_{54,7}(C_{54}(3,7,11,25, 29,43,47,51))$ 

\hfill = $C_{54}(7(3,7,11,25, 29,43,47,51))$ = $C_{54}(21,49,23,13, 41,31,5,33)$ = $C_{54}(5,13,21,23)$.
\\
This implies, $C_{54}(3,7,11,25)$ and $C_{54}(5,13, 21,23)$ are (Type-1) isomorphic. \hfill (ii)
\\
(i) $\&$ (ii) $\Rightarrow$ $C_{54}(1,3,17,19)$ and $C_{54}(5,13,21,23)$ are isomorphic. But they are neither of Type-1 nor of Type-2 (w.r.t. $m$ = 3) by step-1 and step-2. 

Isomorphic circulant graphs $C_{54}(1,3,17,19)$, $C_{54}(3,7,11,25)$ and $C_{54}(5,13,21,23)$ are shown in Figures 2, 3, 4. $C_{54}(3,7,11,25)$ is shown in Figures 3 and 5 as $\theta_{54,3,2}(C_{54}(1,3, 17,19))$ = $C_{54}(3,7, 11,25)$ and  $\varphi_{54,5}(C_{54}(5,13, 21,23))$ = $C_{54}(3,7, 11,25)$, respectively. That is, circulant graph $C_{54}(3,7,11,25)$ obtained from graph $C_{54}(1,3,17,19)$ under the transformation $\theta_{54,3,2}$ and from graph $C_{54}(5,13,21,23)$ under the transformation $\varphi_{54,5}$ are shown in Figures 3 and 5, respectively. \hfill $\Box$
\end{enumerate}

\begin{figure}[ht]
	\centerline{\includegraphics[width=5in]{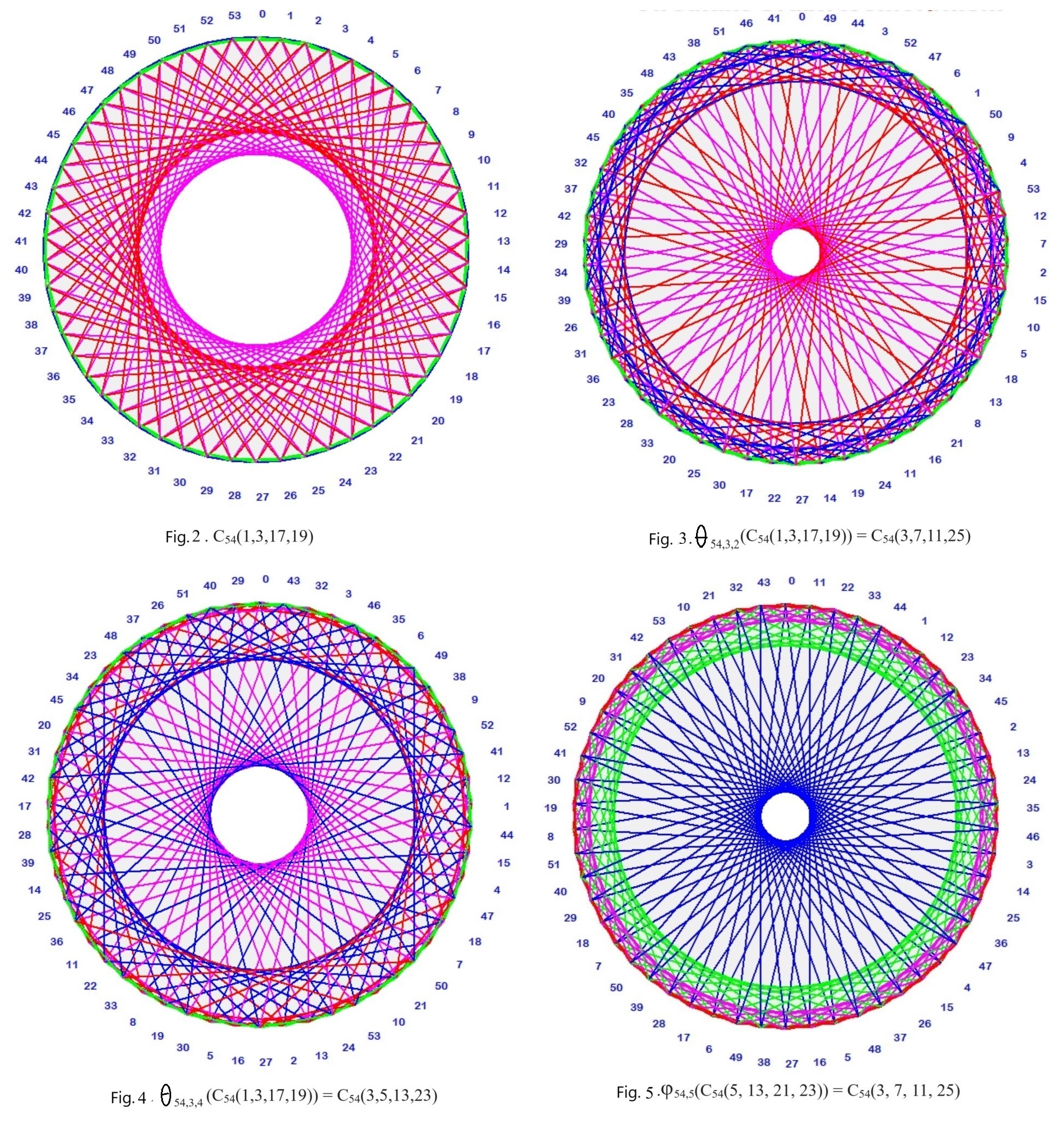}}
\end{figure}

\medskip
\noindent
{\bf Alternate solution.}\quad  In this, we show  $C_{54}(1,3,17,19) \cong C_{54}(5,13,21,23)$ by proving $C_{54}(5,13,21,23)$ $\cong$ $C_{54}(7,11,21,25)$ and $C_{54}(7,11,21,25) \cong C_{54}(1,3,17,19)$ as follows and all other steps are similar. 

$\theta_{54,3,4}(C_{54}(5,13, 21,23))$ = $\theta_{54,3,4}(C_{54}(5,13, 21,23, 31,33,41,49))$ = $C_{54}(29,25,21,47, 43,33, 11,7)$ = $C_{54}(7,11,21,25)$  which implies,  $C_{54}(5,13,21,23) \cong C_{54}(7,11,21,25)$. Also, $\varphi_{54,5}(C_{54}(7,11,21,25))$ = $\varphi_{54,5}(C_{54}(7,11,21,25, 29,33,43,47))$ = $C_{54}(5(7,11,21,25, 29,33,43,47))$ = $C_{54}(35,1,51,17, 37,3,53,19)$ = $C_{54}(1,3,17,19)$ which implies, $C_{54}(7,11,21,25)$ and $C_{54}(1,3,17,19)$ are Type-1 isomorphic. \hfill $\Box$

\begin{prm} \quad \label{p3} {\rm Show that $C_{54}(7,11,21,25)$ and $C_{54}(7,11,15,25)$ are isomorphic but they are neither Type-1 nor Type-2 w.r.t. 3.}
\end{prm}
\noindent
{\bf Solution.}\quad Here, 54 = $2\times 3^3$ and so $m$ = 3 is the only possible value of $m$ such that given two circulant graphs may be Type-2 isomorphic w.r.t. $m$ = 3. We obtain the solution to the problem by proving the following three steps.

Step 1: $C_{54}(7,11,21,25)$ and $C_{54}(7,11,15,25)$ are not Type-1 isomorphic; 

Step 2: $C_{54}(7,11,21,25)$ and $C_{54}(7,11,15,25)$ are not Type-2 isomorphic w.r.t. $m$ = 3; and 

Step 3: $C_{54}(7,11,21,25) \cong C_{54}(5,13,15,23)$ and $C_{54}(5,13,15,23) \cong C_{54}(7,11,15,25)$.
\begin{enumerate}
	\item [\rm Step] 1:  $T1_{54}(C_{54}(7,11,21,25))$ = $\{\varphi_{54,x}(C_{54}(7,11,21,25)): x\in\varphi_{54}\}$ 
	
	= $\{C_{54}(x(7,11,21,25)): x = 1,5,7,$ $11,13,17,19,23,25,29,31,35,37,41,43,47,49,53\}$ 
	
	= $\{C_{54}(7,11,21,25),$ $C_{54}(1,3,17,19),$ $C_{54}(5,$ $13,15,23)\}$ = $\{C_{54}($ $x(7,11,21,25)) : x = 1,5,7\}$. $\Rightarrow$ $C_{54}(7,11,15,25) \notin T1_{54}( C_{54}(7,11,$ $21,25))$ and thereby $C_{54}(7,11,21,25)$ and $C_{54}(7,11,15,25)$ are not Type-1 isomorphic. 
	
	\item [\rm Step] 2: $V_{54,3}(\{7,11,21,25, 29,33,43,47\})$ = $\{\theta_{54,3,t}(\{7,11,21,25, 29,33,43,47\}):$ $t = 0,1,...,\frac{54}{3}-1\}$ 
	
	\hfill = $\{\{7,11,21,25, 29,33,43,47\},$ $\{10,17,21,28, 35,33,46,53\}$, $\{13,23,21,31, 41,33,49,5\}$, 
	
	\hfill $\{16,29,21,34, 47,33,52,11\}$, $\{19,35,21,37, 53,33,1,17\}$, $\{22,41,21,40, 5,33,4,23\}\}$ 
	
	\hspace{.5cm} = $\{\theta_{54,3,t}(\{7,11,21,25, 29,33,43,47\}):$ $t$ = $0,1,2,3,4,5\}$. 
	
	\noindent
	$\Rightarrow$ $\theta_{54,3,0}(C_{54}(7,11,21,25)$ = $C_{54}(7,11,21,25)$, $\theta_{54,3,2}(C_{54}(7,11,21,25))$ = $C_{54}(5,13,21,23)$ and $\theta_{54,3,4}(C_{54}(7,11,21,25))$ = $C_{54}(1,17,19,21)$ are the only graphs of the form $C_{54}(R)$ contained in $V_{54,3}(C_{54}(7,11,21,25))$. 
	
	$\Rightarrow$ $C_{54}(7,11,15,25)\notin V_{54,3}(C_{54}(7,11,21,25))$. 

$\Rightarrow$  $C_{54}(7,11,15,25)\notin T2_{54,3}(C_{54}(7,11,21,25))$ since $T2_{n,m}(C_{n}(R)) \subseteq V_{n,m}(C_{n}(R))$. 

$\Rightarrow$  $C_{54}(7,11,15,25)$ and $C_{54}(7,11,21,25)$ are not Type-2 isomorphic w.r.t. $m$ = 3. 
	
\item [\rm Step] 3: To prove $C_{54}(7,$ $11,21,25) \cong C_{54}(5,13,15,23)$ and $C_{54}(5,13,15,23) \cong C_{54}(7,11,15,23)$.
\\	
Consider, $C_{54}(7(7,11,21,25))$ = $C_{54}(7(7,11,21,25, 29,33,43,47))$ 

\hfill = $C_{54}(49,23,39,13,41,15,31,5)$ = $C_{54}(5,13,15,23)$. 
\\
This implies, $C_{54}(7,11,21,25)$ and $C_{54}(5,13, 15,23)$ are (Type-1) isomorphic. \hfill (i)
	\\
	Also, we have $\theta_{54,3,4}(C_{54}(5,13,15,23))$ = $\theta_{54,3,4}(C_{54}(5,13,15,23, 31,39,41,49))$ 

\hfill = $C_{54}(29,25,$ $15,47, 43,39,11,7)$ = $C_{54}(7,11,15,25, 29,39,43,47) = C_{54}(7,11,15,25)$. 

$\Rightarrow$  $C_{54}(5,13,15,23)$ $\cong$ $C_{54}(7,11,15,25)$. \hfill (ii)
	
 (i) and (ii)	$\Rightarrow$  $C_{54}(7,11,15,25)$ and $C_{54}(7,11,21,25)$ are isomorphic but they are neither of Type-1 nor of Type-2 (w.r.t. 3).  \hfill $\Box$
\end{enumerate}

Corresponding to the above problems, we get the following result.

\begin{theorem}\quad {\rm \label{t24} Let $C_n(R)$ and $C_n(S)$ be Type-2 isomorphic w.r.t. $m$ and $C_n(xR)$ = $C_n(T)$ where $T \neq R$, $m > 1$ divides $\gcd(n, r)$, $m^3$ divides $n$, $r\in R,S$ and $x\in \varphi_n$. Then, $C_n(S)$ and $C_n(T)$ are isomorphic and $S$ $\neq$ $T$ but the two circulant graphs are neither of Type-1 nor of Type-2 w.r.t. $m$. }
\end{theorem}
\begin{proof} \quad Given, $C_n(R)$ and $C_n(S)$ are Type-2 isomorphic w.r.t. $m$ and $C_n(xR)$ = $C_n(T)$ where $T \neq R$, $m > 1$ divides $\gcd(n, r)$, $m^3$ divides $n$, $r\in R,S$ and $x\in \varphi_n$. This implies, $C_n(R)$ $\cong$ $C_n(S)$ $\cong$ $C_n(T)$, $C_n(R),C_n(S)\in T2_{n,m}(C_n(R))$ = $T2_{n,m}(C_n(S))$ and $C_n(R),C_n(T)\in T1_n(C_n(R))$ = $T1_n(C_n(T))$, $R$ $\neq$ $S$ and $R$ $\neq$ $T$ using Theorems \ref{t2.3}, \ref{t2.10} and \ref{t2.12} where $(T1_n(C_n(R)), \circ')$ and $(T2_{n,m}(C_n(R)), \circ)$ are Abelian groups. Suppose, $S$ = $T$. Then, this implies, $C_n(R)$ and $C_n(T)$ are Type-2 isomorphic w.r.t. $m$. This is a contradiction to given condition that $C_n(R)$ and $C_n(T)$ are Type-1 isomorphic and $R$ $\neq$ $T$. This implies, $S$ $\neq$ $T$.
	
Suppose, $C_n(S)$ and $C_n(T)$ be Type-1 isomorphic, $S \neq T$. This implies, $C_n(S),C_n(T)\in T1_n(C_n(S))$ = $T1_n(C_n(T))$, $S \neq T$. Already, we have $C_n(R),C_n(T)\in T1_n(C_n(R))$ = $T1_n(C_n(T))$, $R \neq T$. This implies, $C_n(R),C_n(S),C_n(T)\in T1_n(C_n(R))$ = $T1_n(C_n(S))$ = $T1_n(C_n(T))$, $S \neq T$, $R \neq T$ and $R \neq S$. This implies, $C_n(R)$ and $C_n(S)$ are Type-1 isomorphic using Theorem \ref{t2.3}, $R \neq S$. This is a contradiction to the given condition that $C_n(R)$ and $C_n(S)$ are Type-2 isomorphic w.r.t. $m$ where $m > 1$ divides $\gcd(n, r)$, $m^3$ divides $n$ and $r\in R,S$. This implies, $C_n(S)$ and $C_n(T)$ are isomorphic but they are not of Type-1, $S \neq T$. 
	
On the otherhand, suppose $C_n(S)$ and $C_n(T)$ be Type-2 isomorphic w.r.t. $m$ where $m > 1$ divides $\gcd(n, r)$, $m^3$ divides $n$ and $r\in S,T$. This implies, $C_n(S),C_n(T)\in T2_{n,m}(C_n(S))$ = $T2_{n,m}(C_n(T))$ using Theorem \ref{t2.12}, $S \neq T$. By the given condition $C_n(R)$ and $C_n(S)$ are Type-2 isomorphic w.r.t. $m$ where $m > 1$ divides $\gcd(n, r)$, $m^3$ divides $n$ and $r\in R,S$. This implies, $C_n(R),C_n(S)\in T2_{n,m}(C_n(R))$ = $T2_{n,m}(C_n(S))$. Combining the above two statements, we get, $C_n(R),C_n(S),C_n(T)\in T2_{n,m}(C_n(R))$ = $T2_{n,m}(C_n(S))$ = $T2_{n,m}(C_n(T))$, $R \neq S$, $S \neq T$ and $R \neq T$. This implies that $C_n(R)$ and $C_n(T)$ are Type-2 isomorphic w.r.t. $m$ using Theorem \ref{t2.12}. This is a contradiction to the given condition that $C_n(R)$ and $C_n(T)$ are Type-1 isomorphic. This implies that our assumption is wrong. This implies, $C_n(S)$ and $C_n(T)$ are isomorphic but they are not Type-2 isomorphic w.r.t. $m$. 
	
	Hence we get the result.  
\end{proof}

In the next problem, we find a sequence of isomorphic circulant graphs of given two isomorphic circulant graphs.

\begin{prm} \quad \label{p4} {\rm Show that $C_{54}(1,3,17,19)$ and $C_{54}(7,11,15,25)$ are isomorphic but they are neither of Type-1 nor of Type-2 w.r.t. $m$ = 3. Also, find the isomorphic set and some isomorphism series of the two (isomorphic) circulant graphs.}
\end{prm}
\noindent
{\bf Solution.}\quad In problem \ref{p2}, we proved the next 2 items. 
\begin{enumerate}
	\item [\rm (i)] $C_{54}(1,3,17,19)$ and $C_{54}(5,13,21,23)$ are isomorphic but they are neither Type-1 nor Type-2 w.r.t. $m$ = 3. 
	
	\item [\rm (ii)] $C_{54}(1,3,17,19) \cong_{T2_{54,3,2}} C_{54}(3,7,11,25)  \cong_{T1_{54,7}} C_{54}(5,13,21,23)$. 
		
	\item [\rm (iii)] In problem \ref{p3}, we proved that 
	$C_{54}(7,11,21,25)$ and $C_{54}(7,11,15,25)$ are isomorphic but they are neither Type-1 nor Type-2 isomorphic w.r.t. $m$ = 3. 
	
	\item [\rm (iv)] $C_{54}(7,11,21,25) \cong_{T1_{54,7}} C_{54}(5,13,15,23) \cong_{T2_{54,3,4}} C_{54}(7,11,15,25)$ proved in problem \ref{p3}.
	
	\item [\rm (v)] Here, we prove that $C_{54}(5,13,21,23)$ $\cong_{T2_{54,3,4}} C_{54}(7,11,21,25)$. 
	\\
	In problem \ref{p3.8}, we have
	
	$T1_{54}(C_{54}(7,11,21,25))$ = $\{C_{54}(7,11,21,25), C_{54}(1,3,17,19), C_{54}(5,13,15,23)\}$. 
	
	$\Rightarrow$ $C_{54}(5,13,21,23)\notin T1_{54}(C_{54}(7,11,21,25))$.
	
	$\Rightarrow$ $C_{54}(5,13,21,23)$ and $C_{54}(7,11,21,25)$ are not Type-1 isomorphic. Also,
	
	$\theta_{54,3,4}(C_{54}(5,13,21,23))$ = $\theta_{54,3,4}(C_{54}(5,13,21,23, 31,33,41,49))$ 
	
	\hfill = $C_{54}(29,25,21,47, 43,33,11,7))$ = $C_{54}(7,11,21,25))$.
	
	$\Rightarrow$ $C_{54}(5,13,21,23)$ $\cong_{T2_{54,3,4}} C_{54}(7,11,21,25)$. 
\end{enumerate}
	
Combining $(ii)$, $(iv)$ $\&$ $(v)$, we get,

(1) ~$C_{54}(1,3,17,19) \cong_{T2_{54,3,2}} C_{54}(3,7,11,25)  \cong_{T1_{54,7}} C_{54}(5,13,21,23) $ 
	
	\hfill $\cong_{T2_{54,3,4}} C_{54}(7,11,21,25) \cong_{T1_{54,7}} C_{54}(5,13,15,23) \cong_{T2_{54,3,4}} C_{54}(7,11,15,25)$. 

$\Rightarrow$ $C_{54}(1,3,17,19)$ $\cong$ $C_{54}(7,11,21,25)$.

Also, $C_{54}(1,3,17,19)$ and $C_{54}(7,11,15,25)$ are not Type-2 isomorphic w.r.t. $m$ = 3 since there is no common jump size in the two circulant graphs (only for more clarity, we followed another method to establish this step in problems \ref{p3.7} and \ref{p3.8}). And $C_{54}(7,11,15,25)\notin T1_{54}(C_{54}(1,3,17,19))$ = $\{C_{54}(1,3,17,19), C_{54}(5,13,15,23), C_{54}(7,11,21,25)\}$, see problem \ref{p3.7}. This implies, $C_{54}(1,3,17,19)$ and $C_{54}(7,11,15,25)$ are not Type-1 isomorphic.  Thus, $C_{54}(1,3,17,19)$ and $C_{54}(7,11,15,25)$ are isomorphic but they are neither of Type-1 nor of Type-2 w.r.t. $m$ = 3.

Now, we find the isomorphic set and some isomorphism series of the two isomorphic circulant graphs.

\medskip
\noindent
{\bf Isoset of $C_{54}(1,3,17,19), C_{54}(7,11,21,25)$:}

\medskip
Isoset of $C_{54}(1,3,17,19), C_{54}(7,11,21,25)$ = $Isoset(C_{54}(1,3,17,19), C_{54}(7,11,21,25))$ 

\hfill = $Isoset(C_{54}(1,3,17,19)$ $\cup$ $C_{54}(7,11,21,25))$ = $Isoset(C_{54}(1,3,17,19))$  

\hfill  = $Isoset(C_{54}(7,11,21,25))$ since $C_{54}(1,3,17,19)$ $\cong$ $C_{54}(7,11,21,25)$. ~~~

$T2_{54,3}(C_{54}(1,3,17,19))$ = $\{C_{54}(1,3,17,19), C_{54}(3,7,11,25) = T2_{54,3,2}(C_{54}(1,3,17,19))$, 

\hfill $C_{54}(3,5,13,23) = T2_{54,3,4}(C_{54}(1,3,17,19))\}$,

$T1_{54}(C_{54}(1,3,17,19))$ = $\{C_{54}(1,3,17,19), C_{54}(5,13,15,23) = T1_{54,5}(C_{54}(1,3,17,19)),$ 

\hfill $C_{54}(7,11,21,25) = T1_{54,7}(C_{54}(1,3,17,19))\}$,

$T2_{54,3}(C_{54}(5,13,15,23))$ = $\{C_{54}(5,13,15,23), C_{54}(1,15,17,19) = T2_{54,3,2}(C_{54}(5,13,15,23))$, 

\hfill $C_{54}(7,11,15,25) = T2_{54,3,4}(C_{54}(5,13,15,23))\}$,

$T1_{54}(C_{54}(3,7,11,25))$ = $\{C_{54}(3,7,11,25), C_{54}(1,15,17,19) = T1_{54,5}(C_{54}(3,7,11,25)),$ 

\hfill $C_{54}(5,13,21,23) = T1_{54,7}(C_{54}(3,7,11,25))\}$,

$T2_{54,3}(C_{54}(7,11,21,25))$ = $\{C_{54}(7,11,21,25), C_{54}(5,13,21,23) = T2_{54,3,2}(C_{54}(7,11,21,25))$, 

\hfill $C_{54}(1,17,19,21) = T2_{54,3,4}(C_{54}(7,11,21,25))\}$,

$T1_{54}(C_{54}(3,5,13,23))$ = $\{C_{54}(3,5,13,23), C_{54}(7,11,15,25) = T1_{54,5}(C_{54}(3,5,13,23)),$ 

\hfill $C_{54}(1,17,19,21) = T1_{54,7}(C_{54}(3,5,13,23))\}$. Let

 $X_1$ = $C_{54}(1,3,17,19)$, $X_2$ = $C_{54}(5,13,15,23)$, $X_3$ = $C_{54}(7,11,21,25)$, 

$Y_1$ = $C_{54}(3,7,11,25)$, $Y_2$ = $C_{54}(1,15,17,19)$, $Y_3$ = $C_{54}(5,13,21,23)$, 

$Z_1$ = $C_{54}(3,5,13,23)$, $Z_2$ = $C_{54}(7,11,15,25)$, $Z_3$ = $C_{54}(1,17,19,21)$.  

Then, the isoset of $C_{54}(1,3,17,19), C_{54}(7,11,21,25)$ = $Isoset(C_{54}(1,3,17,19), C_{54}(7,11,21,25))$ 

\hfill = $\{C_{54}(1,3,17,19), C_{54}(5,13,15,23), C_{54}(7,11,21,25), C_{54}(3,7,11,25), C_{54}(3,5,13,23),$ 

\hfill $C_{54}(1,15,17,19), C_{54}(7,11,15,25), C_{54}(5,13,21,23), C_{54}(1,17,19,21)\}$ 

\hfill = $\{C_{54}(X_i), C_{54}(Y_i), C_{54}(Z_i): i = 1,2,3\}$ = $Isoset(C_{54}(X_i))$ = $Isoset(C_{54}(Y_j))$ 

\hfill = $Isoset(C_{54}(Z_k))$, $1 \leq i,j,k \leq 3$.

\medskip
\noindent
{\bf Some isomorphism series of $C_{54}(1,3,17,19)$ and $C_{54}(7,11,21,25)$:}

The followings are some isomorphism series, that includes $(1)$, of the two isomorphic circulant graphs $C_{54}(1,3,17,19)$ and $C_{54}(7,11,15,25)$. 
\begin{enumerate}
\item [\rm (1)] $C_{54}(1,3,17,19) \cong_{T2_{54,3,2}} C_{54}(3,7,11,25)  \cong_{T1_{54,7}} C_{54}(5,13,21,23) $ 
	
	\hfill $\cong_{T2_{54,3,4}} C_{54}(7,11,21,25) \cong_{T1_{54,7}} C_{54}(5,13,15,23) \cong_{T2_{54,3,4}} C_{54}(7,11,15,25)$. 

\item [\rm $(2)$] $C_{54}(1,3,17,19) \cong_{T2_{54,5,2}} C_{54}(3,7,11,25)  \cong_{T1_{54,5}} C_{54}(1,15,17,19)$  $\cong_{T2_{54,3,2}} C_{54}(7,11,15,25)$. 

\item [\rm $(3)$] $C_{54}(1,3,17,19) \cong_{T2_{54,5,4}} C_{54}(3,5,13,23)  \cong_{T1_{54,5}} C_{54}(7,11,15,25)$. 
	
\item [\rm $(4)$] $C_{54}(1,3,17,19) \cong_{T2_{54,5,4}} C_{54}(3,5,13,23) \cong_{T1_{54,7}} C_{54}(1,17,19,21)$  

\hfill $\cong_{T2_{54,5,2}} C_{54}(7,11,21,25)  \cong_{T1_{54,7}} C_{54}(5,13,15,23)$ $\cong_{T2_{54,5,4}} C_{54}(7,11,15,25)$. 
	
\item [\rm $(5)$] $C_{54}(1,3,17,19) \cong_{T1_{54,5}} C_{54}(5,13,15,23)  \cong_{T2_{54,5,4}} C_{54}(7,11,15,25)$.

\item [\rm $(6)$] $C_{54}(1,3,17,19) \cong_{T1_{54,7}} C_{54}(7,11,21,25)  \cong_{T2_{54,3,2}} C_{54}(5,13,21,23)$	

\hfill $\cong_{T2_{54,7}} C_{54}(1,15,17,19) \cong_{T2_{54,3,2}} C_{54}(7,11,15,25)$.
    
\item [\rm $(7)$]  $C_{54}(7,11,15,25) \cong_{T2_{54,3,4}} C_{54}(1,15,17,19)$ 

\hfill $\cong_{T1_{54,7}} C_{54}(3,7,11,25) \cong_{T2_{54,5,2}} C_{54}(3,5,13,23)  \cong_{T1_{54,5}} C_{54}(7,11,15,25)$.  
	
\item [\rm $(8)$] $C_{54}(7,11,15,25) \cong_{T2_{54,5,2}} C_{54}(5,13,15,23)$ 

\hfill $\cong_{T1_{54,5}} C_{54}(7,11,21,25) \cong_{T2_{54,3,2}} C_{54}(5,13,21,23) \cong_{T2_{54,7}} C_{54}(1,15,17,19) $

\hfill $\cong_{T2_{54,5,4}} C_{54}(5,13,15,23) \cong_{T1_{54,7}} C_{54}(1,3,17,19)$.
 	 	
\item [\rm $(9)$]  $C_{54}(7,11,15,25) \cong_{T2_{54,3,4}} C_{54}(1,15,17,19) \cong_{T1_{54,7}} C_{54}(3,7,11,25) \cong_{T2_{54,5,4}} C_{54}(1,3,17,19)$.  
	
\item [\rm $(10)$] $C_{54}(7,11,15,25) \cong_{T1_{54,5}} C_{54}(1,17,19,21) \cong_{T2_{54,5,2}} C_{54}(7,11,21,25)$

\hfill $\cong_{T2_{54,3,2}} C_{54}(5,13,21,23) \cong_{T2_{54,7}} C_{54}(1,15,17,19) $

\hfill $\cong_{T2_{54,5,4}} C_{54}(5,13,15,23) \cong_{T1_{54,7}} C_{54}(1,3,17,19)$.
 	 		
\item [\rm $(11)$] $C_{54}(7,11,15,25) \cong_{T1_{54,7}} C_{54}(3,5,13,23) \cong_{T2_{54,5,2}} C_{54}(1,3,17,19)$.
	
\item [\rm $(12)$] $C_{54}(7,11,15,25) \cong_{T2_{54,5,2}} C_{54}(5,13,15,23) \cong_{T1_{54,7}} C_{54}(1,3,17,19)$.
 		
\item [\rm $(13)$] $C_{54}(7,11,15,25) \cong_{T2_{54,5,4}} C_{54}(1,15,17,19) \cong_{T1_{54,5}} C_{54}(5,13,21,23)$

\hfill $\cong_{T2_{54,5,2}} C_{54}(1,17,19,21)  \cong_{T1_{54,5}} C_{54}(3,5,13,23) \cong_{T2_{54,5,2}} C_{54}(1,3,17,19)$.

\item [\rm $(14)$] $C_{54}(7,11,15,25) \cong_{T2_{54,3,2}} C_{54}(5,13,15,23) \cong_{T1_{54,5}} C_{54}(7,11,21,25)$
	
	\hfill $ \cong_{T2_{54,3,2}} C_{54}(5,13,21,23) \cong_{T1_{54,5}} C_{54}(3,7,11,25) \cong_{T2_{54,3,4}} C_{54}(1,3,17,19)$.  \\
\\
Isoset$(C_{54}(1,3,17,19), C_{54}(7,11,21,25))$ has Eulerian isomorphism series as given in (15). \\

\item [\rm (15)] $C_{54}(X_1)$ $\cong_{T2_{54,3,2}}$ $C_{54}(Y_1)$ $\cong_{T1_{54,5}}$ $C_{54}(Y_2)$ $\cong_{T2_{54,3,2}}$ $C_{54}(Z_2)$ $\cong_{T1_{54,5}}$ $C_{54}(Z_3)$ 

\hfill $\cong_{T2_{54,3,2}}$ $C_{54}(X_3)$ $\cong_{T1_{54,7}}$ $C_{54}(X_2)$ $\cong_{T2_{54,3,4}}$ $C_{54}(Z_2)$ $\cong_{T1_{54,7}}$ $C_{54}(Z_1)$ $\cong_{T2_{54,3,4}}$ $C_{54}(Y_1)$ 

\hfill $\cong_{T1_{54,7}}$ $C_{54}(Y_3)$ $\cong_{T2_{54,3,4}}$ $C_{54}(X_3)$ $\cong_{T1_{54,5}}$ $C_{54}(X_1)$ $\cong_{T2_{54,3,4}}$ $C_{54}(Z_1)$ $\cong_{T1_{54,7}}$ $C_{54}(Z_3)$ 

\hfill $\cong_{T2_{54,3,4}}$ $C_{54}(Y_3)$ $\cong_{T1_{54,7}}$ $C_{54}(Y_2)$ $\cong_{T2_{54,3,4}}$ $C_{54}(X_2)$ $\cong_{T1_{54,5}}$ $C_{54}(X_3)$ 

\hfill $\cong_{T2_{54,3,2}}$ $C_{54}(Y_3)$ $\cong_{T1_{54,5}}$ $C_{54}(Y_1)$ $\cong_{T2_{54,3,2}}$ $C_{54}(Z_1)$ $\cong_{T1_{54,5}}$ $C_{54}(Z_2)$ 

\hfill $\cong_{T2_{54,3,4}}$ $C_{54}(Y_2)$ $\cong_{T1_{54,7}}$ $C_{54}(Y_1)$ $\cong_{T2_{54,3,4}}$ $C_{54}(X_1)$ $\cong_{T1_{54,5}}$ $C_{54}(X_2)$ 

\hfill $\cong_{T2_{54,3,2}}$ $C_{54}(Y_2)$ $\cong_{T1_{54,5}}$ $C_{54}(Y_3)$ $\cong_{T2_{54,3,2}}$ $C_{54}(Z_3)$
$\cong_{T1_{54,5}}$ $C_{54}(Z_1)$ 

\hfill  $\cong_{T2_{54,3,2}}$ $C_{54}(X_1)$ $\cong_{T1_{54,7}}$ $C_{54}(X_3)$ $\cong_{T2_{54,3,4}}$ $C_{54}(Z_3)$ 

\hfill $\cong_{T1_{54,7}}$ $C_{54}(Z_2)$ $\cong_{T2_{54,3,2}}$ $C_{54}(X_2)$ $\cong_{T1_{54,7}}$ $C_{54}(X_1)$.
\end{enumerate}

Figure 6 presents digraph of Eulerian isomorphism series of $C_{54}(R_i)$ by representing $R_i$ in the place of $C_{54}(R_i)$ in the figure where $R_i$ = $X_i, Y_i, Z_i$ and $1 \leq i \leq 3$.  \hfill $\Box$	

\medskip
{\rm 	\begin{center}
		\begin{tikzpicture}  
		[scale=.8,auto=center,every node/.style={draw,circle}]
			
\node (1) at (-10,-0.5) {\tiny{$X_1$}};
\node [scale=.0] (2) at (-9,-.93) [label=90: $>$]{};
\node (3) at (-8,-0.5) {\tiny{$Y_1$}};	
\node [scale=.0] (4) at (-7,-0.93) [label=91: $>$]{};
\node (5) at (-6,-0.5)  {\tiny{$Y_2$}};

\draw[ line width=0.2mm] [blue](1)[dashed]  to (3);
\draw[ line width=0.2mm] [blue](3) to (5);

\node [scale=.0] (6) at (-5,-.93) [label=90: $>$]{};
\node (7) at (-4,-0.5) {\tiny{$Z_2$}};	
\node [scale=.0] (8) at (-3,-0.93) [label=91: $>$]{};
\node (9) at (-2,-0.5)  {\tiny{$Z_3$}};

\draw[ line width=0.2mm] [blue](5)[dashed] to (7);
\draw[ line width=0.2mm] [blue](7) to (9);

\node [scale=.0] (10) at (-1,-.93) [label=90: $>$]{};
\node (11) at (0,-0.5) {\tiny{$X_3$}};	
\node [scale=.0] (12) at (1,-0.93) [label=91: $>$]{};
\node (13) at (2,-0.5)  {\tiny{$X_2$}};

\draw[ line width=0.2mm] [blue](9)[dashed]  to (11);
\draw[ line width=0.2mm] [blue](11) to (13);

\node [scale=.0] (14) at (3,-.93) [label=90: $>$]{};
\node (15) at (4,-0.5) {\tiny{$Z_2$}};	
\node [scale=.0] (16) at (5,-0.93) [label=91: $>$]{};
\node (17) at (6,-0.5)  {\tiny{$Z_1$}};

\draw[ line width=0.2mm] [blue](13)[dashed] to (15);
\draw[ line width=0.2mm] [blue](15) to (17);

\node [scale=.0] (18) at (6,-1.93) [label=91: $\downarrow$]{};
\node (19) at (6,-2.5) {\tiny{$Y_1$}};	
\node [scale=.0] (20) at (5,-2.93) [label=90: $<$]{};
\node (21) at (4,-2.5)  {\tiny{$Y_3$}};

\draw[ line width=0.2mm] [blue](17)[dashed] to (19);
\draw[ line width=0.2mm] [blue](19) to (21);

\node [scale=.0] (22) at (3,-2.93) [label=90: $<$]{};
\node (23) at (2,-2.5) {\tiny{$X_3$}};	
\node [scale=.0] (24) at (1,-2.93) [label=91: $<$]{};
\node (25) at (0,-2.5)  {\tiny{$X_1$}};

\draw[ line width=0.2mm] [blue](21)[dashed]  to (23);
\draw[ line width=0.2mm] [blue](23) to (25);

\node [scale=.0] (26) at (-1,-2.93) [label=90: $<$]{};
\node (27) at (-2,-2.5) {\tiny{$Z_1$}};	
\node [scale=.0] (28) at (-3,-2.93) [label=91: $<$]{};
\node (29) at (-4,-2.5)  {\tiny{$Z_3$}};

\draw[ line width=0.2mm] [blue](25)[dashed] to (27);
\draw[ line width=0.2mm] [blue](27) to (29);

\node [scale=.0] (30) at (-5,-2.93) [label=90: $<$]{};
\node (31) at (-6,-2.5) {\tiny{$Y_3$}};	
\node [scale=.0] (32) at (-7,-2.93) [label=91: $<$]{};
\node (33) at (-8,-2.5)  {\tiny{$Y_2$}};

\draw[ line width=0.2mm] [blue](29)[dashed]  to (31);
\draw[ line width=0.2mm] [blue](31) to (33);

\node [scale=.0] (34) at (-9,-2.93) [label=90: $<$]{};
\node (35) at (-10,-2.5) {\tiny{$X_2$}};	
\node [scale=.0] (36) at (-10,-3.93) [label=91: $\downarrow$]{};
\node (37) at (-10,-4.5)  {\tiny{$X_3$}};

\draw[ line width=0.2mm] [blue](33)[dashed] to (35);
\draw[ line width=0.2mm] [blue](35) to (37);

\node [scale=.0] (38) at (-9,-4.93) [label=90: $>$]{};
\node (39) at (-8,-4.5) {\tiny{$Y_3$}};	
\node [scale=.0] (40) at (-7,-4.93) [label=91: $>$]{};
\node (41) at (-6,-4.5)  {\tiny{$Y_1$}};

\draw[ line width=0.2mm] [blue](37)[dashed]  to (39);
\draw[ line width=0.2mm] [blue](39) to (41);

\node [scale=.0] (42) at (-5,-4.93) [label=90: $>$]{};
\node (43) at (-4,-4.5) {\tiny{$Z_1$}};	
\node [scale=.0] (44) at (-3,-4.93) [label=91: $>$]{};
\node (45) at (-2,-4.5)  {\tiny{$Z_2$}};

\draw[ line width=0.2mm] [blue](41)[dashed] to (43);
\draw[ line width=0.2mm] [blue](43) to (45);

\node [scale=.0] (46) at (-1,-4.93) [label=90: $>$]{};
\node (47) at (0,-4.5) {\tiny{$Y_2$}};	
\node [scale=.0] (48) at (1,-4.93) [label=91: $>$]{};
\node (49) at (2,-4.5)  {\tiny{$Y_1$}};

\draw[ line width=0.2mm] [blue](45)[dashed]  to (47);
\draw[ line width=0.2mm] [blue](47) to (49);

\node [scale=.0] (50) at (3,-4.93) [label=90: $>$]{};
\node (51) at (4,-4.5) {\tiny{$X_1$}};	
\node [scale=.0] (52) at (5,-4.93) [label=91: $>$]{};
\node (53) at (6,-4.5)  {\tiny{$X_2$}};

\draw[ line width=0.2mm] [blue](49)[dashed] to (51);
\draw[ line width=0.2mm] [blue](51) to (53);

\node [scale=.0] (54) at (6,-5.93) [label=91: $\downarrow$]{};
\node (55) at (6,-6.5) {\tiny{$Y_2$}};	
\node [scale=.0] (56) at (5,-6.93) [label=90: $<$]{};
\node (57) at (4,-6.5)  {\tiny{$Y_3$}};

\draw[ line width=0.2mm] [blue](53)[dashed] to (55);
\draw[ line width=0.2mm] [blue](55) to (57);

\node [scale=.0] (58) at (3,-6.93) [label=90: $<$]{};
\node (59) at (2,-6.5) {\tiny{$Z_3$}};	
\node [scale=.0] (60) at (1,-6.93) [label=91: $<$]{};
\node (61) at (0,-6.5)  {\tiny{$Z_1$}};

\draw[ line width=0.2mm] [blue](57)[dashed]  to (59);
\draw[ line width=0.2mm] [blue](59) to (61);

\node [scale=.0] (62) at (-1,-6.93) [label=90: $<$]{};
\node (63) at (-2,-6.5) {\tiny{$X_1$}};	
\node [scale=.0] (64) at (-3,-6.93) [label=91: $<$]{};
\node (65) at (-4,-6.5)  {\tiny{$X_3$}};

\draw[ line width=0.2mm] [blue](61)[dashed] to (63);
\draw[ line width=0.2mm] [blue](63) to (65);

\node [scale=.0] (66) at (-5,-6.93) [label=90: $<$]{};
\node (67) at (-6,-6.5) {\tiny{$Z_3$}};	
\node [scale=.0] (68) at (-7,-6.93) [label=91: $<$]{};
\node (69) at (-8,-6.5)  {\tiny{$Z_2$}};

\draw[ line width=0.2mm] [blue](65)[dashed]  to (67);
\draw[ line width=0.2mm] [blue](67) to (69);

\node [scale=.0] (70) at (-9,-6.93) [label=90: $<$]{};
\node (71) at (-10,-6.5) {\tiny{$X_2$}};	
\node [scale=.0] (72) at (-11,-4.1) [label=91: $\uparrow$]{};

\draw[ line width=0.2mm] [blue](69)[dashed] to (71);
\draw[-, line width=0.2mm] [blue](71) to [out=120,in=240] (1);
	
\end{tikzpicture}

\vspace{.2cm}		
{\small  Fig. 6. Digraph of Eulerian isomorphism series of $C_{54}(R_i)$ with the representation of 

$R_i$ in the place of $C_{54}(R_i)$ where $R_i$ = $X_i, Y_i, Z_i$ for $i$ = 1,2,3.    }
\end{center} }

\medskip
By analysing solutions of the above four problems, we obtain  many properties of isomorphic circulant graphs which are presented in the following lemma.

\begin{lemma} \quad \label{l3.11} {\rm The following holds on isomorphic circulant graphs. 
\begin{enumerate}
	\item [\rm (a1)]   $C_n(R) \cong_{T1_{n, x}} C_n(S)$ $\Leftrightarrow$ $C_n(S) \cong_{T1_{n, x^*}} C_n(R)$, $x,x^*,xx^* = 1\in\varphi_n$.

	\item [\rm (a2)] For $x,y\in\varphi_n$, if $C_n(R) \cong_{T1_{n, x}} C_n(S)$ and $C_n(S) \cong_{T1_{n, y}} C_n(T)$, then $C_n(R) \cong_{T1_{n, yx}} C_n(T)$.
	
	\item [\rm (a3)] In general, for $x_1,x_2,...,x_j\in\varphi_n$, if $C_n(R_1) \cong_{T1_{n, x_1}} C_n(R_2)$, $C_n(R_2) \cong_{T1_{n, x_2}} C_n(R_3)$, $\dots$, $C_n(R_{j}) \cong_{T1_{n, x_j}} C_n(R_{j+1})$, then $C_n(R_1) \cong_{T1_{n, x_jx_{j-1}...X_2x_1}} C_n(R_{j+1})$.
	
	\item [\rm (b1)]   $C_n(R) \cong_{T2_{n, m, t}} C_n(S)$ $\Leftrightarrow$ $C_n(S) \cong_{T2_{n, m, \frac{n}{m}-t}} C_n(R)$ where $r\in R,S$, $m > 1$ divides $\gcd(n, r)$, $m^3$ divides $n$ and $0 \leq t,\frac{n}{m}-t \leq \frac{n}{m}-1$.

	\item [\rm (b2)] For $0 \leq t_1,t_2 \leq \frac{n}{m}-1$, if  $C_n(R) \cong_{T2_{n, m, t_1}} C_n(S)$ and $C_n(S) \cong_{T2_{n, m, t_2}} C_n(T)$ , then $C_n(R)$ $\cong_{T2_{n, m, t_2+t_1}}$ $C_n(T)$ where $t_1+t_2$ = $t_2+t_1$ is calculated under addition modulo $\frac{n}{m}$, $m > 1$ divides $\gcd(n, r)$, $m^3$ divides $n$ and $r\in R,S,T$.
	
	\item [\rm (b3)] In general, for $0 \leq t_1,t_2,...,t_j \leq \frac{n}{m}-1$, if  $C_n(R_1) \cong_{T2_{n, m, t_1}} C_n(R_2)$, $C_n(R_2) \cong_{T2_{n, m, t_2}} C_n(R_3), \dots, C_n(R_{j})$ $\cong_{T2_{n, m, t_j}}$ $C_n(R_{j+1})$, then $C_n(R_1)$ $\cong_{T2_{n, m, t_j+t_{j-1} + \dots +t_2+ t_1}}$ $C_n(R_{j+1})$ where $r\in R_1,R_2,...R_{j+1}$, $m > 1$ divides $\gcd(n, r)$, $m^3$ divides $n$ and $t_1$ + $t_2 + \dots + t_j$ = $t_j+t_{j-1} + \dots + t_1$ is calculated under addition modulo $\frac{n}{m}$.
\end{enumerate}  }
\end{lemma}   
\begin{proof} \quad 
\begin{enumerate}
	\item [\rm (a1)]  $C_n(R) \cong_{T1_{n, x}} C_n(S)$, $x\in\varphi_n$ 
	
	$\Leftrightarrow$  $T1_{n, x}(C_n(R))$ = $C_n(S)$, $x\in\varphi_n$ $\Leftrightarrow$  $C_n(xR))$ = $C_n(S)$, $x\in\varphi_n$ 
	
	$\Leftrightarrow$  $T1_{n, x^*}(C_n(xR))$ = $T1_{n, x^*}(C_n(S))$, $x,x^*,xx^* = 1\in\varphi_n$ 
	
	$\Leftrightarrow$  $C_n(x^*xR))$ = $T1_{n, x^*}(C_n(S))$, $x,x^*,xx^* = 1\in\varphi_n$ 
	
	$\Leftrightarrow$  $C_n(R))$ = $T1_{n, x^*}(C_n(S))$, $x,x^*,xx^* = 1\in\varphi_n$ 
	
	$\Leftrightarrow$  $C_n(S))$ $\cong_{T1_{n, x^*}} C_n(R)$, $x,x^*,xx^* = 1\in\varphi_n$.
	
\item [\rm (a2)]  Let $C_n(R) \cong_{T1_{n, x}} C_n(S)$ and $C_n(S) \cong_{T1_{n, y}} C_n(T)$, $x,y\in\varphi_n$.
	
	$\Rightarrow$  $T1_{n, x}(C_n(R))$ = $C_n(S)$ and $T1_{n, y}(C_n(S))$ = $C_n(T)$, $x,y\in\varphi_n$. 
	
	$\Rightarrow$  $C_n(xR)$ = $C_n(S)$ and $T1_{n, y}(C_n(S))$ = $C_n(T)$, $x,y\in\varphi_n$. 
	
	$\Rightarrow$  $T1_{n, y}(C_n(xR))$ = $C_n(T)$, $x,y\in\varphi_n$. 
	
	$\Rightarrow$  $C_n(y(xR))$ = $C_n(T)$, $x,y\in\varphi_n$. 
	
	$\Rightarrow$  $C_n((yx)R)$ = $C_n(T)$, $x,y\in\varphi_n$. 
	
	$\Rightarrow$  $T1_{n, yx}(C_n(R))$ = $C_n(T)$, $x,y\in\varphi_n$. 
	
	$\Rightarrow$  $C_n(R) \cong_{T1_{n, yx}} C_n(T)$, $x,y,yx\in\varphi_n$.
	
\item [\rm (a3)] Proof follows by the principle of mathematical induction on $j$ and using $(a2)$.\\
	
\item [\rm (b1)] For $0 \leq t,\frac{n}{m}-t \leq \frac{n}{m}-1$,  $m > 1$ divides $\gcd(n, r)$, $m^3$ divides $n$, $S \neq R$ and $r\in R,S$,
	
	~ ~ $C_n(R) \cong_{T2_{n, m, t}} C_n(S)$ 	$\Leftrightarrow$ $\theta_{n, m, t}(C_n(R))$ = $C_n(S)$ and $C_n(S)\in T2_{n,m}(C_n(R))$
	
\hfill	$\Leftrightarrow$ $\theta_{n, m, \frac{n}{m}-t}(\theta_{n, m, t}(C_n(R)))$ = $\theta_{n, m, \frac{n}{m}-t}(C_n(S))$ and $\theta_{n, m, \frac{n}{m}-t}(C_n(S))\in T2_{n,m}(C_n(R))$ 

\hfill since $\theta_{n, m, t}(C_n(S))\in T2_{n,m}(C_n(R))$ $\Leftrightarrow$ $\theta_{n, m, \frac{n}{m}-t}(C_n(S))\in T2_{n,m}(C_n(R))$ = $T2_{n,m}(C_n(S))$ 

\hfill where $S \neq R$, $m > 1$ divides $\gcd(n, r)$, $m^3$ divides $n$ and $r\in R,S$

	$\Leftrightarrow$ $\theta_{n, m, (\frac{n}{m}-t)+t}(C_n(R))$ = $\theta_{n, m, \frac{n}{m}-t}(C_n(S))$ and 
	
	\hfill $\theta_{n, m, t}(C_n(S)),\theta_{n, m, \frac{n}{m}-t}(C_n(S))\in T2_{n,m}(C_n(R))$ 

\hfill where $S \neq R$, $m > 1$ divides $\gcd(n, r)$, $m^3$ divides $n$ and $r\in R,S$

	$\Leftrightarrow$ $C_n(R)$ = $\theta_{n, m, \frac{n}{m}-t}(C_n(S))$ and $\theta_{n, m, \frac{n}{m}-t}(C_n(S))\in T2_{n,m}(C_n(R))$ 

\hfill where $S \neq R$, $m > 1$ divides $\gcd(n, r)$, $m^3$ divides $n$ and $r\in R,S$

	$\Leftrightarrow$ $C_n(S) \cong_{T2_{n, m, \frac{n}{m}-t}} C_n(R)$  where $m > 1$ divides $\gcd(n, r)$, $m^3$ divides $n$, 

\hfill $r\in R,S$ and  $0 \leq t,\frac{n}{m}-t \leq \frac{n}{m}-1$.
	 
\item [\rm (b2)] Let $0 \leq t_1,t_2 \leq \frac{n}{m}-1$, 
$C_n(R) \cong_{T2_{n, m, t_1}} C_n(S)$ and $C_n(S) \cong_{T2_{n, m, t_2}} C_n(T)$. 

$\Rightarrow$ $\theta_{n, m, t_1}(C_n(R))$ = $C_n(S)$ and $C_n(S)\in T2_{n, m}(C_n(R))$ = $T2_{n, m}(C_n(S))$, and 

\hfill $\theta_{n, m, t_2}(C_n(S))$ = $C_n(T)$ and $C_n(T)\in T2_{n, m}(C_n(S))$ = $T2_{n, m}(C_n(T))$

\hfill  where $S \neq R$, $T \neq S$, $m > 1$ divides $\gcd(n, r)$, $m^3$ divides $n$, $r\in R,S,T$. 

$\Rightarrow$ $\theta_{n, m, t_2}(\theta_{n, m, t_1}(C_n(R)))$ = $C_n(T)$ and 

\hfill $C_n(T)\in T2_{n, m}(C_n(R))$ = $T2_{n, m}(C_n(S))$ = $T2_{n, m}(C_n(T))$ 

\hfill where $S \neq R$, $T \neq S$, $R \neq T$, $m > 1$ divides $\gcd(n, r)$, $m^3$ divides $n$, $r\in R,S,T$. 

$\Rightarrow$ $\theta_{n, m, t_2+t_1}(C_n(R))$ = $C_n(T)$ and $C_n(T)\in T2_{n, m}(C_n(R))$ = $T2_{n, m}(C_n(S))$ = $T2_{n, m}(C_n(T))$  

\hfill since $(T2_{n, m}(C_n(R)), \circ)$ is an Abelian group, $0 \leq t_1,t_2 \leq \frac{n}{m}-1$, 

\hfill $t_1+t_2$ = $t_2+t_1$ is calculated under addition modulo $\frac{n}{m}$, $T \neq R$, 

\hfill $m > 1$ divides $\gcd(n, r)$, $m^3$ divides $n$ and $r\in R,S,T$.

$\Rightarrow$ $C_n(R) \cong_{T2_{n, m, t_2+t_1}} C_n(T)$, $0 \leq t_1,t_2 \leq \frac{n}{m}-1$,  $t_1+t_2$ is calculated under addition 

\hfill modulo $\frac{n}{m}$, $m > 1$ divides $\gcd(n, r)$, $m^3$ divides $n$ and $r\in R,S,T$.
\end{enumerate} 

\item [\rm (b3)] Proof follows by the principle of mathematical induction on $j$ and using $(b2)$.
\end{proof} 
 
\section{Isomorphism diagram or digraph and isomorphism graph of circulant graphs}

Definitions of the isomorphic set and isomorphism series of a circulant graph are given in section 3. For a given circulant graph, in this section, we define {\em isomorphism diagram} or {\em isomorphism digraph} and {\em isomorphism graph} which cover all isomorphic circulant graphs and their isomorphic relationships diagramatically and helps to identify easily whether given two circulant graphs are isomorphic or not? And also helps to identify the type of isomorphism that exists among two isomorphic circulant graphs. In these graphs, we have if $C_n(S)$ $\cong_{T1_{n,x}}$ $C_n(T)$, then $C_n(T)$ $\cong_{T1_{n,x^*}}$ $C_n(S)$ using $(a1)$ of Lemma \ref{l3.11} and if $C_n(S)$ $\cong_{T2_{n,m,t}}$ $C_n(T)$, then $C_n(T)$ $\cong_{T2_{n,m,\frac{n}{m}-t}}$ $C_n(S)$ using $(b1)$ of Lemma \ref{l3.11} where $x,x^*\in\varphi_n$, $xx^*$ = 1, $m > 1$ divides $\gcd(n, r)$, $m^3$ divides $n$, $r\in S,T$ and $0 \leq t \leq \frac{n}{m}-1$.  

\begin{definition} \quad \label{d26} Let $\mathcal{C}_{n}$ = $\{C_n(R_i): i = 1,2,\ldots,k ~\text{and}~k \geq 2\}$ be a collection of circulant graphs. Define a simple directed graph $\mathcal{D}$ such that $V(\mathcal{D})$ = $\mathcal{C}_{n}$ and $(C_n(R_i), C_n(R_j))\in E(\mathcal{D})$ if either $C_n(R_i)$ $\cong_{T1_n}$ $C_n(R_j)$, in this case, $C_n(R_i)$ is connected to $C_n(R_j)$ by a directed line, or $C_n(R_i)$ $\cong_{T2_{n, m}}$ $C_n(R_j)$, in this case, $C_n(R_i)$ is connected to $C_n(R_j)$ by a directed dashed line, where $m > 1$ divides $\gcd(n, r)$, $m^3$ divides $n$,  $r\in R_i,R_j$ and $1 \leq i,j \leq k$. In this case, we call the simple digraph $\mathcal{D}$ as {\em an isomorphism digraph} or {\em an isomorphism diagram} of $\mathcal{C}_{n}$ and we denote it by $\mathcal{I}so\mathcal{D}_{n}(\mathcal{C}_n)$. The  underlying simple graph of $\mathcal{D}$ is called the {\em isomorphism graph} of $\mathcal{C}_{n}$ and is denoted by $\mathcal{G}$ = $\mathcal{I}so_{n}(\mathcal{C}_n)$.
\end{definition}
	
\begin{definition} \quad \label{d27} When $\mathcal{C}_{n}$ is a collection of isomorphic circulant graphs, each of order $n$, then the corresponding isomorphism digraph/diagram is denoted by $\mathcal{ID}_{n}(\mathcal{C}_n)$.  Moreover, $\mathcal{I}\mathcal{D}_{n: m_1,m_2,\dots,m_c}(\mathcal{C}_n)$ denotes isomorphism diagram $\mathcal{I}\mathcal{D}_{n}(\mathcal{C}_n)$ when $\mathcal{C}_{n}$ contains isomorphic circulant graphs of Type-2 w.r.t. $m$ $\ni$ $m$ = $m_1,m_2,\dots,m_c$ and its directed dashed line $(C_n(R), C_n(S))$ is assigned with color $x$ whenever $C_n(R)$ $\cong_{T2_{n, m_x}}$ $C_n(S)$, $1 \leq x \leq c$. Isomorphism graph of $\mathcal{ID}_{n}(\mathcal{C}_n)$ is denoted by $\mathcal{I}_{n}(\mathcal{C}_n)$. 
\end{definition}

\begin{obv} {\rm \quad \label{ob1} Let $\mathcal{C}_{n}$ be a collection of isomorphic circulant graphs, each of order $n$, and contain Type-2 isomorphic circulant graphs w.r.t. $m$ $\ni$ $m$ = $m_1,m_2,\dots,m_c$. Let $\mathcal{I}\mathcal{D}_{n: m_1,m_2,\dots,m_c}(\mathcal{C}_n)$ be the isomorphism diagram of $\mathcal{C}_n$. Then, $\mathcal{I}\mathcal{D}_{n: m_1,m_2,\dots,m_c}(\mathcal{C}_n)$ = $\mathcal{I}\mathcal{D}_{n: m_1}(\mathcal{C}_n)$ $\cup$ $\mathcal{I}\mathcal{D}_{n: m_2}(\mathcal{C}_n)$ $\cup$ $\dots$ $\mathcal{I}\mathcal{D}_{n: m_c}(\mathcal{C}_n)$, $c\in\mathbb{N}$. In the construction of $\mathcal{I}\mathcal{D}_{n: m_1,m_2,\dots,m_c}(\mathcal{C}_n)$, the above relationship is used, especially when $n$ is large and $c > 1$.  }
\end{obv}

\begin{definition} \quad \label{d28} Given two isomorphic circulant graphs $C_n(R)$ and $C_n(S)$, let $Iso(C_n(R), C_n(S))$ = $\{C_n(T): C_n(T) \cong C_n(R)\}$. The corresponding isomorphism digraph of $Iso(C_n(R), C_n(S))$ is denoted by $\mathcal{ID}_{n}(C_n(R), C_n(S))$ and the isomorphism graph of $Iso(C_n(R), C_n(S))$ by $\mathcal{I}_{n}(C_n(R), C_n(S))$.
	
Clearly, the isomorphism digraph  $\mathcal{ID}_{n}(C_n(R), C_n(S))$ = $\mathcal{ID}_{n}(C_n(T))$ for each $C_n(T)$ where $C_n(T))$ $\cong$ $C_n(R)$. In general, for $\mathcal{C}_{n}$ = $\{C_n(R_i): C_n(R_i) \cong C_n(R_j),  1 \leq i,j \leq k\}$, the isomorphism digraph $\mathcal{ID}_{n}(\mathcal{C}_{n})$ = $\mathcal{ID}_{n}(C_n(R_i))$ and the isomorphism graph $\mathcal{I}_{n}(\mathcal{C}_{n})$ = $\mathcal{I}_{n}(C_n(R_i))$, $1 \leq i \leq k$. And so we focus on $\mathcal{ID}_{n}(C_n(R))$ and $\mathcal{I}_{n}(C_n(R))$ while studying the isomorphism digraph or the isomorphism graph of isomorphic circulant graphs of order $n$.
\end{definition}

Throughout this study, we consider isomorphism digraphs and isomorphism graphs corresponding to family of circulant graphs which contains circulant graphs of Type-2 isomorphism even though one is free to consider the other case of family of circulant graphs.  

\begin{rem} \quad \label{r29} Given a circulant graph $C_n(R)$, it is possible to construct its isomorphism digraph and isomorphism graph by finding circulant graphs which are isomorphic to it as well as their isomorphism series when $n$ is small. While constructing the isomorphism digraphs and the isomorphism graphs, we use Lemma \ref{l3.11} and consider $C_n(R_1)$ $\cong_{T1_{n, x_j x_{j-1} \dots x_1}}$ $C_n(R_{j+1})$ or/and $C_n(R_1)$ $\cong_{T2_{n, m, t_j t_{j-1} \dots t_1}}$ $C_n(R_{j+1})$ whenever elements of $Iso(C_n(R))$ have isomorphism series $C_n(R_1)$ $\cong_{T1_{n, x_1}}$ $C_n(R_2) \cong_{T1_{n, x_2}} C_n(R_3) \cong_{T1_{n, x_3}}$ $\dots$ $\cong_{T1_{n, x_j}}$ $C_n(R_{j+1})$ or/and $C_n(R_1)$ $\cong_{T2_{n, m, t_1}}$ $C_n(R_2)$ $\cong_{T2_{n, m, t_2}}$ $C_n(R_3) \cong_{T2_{n, m, t_3}}$ $\dots$ $\cong_{T2_{n, m, t_j}}$ $C_n(R_{j+1})$, respectively, $j\in\mathbb{N}$ and $x_1,x_2,\ldots,x_j\in\varphi_n$.  	
\end{rem}

In \cite{v2-2}, we found that Type-2 isomorphic circulant graphs exist only w.r.t. $m$ = 2 among different isomorphic $C_{16}(R)$ and only w.r.t. $m$ = 3 among different isomorphic $C_{27}(R)$. But in section 5, we present the isomorphism digraph and the isomorphism graph of $C_{432}(16, 27, 48, 54, 128, 160, 189)$ which has isomorphic circulant graphs of Type-2 w.r.t. $m$ = 2 as well as $m$ = 3 and in section 6, isomorphism digraph and graph of $C_{6750}(135, 243, 250, 750, 1107, 1593, 2000, 2457, 2500, 2943)$ which has isomorphic circulant graphs of Type-2 w.r.t. $m$ = 3 as well as $m$ = 5 are presented. We use remark \ref{r4.5} to find the isomorphism digraphs and the isomorphism graphs corresponding to different $C_{16}(R)$ and $C_{27}(R)$ graphs and their isomorphism series and also corresponding to $C_{54}(1,3,17,19)$ in the following problems.

\begin{prm} \quad \label{p5} {\rm  Find the isomorphism digraph, the isomorphism graph and different isomorphism series corresponding to each circulant graph which has Type-2 isomorphism and of order 16.   }
\end{prm}

\noindent
{\bf Solution.}\quad In problem 1 in \cite{v2-2}, it is shown that the following 8 pairs of circulant graphs, each of order 16, are the only isomorphic circulant graphs of Type-2 w.r.t. $m$ = 2.

\begin{enumerate}
	\item [\rm (1)]	$C_{16}(1,2,7)$ $\cong_{T2_{16,2,2}}$ $C_{16}(2,3,5)$;
	
	\item [\rm (2)] $C_{16}(1,6,7)$ $\cong_{T2_{16,2,2}}$ $C_{16}(3,5,6)$; 
	
	\item [\rm (3)]	$C_{16}(1,2,4,7)$ $\cong_{T2_{16,2,2}}$ $C_{16}(2,3,4,5)$;
	
	\item [\rm (4)]	$C_{16}(1,2,7,8)$ $\cong_{T2_{16,2,2}}$ $C_{16}(2,3,5,8)$;
	
	\item [\rm (5)]	$C_{16}(1,4,6,7)$ $\cong_{T2_{16,2,2}}$ $C_{16}(3,4,5,6)$;
	
	\item [\rm (6)]	$C_{16}(1,6,7,8)$ $\cong_{T2_{16,2,2}}$ $C_{16}(3,5,6,8)$;
	
	\item [\rm (7)]	$C_{16}(1,2,4,7,8)$ $\cong_{T2_{16,2,2}}$ $C_{16}(2,3,4,5,8)$;
	
	\item [\rm (8)]	$C_{16}(1,4,6,7,8)$ $\cong_{T2_{16,2,2}}$ $C_{16}(3,4,5,6,8)$.  
\end{enumerate}
Let $\mathcal{D}$ be the isomorphism digraph $\mathcal{ID}_{16, 2}(C_{16}(R))$  and $\mathcal{G}$ be the isomorphism graph $\mathcal{I}_{16, 2}(C_{16}(R))$ for different $R$. Our aim is to find $\mathcal{D}$ and $\mathcal{G}$ for different $R$ where $V(\mathcal{D})$ = $Iso(C_{16}(R))$ = $\{C_{16}(S):$ $C_{16}(S)$ $\cong$ $C_{16}(R)\}$ = $V(\mathcal{G})$.  

In each case, we find the isomorphism digraph $\mathcal{ID}_{16, 2}(C_{16}(R))$ at first and then from this digraph, we obtain the isomorphism graph $\mathcal{I}_{16, 2}(C_{16}(R))$ as well as different isomorphism series. 

\begin{enumerate}
\item [\rm (1)] Here, $C_{16}(R)$ = $C_{16}(1,2,7)$,  

$T2_{16,2}(C_{16}(1,2,7))$ = $\{C_{16}(1,2,7), C_{16}(2,3,5) = T2_{16,2,2}(C_{16}(1,2,7))\}$ = $T2_{16,2}(C_{16}(2,3,5))$,

$T1_{16}(C_{16}(1,2,7))$ = $\{C_{16}(1,2,7), C_{16}(3,5,6) = T1_{16,3}(C_{16}(1,2,7))\}$ = $T1_{16,3}(C_{16}(3,5,6))$,

$T2_{16,2}(C_{16}(3,5,6))$ = $\{C_{16}(3,5,6), C_{16}(1,6,7) = T2_{16,2,2}(C_{16}(3,5,6))\}$ = $T2_{16,2}(C_{16}(1,6,7))$,

$T1_{16}(C_{16}(2,3,5))$ = $\{C_{16}(2,3,5), C_{16}(1,6,7) = T1_{16,3}(C_{16}(2,3,5))\}$ = $T1_{16}(C_{16}(2,3,5))$ and
\\
$V(\mathcal{D})$ = $Iso(C_{16}(R))$ = $\{C_{16}(1,2,7), C_{16}(2,3,5), C_{16}(1,6,7), C_{16}(3,5,6)$ = $V(\mathcal{G})$. 

Corresponding to $V(\mathcal{D})$ = $V(\mathcal{G})$, we draw the isomorphism digraph $\mathcal{ID}_{16, 2}(C_{16}(1,2,7))$ and from this digraph we obtain the isomorphism graph $\mathcal{I}_{16, 2}(C_{16}(1,2,7))$ which are presented in Figures 7 and 8, respectively. 

\item [\rm (2)] Here, $C_{16}(R)$ = $C_{16}(1,6,7)$ and $C_{16}(1,6,7)$ $\cong_{T2_{16,2,2}}$ $C_{16}(3,5,6)$. From $(1)$, we get,  

$V(\mathcal{D})$ = $V(\mathcal{G})$ = $Iso(C_{16}(1,6,7))$ = $Iso(C_{16}(3,5,6))$ = $Iso(C_{16}(2,3,5))$ = $Iso(C_{16}(1,2,7))$. 

The isomorphism digraph $\mathcal{ID}_{16, 2}(C_{16}(1,2,7))$ = $\mathcal{ID}_{16, 2}(C_{16}(1,6,7))$ = $\mathcal{ID}_{16, 2}(C_{16}(3,5,6))$ = $\mathcal{ID}_{16, 2}(C_{16}(2,3,5))$ and the isomorphism graph $\mathcal{I}_{16, 2}(C_{16}(1,2,7))$ =
\\
$\mathcal{I}_{16, 2}(C_{16}(1,6,7))$ = $\mathcal{I}_{16, 2}(C_{16}(3,5,6))$ = $\mathcal{I}_{16, 2}(C_{16}(2,3,5))$ are same as these are given in Figures 7 and 8.

{\rm 	\begin{center}
	\begin{tikzpicture}  
		[scale=.55,auto=center,every node/.style={draw,circle}]  
	
		\node (1) at (-7,-4) {\tiny{$C_{16}(S_1)$}};
                   \node (2) at (0,-4) {\tiny{$C_{16}(R_1)$}};	
		\node (3) at (0,2)  {\tiny{$C_{16}(R_2)$}};
		\node (4) at (-7,2) {\tiny{$C_{16}(S_2)$}};

		\node [scale=.0] (12) at (-3.5,-5.5) [label=90: $>$]{};
		\node [scale=.0] (21) at (-3.5,-3.7) [label=91: $<$]{};
		\node [scale=.0] (23) at (-.95,-1.3) [label=92: $\downarrow$]{};
                   \node [scale=.0] (32) at (.95,-1.5) [label=93: $\uparrow$]{};
		\node [scale=.0] (34) at (-3.5,2.27) [label=90: $<$]{};
		\node [scale=.0] (43) at (-3.5,.52) [label=91: $>$]{};
		\node [scale=.0] (41) at (-7.95,-1.5) [label=92: $\downarrow$]{};
                   \node [scale=.0] (14) at (-6.05,-1.45) [label=93: $\uparrow$]{};
      			
\draw[line width=0.2mm] [blue](3)[dashed] to [out=160,in=20] (4);
\draw[line width=0.2mm] [blue](4)[dashed] to [out=340,in=200] (3);
\draw[line width=0.2mm] [blue](1)[dashed] to [out=340,in=200] (2);
\draw[line width=0.2mm] [blue](2)[dashed] to [out=160,in=20] (1);

   	\draw[line width=0.2mm] [blue](3) to [out=245,in=115] (2);
   	\draw[line width=0.2mm] [blue](2) to [out=65,in=295] (3);
    \draw[line width=0.2mm] [blue](4) to [out=245,in=115] (1);
   	\draw[line width=0.2mm] [blue](1) to [out=65,in=295] (4);
    	
 \node [scale=.0] (6) at (-3.5,2.2) [label=90:\tiny{$T2_{16,2,2}$}]{};
 \node [scale=.0] (7) at (-3.5,-0.3)[label=90:\tiny{$T2_{16,2,2}$}]{};
 \node [scale=.0] (8) at (-3.5,-3.8)[label=90:\tiny{$T2_{16,2,2}$}]{};
 \node [scale=.0] (9) at (-3.5,-6.4)[label=90:\tiny{$T2_{16,2,2}$}]{};

 \node [scale=.0] (10) at (-5.2,-1.5) [label=90:\tiny{$T1_{16,3}$}]{};
 \node [scale=.0] (11) at (-8.7,-2.5) [label=90:\tiny{$T1_{16,3}$}]{};
 \node [scale=.0] (12) at (1.75,-1.5) [label=90:\tiny{$T1_{16,3}$}]{};
\node [scale=.0] (13) at (-1.6,-2.5) [label=90:\tiny{$T1_{16,3}$}]{};

\node (3) at (12.45,2)  {\tiny{$C_{16}(R_2)$}};
\node (4) at (5.45,2) {\tiny{$C_{16}(S_2)$}};
\node (1) at (5.45,-4) {\tiny{$C_{16}(S_1)$}};
\node (2) at (12.45,-4) {\tiny{$C_{16}(R_1)$}};	

\draw[-, line width=0.2mm] [blue](1)[dashed] to (2);
\draw[-, line width=0.2mm] [blue](3)[dashed] to (4);

\draw[-, line width=0.2mm] [blue](2) to (3);
\draw[-, line width=0.2mm] [blue](1) to (4);

\node [scale=.0] (6) at (8.95,1.3) [label=90:\tiny{$T2_{16,2}$}]{};
\node [scale=.0] (8) at (8.95,-4.7) [label=90:\tiny{$T2_{16,2}$}]{};

\node [scale=.0] (10) at (6.3,-1.5) [label=90:\tiny{$T1_{16}$}]{};
\node [scale=.0] (12) at (11.7,-1.5) [label=90:\tiny{$T1_{16}$}]{};
\end{tikzpicture}

\hspace{.8cm} Figure  7. {\small Digraph $\mathcal{D}$ with   \hspace{3cm} Figure  8. Graph $\mathcal{G}$ with  \hspace{1.5cm} \\
 \hspace{.5cm} $R_1$ = $\{1,2,7\}$, $S_1$ = $\{2,3,5\}$, $R_2$ = $\{3,5,6\}$, $S_2$ = $\{1,6,7\}$.   } 
\end{center} }
	
\item [\rm (3)]	Here, $C_{16}(R)$ = $C_{16}(1,2,4,7)$,  

$T2_{16,2}(C_{16}(1,2,4,7))$ = $\{C_{16}(1,2,4,7),$  

\hfill $C_{16}(2,3,4,5)$ = $T2_{16,2,2}(C_{16}(1,2,4,7))\}$ = $T2_{16,2}(C_{16}(2,3,4,5))$,

$T1_{16}(C_{16}(1,2,4,7))$ = $\{C_{16}(1,2,4,7),$  

\hfill $C_{16}(3,4,5,6)$ = $T1_{16,3}(C_{16}(1,2,4,7))\}$ = $T1_{16,3}(C_{16}(3,4,5,6))$,

$T2_{16,2}(C_{16}(3,4,5,6))$ = $\{C_{16}(3,4,5,6),$ 

\hfill $C_{16}(1,4,6,7)$ = $T2_{16,2,2}(C_{16}(3,4,5,6))\}$ = $T2_{16,2}(C_{16}(1,4,6,7))$,

$T1_{16}(C_{16}(2,3,4,5))$ = $\{C_{16}(2,3,4,5),$  

\hfill $C_{16}(1,4,6,7)$ = $T1_{16,3}(C_{16}(2,3,4,5))\}$ = $T1_{16}(C_{16}(2,3,4,5))$, and
\\
$V(\mathcal{D})$ = $Iso(C_{16}(R))$ = $\{C_{16}(1,2,4,7), C_{16}(2,3,4,5), C_{16}(1,4,6,7), C_{16}(3,4,5,6)$ = $V(\mathcal{G})$. 

The isomorphism digraph $\mathcal{ID}_{16, 2}(C_{16}(1,2,4,7))$ = $\mathcal{ID}_{16, 2}(C_{16}(2,3,4,5))$ = $\mathcal{I}D_{16, 2}(C_{16}(1,4,6,7))$ = $\mathcal{ID}_{16, 2}(C_{16}(3,4,5,6))$ and the isomorphism graph $\mathcal{I}_{16, 2}(C_{16}(1,2,4,7))$ = $\mathcal{I}_{16, 2}(C_{16}(2,3,4,5))$ = $\mathcal{I}_{16, 2}(C_{16}(1,4,6,7))$ = $\mathcal{I}_{16, 2}(C_{16}(3,4,5,6))$ are given in Figures 9 and 10.

Clearly, graphs in Figures 7 and 8 are same as in Figures 9 and 10, except in Figures 9 and 10, $S_i$ = $R_i \cup \{4\}$ for $i$ = 1 to 4. Moreover, $V(\mathcal{D})$ = $V(\mathcal{G})$ = $Iso(C_{16}(1,2,4,7))$ = $Iso(C_{16}(2,3,4,5))$ = $Iso(C_{16}(1,4,6,7))$ = $Iso(C_{16}(3,4,5,6))$. 

\item [\rm (4)]	Here, $C_{16}(X)$ = $C_{16}(1,2,7,8)$,  

$T2_{16,2}(C_{16}(1,2,7,8))$ = $\{C_{16}(1,2,7,8),$ 

\hfill $C_{16}(2,3,5,8)$ = $T2_{16,2,2}(C_{16}(1,2,7,8))\}$ = $T2_{16,2}(C_{16}(2,3,5,8))$,

$T1_{16}(C_{16}(1,6,7,8))$ = $\{C_{16}(1,6,7,8)$, 

\hfill $C_{16}(3,5,6,8)$ = $T1_{16,3}(C_{16}(1,6,7,8))\}$ = $T1_{16,3}(C_{16}(3,5,6,8))$,

$T2_{16,2}(C_{16}(3,5,6,8))$ = $\{C_{16}(3,5,6,8),$ 

\hfill $C_{16}(1,6,7,8)$ = $T2_{16,2,2}(C_{16}(3,5,6,8))\}$ = $T2_{16,2}(C_{16}(1,6,7,8))$,

$T1_{16}(C_{16}(2,3,5,8))$ = $\{C_{16}(2,3,5,8),$ 

\hfill $C_{16}(1,6,7,8)$ = $T1_{16,3}(C_{16}(2,3,5,8))\}$ = $T1_{16}(C_{16}(2,3,5,8))$ and
\\
$V(\mathcal{D})$ = $Iso(C_{16}(R))$ = $\{C_{16}(1,2,7,8), C_{16}(2,3,5,8), C_{16}(1,6,7,8), C_{16}(3,5,6,8)$ = $V(\mathcal{G})$.

The isomorphism digraph $\mathcal{ID}_{16, 2}(C_{16}(1,2,7,8))$ = $\mathcal{ID}_{16, 2}(C_{16}(2,3,5,8))$ = $\mathcal{ID}_{16, 2}(C_{16}(1,6,7,8))$ = $\mathcal{ID}_{16, 2}(C_{16}(3,5,6,8))$ and the isomorphism graph $\mathcal{I}_{16, 2}(C_{16}(1,2,7,8))$ = $\mathcal{I}_{16, 2}(C_{16}(2,3,5,8))$ = $\mathcal{I}_{16, 2}(C_{16}(1,6,7,8))$ = $\mathcal{I}_{16, 2}(C_{16}(3,5,6,8))$ are given in Figures 11 and 12.
    
    Clearly, graphs in Figures 11 and 12 are same as in Figures 9 and 10, except in Figures 11 and 12, $X_i$ = $(S_i \setminus \{4\}) \cup \{8\}$ for $i$ = 1 to 4.

{\rm 	\begin{center}
	\begin{tikzpicture}  
		[scale=.55, auto=center,every node/.style={draw,circle}]  
	
		\node (1) at (-7,-4) {\tiny{$C_{16}(S_1)$}};
                   \node (2) at (0,-4) {\tiny{$C_{16}(R_1)$}};	
		\node (3) at (0,2)  {\tiny{$C_{16}(R_2)$}};
		\node (4) at (-7,2) {\tiny{$C_{16}(S_2)$}};

		\node [scale=.0] (12) at (-3.5,-5.45) [label=90: $>$]{};
		\node [scale=.0] (21) at (-3.5,-3.8) [label=91: $<$]{};
		\node [scale=.0] (23) at (-.95,-1.3) [label=92: $\downarrow$]{};
                   \node [scale=.0] (32) at (.95,-1.5) [label=93: $\uparrow$]{};
		\node [scale=.0] (34) at (-3.5,2.27) [label=90: $<$]{};
		\node [scale=.0] (43) at (-3.5,.53) [label=91: $>$]{};
		\node [scale=.0] (41) at (-7.95,-1.5) [label=92: $\downarrow$]{};
                   \node [scale=.0] (14) at (-6.05,-1.45) [label=93: $\uparrow$]{};
      			
\draw[line width=0.2mm] [blue](3)[dashed] to [out=160,in=20] (4);
\draw[line width=0.2mm] [blue](4)[dashed] to [out=340,in=200] (3);
\draw[line width=0.2mm] [blue](1)[dashed] to [out=340,in=200] (2);
\draw[line width=0.2mm] [blue](2)[dashed] to [out=160,in=20] (1);

   	\draw[line width=0.2mm] [blue](3) to [out=245,in=115] (2);
   	\draw[line width=0.2mm] [blue](2) to [out=65,in=295] (3);
    \draw[line width=0.2mm] [blue](4) to [out=245,in=115] (1);
   	\draw[line width=0.2mm] [blue](1) to [out=65,in=295] (4);
    	
 \node [scale=.0] (6) at (-3.5,2.2) [label=90:\tiny{$T2_{16,2,2}$}]{};
 \node [scale=.0] (7) at (-3.5,-0.4)[label=90:\tiny{$T2_{16,2,2}$}]{};
 \node [scale=.0] (8) at (-3.5,-3.8)[label=90:\tiny{$T2_{16,2,2}$}]{};
 \node [scale=.0] (9) at (-3.5,-6.6)[label=90:\tiny{$T2_{16,2,2}$}]{};

 \node [scale=.0] (10) at (-5.2,-1.5) [label=90:\tiny{$T1_{16,3}$}]{};
 \node [scale=.0] (11) at (-8.7,-2.5) [label=90:\tiny{$T1_{16,3}$}]{};
 \node [scale=.0] (12) at (1.75,-1.5) [label=90:\tiny{$T1_{16,3}$}]{};
\node [scale=.0] (13) at (-1.6,-2.5) [label=90:\tiny{$T1_{16,3}$}]{};

\node (2) at (12.45,-4) {\tiny{$C_{16}(R_1)$}};	
\node (1) at (5.45,-4) {\tiny{$C_{16}(S_1)$}};
\node (3) at (12.45,2)  {\tiny{$C_{16}(R_2)$}};
\node (4) at (5.45,2) {\tiny{$C_{16}(S_2)$}};

\draw[-, line width=0.2mm] [blue](1)[dashed] to (2);
\draw[-, line width=0.2mm] [blue](3)[dashed] to (4);

\draw[-, line width=0.2mm] [blue](2) to (3);
\draw[-, line width=0.2mm] [blue](1) to (4);

\node [scale=.0] (6) at (8.95,1.3) [label=90:\tiny{$T2_{16,2}$}]{};
\node [scale=.0] (8) at (8.95,-4.7) [label=90:\tiny{$T2_{16,2}$}]{};

\node [scale=.0] (10) at (6.3,-1.5) [label=90:\tiny{$T1_{16}$}]{};
\node [scale=.0] (12) at (11.7,-1.5) [label=90:\tiny{$T1_{16}$}]{};
\end{tikzpicture}

\hspace{.8cm} Figure  9. {\small Digraph $\mathcal{D}$ with   \hspace{3cm} Figure  10. Graph $\mathcal{G}$ with  \hspace{1.5cm} \\
\hspace{.5cm} $R_1$ = $\{1,2,4,7\}$, $S_1$ = $\{2,3,4,5\}$, $R_2$ = $\{3,4,5,6\}$, $S_2$ = $\{1,4,6,7\}$. } 
\end{center} }
    
\item [\rm (5)] Here, $C_{16}(S)$ = $C_{16}(1,4,6,7)$ and $C_{16}(1,4,6,7)$ $\cong_{T2_{16,2,2}}$ $C_{16}(3,4,5,6)$. This is same as case (3) since $V(\mathcal{D})$ = $V(\mathcal{G})$ = $Iso(C_{16}(1,4,6,7))$ = $Iso(C_{16}(3,4,5,6))$ =  $Iso(C_{16}(1,2,4,7))$ = $Iso(C_{16}(2,3,4,5))$. See Figures 9 and 10.
    
\item [\rm (6)] Here, $C_{16}(X)$ = $C_{16}(1,6,7,8)$ and $C_{16}(1,6,7,8)$ $\cong_{T2_{16,2,2}}$ $C_{16}(3,5,6,8)$. This is same as case (4) since $V(\mathcal{D})$ = $V(\mathcal{G})$ = $Iso(C_{16}(1,6,7,8))$ = $Iso(C_{16}(3,5,6,8))$ =  $Iso(C_{16}(1,2,7,8))$ = $Iso(C_{16}(2,3,5,8))$. See Figures 11 and 12.

{\rm 	\begin{center}
	\begin{tikzpicture}  
		[scale=.6,auto=center,every node/.style={draw,circle}]  
	
		\node (1) at (-7,-4) {\tiny{$C_{16}(Y_1)$}};
                   \node (2) at (0,-4) {\tiny{$C_{16}(X_1)$}};	
		\node (3) at (0,2)  {\tiny{$C_{16}(X_2)$}};
		\node (4) at (-7,2) {\tiny{$C_{16}(Y_2)$}};

		\node [scale=.0] (12) at (-3.5,-5.45) [label=90: $>$]{};
		\node [scale=.0] (21) at (-3.5,-3.7) [label=91: $<$]{};
		\node [scale=.0] (23) at (-.95,-1.3) [label=92: $\downarrow$]{};
                   \node [scale=.0] (32) at (.95,-1.5) [label=93: $\uparrow$]{};
		\node [scale=.0] (34) at (-3.5,2.27) [label=90: $<$]{};
		\node [scale=.0] (43) at (-3.5,.57) [label=91: $>$]{};
		\node [scale=.0] (41) at (-7.95,-1.5) [label=92: $\downarrow$]{};
                   \node [scale=.0] (14) at (-6.05,-1.45) [label=93: $\uparrow$]{};
      			
\draw[line width=0.2mm] [blue](3)[dashed] to [out=160,in=20] (4);
\draw[line width=0.2mm] [blue](4)[dashed] to [out=340,in=200] (3);
\draw[line width=0.2mm] [blue](1)[dashed] to [out=340,in=200] (2);
\draw[line width=0.2mm] [blue](2)[dashed] to [out=160,in=20] (1);

   	\draw[line width=0.2mm] [blue](3) to [out=245,in=115] (2);
   	\draw[line width=0.2mm] [blue](2) to [out=65,in=295] (3);
    \draw[line width=0.2mm] [blue](4) to [out=245,in=115] (1);
   	\draw[line width=0.2mm] [blue](1) to [out=65,in=295] (4);
    	
 \node [scale=.0] (6) at (-3.5,2.2) [label=90:\tiny{$T2_{16,2,2}$}]{};
 \node [scale=.0] (7) at (-3.5,-0.3)[label=90:\tiny{$T2_{16,2,2}$}]{};
 \node [scale=.0] (8) at (-3.5,-3.8)[label=90:\tiny{$T2_{16,2,2}$}]{};
 \node [scale=.0] (9) at (-3.5,-6.4)[label=90:\tiny{$T2_{16,2,2}$}]{};

 \node [scale=.0] (10) at (-5.2,-1.5) [label=90:\tiny{$T1_{16,3}$}]{};
 \node [scale=.0] (11) at (-8.7,-2.5) [label=90:\tiny{$T1_{16,3}$}]{};
 \node [scale=.0] (12) at (1.75,-1.5) [label=90:\tiny{$T1_{16,3}$}]{};
\node [scale=.0] (13) at (-1.6,-2.5) [label=90:\tiny{$T1_{16,3}$}]{};

\node (3) at (12.45,2)  {\tiny{$C_{16}(X_2)$}};
\node (4) at (5.45,2) {\tiny{$C_{16}(Y_2)$}};
\node (1) at (5.45,-4) {\tiny{$C_{16}(Y_1)$}};
\node (2) at (12.45,-4) {\tiny{$C_{16}(X_1)$}};	

\draw[-, line width=0.2mm] [blue](1)[dashed] to (2);
\draw[-, line width=0.2mm] [blue](3)[dashed] to (4);

\draw[-, line width=0.2mm] [blue](2) to (3);
\draw[-, line width=0.2mm] [blue](1) to (4);

\node [scale=.0] (6) at (8.95,1.3) [label=90:\tiny{$T2_{16,2}$}]{};
\node [scale=.0] (8) at (8.95,-4.7) [label=90:\tiny{$T2_{16,2}$}]{};

\node [scale=.0] (10) at (6.3,-1.5) [label=90:\tiny{$T1_{16}$}]{};
\node [scale=.0] (12) at (11.7,-1.5) [label=90:\tiny{$T1_{16}$}]{};
\end{tikzpicture}
			
\hspace{.8cm} Figure  11. {\small Digraph $\mathcal{D}$ with   \hspace{3cm} Figure  12. Graph $\mathcal{G}$ with  \hspace{1.5cm} \\
\hspace{.5cm} $X_1$ = $\{1,2,7,8\}$, $Y_1$ = $\{2,3,5,8\}$, $X_2$ = $\{3,5,6,8\}$, $Y_2$ = $\{1,6,7,8\}$. } 
\end{center} }

{\rm 	\begin{center}
	\begin{tikzpicture}  
		[scale=.6,auto=center,every node/.style={draw,circle}]  
	
		\node (1) at (-7,-4) {\tiny{$C_{16}(Y_1)$}};
    \node (2) at (0,-4) {\tiny{$C_{16}(X_1)$}};	
		\node (3) at (0,2)  {\tiny{$C_{16}(X_2)$}};
		\node (4) at (-7,2) {\tiny{$C_{16}(Y_2)$}};

		\node [scale=.0] (12) at (-3.5,-5.45) [label=90: $>$]{};
		\node [scale=.0] (21) at (-3.5,-3.7) [label=91: $<$]{};
		\node [scale=.0] (23) at (-.95,-1.3) [label=92: $\downarrow$]{};
                   \node [scale=.0] (32) at (.95,-1.5) [label=93: $\uparrow$]{};
		\node [scale=.0] (34) at (-3.5,2.27) [label=90: $<$]{};
		\node [scale=.0] (43) at (-3.5,.57) [label=91: $>$]{};
		\node [scale=.0] (41) at (-7.95,-1.5) [label=92: $\downarrow$]{};
                   \node [scale=.0] (14) at (-6.05,-1.45) [label=93: $\uparrow$]{};
      			
\draw[line width=0.2mm] [blue](3)[dashed] to [out=160,in=20] (4);
\draw[line width=0.2mm] [blue](4)[dashed] to [out=340,in=200] (3);
\draw[line width=0.2mm] [blue](1)[dashed] to [out=340,in=200] (2);
\draw[line width=0.2mm] [blue](2)[dashed] to [out=160,in=20] (1);

   	\draw[line width=0.2mm] [blue](3) to [out=245,in=115] (2);
   	\draw[line width=0.2mm] [blue](2) to [out=65,in=295] (3);
    \draw[line width=0.2mm] [blue](4) to [out=245,in=115] (1);
   	\draw[line width=0.2mm] [blue](1) to [out=65,in=295] (4);
    	
 \node [scale=.0] (6) at (-3.5,2.2) [label=90:\tiny{$T2_{16,2,2}$}]{};
 \node [scale=.0] (7) at (-3.5,-0.3)[label=90:\tiny{$T2_{16,2,2}$}]{};
 \node [scale=.0] (8) at (-3.5,-3.8)[label=90:\tiny{$T2_{16,2,2}$}]{};
 \node [scale=.0] (9) at (-3.5,-6.4)[label=90:\tiny{$T2_{16,2,2}$}]{};

 \node [scale=.0] (10) at (-5.2,-1.5) [label=90:\tiny{$T1_{16,3}$}]{};
 \node [scale=.0] (11) at (-8.7,-2.5) [label=90:\tiny{$T1_{16,3}$}]{};
 \node [scale=.0] (12) at (1.75,-1.5) [label=90:\tiny{$T1_{16,3}$}]{};
\node [scale=.0] (13) at (-1.6,-2.5) [label=90:\tiny{$T1_{16,3}$}]{};

\node (3) at (12.45,2)  {\tiny{$C_{16}(X_2)$}};
\node (4) at (5.45,2) {\tiny{$C_{16}(Y_2)$}};
\node (1) at (5.45,-4) {\tiny{$C_{16}(Y_1)$}};
\node (2) at (12.45,-4) {\tiny{$C_{16}(X_1)$}};	

\draw[-, line width=0.2mm] [blue](1)[dashed] to (2);
\draw[-, line width=0.2mm] [blue](3)[dashed] to (4);

\draw[-, line width=0.2mm] [blue](2) to (3);
\draw[-, line width=0.2mm] [blue](1) to (4);

\node [scale=.0] (6) at (8.95,1.3) [label=90:\tiny{$T2_{16,2}$}]{};
\node [scale=.0] (8) at (8.95,-4.7) [label=90:\tiny{$T2_{16,2}$}]{};

\node [scale=.0] (10) at (6.3,-1.5) [label=90:\tiny{$T1_{16}$}]{};
\node [scale=.0] (12) at (11.7,-1.5) [label=90:\tiny{$T1_{16}$}]{};
\end{tikzpicture}
			
\hspace{.8cm} Figure  13. {\small Digraph $\mathcal{D}$ with   \hspace{3cm} Figure  14. Graph $\mathcal{G}$ with  \hspace{1.5cm} \\
\hspace{.5cm} $X_1$ = $\{1,2,4,7,8\}$, $Y_1$ = $\{2,3,4,5,8\}$, $X_2$ = $\{3,4,5,6,8\}$, $Y_2$ = $\{1,4,6,7,8\}$. } 

\end{center} }

\item [\rm (7)]	Here, $C_{16}(Y)$ = $C_{16}(1,2,4,7,8)$,  

$T2_{16,2}(C_{16}(1,2,4,7,8))$ = $\{C_{16}(1,2,4,7,8),$ 

\hfill $C_{16}(2,3,4,5,8)$ = $T2_{16,2,2}(C_{16}(1,2,4,7,8))\}$ = $T2_{16,2}(C_{16}(2,3,4,5,8))$,

$T1_{16}(C_{16}(1,2,4,7,8))$ = $\{C_{16}(1,2,4,7,8)$, 

\hfill $C_{16}(3,4,5,6,8)$ = $T1_{16,3}(C_{16}(1,2,4,7,8))\}$ = $T1_{16,3}(C_{16}(3,4,5,6,8))$,

$T2_{16,2}(C_{16}(3,4,5,6,8))$ = $\{C_{16}(3,4,5,6,8)$, 

\hfill $C_{16}(1,4,6,7,8)$ = $T2_{16,2,2}(C_{16}(3,4,5,6,8))\}$ = $T2_{16,2}(C_{16}(1,4,6,7,8))$,

$T1_{16}(C_{16}(2,3,4,5,8))$ = $\{C_{16}(2,3,4,5,8)$, 

\hfill $C_{16}(1,4,6,7,8)$ = $T1_{16,3}(C_{16}(2,3,4,5,8))\}$ = $T1_{16}(C_{16}(2,3,4,5,8))$, and
\\
$V(\mathcal{D})$ = $Iso(C_{16}(R))$ 

\hfill = $\{C_{16}(1,2,4,7,8), C_{16}(2,3,4,5,8), C_{16}(1,4,6,7,8), C_{16}(3,4,5,6,8)$ = $V(\mathcal{G})$. 

Isomorphism digraph 
$\mathcal{ID}_{16, 2}(C_{16}(1,2,4,7,8))$ = $\mathcal{ID}_{16, 2}(C_{16}(2,3,4,5,8))$ = $\mathcal{ID}_{16, 2}(C_{16}(1,4$, $6,7,8))$ = $\mathcal{ID}_{16, 2}(C_{16}(3,4,5,6,8))$ and 
isomorphism graph $\mathcal{I}_{16, 2}(C_{16}(1,2,4,7,8))$ = $\mathcal{I}_{16, 2}(C_{16}(2$, $3,4,5,8))$ = $\mathcal{I}_{16, 2}(C_{16}(1,4,6,7,8))$ = $\mathcal{I}_{16, 2}(C_{16}(3,4,5,6,8))$ are given in Figures 13 and 14.

Clearly, graphs in Figures 13 and 14 are similar to graphs in Figures 11 and 12, respectively.

\item [\rm (8)]	Here, $C_{16}(Y)$ = $C_{16}(1,4,6,7,8)$ and $C_{16}(1,4,6,7,8)$ $\cong_{T2_{16,2,2}}$ $C_{16}(3,4,5,6,8)$. This case is same as case (7) since $V(\mathcal{D})$ = $V(\mathcal{G})$ = $Iso(C_{16}(1,4,6,7,8))$ = = $Iso(C_{16}(3,4,5,6,8))$ =  $Iso(C_{16}(1,2,4,7,8))$ = $Iso(C_{16}(2,3,4,5,8))$. See Figures 13 and 14.  
\end{enumerate}

It is noted that in all the 8 cases, all the unlabeled digraphs are the same as well as their unlabeled underlying graphs are the same and all isomorphism series corresponding to one case, say in case 8, only are presented. It is easy to write all isomorphism series corresponding to other cases.

In Figure 13, corresponding to case (8), let $X_1$ = $\{1,2,4,7,8\}$, $Y_1$ = $\{2,3,4,5,8\}$, $X_2$ = $\{3,4,5,6,8\}$ and $Y_2$ = $\{1,4,6,7,8\}$. From Figure 13, we get the following isomorphism series.
\begin{enumerate}
\item [\rm (8.1)] $C_{16}(X_1)$ $\cong$ $C_{16}(Y_1)$:
	
	$(i)$ $C_{16}(X_1)$ $\cong_{T2_{16,2,2}}$ $C_{16}(Y_1)$ or 
	
	$(ii)$ $C_{16}(X_1)$ $\cong_{T1_{16,3}}$ $C_{16}(X_2)$ $\cong_{T2_{16,2,2}}$ $C_{16}(Y_2)$  $\cong_{T1_{16,3}}$ $C_{16}(Y_1)$;	 

\item [\rm (8.2)] $C_{16}(X_1)$ $\cong$ $C_{16}(Y_2)$:
	
	 $(i)$ $C_{16}(X_1)$ $\cong_{T2_{16,2,2}}$ $C_{16}(Y_1)$ $\cong_{T1_{16,3}}$ $C_{16}(Y_2)$ or 
	 
	 $(ii)$  $C_{16}(X_1)$ $\cong_{T1_{16,3}}$   $C_{16}(X_2)$ $\cong_{T2_{16,2,2}}$ $C_{16}(Y_2)$;
	 
\item [\rm (8.3)] $C_{16}(X_1)$ $\cong$ $C_{16}(X_2)$:	
	
	 $(i)$ $C_{16}(X_1)$ $\cong_{T1_{16,3}}$ $C_{16}(X_2)$ or 

     $(ii)$  $C_{16}(X_1)$ $\cong_{T2_{16,2,2}}$ $C_{16}(Y_1)$ $\cong_{T1_{16,3}}$ $C_{16}(Y_2)$ $\cong_{T2_{16,2,2}}$ $C_{16}(X_2)$; 		

\item [\rm (8.4)] $C_{16}(Y_1)$ $\cong$ $C_{16}(Y_2)$:
	
	$(i)$ $C_{16}(Y_1)$ $\cong_{T1_{16,3}}$ $C_{16}(Y_2)$ or 
	
	$(ii)$  $C_{16}(Y_1)$ $\cong_{T2_{16,2,2}}$ $C_{16}(X_1)$ $\cong_{T1_{16,3}}$   $C_{16}(X_2)$ $\cong_{T2_{16,2,2}}$ $C_{16}(Y_2)$;		

\item [\rm (8.5)] $C_{16}(Y_1)$ $\cong$ $C_{16}(X_2)$:

	$(i)$ $C_{16}(Y_1)$ $\cong_{T1_{16,3}}$ $C_{16}(Y_2)$ $\cong_{T2_{16,2,2}}$ $C_{16}(X_2)$ or 

    $(ii)$  $C_{16}(Y_1)$ $\cong_{T2_{16,2,2}}$ $C_{16}(X_1)$ $\cong_{T1_{16,3}}$ $C_{16}(X_2)$;		
	
\item [\rm (8.6)] $C_{16}(Y_2)$ $\cong$ $C_{16}(X_2)$:
	
	$(i)$ $C_{16}(Y_2)$ $\cong_{T2_{16,2,2}}$ $C_{16}(X_2)$ or 

   $(ii)$ $C_{16}(Y_2)$ $\cong_{T1_{16,3}}$ $C_{16}(Y_1)$ $\cong_{T2_{16,2,2}}$ $C_{16}(X_1)$ $\cong_{T1_{16,3}}$ $C_{16}(X_2)$. 	
\end{enumerate}

In all these cases, it is easy to find from their isomorphism digraphs that they don't have Eulerian isomorphism series. \hfill $\Box$
 
 \begin{prm} \quad \label{p6} {\rm Find the isomorphism digraph, the isomorphism graph and different isomorphism series corresponding to each circulant graph which has Type-2 isomorphism and of order 27. }
 \end{prm}
 
 \noindent
 {\bf Solution.}\quad In problem 6 in \cite{v2-2}, it is shown that the following 12 triples of circulant graphs, each of order 27, are the only isomorphic circulant graphs of Type-2 w.r.t. $m$ = 3.
 
 \begin{enumerate}
\item [\rm (1)]	$C_{27}(1,3,8,10)$, $C_{27}(2,3,7,11)$ and $C_{27}(1,3,8,10)$,

\hfill $C_{27}(1,3,8,10)$ $\cong_{T2_{27,3,1 }}$  $C_{27}(2,3,7,11)$ and $C_{27}(1,3,8,10)$ $\cong_{T2_{27,3,2}}$ $C_{27}(3,4,5,13)$; 	

\item [\rm (2)]	 $C_{27}(1,6,8,10)$, $C_{27}(2,6,7,11)$ and $C_{27}(1,6,8,10)$,
	
\hfill $C_{27}(1,6,8,10)$ $\cong_{T2_{27,3, 1}}$ $C_{27}(2,6,7,11)$ and $C_{27}(1,6,8,10)$ $\cong_{T2_{27,3, 2}}$ $C_{27}(4,5,6,13)$;
	
\item [\rm (3)]	$C_{27}(1,8,10,12)$, $C_{27}(2,7,11,12)$ and $C_{27}(1,8,10,12)$,

\hfill $C_{27}(1,8,10,12)$ $\cong_{T2_{27,3, 1}}$ $C_{27}(2,7,11,12)$ and $C_{27}(1,8,10,12)$ $\cong_{T2_{27,3,2 }}$ $C_{27}(4,5,12,13)$;

\item [\rm (4)] $C_{27}(1,3,6,8,10)$, $C_{27}(2,3,6,7,11)$ and $C_{27}(1,3,6,8,10)$,

\hfill $C_{27}(1,3,6,8,10)$ $\cong_{T2_{27,3, 1}}$ $C_{27}(2,3,6,7,11)$ and $C_{27}(1,3,6,8,10)$ $\cong_{T2_{27,3, 2}}$ $C_{27}(3,4,5,6,13)$;

\item [\rm (5)]	$C_{27}(1,3,8,9,10)$, $C_{27}(2,3,7,9,11)$ and $C_{27}(1,3,8,9,10)$,

\hfill $C_{27}(1,3,8,9,10)$ $\cong_{T2_{27,3, 1}}$ $C_{27}(2,3,7,9,11)$ and $C_{27}(1,3,8,9,10)$ $\cong_{T2_{27,3, 2}}$ $C_{27}(3,4,5,9,13)$;

\item [\rm (6)]	$C_{27}(1,3,8,10,12)$, $C_{27}(2,3,7,11,12)$ and $C_{27}(1,3,8,10,12)$,
\\
$C_{27}(1,3,8,10,12)$ $\cong_{T2_{27,3, 1}}$ $C_{27}(2,3,7,11,12)$ and $C_{27}(1,3,8,10,12)$ $\cong_{T2_{27,3, 2}}$ $C_{27}(3,4,5,12,13)$;

\item [\rm (7)]	$C_{27}(1,6,8,9,10)$, $C_{27}(2,6,7,9,11)$ and $C_{27}(1,6,8,9,10)$,

\hfill $C_{27}(1,6,8,9,10)$ $\cong_{T2_{27,3, 1}}$ $C_{27}(2,6,7,9,11)$ and $C_{27}(1,6,8,9,10)$ $\cong_{T2_{27,3, 2}}$ $C_{27}(4,5,6,9,13)$;

\item [\rm (8)]	$C_{27}(1,6,8,10,12)$, $C_{27}(2,6,7,11,12)$ and $C_{27}(1,6,8,10,12)$,
\\
$C_{27}(1,6,8,10,12)$ $\cong_{T2_{27,3, 1}}$ $C_{27}(2,6,7,11,12)$ and $C_{27}(1,6,8,10,12)$ $\cong_{T2_{27,3, 2}}$ $C_{27}(4,5,6,12,13)$;

\item [\rm (9)] $C_{27}(1,8,9,10,12)$, $C_{27}(2,7,9,11,12)$ and $C_{27}(1,8,9,10,12)$,
\\
$C_{27}(1,8,9,10,12)$ $\cong_{T2_{27,3, 1}}$ $C_{27}(2,7,9,11,12)$ and $C_{27}(1,8,9,10,12)$ $\cong_{T2_{27,3, 2}}$ $C_{27}(4,5,9,12,13)$;

\item [\rm (10)] $C_{27}(1,3,6,8,9,10)$, $C_{27}(2,3,6,7,9,11)$ and $C_{27}(1,3,6,8,9,10)$,

$C_{27}(1,3,6,8,9,10)$ $\cong_{T2_{27,3, 1}}$  
$C_{27}(2,3,6,7,9,11)$

\hfill and $C_{27}(1,3,6,8,9,10)$ $\cong_{T2_{27,3, 2}}$ $C_{27}(3,4,5,6,9,13)$;

\item [\rm (11)] $C_{27}(1,3,8,9,10,12)$, $C_{27}(3,4,5,9,12,13)$, and $C_{27}(1,3,8,9,10,12)$,

$C_{27}(1,3,8,9,10,12)$ $\cong_{T2_{27,3, }}$ $C_{27}(3,4,5,9,12,13)$  

\hfill and $C_{27}(1,3,8,9,10,12)$ $\cong_{T2_{27,3, 1}}$ $C_{27}(2,3,7,9,11,12)$;

\item [\rm (12)] $C_{27}(1,6,8,9,10,12)$, $C_{27}(2,6,7,9,11,12)$ and $C_{27}(1,6,8,9,10,12)$,

$C_{27}(1,6,8,9,10,12)$ $\cong_{T2_{27,3, 1}}$   
$C_{27}(2,6,7,9,11,12)$

\hfill and $C_{27}(1,6,8,9,10,12)$ $\cong_{T2_{27,3, 2}}$ $C_{27}(4,5,6,9,12,13)$.  
 \end{enumerate}

Let $\mathcal{D}$ be the isomorphism digraph $\mathcal{ID}_{27, 3}(C_{27}(R))$  and $\mathcal{G}$ be the isomorphism graph $\mathcal{I}_{27, 3}(C_{27}(R))$ w.r.t. $m$ = 3 for different $R$. Our aim is to find $\mathcal{D}$ and $\mathcal{G}$ for different $R$ where $V(\mathcal{D})$ = $Iso(C_{27}(R))$ = $\{C_{27}(S):$ $C_{27}(S)$ $\cong$ $C_{27}(R)\}$ = $V(\mathcal{G})$.  

In each case, we find the isomorphism digraph $\mathcal{ID}_{27, 3}(C_{27}(R))$ at first and then from this digraph, we obtain the isomorphism graph $\mathcal{I}_{27, 3}(C_{27}(R))$ as well as different isomorphism series. 

 \begin{enumerate}
 	\item [\rm (1)] Here, $C_{27}(R)$ = $C_{27}(1,3,8,10)$, 

$T2_{27,3 }(C_{27}(1,3,8,10))$ = $\{C_{27}(1,3,8,10), C_{27}(3,4,5,13) = T2_{27,3, 1}(C_{27}(1,3,8,10))$,

\hfill $C_{27}(2,3,7,11) = T2_{27,3, 2}(C_{27}(1,3,8,10))\}$ 

\hfill = $T2_{27,3 }(C_{27}(3,4,5,13))$ = $T2_{27,3 }(C_{27}(2,3,7,11))$,

$T1_{27}(C_{27}(1,3,8,10))$ = $\{C_{27}(1,3,8,10), C_{27}(2,6,7,11) = T1_{27, 2}(C_{27}(1,3,8,10))$,

\hfill $C_{27}(4,5,12,13) = T1_{27, 4}(C_{27}(1,3,8,10))\}$ 

\hfill = $T1_{27}(C_{27}(2,6,7,11))$ = $T1_{27}(C_{27}(4,5,12,13))$,

$T2_{27,3 }(C_{27}(2,6,7,11))$ = $\{C_{27}(2,6,7,11), C_{27}(1,6,8,10) = T2_{27,3, 1}(C_{27}(2,6,7,11))$,

\hfill $C_{27}(4,5,6,13) = T2_{27,3, 2}(C_{27}(2,6,7,11))\}$ 

\hfill = $T2_{27,3 }(C_{27}(1,6,8,10))$ = $T2_{27,3 }(C_{27}(4,5,6,13))$,

$T1_{27}(C_{27}(1,6,8,10))$ = $\{C_{27}(1,6,8,10), C_{27}(2,7,11,12) = T1_{27, 2}(C_{27}1,6,8,10))$,

\hfill $C_{27}(3,4,5,13) = T1_{27, 4}(C_{27}(1,6,8,10))\}$ 

\hfill = $T1_{27}(C_{27}(2,7,11,12))$ = $T1_{27}(C_{27}(3,4,5,13))$,

$T2_{27,3 }(C_{27}(2,7,11,12))$ = $\{C_{27}(2,7,11,12), C_{27}(1,8,10,12) = T2_{27,3, 1}(C_{27}(2,7,11,12))$,

\hfill $C_{27}(4,5,12,13) = T2_{27,3, 2}(C_{27}(2,7,11,12))\}$ 

\hfill = $T2_{27,3 }(C_{27}(1,8,10,12))$ = $T2_{27,3 }(C_{27}(4,5,12,13))$,

$T1_{27}(C_{27}(1,8,10,12))$ = $\{C_{27}(1,8,10,12), C_{27}(2,3,7,11) = T1_{27, 2}(C_{27}(1,8,10,12))$,

\hfill $C_{27}(4,5,6,13) = T1_{27, 4}(C_{27}(1,8,10,12))\}$ 

\hfill = $T1_{27}(C_{27}(2,3,7,11))$ = $T1_{27}(C_{27}(4,5,6,13))$ and
\\
 $V(\mathcal{D})$ = $V(\mathcal{G})$ = $Iso(C_{27}(R))$ = $Iso(C_{27}(1,3,8,10))$ = $Iso(C_{27}(2,6,7,11))$ = $Iso(C_{27}(4,5,12,13))$   
 	
 \hfill 	 = $\{C_{27}(1,3,8,10), C_{27}(2,6,7,11),  C_{27}(4,5,12,13), C_{27}(3,4,5,13),  C_{27}(1,6,8,10)$, 

 	\hfill	$C_{27}(2,7,11,12), C_{27}(2,3,7,11), C_{27}(4,5,6,13), C_{27}(1,8,10,12)\}$
    
 \hfill   = $Iso(C_{27}(3,4,5,13))$ = $Iso(C_{27}(1,6,8,10))$ = $Iso(C_{27}(2,7,11,12))$ 

\hfill    = $Iso(C_{27}((2,3,7,11))$ = $Iso(C_{27}(4,5,6,13))$ = $Iso(C_{27}(1,8,10,12))$. 
    
    Corresponding to $V(\mathcal{D})$ = $V(\mathcal{G})$, we draw the isomorphism digraph $\mathcal{ID}_{27, 3}(C_{27}(1,3,8,10))$ w.r.t. $m$ = 3 and from this digraph we obtain the isomorphism graph $\mathcal{ID}_{27, 3}(C_{27}(1,3,8,10))$ w.r.t. $m$ = 3. The isomorphism digraph  $\mathcal{D}$ = $\mathcal{ID}_{27, 3}(C_{27}(1,3,8,10))$ = $\mathcal{ID}_{27, 3}(C_{27}(2,6,7,11))$ = $\mathcal{ID}_{27, 3}(C_{27}(4,5,12,13))$ = $\mathcal{ID}_{27, 3}(C_{27}(3,4,5,13))$ = $\mathcal{ID}_{27, 3}(C_{27}(1,6,8,10))$ 
\\ = $\mathcal{ID}_{27, 3}(C_{27}(2,7,11,12))$ = $\mathcal{ID}_{27, 3}(C_{27}(2,3,7,11))$ = $\mathcal{ID}_{27, 3}(C_{27}(4,5,6,13))$ 
\\ = $\mathcal{ID}_{27, 3}(C_{27}(1,8,10,12))$ and the isomorphism graph
\\
 $\mathcal{G}$ = $\mathcal{I}_{27, 3}(C_{27}(1,3,8,10))$ = $\mathcal{I}_{27, 3}(C_{27}(2,6,7,11))$ = $\mathcal{I}_{27, 3}(C_{27}(4,5,12,13))$ 
 
 = $\mathcal{I}_{27, 3}(C_{27}(3,4,5,13))$ = $\mathcal{I}_{27, 3}(C_{27}(1,6,8,10))$ = $\mathcal{I}_{27, 3}(C_{27}(2,7,11,12))$ 
 
 = $\mathcal{I}_{27, 3}(C_{27}(2,3,7,11))$ = $\mathcal{I}_{27, 3}(C_{27}(4,5,6,13))$ = $\mathcal{I}_{27, 3}(C_{27}(1,8,10,12))$ are given in Figures 15 and 16, respectively. 
\end{enumerate}

{\rm 	\begin{center}
		\begin{tikzpicture}  
		[scale=.37,auto=center,every node/.style={draw,circle}] 
		
		\node (1) at (0,0) {\tiny{$C_{27}(R_1)$}};
		\node (2) at (9,6) {\tiny{$C_{27}(T_1)$}};	
		\node (3) at (-9,6){\tiny{$C_{27}(S_1)$}};

\draw[->, line width=0.2mm] [blue](1)[dashed] to [out=50,in=200] (2);
\draw[->, line width=0.2mm] [blue](2)[dashed] to [out=230,in=20] (1);
\draw[->, line width=0.2mm] [blue](2)[dashed] to [out=175,in=5] (3);
\draw[->, line width=0.2mm] [blue](3)[dashed] to [out=355,in=185](2);
\draw[->, line width=0.2mm] [blue](3)[dashed] to [out=340,in=140](1);
\draw[->, line width=0.2mm] [blue](1)[dashed] to [out=160,in=310](3);
		
\node [scale=.0] (11) at (6,3) [label=90:\tiny{$T2_{27,3,2}$}]{};
\node [scale=.0] (12) at (3.5,-1) [label=90:\tiny{$T2_{27,3,1}$}]{};
\node [scale=.0] (13) at (-4.5,5) [label=90:\tiny{$T2_{27,3,2}$}]{};
\node [scale=.0] (14) at (3.25,3.5) [label=90:\tiny{$T2_{27,3,1}$}]{};
\node [scale=.0] (15) at (-7.2,1.75)[label=90:\tiny{$T2_{27,3,1}$}]{};
\node [scale=.0] (16) at (-3,1) [label=90:\tiny{$T2_{27,3,2}$}]{};
		\node (4) at (0,10) {\tiny{$C_{27}(R_2)$}};
		\node (5) at (9,16) {\tiny{$C_{27}(T_2)$}};	
		\node (6) at (-9,16){\tiny{$C_{27}(S_2)$}};
		
\draw[->, line width=0.2mm] [blue](4)[dashed] to [out=50,in=200] (5);
\draw[->, line width=0.2mm] [blue](5)[dashed] to [out=230,in=20] (4);
\draw[->, line width=0.2mm] [blue](5)[dashed] to [out=175,in=5] (6);
\draw[->, line width=0.2mm] [blue](6)[dashed] to [out=355,in=185] (5);
\draw[->, line width=0.2mm] [blue](6)[dashed] to [out=340,in=140] (4);
\draw[->, line width=0.2mm] [blue](4)[dashed] to [out=160,in=320] (6);
		
\node [scale=.0] (17) at (6.2,13) [label=90:\tiny{$T2_{27,3,2}$}]{};
\node [scale=.0] (18) at (3.2,9.2) [label=90:\tiny{$T2_{27,3,1}$}]{};
\node [scale=.0] (19) at (-5.5,15)[label=90:\tiny{$T2_{27,3,2}$}]{};
\node [scale=.0] (20) at (5,14.2) [label=90:\tiny{$T2_{27,3,1}$}]{};
\node [scale=.0](21) at (-6.5,12.3)[label=90:\tiny{$T2_{27,3,1}$}]{};
\node [scale=.0] (22) at (-2,10) [label=90:\tiny{$T2_{27,3,2}$}]{};
		\node (7) at (0,20) {\tiny{$C_{27}(R_3)$}};
		\node (8) at (9,26) {\tiny{$C_{27}(T_3)$}};	
		\node (9) at (-9,26){\tiny{$C_{27}(S_3)$}};
		
\draw[->, line width=0.2mm] [blue](7)[dashed] to [out=50,in=200] (8);
\draw[->, line width=0.2mm] [blue](8)[dashed] to [out=230,in=20] (7);
\draw[->, line width=0.2mm] [blue](8)[dashed] to [out=175,in=5] (9);
\draw[->, line width=0.2mm] [blue](9)[dashed] to [out=355,in=185] (8);
\draw[->, line width=0.2mm] [blue](9)[dashed] to [out=340,in=140] (7);
\draw[->, line width=0.2mm] [blue](7)[dashed] to [out=160,in=320] (9);
		
\node [scale=.0] (17) at (5.5,23) [label=90:\tiny{$T2_{27,3,2}$}]{};
\node [scale=.0] (18) at (3.5,19.25)[label=90:\tiny{$T2_{27,3,1}$}]{};
\node [scale=.0](19) at (-5,25.1) [label=90:\tiny{$T2_{27,3,2}$}]{};
\node [scale=.0] (20) at (2,24.1) [label=90:\tiny{$T2_{27,3,1}$}]{};
\node [scale=.0] (21) at (-6.5,22.3)[label=90:\tiny{$T2_{27,3,1}$}]{};
\node [scale=.0](22) at (-1.25,20.2)[label=90:\tiny{$T2_{27,3,2}$}]{};
		\draw[->, line width=0.2mm] [blue](1) to [out=110,in=250] (4);
		\draw[->, line width=0.2mm] [blue](4) to [out=290,in=70] (1);
		\draw[->, line width=0.2mm] [blue](1) to [out=120,in=240] (7);
		\draw[->, line width=0.2mm] [blue](7) to [out=300,in=60] (1);
  	\draw[->, line width=0.2mm] [blue](2) to [out=110,in=250] (5);
		\draw[->, line width=0.2mm] [blue](5) to [out=290,in=70] (2);
		\draw[->, line width=0.2mm] [blue](2) to [out=120,in=240] (8);
		\draw[->, line width=0.2mm] [blue](8) to [out=300,in=60] (2);
  	\draw[->, line width=0.2mm] [blue](3) to [out=110,in=250] (6);
		\draw[->, line width=0.2mm] [blue](6) to [out=290,in=70] (3);
		\draw[->, line width=0.2mm] [blue](3) to [out=120,in=240] (9);
		\draw[->, line width=0.2mm] [blue](9) to [out=300,in=60] (3);
 		\draw[->, line width=0.2mm] [blue](4) to [out=110,in=250] (7);
		\draw[->, line width=0.2mm] [blue](7) to [out=290,in=70] (4);
		\draw[->, line width=0.2mm] [blue](5) to [out=110,in=250] (8);
		\draw[->, line width=0.2mm] [blue](8) to [out=290,in=70] (5);
		\draw[->, line width=0.2mm] [blue](6) to [out=110,in=250] (9);
		\draw[->, line width=0.2mm] [blue](9) to [out=290,in=70] (6);
		
\node [scale=.0] (39) at (-1,6) [label=90:\tiny{$T1_{27,2}$}]{};
\node [scale=.0] (45) at (1,2.5) [label=90:\tiny{$T1_{27,4}$}]{};
\node [scale=.0] (43) at (-3.5,7.5) [label=90:\tiny{$T1_{27,4}$}]{};
\node [scale=.0] (40) at (4.5,7.5) [label=90:\tiny{$T1_{27,2}$}]{};
\node [scale=.0] (33) at (10,9) [label=90:\tiny{$T1_{27,4}$}]{};
\node [scale=.0] (44) at (8,11) [label=90:\tiny{$T1_{27,2}$}]{};
\node [scale=.0] (35) at (6.5,19) [label=90:\tiny{$T1_{27,4}$}]{};
\node [scale=.0] (34) at (12.5,13) [label=90:\tiny{$T1_{27,2}$}]{};

\node [scale=.0] (47) at (-10,11) [label=90:\tiny{$T1_{27,2}$}]{};
\node [scale=.0] (42) at (-8,9) [label=90:\tiny{$T1_{27,4}$}]{};
\node [scale=.0] (41) at (-12,19) [label=90:\tiny{$T1_{27,4}$}]{};
\node [scale=.0] (36) at (-6,10) [label=90:\tiny{$T1_{27,2}$}]{};

\node [scale=.0] (46) at (1.5,12) [label=90:\tiny{$T1_{27,4}$}]{};
\node [scale=.0] (37) at (-1,13) [label=90:\tiny{$T1_{27,2}$}]{};
\node [scale=.0] (31) at (8.5,21) [label=90:\tiny{$T1_{27,2}$}]{};
\node [scale=.0] (38) at (10,18) [label=90:\tiny{$T1_{27,4}$}]{};
\node [scale=.0] (31) at (-9.5,21) [label=90:\tiny{$T1_{27,2}$}]{};
\node [scale=.0] (32) at (-8,18) [label=90:\tiny{$T1_{27,4}$}]{};
\end{tikzpicture}		
		
 {\footnotesize  \hspace{1cm} Fig.  15.   Digraph $\mathcal{D}$ with   

 $R_1$ = $\{1,3,8,10\}$, $S_1$ = $\{3,4,5,13\}$, $T_1$ = $\{2,3,7,11\}$,  
  
$R_2$ = $\{2,6,7,11\}$, $S_2$ = $\{1,6,8,10\}$, $T_2$ = $\{4,5,6,13\}$, 
  
 ~ $R_3$ = $\{4,5,12,13\}$, $S_3$ = $\{2,7,11,12\}$,  $T_3$ = $\{1,8,10,12\}$. 
 }  
\end{center} }

\begin{enumerate} 
\item [\rm (2)] Here, $C_{27}(R)$ = $C_{27}(1,6,8,10)$,  

$T2_{27,3 }(C_{27}(1,6,8,10))$ = $\{C_{27}(1,6,8,10), C_{27}(4,5,6,13) = T2_{27,3, 1}(C_{27}(1,6,8,10))$,

\hfill $C_{27}(2,6,7,11) = T2_{27,3, 2}(C_{27}(1,6,8,10))\}$ 

\hfill = $T2_{27,3 }(C_{27}(4,5,6,13))$ = $T2_{27,3 }(C_{27}(2,6,7,11))$,

$T1_{27}(C_{27}(1,6,8,10))$ = $\{C_{27}(1,6,8,10), C_{27}(2,7,11,12) = T1_{27, 2}(C_{27}(1,6,8,10))$,

\hfill $C_{27}(3,4,5,13) = T1_{27, 4}(C_{27}(1,6,8,10))\}$ 

\hfill = $T1_{27}(C_{27}(2,7,11,12))$ = $T1_{27}(C_{27}(3,4,5,13))$,

$T2_{27,3 }(C_{27}(2,7,11,12))$ = $\{C_{27}(2,7,11,12), C_{27}(1,8,10,12) = T2_{27,3, 1}(C_{27}(2,7,11,12))$,

\hfill $C_{27}(4,5,12,13) = T2_{27,3, 2}(C_{27}(2,7,11,12))\}$ 

\hfill = $T2_{27,3 }(C_{27}(1,8,10,12))$ = $T2_{27,3 }(C_{27}(4,5,12,13))$,

$T1_{27}(C_{27}(1,8,10,12))$ = $\{C_{27}(1,8,10,12), C_{27}(2,3,7,11) = T1_{27, 2}(C_{27}1,8,10,12))$,

\hfill $C_{27}(4,5,6,13) = T1_{27, 4}(C_{27}(1,8,10,12))\}$ 

\hfill = $T1_{27}(C_{27}(2,3,7,11))$ = $T1_{27}(C_{27}(4,5,6,13))$,

$T2_{27,3 }(C_{27}(2,3,7,11))$ = $\{C_{27}(2,3,7,11), C_{27}(1,3,8,10) = T2_{27,3, 1}(C_{27}(2,3,7,11))$,

\hfill $C_{27}(3,4,5,13) = T2_{27,3, 2}(C_{27}(2,3,7,11))\}$ 

\hfill = $T2_{27,3 }(C_{27}(1,3,8,10))$ = $T2_{27,3 }(C_{27}(3,4,5,13))$,

$T1_{27}(C_{27}(1,3,8,10))$ = $\{C_{27}(1,3,8,10), C_{27}(2,6,7,11) = T1_{27, 2}(C_{27}(1,3.8,10))$,

\hfill $C_{27}(4,5,12,13) = T1_{27, 4}(C_{27}(1,3,8,10))\}$ 

\hfill = $T1_{27}(C_{27}(2,6,7,11))$ = $T1_{27}(C_{27}(4,5,12,13))$ and
\\
 $V(\mathcal{D})$ = $V(\mathcal{G})$ = $Iso(C_{27}(R))$ = $Iso(C_{27}(1,6,8,10))$ = $Iso(C_{27}(2,7,11,12))$ = $Iso(C_{27}(3,4,5,13))$   
 	
 \hfill 	 = $\{C_{27}(1,6,8,10), C_{27}(2,7,11,12),  C_{27}(3,4,5,13), C_{27}(4,5,6,13), C_{27}(1,8,10,12)$, 

 	\hfill	$C_{27}(2,3,7,11), C_{27}(2,6,7,11), C_{27}(4,5,12,13), C_{27}(1,3,8,10)\}$
    
 \hfill   = $Iso(C_{27}(4,5,6,13))$ = $Iso(C_{27}(1,8,10,12))$ = $Iso(C_{27}(2,3,7,11))$ 

\hfill    = $Iso(C_{27}((2,6,7,11))$ = $Iso(C_{27}(4,5,12,13))$ = $Iso(C_{27}(1,3,8,10))$. 
    
 	The isomorphism digraph $\mathcal{ID}_{27, 3}(C_{27}(1,6,8,10))$ = $\mathcal{ID}_{27, 3}(C_{27}(2,7,11,12))$ 
\\
= $\mathcal{ID}_{27, 3}(C_{27}(3,4,5,13))$ = $\mathcal{ID}_{27, 3}(C_{27}(4,5,6,13))$ = $\mathcal{ID}_{27, 3}(C_{27}(1,8,10,12))$ 
\\
= $\mathcal{ID}_{27, 3}(C_{27}(2,3,7,11))$ = $\mathcal{ID}_{27, 3}(C_{27}(2,6,7,11))$ = $\mathcal{ID}_{27, 3}(C_{27}(4,5,12,13))$ = \\
$\mathcal{ID}_{27, 3}(C_{27}(1,3,8,10))$ and the isomorphism graph $\mathcal{I}_{27, 3}(C_{27}(1,6,8,10))$ = $\mathcal{I}_{27, 3}(C_{27}(2,7,11,12))$ = $\mathcal{I}_{27, 3}(C_{27}(3,4,5,13))$ = $\mathcal{I}_{27, 3}(C_{27}(4,5,6,13))$ = $\mathcal{I}_{27, 3}(C_{27}(1,8,10,12))$ = $\mathcal{I}_{27, 3}(C_{27}(2,3,7,11))$ = $\mathcal{I}_{27, 3}(C_{27}(2,6,7,11))$ = $\mathcal{I}_{27, 3}(C_{27}(4,5,12,13))$ = $\mathcal{I}_{27, 3}(C_{27}(1,3,8,10))$ are given in Figures 15 and 16, respectively. 

{\rm 	\begin{center}
		\begin{tikzpicture}  
		[scale=.25, auto=center,every node/.style={draw,circle}] 

		\node (1) at (0,0) {\tiny{$C_{27}(R_1)$}};
		\node (2) at (9,6) {\tiny{$C_{27}(T_1)$}};	
		\node (3) at (-9,6) {\tiny{$C_{27}(S_1)$}};

\draw[line width=0.2mm] [blue](1)[dashed] to (2);
\draw[line width=0.2mm] [blue](2)[dashed] to (3);
\draw[line width=0.2mm] [blue](3)[dashed] to (1);
		
\node [scale=.0] (11) at (5.5,1) [label=90:\tiny{$T2_{27,3}$}]{};
\node [scale=.0] (13) at (3,5) [label=90:\tiny{$T2_{27,3}$}]{};
\node [scale=.0] (15) at (-5.5,1)[label=90:\tiny{$T2_{27,3}$}]{};
		\node (4) at (0,10) {\tiny{$C_{27}(R_2)$}};
		\node (5) at (9,16) {\tiny{$C_{27}(T_2)$}};	
		\node (6) at (-9,16)  {\tiny{$C_{27}(S_2)$}};
		
\draw[line width=0.2mm] [blue](4)[dashed] to (5);
\draw[line width=0.2mm] [blue](5)[dashed] to (6);
\draw[line width=0.2mm] [blue](6)[dashed] to (4);
		
\node [scale=.0] (17) at (4.5,11) [label=90:\tiny{$T2_{27,3}$}]{};
\node [scale=.0] (19) at (2.5,15)[label=90:\tiny{$T2_{27,3}$}]{};
\node [scale=.0](21) at (-5,11.5)[label=90:\tiny{$T2_{27,3}$}]{};
		\node (7) at (0,20) {\tiny{$C_{27}(R_3)$}};
		\node (8) at (9,26) {\tiny{$C_{27}(T_3)$}};	
		\node (9) at (-9,26)  {\tiny{$C_{27}(S_3)$}};
		
\draw[line width=0.2mm] [blue](7)[dashed] to (8);
\draw[line width=0.2mm] [blue](8)[dashed] to (9);
\draw[line width=0.2mm] [blue](9)[dashed] to (7);
		
\node [scale=.0] (17) at (4,22) [label=90:\tiny{$T2_{27,3}$}]{};
\node [scale=.0](19) at (0,25.1) [label=90:\tiny{$T2_{27,3}$}]{};
\node [scale=.0] (21) at (-3,21.5)[label=90:\tiny{$T2_{27,3}$}]{};
		\draw[line width=0.2mm] [blue](1) to (4);
		\draw[line width=0.2mm] [blue](1) to [out=120,in=240] (7);
		\draw[line width=0.2mm] [blue](2) to (5);
		\draw[line width=0.2mm] [blue](2) to [out=120,in=240] (8);
		\draw[line width=0.2mm] [blue](3) to (6);
		\draw[line width=0.2mm] [blue](3) to [out=120,in=240] (9);
		\draw[line width=0.2mm] [blue](4) to (7);
		\draw[line width=0.2mm] [blue](5) to (8);
		\draw[line width=0.2mm] [blue](6) to (9);
		
\node [scale=.0] (39) at (1,3) [label=90:\tiny{$T1_{27}$}]{};
\node [scale=.0] (43) at (-3.75,7.5) [label=90:\tiny{$T1_{27}$}]{};
\node [scale=.0] (33) at (10,9.5) [label=90:\tiny{$T1_{27}$}]{};
\node [scale=.0] (35) at (6.75,18) [label=90:\tiny{$T1_{27}$}]{};

\node [scale=.0] (47) at (-8,10) [label=90:\tiny{$T1_{27}$}]{};
\node [scale=.0] (41) at (-11,18.5) [label=90:\tiny{$T1_{27}$}]{};

\node [scale=.0] (46) at (1,13) [label=90:\tiny{$T1_{27}$}]{};
\node [scale=.0] (31) at (10,20) [label=90:\tiny{$T1_{27}$}]{};
\node [scale=.0] (31) at (-8,20) [label=90:\tiny{$T1_{27}$}]{};

\end{tikzpicture}		

 {\footnotesize  \hspace{1cm} Fig.  16.   Graph $\mathcal{G}$ with   

 $R_1$ = $\{1,3,8,10\}$, $S_1$ = $\{3,4,5,13\}$, $T_1$ = $\{2,3,7,11\}$,  
  
$R_2$ = $\{2,6,7,11\}$, $S_2$ = $\{1,6,8,10\}$, $T_2$ = $\{4,5,6,13\}$, 
  
 ~ $R_3$ = $\{4,5,12,13\}$, $S_3$ = $\{2,7,11,12\}$,  $T_3$ = $\{1,8,10,12\}$. 
 } 
\end{center} }
	
\item [\rm (3)]	Here, $C_{27}(R)$ = $C_{27}(1,8,10,12)$,  $C_{27}(1,8,10,12)$ $\cong_{T2_{27,3, 1}}$ $C_{27}(2,7,11,12)$ and $C_{27}(1,8,10,12)$ $\cong_{T2_{27,3,2 }}$ $C_{27}(4,5,12,13)$. From (1) and (2), we have $V(\mathcal{D})$ = $V(\mathcal{G})$ = $Iso(C_{27}(1,8,10,12))$ = $Iso(C_{27}(2,7,11,12))$ = $Iso(C_{27}(4,5,12,13))$ = $Iso(C_{27}(1,3,8,10))$ = $Iso(C_{27}(2,3,7,11))$ = $Iso(C_{27}(3,4,5,13))$ = $Iso(C_{27}(1,6,8,10))$ = $Iso(C_{27}(2,6,7,11))$ = $Iso(C_{27}(4,5,6,13))$ and hence the isomorphism digraph $\mathcal{ID}_{27, 3}(C_{27}(1,3,8,10))$ = $\mathcal{ID}_{27, 3}(C_{27}(2,3,7,11))$ 
\\
= $\mathcal{ID}_{27, 3}(C_{27}(3,4,5,13))$ = $\mathcal{ID}_{27, 3}(C_{27}(1,6,8,10))$ = $\mathcal{ID}_{27, 3}(C_{27}(2,6,7,11))$ 
    \\
    = $\mathcal{ID}_{27, 3}(C_{27}(4,5,6,13))$ = $\mathcal{ID}_{27, 3}(C_{27}(1,8,10,12))$ = $\mathcal{ID}_{27, 3}(C_{27}(2,7,11,12))$ = \\
    $\mathcal{ID}_{27, 3}(C_{27}(4,5,12,13))$ and the isomorphism graph $\mathcal{I}_{27, 3}(C_{27}(1,3,8,10))$ = $\mathcal{I}_{27, 3}(C_{27}(2,3,7,11))$ = $\mathcal{I}_{27, 3}(C_{27}(3,4,5,13))$ = $\mathcal{I}_{27, 3}(C_{27}(1,6,8,10))$ = $\mathcal{I}_{27, 3}(C_{27}(2,6,7,11))$ = $\mathcal{I}_{27, 3}(C_{27}(4,5,6,13))$ = $\mathcal{I}_{27, 3}(C_{27}(1,8,10,12))$ = $\mathcal{I}_{27, 3}(C_{27}(2,7,11,12))$ = $\mathcal{I}_{27, 3}(C_{27}(4,5,12,13))$ are same as in case (1) and (2) and are given in  Figures 15 and 16, respectively. 
    
\item [\rm (4)] Here, $C_{27}(R)$ = $C_{27}(1,3,6,8,10)$,  

$T2_{27,3 }(C_{27}(1,3,6,8,10))$ = $\{C_{27}(1,3,6,8,10), C_{27}(3,4,5,6,13) = T2_{27,3, 1}(C_{27}(1,3,6,8,10))$,

\hfill $C_{27}(2,3,6,7,11) = T2_{27,3, 2}(C_{27}(1,3,6,8,10))\}$ 

\hfill = $T2_{27,3 }(C_{27}(3,4,5,6,13))$ = $T2_{27,3 }(C_{27}(2,3,6,7,11))$,

$T1_{27}(C_{27}(1,3,6,8,10))$ = $\{C_{27}(1,3,6,8,10), C_{27}(2,6,7,11,12) = T1_{27, 2}(C_{27}(1,3,6,8,10))$,

\hfill $C_{27}(3,4,5,12,13) = T1_{27, 4}(C_{27}(1,3,6,8,10))\}$ 

\hfill = $T1_{27}(C_{27}(2,6,7,11,12))$ = $T1_{27}(C_{27}(3,4,5,12,13))$,

$T2_{27,3 }(C_{27}(2,6,7,11,12))$ = $\{C_{27}(2,6,7,11,12),$ 

\hfill $C_{27}(1,6,8,10,12) = T2_{27,3, 1}(C_{27}(2,6,7,11,12))$,

\hfill $C_{27}(4,5,6,12,13) = T2_{27,3, 2}(C_{27}(2,6,7,11,12))\}$ 

\hfill = $T2_{27,3 }(C_{27}(1,6,8,10,12))$ = $T2_{27,3 }(C_{27}(4,5,6,12,13))$,

$T1_{27}(C_{27}(1,6,8,10,12))$ = $\{C_{27}(1,6,8,10,12), C_{27}(2,3,7,11,12) = T1_{27, 2}(C_{27}1,6,8,10,12))$,

\hfill $C_{27}(3,4,5,6,13) = T1_{27, 4}(C_{27}(1,6,8,10,12))\}$ 

\hfill = $T1_{27}(C_{27}(2,3,7,11,12))$ = $T1_{27}(C_{27}(3,4,5,6,13))$,

$T2_{27,3 }(C_{27}(2,3,7,11,12))$ = $\{C_{27}(2,3,7,11,12),$ 

\hfill $C_{27}(1,3,8,10,12) = T2_{27,3, 1}(C_{27}(2,3,7,11,12))$,

\hfill $C_{27}(3,4,5,13,12) = T2_{27,3, 2}(C_{27}(2,3,7,11,12))\}$ 

\hfill = $T2_{27,3 }(C_{27}(1,3,8,10,12))$ = $T2_{27,3 }(C_{27}(3,4,5,12,13))$,

$T1_{27}(C_{27}(1,3,8,10,12))$ = $\{C_{27}(1,3,8,10,12), C_{27}(2,3,6,7,11) = T1_{27, 2}(C_{27}(1,3.8,10,12))$,

\hfill $C_{27}(4,5,6,12,13) = T1_{27, 4}(C_{27}(1,3,8,10,12))\}$ 

\hfill = $T1_{27}(C_{27}(2,3,6,7,11))$ = $T1_{27}(C_{27}(4,5,6,12,13))$ and
\\
 $V(\mathcal{D})$ = $V(\mathcal{G})$ = $Iso(C_{27}(R))$ = $Iso(C_{27}(1,3,6,8,10))$ = $Iso(C_{27}(2,6,7,11,12))$

\hfill  = $Iso(C_{27}(3,4,5,12,13))$  = $\{C_{27}(1,3,6,8,10), C_{27}(2,6,7,11,12),  C_{27}(3,4,5,12,13),$ 
 	
 \hfill 	 $ C_{27}(3,4,5,6,13), C_{27}(1,6,8,10,12)$, $C_{27}(2,3,7,11,12),$

 	\hfill	$C_{27}(2,3,6,7,11), C_{27}(4,5,6,12,13), C_{27}(1,3,8,10,12)\}$
    
 \hfill   = $Iso(C_{27}(3,4,5,6,13))$ = $Iso(C_{27}(1,6,8,10,12))$ = $Iso(C_{27}(2,3,7,11,12))$ 

\hfill    = $Iso(C_{27}((2,3,6,7,11))$ = $Iso(C_{27}(4,5,6,12,13))$ = $Iso(C_{27}(1,3,8,10,12))$. 
    
 	$\Rightarrow$ The isomorphism digraph $\mathcal{ID}_{27, 3}(C_{27}(1,3,6,8,10))$ = $\mathcal{ID}_{27, 3}(C_{27}(2,6,7,11,12))$ 

\hfill = $\mathcal{ID}_{27, 3}(C_{27}(3,4,5,12,13))$ = $\mathcal{ID}_{27, 3}(C_{27}(3,4,5,6,13))$ = $\mathcal{ID}_{27, 3}(C_{27}(1,6,8,10,12))$ 

\hfill = $\mathcal{ID}_{27, 3}(C_{27}(2,3,7,11,12))$ = $\mathcal{ID}_{27, 3}(C_{27}(2,3,6,7,11))$ = $\mathcal{ID}_{27, 3}(C_{27}(4,5,6,12,13))$ 

\hfill = $\mathcal{ID}_{27, 3}(C_{27}(1,3,8,10,12))$ and the isomorphism graph $\mathcal{I}_{27, 3}(C_{27}(1,3,6,8,10))$ 

\hfill = $\mathcal{I}_{27, 3}(C_{27}(2,6,7,11,12))$ = $\mathcal{I}_{27, 3}(C_{27}(3,4,5,6,13))$ = $\mathcal{I}_{27, 3}(C_{27}(3,4,5,6,13))$ 

\hfill = $\mathcal{I}_{27, 3}(C_{27}(1,6,8,10,12))$ = $\mathcal{I}_{27, 3}(C_{27}(2,3,7,11,12))$ = $\mathcal{I}_{27, 3}(C_{27}(2,3,6,7,11))$ 
\\
= $\mathcal{I}_{27, 3}(C_{27}(4,5,6,12,13))$ = $\mathcal{I}_{27, 3}(C_{27}(1,3,8,10,12))$  which are same as given in Figures 15 and 16, except one additional element is found in each jump size set $X$ of $C_{27}(X)$. The isomorphism digraph and graph are presented in Figures 17 and 18. Clearly, unlabeled digraphs in Figures 15 and 17 are same as the unlabeled graphs in Figures 16 and 18.
\end{enumerate}

{\rm 	\begin{center}
 		\begin{tikzpicture}  
 		[scale=.32,auto=center,every node/.style={draw,circle}] 
 		
 		\node (1) at (0,0) {\tiny{$C_{27}(U_1)$}};
		\node (2) at (9,6) {\tiny{$C_{27}(W_1)$}};	
		\node (3) at (-9,6){\tiny{$C_{27}(V_1)$}};

\draw[->, line width=0.2mm] [blue](1)[dashed] to [out=50,in=200] (2);
\draw[->, line width=0.2mm] [blue](2)[dashed] to [out=230,in=20] (1);
\draw[->, line width=0.2mm] [blue](2)[dashed] to [out=175,in=5] (3);
\draw[->, line width=0.2mm] [blue](3)[dashed] to [out=355,in=185](2);
\draw[->, line width=0.2mm] [blue](3)[dashed] to [out=340,in=140](1);
\draw[->, line width=0.2mm] [blue](1)[dashed] to [out=160,in=310](3);
		
\node [scale=.0] (11) at (5,2) [label=90:\tiny{$T2_{27,3,2}$}]{};
\node [scale=.0] (12) at (5,-1) [label=90:\tiny{$T2_{27,3,1}$}]{};
\node [scale=.0] (13) at (-4.5,4.5)[label=90:\tiny{$T2_{27,3,2}$}]{};
\node [scale=.0] (14) at (4.8,3.5)[label=90:\tiny{$T2_{27,3,1}$}]{};
\node [scale=.0](15) at (-6,.5)[label=90:\tiny{$T2_{27,3,1}$}]{};
\node [scale=.0] (16) at (-4,1.5) [label=90:\tiny{$T2_{27,3,2}$}]{};
		\node (4) at (0,10) {\tiny{$C_{27}(U_2)$}};
		\node (5) at (9,16) {\tiny{$C_{27}(W_2)$}};	
		\node (6) at (-9,16){\tiny{$C_{27}(V_2)$}};
		
\draw[->, line width=0.2mm] [blue](4)[dashed] to [out=50,in=200] (5);
\draw[->, line width=0.2mm] [blue](5)[dashed] to [out=230,in=20] (4);
\draw[->, line width=0.2mm] [blue](5)[dashed] to [out=175,in=5] (6);
\draw[->, line width=0.2mm] [blue](6)[dashed] to [out=355,in=185] (5);
\draw[->, line width=0.2mm] [blue](6)[dashed] to [out=340,in=140] (4);
\draw[->, line width=0.2mm] [blue](4)[dashed] to [out=160,in=320] (6);
		
\node [scale=.0] (17) at (5,12) [label=90:\tiny{$T2_{27,3,2}$}]{};
\node [scale=.0] (18) at (5.2,9.7) [label=90:\tiny{$T2_{27,3,1}$}]{};
\node [scale=.0](19) at (-3.5,14.7)[label=90:\tiny{$T2_{27,3,2}$}]{};
\node [scale=.0] (20) at (4.1,13.8)[label=90:\tiny{$T2_{27,3,1}$}]{};
\node [scale=.0](21) at (-6,11.5)[label=90:\tiny{$T2_{27,3,1}$}]{};
\node [scale=.0] (22) at (-2.5,10) [label=90:\tiny{$T2_{27,3,2}$}]{};
		\node (7) at (0,20) {\tiny{$C_{27}(U_3)$}};
		\node (8) at (9,26) {\tiny{$C_{27}(W_3)$}};	
		\node (9) at (-9,26){\tiny{$C_{27}(V_3)$}};
		
\draw[->, line width=0.2mm] [blue](7)[dashed] to [out=50,in=200] (8);
\draw[->, line width=0.2mm] [blue](8)[dashed] to [out=230,in=20] (7);
\draw[->, line width=0.2mm] [blue](8)[dashed] to [out=175,in=5] (9);
\draw[->, line width=0.2mm] [blue](9)[dashed] to [out=355,in=185] (8);
\draw[->, line width=0.2mm] [blue](9)[dashed] to [out=340,in=140] (7);
\draw[->, line width=0.2mm] [blue](7)[dashed] to [out=160,in=320] (9);
		
\node [scale=.0] (17) at (5,22) [label=90:\tiny{$T2_{27,3,2}$}]{};
\node [scale=.0](18) at (3.5,19.25)[label=90:\tiny{$T2_{27,3,1}$}]{};
\node [scale=.0](19) at (-4,24.7) [label=90:\tiny{$T2_{27,3,2}$}]{};
\node [scale=.0] (20) at (2,23.8) [label=90:\tiny{$T2_{27,3,1}$}]{};
\node [scale=.0](21) at (-5.5,21)[label=90:\tiny{$T2_{27,3,1}$}]{};
\node [scale=.0](22) at (-1.25,20.5)[label=90:\tiny{$T2_{27,3,2}$}]{};
		\draw[->, line width=0.2mm] [blue](1) to [out=110,in=250] (4);
		\draw[->, line width=0.2mm] [blue](4) to [out=290,in=70] (1);
		\draw[->, line width=0.2mm] [blue](1) to [out=120,in=240] (7);
		\draw[->, line width=0.2mm] [blue](7) to [out=300,in=60] (1);
  	\draw[->, line width=0.2mm] [blue](2) to [out=110,in=250] (5);
		\draw[->, line width=0.2mm] [blue](5) to [out=290,in=70] (2);
		\draw[->, line width=0.2mm] [blue](2) to [out=120,in=240] (8);
		\draw[->, line width=0.2mm] [blue](8) to [out=300,in=60] (2);
  	\draw[->, line width=0.2mm] [blue](3) to [out=110,in=250] (6);
		\draw[->, line width=0.2mm] [blue](6) to [out=290,in=70] (3);
		\draw[->, line width=0.2mm] [blue](3) to [out=120,in=240] (9);
		\draw[->, line width=0.2mm] [blue](9) to [out=300,in=60] (3);
 		\draw[->, line width=0.2mm] [blue](4) to [out=110,in=250] (7);
		\draw[->, line width=0.2mm] [blue](7) to [out=290,in=70] (4);
		\draw[->, line width=0.2mm] [blue](5) to [out=110,in=250] (8);
		\draw[->, line width=0.2mm] [blue](8) to [out=290,in=70] (5);
		\draw[->, line width=0.2mm] [blue](6) to [out=110,in=250] (9);
		\draw[->, line width=0.2mm] [blue](9) to [out=290,in=70] (6);
		
\node [scale=.0] (39) at (-1,5.5) [label=90:\tiny{$T1_{27,2}$}]{};
\node [scale=.0] (45) at (1,2.5) [label=90:\tiny{$T1_{27,4}$}]{};
\node [scale=.0] (43) at (-3.5,7.5) [label=90:\tiny{$T1_{27,4}$}]{};
\node [scale=.0] (40) at (4,7.5) [label=90:\tiny{$T1_{27,2}$}]{};
\node [scale=.0] (33) at (10,9) [label=90:\tiny{$T1_{27,4}$}]{};
\node [scale=.0] (44) at (8,11) [label=90:\tiny{$T1_{27,2}$}]{};
\node [scale=.0] (35) at (6.5,19) [label=90:\tiny{$T1_{27,4}$}]{};
\node [scale=.0] (34) at (12.5,13) [label=90:\tiny{$T1_{27,2}$}]{};

\node [scale=.0] (47) at (-10,11) [label=90:\tiny{$T1_{27,2}$}]{};
\node [scale=.0] (42) at (-8,9) [label=90:\tiny{$T1_{27,4}$}]{};
\node [scale=.0] (41) at (-12,19) [label=90:\tiny{$T1_{27,4}$}]{};
\node [scale=.0] (36) at (-6,10) [label=90:\tiny{$T1_{27,2}$}]{};

\node [scale=.0] (46) at (1.5,12) [label=90:\tiny{$T1_{27,4}$}]{};
\node [scale=.0] (37) at (-1,13) [label=90:\tiny{$T1_{27,2}$}]{};
\node [scale=.0] (31) at (8.5,21) [label=90:\tiny{$T1_{27,2}$}]{};
\node [scale=.0] (38) at (10,18) [label=90:\tiny{$T1_{27,4}$}]{};
\node [scale=.0] (31) at (-9.5,21) [label=90:\tiny{$T1_{27,2}$}]{};
\node [scale=.0] (32) at (-8,18) [label=90:\tiny{$T1_{27,4}$}]{};
				
		\node (1) at (26,2) {\tiny{$C_{27}(U_1)$}};
		\node (2) at (33,7) {\tiny{$C_{27}(W_1)$}};	
		\node (3) at (19,7)  {\tiny{$C_{27}(V_1)$}};

\draw[line width=0.2mm] [blue](1)[dashed] to (2);
\draw[line width=0.2mm] [blue](2)[dashed] to (3);
\draw[line width=0.2mm] [blue](3)[dashed] to (1);
		
\node [scale=.0] (11) at (31,2) [label=90:\tiny{$T2_{27,3}$}]{};
\node [scale=.0] (13) at (28.5,5.5) [label=90:\tiny{$T2_{27,3}$}]{};
\node [scale=.0] (15) at (22.5,2.5)[label=90:\tiny{$T2_{27,3}$}]{};
		\node (4) at (26,10.5) {\tiny{$C_{27}(U_2)$}};
		\node (5) at (33,15.5) {\tiny{$C_{27}(W_2)$}};	
		\node (6) at (19,15.5)  {\tiny{$C_{27}(V_2)$}};
		
\draw[line width=0.2mm] [blue](4)[dashed] to (5);
\draw[line width=0.2mm] [blue](5)[dashed] to (6);
\draw[line width=0.2mm] [blue](6)[dashed] to (4);
		
\node [scale=.0] (17) at (28.5,10.75) [label=90:\tiny{$T2_{27,3}$}]{};
\node [scale=.0] (19) at (28.5,14)[label=90:\tiny{$T2_{27,3}$}]{};
\node [scale=.0](21) at (22.5,11.5)[label=90:\tiny{$T2_{27,3}$}]{};
		\node (7) at (26,19) {\tiny{$C_{27}(U_3)$}};
		\node (8) at (33,24) {\tiny{$C_{27}(W_3)$}};	
		\node (9) at (19,24)  {\tiny{$C_{27}(V_3)$}};
		
\draw[line width=0.2mm] [blue](7)[dashed] to (8);
\draw[line width=0.2mm] [blue](8)[dashed] to (9);
\draw[line width=0.2mm] [blue](9)[dashed] to (7);
		
\node [scale=.0] (17) at (23,20) [label=90:\tiny{$T2_{27,3}$}]{};
\node [scale=.0] (19) at (29,20)[label=90:\tiny{$T2_{27,3}$}]{};
\node [scale=.0](21) at (26.5,22.5) [label=90:\tiny{$T2_{27,3}$}]{};
		\draw[line width=0.2mm] [blue](1) to (4);
		\draw[line width=0.2mm] [blue](1) to [out=120,in=240] (7);
		\draw[line width=0.2mm] [blue](2) to (5);
		\draw[line width=0.2mm] [blue](2) to [out=120,in=240] (8);
		\draw[line width=0.2mm] [blue](3) to (6);
		\draw[line width=0.2mm] [blue](3) to [out=120,in=240] (9);
		\draw[line width=0.2mm] [blue](4) to (7);
		\draw[line width=0.2mm] [blue](5) to (8);
		\draw[line width=0.2mm] [blue](6) to (9);
		
\node [scale=.0] (39) at (27.25,4) [label=90:\tiny{$T1_{27}$}]{};
\node [scale=.0] (43) at (22,8) [label=90:\tiny{$T1_{27}$}]{};
\node [scale=.0] (33) at (34.2,9.5) [label=90:\tiny{$T1_{27}$}]{};

\node [scale=.0] (47) at (20.2,10) [label=90:\tiny{$T1_{27}$}]{};
\node [scale=.0] (46) at (27.2,12.5) [label=90:\tiny{$T1_{27}$}]{};
\node [scale=.0] (35) at (34.2,18) [label=90:\tiny{$T1_{27}$}]{};

\node [scale=.0] (41) at (16,19) [label=90:\tiny{$T1_{27}$}]{};
\node [scale=.0] (32) at (20.2,18) [label=90:\tiny{$T1_{27}$}]{};
\node [scale=.0] (31) at (30.5,18) [label=90:\tiny{$T1_{27}$}]{};
 \end{tikzpicture}		

	\vspace{.1cm}		
 {\small  \hspace{1.5cm}	Fig.  17.   Digraph $\mathcal{D}$  \hspace{5.5cm} Figure  18. Graph $\mathcal{G}$ 
\\		  
 \vspace{.1cm}	
   \hspace{1.4cm} with	$U_1$ = $\{1,3,6,8,10\}$, ~$V_1$ = $\{3,4,5,6,13\}$, ~$W_1$ = $\{2,3,6,7,11\}$,   
 			
	\hspace{2.5cm}	$U_2$ = $\{2,6,7,11,12\}$, $V_2$ = $\{1,6,8,10,12\}$, $W_2$ = $\{4,5,6,12,13\}$,  
 			
	\hspace{2.5cm}	$U_3$ = $\{3,4,5,12,13\}$, $V_3$ = $\{2,3,7,11,12\}$, $W_3$ = $\{1,3,8,10,12\}$.  
} 
\end{center} }

\begin{enumerate}
\item [\rm (5)] Here, $C_{27}(R)$ = $C_{27}(1,3,8,9,10)$,  

$T2_{27,3 }(C_{27}(1,3,8,9,10))$ = $\{C_{27}(1,3,8,9,10), C_{27}(3,4,5,9,13) = T2_{27,3, 1}(C_{27}(1,3,8,9,10))$,

\hfill $C_{27}(2,3,7,9,11) = T2_{27,3, 2}(C_{27}(1,3,8,9,10))\}$ 

\hfill = $T2_{27,3 }(C_{27}(3,4,5,9,13))$ = $T2_{27,3 }(C_{27}(2,3,7,9,11))$,

$T1_{27}(C_{27}(1,3,8,9,10))$ = $\{C_{27}(1,3,8,9,10), C_{27}(2,6,7,9,11) = T1_{27, 2}(C_{27}(1,3,8,9,10))$,

\hfill $C_{27}(4,5,9,12,13) = T1_{27, 4}(C_{27}(1,3,8,9,10))\}$ 

\hfill = $T1_{27}(C_{27}(2,6,7,9,11))$ = $T1_{27}(C_{27}(4,5,9,12,13))$,

$T2_{27,3 }(C_{27}(2,6,7,9,11))$ = $\{C_{27}(2,6,7,9,11), C_{27}(1,6,8,9,10) = T2_{27,3, 1}(C_{27}(2,6,7,9,11))$,

\hfill $C_{27}(4,5,6,9,13) = T2_{27,3, 2}(C_{27}(2,6,7,9,11))\}$ 

\hfill = $T2_{27,3 }(C_{27}(1,6,8,9,10))$ = $T2_{27,3 }(C_{27}(4,5,6,9,13))$,

$T1_{27}(C_{27}(1,6,8,9,10))$ = $\{C_{27}(1,6,8,9,10), C_{27}(2,7,9,11,12) = T1_{27, 2}(C_{27}1,6,8,9,10))$,

\hfill $C_{27}(3,4,5,9,13) = T1_{27, 4}(C_{27}(1,6,8,9,10))\}$ 

\hfill = $T1_{27}(C_{27}(2,7,9,11,12))$ = $T1_{27}(C_{27}(3,4,5,9,13))$,

$T2_{27,3 }(C_{27}(2,7,9,11,12))$ = $\{C_{27}(2,7,9,11,12),$

\hfill $C_{27}(1,8,9,10,12) = T2_{27,3, 1}(C_{27}(2,7,9,11,12))$, 

\hfill $C_{27}(4,5,9,13,12) = T2_{27,3, 2}(C_{27}(2,7,9,11,12))\}$ 

\hfill = $T2_{27,3 }(C_{27}(1,8,9,10,12))$ = $T2_{27,3 }(C_{27}(4,5,9,12,13))$,

$T1_{27}(C_{27}(1,8,9,10,12))$ = $\{C_{27}(1,8,9,10,12), C_{27}(2,3,7,9,11) = T1_{27, 2}(C_{27}(1,8,9,10,12))$,

\hfill $C_{27}(4,5,6,9,13) = T1_{27, 4}(C_{27}(1,8,9,10,12))\}$ 

\hfill = $T1_{27}(C_{27}(2,3,7,9,11))$ = $T1_{27}(C_{27}(4,5,6,9,13))$ and
\\
 $V(\mathcal{D})$ = $V(\mathcal{G})$ = $Iso(C_{27}(R))$ = $Iso(C_{27}(1,3,8,9,10))$ = $Iso(C_{27}(2,6,7,9,11))$ 

\hfill = $Iso(C_{27}(4,5,9,12,13))$ = $\{C_{27}(1,3,8,9,10), C_{27}(2,6,7,9,11),  C_{27}(4,5,9,12,13),$   
 	
 \hfill 	 $C_{27}(3,4,5,9,13), C_{27}(1,6,8,9,10)$, $C_{27}(2,7,9,11,12),$

 	\hfill	$C_{27}(2,3,7,9,11), C_{27}(4,5,6,9,13), C_{27}(1,8,9,10,12)\}$
    
 \hfill   = $Iso(C_{27}(3,4,5,9,13))$ = $Iso(C_{27}(1,6,8,9,10))$ = $Iso(C_{27}(2,7,9,11,12))$ 

\hfill    = $Iso(C_{27}((2,3,7,9,11))$ = $Iso(C_{27}(4,5,6,9,13))$ = $Iso(C_{27}(1,8,9,10,12))$. 
    
  $\Rightarrow$	The isomorphism digraph $\mathcal{ID}_{27, 3}(C_{27}(1,3,8,9,10))$ = $\mathcal{ID}_{27, 3}(C_{27}(2,6,7,9,11))$ 

\hfill = $\mathcal{ID}_{27, 3}(C_{27}(4,5,9,12,13))$ = $\mathcal{ID}_{27, 3}(C_{27}(3,4,5,9,13))$ = $\mathcal{ID}_{27, 3}(C_{27}(1,6,8,9,10))$ 

\hfill = $\mathcal{ID}_{27, 3}(C_{27}(2,7,9,11,12))$ = $\mathcal{ID}_{27, 3}(C_{27}(2,3,7,9,11))$ = $\mathcal{ID}_{27, 3}(C_{27}(4,5,6,9,13))$  

\hfill = $\mathcal{ID}_{27, 3}(C_{27}(1,8,9,10,12))$ and the isomorphism graph $\mathcal{I}_{27, 3}(C_{27}(1,3,8,9,10))$ 

\hfill = $\mathcal{I}_{27, 3}(C_{27}(2,6,7,9,11))$ = $\mathcal{I}_{27, 3}(C_{27}(4,5,9,12,13))$ = $\mathcal{I}_{27, 3}(C_{27}(3,4,5,9,13))$ 

\hfill  = $\mathcal{I}_{27, 3}(C_{27}(1,6,8,9,10))$ = $\mathcal{I}_{27, 3}(C_{27}(2,7,9,11,12))$ = $\mathcal{I}_{27, 3}(C_{27}(2,3,7,9,11))$ 

\hfill = $\mathcal{I}_{27, 3}(C_{27}(4,5,6,9,13))$ = $\mathcal{I}_{27, 3}(C_{27}(1,8,9,10,12))$ are given in Figures 19 and 20. 

\item [\rm (6)] Here, $C_{27}(R)$ = $C_{27}(1,3,8,10,12)$, 

$T2_{27,3 }(C_{27}(1,3,8,10,12))$ = $\{C_{27}(1,3,8,10,12),$ 

\hfill $C_{27}(3,4,5,12,13) = T2_{27,3, 1}(C_{27}(1,3,8,10,12))$,

\hfill $C_{27}(2,3,7,11,12) = T2_{27,3, 2}(C_{27}(1,3,8,10,12))\}$ 

\hfill = $T2_{27,3 }(C_{27}(3,4,5,12,13))$ = $T2_{27,3 }(C_{27}(2,3,7,11,12))$,

$T1_{27}(C_{27}(1,3,8,10,12))$ = $\{C_{27}(1,3,8,10,12), C_{27}(2,3,6,7,11) = T1_{27, 2}(C_{27}(1,3,8,10,12))$,

\hfill $C_{27}(4,5,6,12,13) = T1_{27, 4}(C_{27}(1,3,8,10,12))\}$ 

\hfill = $T1_{27}(C_{27}(2,3,6,7,11))$ = $T1_{27}(C_{27}(4,5,6,12,13))$,

$T2_{27,3 }(C_{27}(2,3,6,7,11))$ = $\{C_{27}(2,3,6,7,11), C_{27}(1,3,6,8,10) = T2_{27,3, 1}(C_{27}(2,3,6,7,11))$,

\hfill $C_{27}(3,4,5,6,13) = T2_{27,3, 2}(C_{27}(2,3,6,7,11))\}$ 

\hfill = $T2_{27,3 }(C_{27}(1,3,6,8,10))$ = $T2_{27,3 }(C_{27}(3,4,5,6,13))$,

$T1_{27}(C_{27}(1,3,6,8,10))$ = $\{C_{27}(1,3,6,8,10), C_{27}(2,6,7,11,12) = T1_{27, 2}(C_{27}1,3,6,8,10))$,

\hfill $C_{27}(3,4,5,12,13) = T1_{27, 4}(C_{27}(1,3,6,8,10))\}$ 

\hfill = $T1_{27}(C_{27}(2,6,7,11,12))$ = $T1_{27}(C_{27}(3,4,5,12,13))$,

$T2_{27,3 }(C_{27}(2,6,7,11,12))$ = $\{C_{27}(2,6,7,11,12),$ 

\hfill $C_{27}(1,6,8,10,12) = T2_{27,3, 1}(C_{27}(2,6,7,11,12))$,

\hfill $C_{27}(4,5,6,12,13) = T2_{27,3, 2}(C_{27}(2,6,7,11,12))\}$ 

\hfill = $T2_{27,3 }(C_{27}(1,6,8,10,12))$ = $T2_{27,3 }(C_{27}(4,5,6,12,13))$,

$T1_{27}(C_{27}(1,6,8,10,12))$ = $\{C_{27}(1,6,8,10,12), C_{27}(2,3,7,11,12) = T1_{27, 2}(C_{27}(1,6,8,10,12))$,

\hfill $C_{27}(3,4,5,6,13) = T1_{27, 4}(C_{27}(1,6,8,10,12))\}$ 

\hfill = $T1_{27}(C_{27}(2,3,7,11,12))$ = $T1_{27}(C_{27}(3,4,5,6,13))$ and
\\
 $V(\mathcal{D})$ = $V(\mathcal{G})$ = $Iso(C_{27}(R))$ = $Iso(C_{27}(1,3,8,10,12))$ 

\hfill = $Iso(C_{27}(2,3,6,7,11))$ = $Iso(C_{27}(4,5,6,12,13))$  = $\{C_{27}(1,3,8,10,12), C_{27}(2,3,6,7,11), $ 
 	
 \hfill 	 $C_{27}(4,5,6,12,13), C_{27}(3,4,5,12,13),  C_{27}(1,3,6,8,10), C_{27}(2,6,7,11,12)$, 

 	\hfill	$C_{27}(2,3,7,11,12), C_{27}(3,4,5,6,13), C_{27}(1,6,8,10,12)\}$
    
 \hfill   = $Iso(C_{27}(3,4,5,12,13))$ = $Iso(C_{27}(1,3,6,8,10))$ = $Iso(C_{27}(2,6,7,11,12))$ 

\hfill    = $Iso(C_{27}((2,3,7,11,12))$ = $Iso(C_{27}(3,4,5,6,13))$ = $Iso(C_{27}(1,6,8,10,12))$. 
    
    Corresponding to $V(\mathcal{D})$ = $V(\mathcal{G})$, we draw the isomorphism digraph $\mathcal{ID}_{27, 3}(C_{27}(1,3,8,10,12))$ w.r.t. $m$ = 3 and from this digraph we obtain the isomorphism graph $\mathcal{ID}_{27, 3}(C_{27}(1,3,8,10,12))$ w.r.t. $m$ = 3. The isomorphism digraph is $\mathcal{ID}_{27, 3}(C_{27}(1,3,8,10,12))$ = $\mathcal{ID}_{27, 3}(C_{27}(2,3,6,7,11))$ = $\mathcal{ID}_{27, 3}(C_{27}(4,5,6,12,13))$ = $\mathcal{ID}_{27, 3}(C_{27}(3,4,5,12,13))$ = $\mathcal{ID}_{27, 3}(C_{27}(1,3,6,8,10))$ 

\hfill = $\mathcal{ID}_{27, 3}(C_{27}(2,6,7,11,12))$ = $\mathcal{ID}_{27, 3}(C_{27}(2,3,7,11,12))$ = $\mathcal{ID}_{27, 3}(C_{27}(3,4,5,6,13))$ \\
= $\mathcal{ID}_{27, 3}(C_{27}(1,6,8,10,12))$ and the isomorphism graph $\mathcal{I}_{27, 3}(C_{27}(1,3,8,10,12))$ 

\hfill = $\mathcal{I}_{27, 3}(C_{27}(2,3,6,7,11))$ = $\mathcal{I}_{27, 3}(C_{27}(4,5,6,12,13))$ = $\mathcal{I}_{27, 3}(C_{27}(3,4,5,12,13))$ 

\hfill = $\mathcal{I}_{27, 3}(C_{27}(1,3,6,8,10))$ = $\mathcal{I}_{27, 3}(C_{27}(2,6,7,11,12))$ = $\mathcal{I}_{27, 3}(C_{27}(2,3,7,11,12))$ 

\hfill = $\mathcal{I}_{27, 3}(C_{27}(3,4,5,6,13))$ = $\mathcal{I}_{27, 3}(C_{27}(1,6,8,10,12))$ are given in Figures 17 and 18 which are same as graphs given in Figures 15 and 16, respectively, except one additional element is found in each jump size $X$ of $C_{27}(X)$, the vertex labels of the graphs in this case. \\

\item [\rm (7)]	In this case, $C_{27}(R)$ = $C_{27}(1,6,8,9,10)$,  $C_{27}(1,6,8,9,10)$ $\cong_{T2_{27,3, 1}}$ $C_{27}(2,6,7,9,11)$ and 
 	
 	$C_{27}(1,6,8,9,10)$ $\cong_{T2_{27,3, 2}}$ $C_{27}(4,5,6,9,13)$. This is similar to case (5). In this case,  
 	\\
 	$V(\mathcal{D})$ = $V(\mathcal{G})$ = $Iso(C_{27}(1,6,8,9,10))$ = $Iso(C_{27}(2,6,7,9,11))$ = $Iso(C_{27}(4,5,6,9,13))$ 
 	
 	= $Iso(C_{27}(1,8,9,10,12))$ = $Iso(C_{27}(2,7,9,11,12))$ = $Iso(C_{27}(4,5,9,12,13))$ 
 	
 	= $Iso(C_{27}(1,3,8,9,10))$ = $Iso(C_{27}(2,3,7,9,11))$ = $Iso(C_{27}(3,4,5,9,13))$, 
\\
the isomorphism digraph $\mathcal{ID}_{27, 3}(C_{27}(1,6,8,9,10))$ = $\mathcal{ID}_{27, 3}(C_{27}(2,6,7,9,11))$ 

= $\mathcal{ID}_{27, 3}(C_{27}(4,5,6,9,13))$ = $\mathcal{ID}_{27, 3}(C_{27}(1,8,9,10,12,))$ = $\mathcal{ID}_{27, 3}(C_{27}(2,7,9,11,12))$ 

= $\mathcal{ID}_{27, 3}(C_{27}(4,5,9,12,13))$ = $\mathcal{ID}_{27, 3}(C_{27}(1,3,8,9,10))$ = $\mathcal{ID}_{27, 3}(C_{27}(2,3,7,9,11))$ 

 = $\mathcal{ID}_{27, 3}(C_{27}(3,4,5,9,13))$ and 	the isomorphism graph $\mathcal{I}_{27, 3}(C_{27}(1,6,8,9,10))$ 

= $\mathcal{I}_{27, 3}(C_{27}(2,6,7,9,11))$ = $\mathcal{I}_{27, 3}(C_{27}(4,5,6,9,13))$ = $\mathcal{I}_{27, 3}(C_{27}(1,8,9,10,12,))$ 

= $\mathcal{I}_{27, 3}(C_{27}(2,7,9,11,12))$ = $\mathcal{I}_{27, 3}(C_{27}(4,5,9,12,13))$ = $\mathcal{I}_{27, 3}(C_{27}(1,3,8,9,10))$ 

= $\mathcal{I}_{27, 3}(C_{27}(2,3,7,9,11))$ = $\mathcal{I}_{27, 3}(C_{27}(3,4,5,9,13))$. The isomorphism digraph 

$\mathcal{D}$ = $\mathcal{ID}_{27, 3}(C_{27}(1,6,8,9,10))$ and the isomorphism graph $\mathcal{G}$ = $\mathcal{I}_{27, 3}(C_{27}(1,6,8,9,10))$ are given in Figures 19 and 20.
\end{enumerate}

{\rm 	\begin{center}
 		\begin{tikzpicture}  
 		[scale=.32,auto=center,every node/.style={draw,circle}] 
 		
 		\node (1) at (0,0) {\tiny{$C_{27}(R_1)$}};
		\node (2) at (9,6) {\tiny{$C_{27}(T_1)$}};	
		\node (3) at (-9,6){\tiny{$C_{27}(S_1)$}};

\draw[->, line width=0.2mm] [blue](1)[dashed] to [out=50,in=200] (2);
\draw[->, line width=0.2mm] [blue](2)[dashed] to [out=230,in=20] (1);
\draw[->, line width=0.2mm] [blue](2)[dashed] to [out=175,in=5] (3);
\draw[->, line width=0.2mm] [blue](3)[dashed] to [out=355,in=185](2);
\draw[->, line width=0.2mm] [blue](3)[dashed] to [out=340,in=140](1);
\draw[->, line width=0.2mm] [blue](1)[dashed] to [out=160,in=310](3);
		
\node [scale=.0] (11) at (5,2) [label=90:\tiny{$T2_{27,3,2}$}]{};
\node [scale=.0] (12) at (5,-1) [label=90:\tiny{$T2_{27,3,1}$}]{};
\node [scale=.0] (13) at (-4.5,4.5)[label=90:\tiny{$T2_{27,3,2}$}]{};
\node [scale=.0] (14) at (4.8,3.5)[label=90:\tiny{$T2_{27,3,1}$}]{};
\node [scale=.0](15) at (-6,.5)[label=90:\tiny{$T2_{27,3,1}$}]{};
\node [scale=.0] (16) at (-4,1.5) [label=90:\tiny{$T2_{27,3,2}$}]{};
		\node (4) at (0,10) {\tiny{$C_{27}(R_2)$}};
		\node (5) at (9,16) {\tiny{$C_{27}(T_2)$}};	
		\node (6) at (-9,16){\tiny{$C_{27}(S_2)$}};
		
\draw[->, line width=0.2mm] [blue](4)[dashed] to [out=50,in=200] (5);
\draw[->, line width=0.2mm] [blue](5)[dashed] to [out=230,in=20] (4);
\draw[->, line width=0.2mm] [blue](5)[dashed] to [out=175,in=5] (6);
\draw[->, line width=0.2mm] [blue](6)[dashed] to [out=355,in=185] (5);
\draw[->, line width=0.2mm] [blue](6)[dashed] to [out=340,in=140] (4);
\draw[->, line width=0.2mm] [blue](4)[dashed] to [out=160,in=320] (6);
		
\node [scale=.0] (17) at (5,12) [label=90:\tiny{$T2_{27,3,2}$}]{};
\node [scale=.0] (18) at (5.2,9.7) [label=90:\tiny{$T2_{27,3,1}$}]{};
\node [scale=.0](19) at (-3.5,14.7)[label=90:\tiny{$T2_{27,3,2}$}]{};
\node [scale=.0] (20) at (4.1,13.8)[label=90:\tiny{$T2_{27,3,1}$}]{};
\node [scale=.0](21) at (-6,11.5)[label=90:\tiny{$T2_{27,3,1}$}]{};
\node [scale=.0] (22) at (-2.5,10) [label=90:\tiny{$T2_{27,3,2}$}]{};
		\node (7) at (0,20) {\tiny{$C_{27}(R_3)$}};
		\node (8) at (9,26) {\tiny{$C_{27}(T_3)$}};	
		\node (9) at (-9,26){\tiny{$C_{27}(S_3)$}};
		
\draw[->, line width=0.2mm] [blue](7)[dashed] to [out=50,in=200] (8);
\draw[->, line width=0.2mm] [blue](8)[dashed] to [out=230,in=20] (7);
\draw[->, line width=0.2mm] [blue](8)[dashed] to [out=175,in=5] (9);
\draw[->, line width=0.2mm] [blue](9)[dashed] to [out=355,in=185] (8);
\draw[->, line width=0.2mm] [blue](9)[dashed] to [out=340,in=140] (7);
\draw[->, line width=0.2mm] [blue](7)[dashed] to [out=160,in=320] (9);
		
\node [scale=.0] (17) at (5,22) [label=90:\tiny{$T2_{27,3,2}$}]{};
\node [scale=.0](18) at (3.5,19.25)[label=90:\tiny{$T2_{27,3,1}$}]{};
\node [scale=.0](19) at (-4,24.7) [label=90:\tiny{$T2_{27,3,2}$}]{};
\node [scale=.0] (20) at (2,23.8) [label=90:\tiny{$T2_{27,3,1}$}]{};
\node [scale=.0](21) at (-5.5,21)[label=90:\tiny{$T2_{27,3,1}$}]{};
\node [scale=.0](22) at (-1.25,20.5)[label=90:\tiny{$T2_{27,3,2}$}]{};
		\draw[->, line width=0.2mm] [blue](1) to [out=110,in=250] (4);
		\draw[->, line width=0.2mm] [blue](4) to [out=290,in=70] (1);
		\draw[->, line width=0.2mm] [blue](1) to [out=120,in=240] (7);
		\draw[->, line width=0.2mm] [blue](7) to [out=300,in=60] (1);
  	\draw[->, line width=0.2mm] [blue](2) to [out=110,in=250] (5);
		\draw[->, line width=0.2mm] [blue](5) to [out=290,in=70] (2);
		\draw[->, line width=0.2mm] [blue](2) to [out=120,in=240] (8);
		\draw[->, line width=0.2mm] [blue](8) to [out=300,in=60] (2);
  	\draw[->, line width=0.2mm] [blue](3) to [out=110,in=250] (6);
		\draw[->, line width=0.2mm] [blue](6) to [out=290,in=70] (3);
		\draw[->, line width=0.2mm] [blue](3) to [out=120,in=240] (9);
		\draw[->, line width=0.2mm] [blue](9) to [out=300,in=60] (3);
 		\draw[->, line width=0.2mm] [blue](4) to [out=110,in=250] (7);
		\draw[->, line width=0.2mm] [blue](7) to [out=290,in=70] (4);
		\draw[->, line width=0.2mm] [blue](5) to [out=110,in=250] (8);
		\draw[->, line width=0.2mm] [blue](8) to [out=290,in=70] (5);
		\draw[->, line width=0.2mm] [blue](6) to [out=110,in=250] (9);
		\draw[->, line width=0.2mm] [blue](9) to [out=290,in=70] (6);
		
\node [scale=.0] (39) at (-1,5.5) [label=90:\tiny{$T1_{27,2}$}]{};
\node [scale=.0] (45) at (1,2.5) [label=90:\tiny{$T1_{27,4}$}]{};
\node [scale=.0] (43) at (-3.5,7.5) [label=90:\tiny{$T1_{27,4}$}]{};
\node [scale=.0] (40) at (4,7.5) [label=90:\tiny{$T1_{27,2}$}]{};
\node [scale=.0] (33) at (10,9) [label=90:\tiny{$T1_{27,4}$}]{};
\node [scale=.0] (44) at (8,11) [label=90:\tiny{$T1_{27,2}$}]{};
\node [scale=.0] (35) at (6.5,19) [label=90:\tiny{$T1_{27,4}$}]{};
\node [scale=.0] (34) at (12.5,13) [label=90:\tiny{$T1_{27,2}$}]{};

\node [scale=.0] (47) at (-10,11) [label=90:\tiny{$T1_{27,2}$}]{};
\node [scale=.0] (42) at (-8,9) [label=90:\tiny{$T1_{27,4}$}]{};
\node [scale=.0] (41) at (-12,19) [label=90:\tiny{$T1_{27,4}$}]{};
\node [scale=.0] (36) at (-6,10) [label=90:\tiny{$T1_{27,2}$}]{};

\node [scale=.0] (46) at (1.5,12) [label=90:\tiny{$T1_{27,4}$}]{};
\node [scale=.0] (37) at (-1,13) [label=90:\tiny{$T1_{27,2}$}]{};
\node [scale=.0] (31) at (8.5,21) [label=90:\tiny{$T1_{27,2}$}]{};
\node [scale=.0] (38) at (10,18) [label=90:\tiny{$T1_{27,4}$}]{};
\node [scale=.0] (31) at (-9.5,21) [label=90:\tiny{$T1_{27,2}$}]{};
\node [scale=.0] (32) at (-8,18) [label=90:\tiny{$T1_{27,4}$}]{};
				
		\node (1) at (26,2) {\tiny{$C_{27}(R_1)$}};
		\node (2) at (33,7) {\tiny{$C_{27}(T_1)$}};	
		\node (3) at (19,7)  {\tiny{$C_{27}(S_1)$}};

\draw[line width=0.2mm] [blue](1)[dashed] to (2);
\draw[line width=0.2mm] [blue](2)[dashed] to (3);
\draw[line width=0.2mm] [blue](3)[dashed] to (1);
		
\node [scale=.0] (11) at (31,2) [label=90:\tiny{$T2_{27,3}$}]{};
\node [scale=.0] (13) at (28.5,5.5) [label=90:\tiny{$T2_{27,3}$}]{};
\node [scale=.0] (15) at (22.5,2.5)[label=90:\tiny{$T2_{27,3}$}]{};
		\node (4) at (26,10.5) {\tiny{$C_{27}(R_2)$}};
		\node (5) at (33,15.5) {\tiny{$C_{27}(T_2)$}};	
		\node (6) at (19,15.5)  {\tiny{$C_{27}(S_2)$}};
		
\draw[line width=0.2mm] [blue](4)[dashed] to (5);
\draw[line width=0.2mm] [blue](5)[dashed] to (6);
\draw[line width=0.2mm] [blue](6)[dashed] to (4);
		
\node [scale=.0] (17) at (28.5,10.75) [label=90:\tiny{$T2_{27,3}$}]{};
\node [scale=.0] (19) at (28.5,14)[label=90:\tiny{$T2_{27,3}$}]{};
\node [scale=.0](21) at (22.5,11.5)[label=90:\tiny{$T2_{27,3}$}]{};
		\node (7) at (26,19) {\tiny{$C_{27}(R_3)$}};
		\node (8) at (33,24) {\tiny{$C_{27}(T_3)$}};	
		\node (9) at (19,24)  {\tiny{$C_{27}(S_3)$}};
		
\draw[line width=0.2mm] [blue](7)[dashed] to (8);
\draw[line width=0.2mm] [blue](8)[dashed] to (9);
\draw[line width=0.2mm] [blue](9)[dashed] to (7);
		
\node [scale=.0] (17) at (23,20) [label=90:\tiny{$T2_{27,3}$}]{};
\node [scale=.0] (19) at (29,20)[label=90:\tiny{$T2_{27,3}$}]{};
\node [scale=.0](21) at (26.5,22.5) [label=90:\tiny{$T2_{27,3}$}]{};
		\draw[line width=0.2mm] [blue](1) to (4);
		\draw[line width=0.2mm] [blue](1) to [out=120,in=240] (7);
		\draw[line width=0.2mm] [blue](2) to (5);
		\draw[line width=0.2mm] [blue](2) to [out=120,in=240] (8);
		\draw[line width=0.2mm] [blue](3) to (6);
		\draw[line width=0.2mm] [blue](3) to [out=120,in=240] (9);
		\draw[line width=0.2mm] [blue](4) to (7);
		\draw[line width=0.2mm] [blue](5) to (8);
		\draw[line width=0.2mm] [blue](6) to (9);
		
\node [scale=.0] (39) at (27.25,4) [label=90:\tiny{$T1_{27}$}]{};
\node [scale=.0] (43) at (22,8) [label=90:\tiny{$T1_{27}$}]{};
\node [scale=.0] (33) at (34.2,9.5) [label=90:\tiny{$T1_{27}$}]{};

\node [scale=.0] (47) at (20.2,10) [label=90:\tiny{$T1_{27}$}]{};
\node [scale=.0] (46) at (27.2,12.5) [label=90:\tiny{$T1_{27}$}]{};
\node [scale=.0] (35) at (34.2,18) [label=90:\tiny{$T1_{27}$}]{};

\node [scale=.0] (41) at (16,19) [label=90:\tiny{$T1_{27}$}]{};
\node [scale=.0] (32) at (20.2,18) [label=90:\tiny{$T1_{27}$}]{};
\node [scale=.0] (31) at (30.5,18) [label=90:\tiny{$T1_{27}$}]{};
 \end{tikzpicture}		
	
	\vspace{.1cm}		
 {\small  \hspace{1.5cm}	Fig.  19.   Digraph $\mathcal{D}$  \hspace{5.5cm} Figure  20. Graph $\mathcal{G}$ 
\\		  
 \vspace{.1cm}	
   \hspace{1.4cm} with	$R_1$ = $\{1,3,8,9,10\}$, ~$S_1$ = $\{3,4,5,9,13\}$, ~$T_1$ = $\{2,3,7,9,11\}$,   
 			
	\hspace{2.1cm}	$R_2$ = $\{2,6,7,9,11\}$, ~$S_2$ = $\{1,6,8,9,10\}$, ~$T_2$ =  $\{4,5,6,9,13\}$,  
 			
	\hspace{2.4cm}	$R_3$ = $\{4,5,9,12,13\}$, $S_3$ = $\{2,7,9,11,12\}$, $T_3$ = $\{1,8,9,10,12\}$.   
}		
\end{center} }

\begin{enumerate} 		 		
\item [\rm (8)]	Here, $C_{27}(R)$ = $C_{27}(1,6,8,10,12)$,	$C_{27}(1,6,8,10,12)$ $\cong_{T2_{27,3, 1}}$ $C_{27}(2,6,7,11,12)$ and 

      $C_{27}(1,6,8,10,12)$ $\cong_{T2_{27,3, 2}}$ $C_{27}(4,5,6,12,13)$. This is similar to case (6) and the isomorphism digraph and the isomorphism graph are given in Figures 17 and 18..

\item [\rm (9)] Here, $C_{27}(R)$ = $C_{27}(1,8,9,10,12)$, $C_{27}(1,8,9,10,12)$ $\cong_{T2_{27,3, 1}}$ $C_{27}(2,7,9,11,12)$ and 
 	
 	$C_{27}(1,8,9,10,12)$ $\cong_{T2_{27,3, 2}}$ $C_{27}(4,5,9,12,13)$. This is similar to case (7) and the isomorphism digraph and the isomorphism graph are given in Figures 19 and 20.

\item [\rm (10)] Here, $C_{27}(R)$ = $C_{27}(1,3,6,8,9,10)$,  

$T2_{27,3 }(C_{27}(1,3,6,8,9,10))$ = $\{C_{27}(1,3,6,8,9,10)$, 

\hfill $C_{27}(3,4,5,6,9,13) = T2_{27,3, 1}(C_{27}(1,3,6,8,9,10))$,

\hfill $C_{27}(2,3,6,7,9,11) = T2_{27,3, 2}(C_{27}(1,3,6,8,9,10))\}$ 

\hfill = $T2_{27,3 }(C_{27}(3,4,5,6,9,13))$ = $T2_{27,3 }(C_{27}(2,3,6,7,9,11))$,

$T1_{27}(C_{27}(1,3,6,8,9,10))$ = $\{C_{27}(1,3,6,8,9,10)$, 

\hfill $C_{27}(2,6,7,9,11,12) = T1_{27, 2}(C_{27}(1,3,6,8,9,10))$,

\hfill $C_{27}(3,4,5,9,12,13) = T1_{27, 4}(C_{27}(1,3,6,8,9,10))\}$ 

\hfill = $T1_{27}(C_{27}(2,6,7,9,11,12))$ = $T1_{27}(C_{27}(3,4,5,9,12,13))$,

$T2_{27,3 }(C_{27}(2,6,7,9,11,12))$ = $\{C_{27}(2,6,7,9,11,12)$, 

\hfill $C_{27}(1,6,8,9,10,12) = T2_{27,3, 1}(C_{27}(2,6,7,9,11,12))$,

\hfill $C_{27}(4,5,6,9,12,13) = T2_{27,3, 2}(C_{27}(2,6,7,9,11,12))\}$ 

\hfill = $T2_{27,3 }(C_{27}(1,6,8,9,10,12))$ = $T2_{27,3 }(C_{27}(4,5,6,9,12,13))$,

$T1_{27}(C_{27}(1,6,8,9,10,12))$ = $\{C_{27}(1,6,8,9,10,12)$, 

\hfill $C_{27}(2,3,7,9,11,12) = T1_{27, 2}(C_{27}1,6,8,9,10,12))$,

\hfill $C_{27}(3,4,5,6,9,13) = T1_{27, 4}(C_{27}(1,6,8,9,10,12))\}$ 

\hfill = $T1_{27}(C_{27}(2,3,7,9,11,12))$ = $T1_{27}(C_{27}(3,4,5,6,9,13))$,

$T2_{27,3 }(C_{27}(2,3,7,9,11,12))$ = $\{C_{27}(2,3,7,9,11,12),$

\hfill $C_{27}(1,3,8,9,10,12) = T2_{27,3, 1}(C_{27}(2,3,7,9,11,12))$, 

\hfill $C_{27}(3,4,5,9,13,12) = T2_{27,3, 2}(C_{27}(2,3,7,9,11,12))\}$ 

\hfill = $T2_{27,3 }(C_{27}(1,3,8,9,10,12))$ = $T2_{27,3 }(C_{27}(3,4,5,9,12,13))$,

$T1_{27}(C_{27}(1,3,8,9,10,12))$ = $\{C_{27}(1,3,8,9,10,12)$, 

\hfill $C_{27}(2,3,6,7,9,11) = T1_{27, 2}(C_{27}(1,3,8,9,10,12))$,

\hfill $C_{27}(4,5,6,9,12,13) = T1_{27, 4}(C_{27}(1,3,8,9,10,12))\}$ 

\hfill = $T1_{27}(C_{27}(2,3,6,7,9,11))$ = $T1_{27}(C_{27}(3,4,5,6,9,12,13))$ and
\\
 $V(\mathcal{D})$ = $V(\mathcal{G})$ = $Iso(C_{27}(R))$ = $Iso(C_{27}(1,3,6,8,9,10))$ = $Iso(C_{27}(2,6,7,9,11,12))$ 

\hfill = $Iso(C_{27}(3,4,5,9,12,13))$ = $\{C_{27}(1,3,6,8,9,10), C_{27}(2,6,7,9,11,12),  C_{27}(3,4,5,9,12,13),$   
 	
 \hfill 	 $C_{27}(3,4,5,6,9,13), C_{27}(1,6,8,9,10,12)$, $C_{27}(2,3,7,9,11,12),$

 	\hfill	$C_{27}(2,3,6,7,9,11), C_{27}(4,5,6,9,12,13), C_{27}(1,3,8,9,10,12)\}$
    
 \hfill   = $Iso(C_{27}(3,4,5,6,9,13))$ = $Iso(C_{27}(1,6,8,9,10,12))$ = $Iso(C_{27}(2,3,6,7,9,11,12))$ 

\hfill    = $Iso(C_{27}((2,3,7,9,11))$ = $Iso(C_{27}(4,5,6,9,12,13))$ = $Iso(C_{27}(1,3,8,9,10,12))$. 
 \\   
  $\Rightarrow$	The isomorphism digraph $\mathcal{ID}_{27, 3}(C_{27}(1,3,6,8,9,10))$ = $\mathcal{ID}_{27, 3}(C_{27}(2,6,7,9,11,12))$ 

\hfill = $\mathcal{ID}_{27, 3}(C_{27}(3,4,5,9,12,13))$ = $\mathcal{ID}_{27, 3}(C_{27}(3,4,5,6,9,13))$ = $\mathcal{ID}_{27, 3}(C_{27}(1,6,8,9,10,12))$ 

\hfill = $\mathcal{ID}_{27, 3}(C_{27}(2,3,7,9,11,12))$ = $\mathcal{ID}_{27, 3}(C_{27}(2,3,6,7,9,11))$ = $\mathcal{ID}_{27, 3}(C_{27}(4,5,6,9,12,13))$  

\hfill = $\mathcal{ID}_{27, 3}(C_{27}(1,3,8,9,10,12))$ and the isomorphism graph $\mathcal{I}_{27, 3}(C_{27}(1,3,6,8,9,10))$ 

\hfill = $\mathcal{I}_{27, 3}(C_{27}(2,6,7,9,11,12))$ = $\mathcal{I}_{27, 3}(C_{27}(3,4,5,9,12,13))$ = $\mathcal{I}_{27, 3}(C_{27}(3,4,5,6,9,13))$ 

\hfill  = $\mathcal{I}_{27, 3}(C_{27}(1,6,8,9,10,12))$ = $\mathcal{I}_{27, 3}(C_{27}(2,3,7,9,11,12))$ = $\mathcal{I}_{27, 3}(C_{27}(2,3,6,7,9,11))$ 

 = $\mathcal{I}_{27, 3}(C_{27}(4,5,6,9,12,13))$ = $\mathcal{I}_{27, 3}(C_{27}(1,3,8,9,10,12))$ and are given in Figures 21 and 22, respectively. 
\end{enumerate}

{\rm 	\begin{center}
 		\begin{tikzpicture}  
 		[scale=.32,auto=center,every node/.style={draw,circle}] 
 		
 		\node (1) at (0,0) {\tiny{$C_{27}(R_1)$}};
		\node (2) at (9,6) {\tiny{$C_{27}(T_1)$}};	
		\node (3) at (-9,6){\tiny{$C_{27}(S_1)$}};

\draw[->, line width=0.2mm] [blue](1)[dashed] to [out=50,in=200] (2);
\draw[->, line width=0.2mm] [blue](2)[dashed] to [out=230,in=20] (1);
\draw[->, line width=0.2mm] [blue](2)[dashed] to [out=175,in=5] (3);
\draw[->, line width=0.2mm] [blue](3)[dashed] to [out=355,in=185](2);
\draw[->, line width=0.2mm] [blue](3)[dashed] to [out=340,in=140](1);
\draw[->, line width=0.2mm] [blue](1)[dashed] to [out=160,in=310](3);
		
\node [scale=.0] (11) at (5,2) [label=90:\tiny{$T2_{27,3,2}$}]{};
\node [scale=.0] (12) at (5,-1) [label=90:\tiny{$T2_{27,3,1}$}]{};
\node [scale=.0] (13) at (-4.5,4.5)[label=90:\tiny{$T2_{27,3,2}$}]{};
\node [scale=.0] (14) at (4.8,3.5)[label=90:\tiny{$T2_{27,3,1}$}]{};
\node [scale=.0](15) at (-6,.5)[label=90:\tiny{$T2_{27,3,1}$}]{};
\node [scale=.0] (16) at (-4,1.5) [label=90:\tiny{$T2_{27,3,2}$}]{};
		\node (4) at (0,10) {\tiny{$C_{27}(R_2)$}};
		\node (5) at (9,16) {\tiny{$C_{27}(T_2)$}};	
		\node (6) at (-9,16){\tiny{$C_{27}(S_2)$}};
		
\draw[->, line width=0.2mm] [blue](4)[dashed] to [out=50,in=200] (5);
\draw[->, line width=0.2mm] [blue](5)[dashed] to [out=230,in=20] (4);
\draw[->, line width=0.2mm] [blue](5)[dashed] to [out=175,in=5] (6);
\draw[->, line width=0.2mm] [blue](6)[dashed] to [out=355,in=185] (5);
\draw[->, line width=0.2mm] [blue](6)[dashed] to [out=340,in=140] (4);
\draw[->, line width=0.2mm] [blue](4)[dashed] to [out=160,in=320] (6);
		
\node [scale=.0] (17) at (5,12) [label=90:\tiny{$T2_{27,3,2}$}]{};
\node [scale=.0] (18) at (5.2,9.7) [label=90:\tiny{$T2_{27,3,1}$}]{};
\node [scale=.0](19) at (-3.5,14.7)[label=90:\tiny{$T2_{27,3,2}$}]{};
\node [scale=.0] (20) at (4.1,13.8)[label=90:\tiny{$T2_{27,3,1}$}]{};
\node [scale=.0](21) at (-6,11.5)[label=90:\tiny{$T2_{27,3,1}$}]{};
\node [scale=.0] (22) at (-2.5,10) [label=90:\tiny{$T2_{27,3,2}$}]{};
		\node (7) at (0,20) {\tiny{$C_{27}(R_3)$}};
		\node (8) at (9,26) {\tiny{$C_{27}(T_3)$}};	
		\node (9) at (-9,26){\tiny{$C_{27}(S_3)$}};
		
\draw[->, line width=0.2mm] [blue](7)[dashed] to [out=50,in=200] (8);
\draw[->, line width=0.2mm] [blue](8)[dashed] to [out=230,in=20] (7);
\draw[->, line width=0.2mm] [blue](8)[dashed] to [out=175,in=5] (9);
\draw[->, line width=0.2mm] [blue](9)[dashed] to [out=355,in=185] (8);
\draw[->, line width=0.2mm] [blue](9)[dashed] to [out=340,in=140] (7);
\draw[->, line width=0.2mm] [blue](7)[dashed] to [out=160,in=320] (9);
		
\node [scale=.0] (17) at (5,22) [label=90:\tiny{$T2_{27,3,2}$}]{};
\node [scale=.0](18) at (3.5,19.25)[label=90:\tiny{$T2_{27,3,1}$}]{};
\node [scale=.0](19) at (-4,24.7) [label=90:\tiny{$T2_{27,3,2}$}]{};
\node [scale=.0] (20) at (2,23.8) [label=90:\tiny{$T2_{27,3,1}$}]{};
\node [scale=.0](21) at (-5.5,21)[label=90:\tiny{$T2_{27,3,1}$}]{};
\node [scale=.0](22) at (-1.25,20.5)[label=90:\tiny{$T2_{27,3,2}$}]{};
		\draw[->, line width=0.2mm] [blue](1) to [out=110,in=250] (4);
		\draw[->, line width=0.2mm] [blue](4) to [out=290,in=70] (1);
		\draw[->, line width=0.2mm] [blue](1) to [out=120,in=240] (7);
		\draw[->, line width=0.2mm] [blue](7) to [out=300,in=60] (1);
  	\draw[->, line width=0.2mm] [blue](2) to [out=110,in=250] (5);
		\draw[->, line width=0.2mm] [blue](5) to [out=290,in=70] (2);
		\draw[->, line width=0.2mm] [blue](2) to [out=120,in=240] (8);
		\draw[->, line width=0.2mm] [blue](8) to [out=300,in=60] (2);
  	\draw[->, line width=0.2mm] [blue](3) to [out=110,in=250] (6);
		\draw[->, line width=0.2mm] [blue](6) to [out=290,in=70] (3);
		\draw[->, line width=0.2mm] [blue](3) to [out=120,in=240] (9);
		\draw[->, line width=0.2mm] [blue](9) to [out=300,in=60] (3);
 		\draw[->, line width=0.2mm] [blue](4) to [out=110,in=250] (7);
		\draw[->, line width=0.2mm] [blue](7) to [out=290,in=70] (4);
		\draw[->, line width=0.2mm] [blue](5) to [out=110,in=250] (8);
		\draw[->, line width=0.2mm] [blue](8) to [out=290,in=70] (5);
		\draw[->, line width=0.2mm] [blue](6) to [out=110,in=250] (9);
		\draw[->, line width=0.2mm] [blue](9) to [out=290,in=70] (6);
		
\node [scale=.0] (39) at (-1,5.5) [label=90:\tiny{$T1_{27,2}$}]{};
\node [scale=.0] (45) at (1,2.5) [label=90:\tiny{$T1_{27,4}$}]{};
\node [scale=.0] (43) at (-3.5,7.5) [label=90:\tiny{$T1_{27,4}$}]{};
\node [scale=.0] (40) at (4,7.5) [label=90:\tiny{$T1_{27,2}$}]{};
\node [scale=.0] (33) at (10,9) [label=90:\tiny{$T1_{27,4}$}]{};
\node [scale=.0] (44) at (8,11) [label=90:\tiny{$T1_{27,2}$}]{};
\node [scale=.0] (35) at (6.5,19) [label=90:\tiny{$T1_{27,4}$}]{};
\node [scale=.0] (34) at (12.5,13) [label=90:\tiny{$T1_{27,2}$}]{};

\node [scale=.0] (47) at (-10,11) [label=90:\tiny{$T1_{27,2}$}]{};
\node [scale=.0] (42) at (-8,9) [label=90:\tiny{$T1_{27,4}$}]{};
\node [scale=.0] (41) at (-12,19) [label=90:\tiny{$T1_{27,4}$}]{};
\node [scale=.0] (36) at (-6,10) [label=90:\tiny{$T1_{27,2}$}]{};

\node [scale=.0] (46) at (1.5,12) [label=90:\tiny{$T1_{27,4}$}]{};
\node [scale=.0] (37) at (-1,13) [label=90:\tiny{$T1_{27,2}$}]{};
\node [scale=.0] (31) at (8.5,21) [label=90:\tiny{$T1_{27,2}$}]{};
\node [scale=.0] (38) at (10,18) [label=90:\tiny{$T1_{27,4}$}]{};
\node [scale=.0] (31) at (-9.5,21) [label=90:\tiny{$T1_{27,2}$}]{};
\node [scale=.0] (32) at (-8,18) [label=90:\tiny{$T1_{27,4}$}]{};
				
		\node (1) at (26,2) {\tiny{$C_{27}(R_1)$}};
		\node (2) at (33,7) {\tiny{$C_{27}(T_1)$}};	
		\node (3) at (19,7)  {\tiny{$C_{27}(S_1)$}};

\draw[line width=0.2mm] [blue](1)[dashed] to (2);
\draw[line width=0.2mm] [blue](2)[dashed] to (3);
\draw[line width=0.2mm] [blue](3)[dashed] to (1);
		
\node [scale=.0] (11) at (31,2) [label=90:\tiny{$T2_{27,3}$}]{};
\node [scale=.0] (13) at (28.5,5.5) [label=90:\tiny{$T2_{27,3}$}]{};
\node [scale=.0] (15) at (22.5,2.5)[label=90:\tiny{$T2_{27,3}$}]{};
		\node (4) at (26,10.5) {\tiny{$C_{27}(R_2)$}};
		\node (5) at (33,15.5) {\tiny{$C_{27}(T_2)$}};	
		\node (6) at (19,15.5)  {\tiny{$C_{27}(S_2)$}};
		
\draw[line width=0.2mm] [blue](4)[dashed] to (5);
\draw[line width=0.2mm] [blue](5)[dashed] to (6);
\draw[line width=0.2mm] [blue](6)[dashed] to (4);
		
\node [scale=.0] (17) at (28.5,10.75) [label=90:\tiny{$T2_{27,3}$}]{};
\node [scale=.0] (19) at (28.5,14)[label=90:\tiny{$T2_{27,3}$}]{};
\node [scale=.0](21) at (22.5,11.5)[label=90:\tiny{$T2_{27,3}$}]{};
		\node (7) at (26,19) {\tiny{$C_{27}(R_3)$}};
		\node (8) at (33,24) {\tiny{$C_{27}(T_3)$}};	
		\node (9) at (19,24)  {\tiny{$C_{27}(S_3)$}};
		
\draw[line width=0.2mm] [blue](7)[dashed] to (8);
\draw[line width=0.2mm] [blue](8)[dashed] to (9);
\draw[line width=0.2mm] [blue](9)[dashed] to (7);
		
\node [scale=.0] (17) at (23,20) [label=90:\tiny{$T2_{27,3}$}]{};
\node [scale=.0] (19) at (29,20)[label=90:\tiny{$T2_{27,3}$}]{};
\node [scale=.0](21) at (26.5,22.5) [label=90:\tiny{$T2_{27,3}$}]{};
		\draw[line width=0.2mm] [blue](1) to (4);
		\draw[line width=0.2mm] [blue](1) to [out=120,in=240] (7);
		\draw[line width=0.2mm] [blue](2) to (5);
		\draw[line width=0.2mm] [blue](2) to [out=120,in=240] (8);
		\draw[line width=0.2mm] [blue](3) to (6);
		\draw[line width=0.2mm] [blue](3) to [out=120,in=240] (9);
		\draw[line width=0.2mm] [blue](4) to (7);
		\draw[line width=0.2mm] [blue](5) to (8);
		\draw[line width=0.2mm] [blue](6) to (9);
		
\node [scale=.0] (39) at (27.25,4) [label=90:\tiny{$T1_{27}$}]{};
\node [scale=.0] (43) at (22,8) [label=90:\tiny{$T1_{27}$}]{};
\node [scale=.0] (33) at (34.2,9.5) [label=90:\tiny{$T1_{27}$}]{};

\node [scale=.0] (47) at (20.2,10) [label=90:\tiny{$T1_{27}$}]{};
\node [scale=.0] (46) at (27.2,12.5) [label=90:\tiny{$T1_{27}$}]{};
\node [scale=.0] (35) at (34.2,18) [label=90:\tiny{$T1_{27}$}]{};

\node [scale=.0] (41) at (16,19) [label=90:\tiny{$T1_{27}$}]{};
\node [scale=.0] (32) at (20.2,18) [label=90:\tiny{$T1_{27}$}]{};
\node [scale=.0] (31) at (30.5,18) [label=90:\tiny{$T1_{27}$}]{};
 \end{tikzpicture}		
		
{\small  \hspace{1cm} Fig.  21.   Digraph $\mathcal{D}$  \hspace{5cm} Figure  22. Graph $\mathcal{G}$   

 \vspace{.1cm}	
 \hspace{1.6cm}	with	$R_1$ = $\{1,3,6,8,9,10\}$, ~~$S_1$ = $\{3,4,5,6,9,13\}$, $T_1$ = $\{2,3,6,7,9,11\}$ ,  
			
	\hspace{2.6cm}	   $R_2$ = $\{2,6,7,9,11,12\}$,  $S_2$ = $\{1,6,8,9,10,12\}$, $T_2$ = $\{4,5,6,9,12,13\}$, 
			
	\hspace{2.5cm}	    $R_3$ = $\{3,4,5,9,12,13\}$,  $S_3$ = $\{2,3,7,9,11,12\}$, $T_3$ = $\{1,3,8,9,10,12\}$.  }
\end{center} }

\begin{enumerate}
 \item [\rm (11)] Here, $C_{27}(R)$ = $C_{27}(1,3,8,9,10,12)$, 

$T2_{27,3 }(C_{27}(1,3,8,9,10,12))$ = $\{C_{27}(1,3,8,9,10,12),$ 

\hfill $C_{27}(3,4,5,9,12,13) = T2_{27,3, 1}(C_{27}(1,3,8,9,10,12))$,

\hfill $C_{27}(2,3,7,9,11,12) = T2_{27,3, 2}(C_{27}(1,3,8,9,10,12))\}$ 

\hfill = $T2_{27,3 }(C_{27}(3,4,5,9,12,13))$ = $T2_{27,3 }(C_{27}(2,3,7,9,11,12))$,

$T1_{27}(C_{27}(1,3,8,9,10,12))$ = $\{C_{27}(1,3,8,9,10,12)$, 

\hfill $C_{27}(2,3,6,7,9,11) = T1_{27, 2}(C_{27}(1,3,8,9,10,12))$,

\hfill $C_{27}(4,5,6,9,12,13) = T1_{27, 4}(C_{27}(1,3,8,9,10,12))\}$ 

\hfill = $T1_{27}(C_{27}(2,3,6,7,9,11))$ = $T1_{27}(C_{27}(4,5,6,9,12,13))$,

$T2_{27,3 }(C_{27}(2,3,6,7,9,11))$ = $\{C_{27}(2,3,6,7,9,11)$, 

\hfill $C_{27}(1,3,6,8,9,10) = T2_{27,3, 1}(C_{27}(2,3,6,7,9,11))$,

\hfill $C_{27}(3,4,5,6,9,13) = T2_{27,3, 2}(C_{27}(2,3,6,7,9,11))\}$ 

\hfill = $T2_{27,3 }(C_{27}(1,3,6,8,9,10))$ = $T2_{27,3 }(C_{27}(3,4,5,6,9,13))$,

$T1_{27}(C_{27}(1,3,6,8,9,10))$ = $\{C_{27}(1,3,6,8,9,10)$, 

\hfill $C_{27}(2,6,7,9,11,12) = T1_{27, 2}(C_{27}1,3,6,8,9,10))$,

\hfill $C_{27}(3,4,5,9,12,13) = T1_{27, 4}(C_{27}(1,3,6,8,9,10))\}$ 

\hfill = $T1_{27}(C_{27}(2,6,7,9,11,12))$ = $T1_{27}(C_{27}(3,4,5,9,12,13))$,

$T2_{27,3 }(C_{27}(2,6,7,9,11,12))$ = $\{C_{27}(2,6,7,9,11,12),$ 

\hfill $C_{27}(1,6,8,9,10,12) = T2_{27,3, 1}(C_{27}(2,6,7,9,11,12))$,

\hfill $C_{27}(4,5,6,9,12,13) = T2_{27,3, 2}(C_{27}(2,6,7,9,11,12))\}$ 

\hfill = $T2_{27,3 }(C_{27}(1,6,8,9,10,12))$ = $T2_{27,3 }(C_{27}(4,5,6,9,12,13))$,

$T1_{27}(C_{27}(1,6,8,9,10,12))$ = $\{C_{27}(1,6,8,9,10,12)$, 

\hfill $C_{27}(2,3,7,9,11,12) = T1_{27, 2}(C_{27}(1,6,8,9,10,12))$,

\hfill $C_{27}(3,4,5,6,9,13) = T1_{27, 4}(C_{27}(1,6,8,9,10,12))\}$ 

\hfill = $T1_{27}(C_{27}(2,3,7,9,11,12))$ = $T1_{27}(C_{27}(3,4,5,6,9,13))$ and
\\
 $V(\mathcal{D})$ = $V(\mathcal{G})$ = $Iso(C_{27}(R))$ = $Iso(C_{27}(1,3,8,9,10,12))$ 

\hfill = $Iso(C_{27}(2,3,6,7,9,11))$ = $Iso(C_{27}(4,5,6,9,12,13))$

\hfill  = $\{C_{27}(1,3,8,9,10,12), C_{27}(2,3,6,7,9,11), C_{27}(4,5,6,9,12,13), $ 
 	
 \hfill 	 $C_{27}(3,4,5,9,12,13),  C_{27}(1,3,6,8,9,10), C_{27}(2,6,7,9,11,12)$, 

 	\hfill	$C_{27}(2,3,7,9,11,12), C_{27}(3,4,5,6,9,13), C_{27}(1,6,8,9,10,12)\}$
    
 \hfill   = $Iso(C_{27}(3,4,5,9,12,13))$ = $Iso(C_{27}(1,3,6,8,9,10))$ = $Iso(C_{27}(2,6,7,9,11,12))$ 

\hfill    = $Iso(C_{27}((2,3,7,9,11,12))$ = $Iso(C_{27}(3,4,5,6,9,13))$ = $Iso(C_{27}(1,6,8,9,10,12))$. 
\\    
    Corresponding to $V(\mathcal{D})$ = $V(\mathcal{G})$, we draw the isomorphism digraph $\mathcal{ID}_{27, 3}(C_{27}(1,3,8,9,10,12))$ w.r.t. $m$ = 3 and from this digraph we obtain the isomorphism graph $\mathcal{ID}_{27, 3}(C_{27}(1,3,8,9,10,12))$ w.r.t. $m$ = 3. 
\\
The isomorphism digraph $\mathcal{D}$ = $\mathcal{ID}_{27, 3}(C_{27}(1,3,8,9,10,12))$ = $\mathcal{ID}_{27, 3}(C_{27}(2,3,6,7,9,11))$ 

= $\mathcal{ID}_{27, 3}(C_{27}(4,5,6,9,12,13))$ = $\mathcal{ID}_{27, 3}(C_{27}(3,4,5,9,12,13))$ = $\mathcal{ID}_{27, 3}(C_{27}(1,3,6,8,9,10))$ 

= $\mathcal{ID}_{27, 3}(C_{27}(2,6,7,9,11,12))$ = $\mathcal{ID}_{27, 3}(C_{27}(2,3,7,9,11,12))$ = $\mathcal{ID}_{27, 3}(C_{27}(3,4,5,6,9,13))$ 

\hfill = $\mathcal{ID}_{27, 3}(C_{27}(1,6,8,9,10,12))$. 
\\
The isomorphism graph $\mathcal{G}$ = $\mathcal{I}_{27, 3}(C_{27}(1,3,8,9,10,12))$ = $\mathcal{I}_{27, 3}(C_{27}(2,3,6,7,9,11))$ 

= $\mathcal{I}_{27, 3}(C_{27}(4,5,6,9,12,13))$ = $\mathcal{I}_{27, 3}(C_{27}(3,4,5,9,12,13))$ = $\mathcal{I}_{27, 3}(C_{27}(1,3,6,8,9,10))$ 

\hfill = $\mathcal{I}_{27, 3}(C_{27}(2,6,7,9,11,12))$ = $\mathcal{I}_{27, 3}(C_{27}(2,3,7,9,11,12))$. 
 
This implies that in this case, the isomorphism digraph and the isomorphism graph are same as given in Figures 21 and 22, respectively. 

\item [\rm (12)] Here, $C_{27}(R)$ = $C_{27}(1,6,8,9,10,12)$,	$C_{27}(1,6,8,9,10,12)$ $\cong_{T2_{27,3, 1}}$ $C_{27}(2,6,7,9,11,12)$ and 
	\\
	$C_{27}(1,6,8,9,10,12)$ $\cong_{T2_{27,3, 2}}$ $C_{27}(4,5,6,9,12,13)$. This implies that this case is similar to case (11) and the isomorphism digraphs and the isomorphism graphs in this case are given in Figures 21 and 22, respectively.
\end{enumerate} 	

It is noted that the unlabeled digraphs in all the above cases are the same as well as their unlabeled graphs. Here, we present all isomorphism series corresponding to Figure 21 of cases (10), (11) and (12). Corresponding to cases (10), (11) and (12), in Figure 21, let $R_1$ = $\{1,3,6,8,9,10\}$, $S_1$ = $\{3,4,5,6,9,13\}$,  $T_1$ = $\{2,3,6,7,9,11\}$,  $R_2$ = $\{2,6,7,9,11,12\}$, $S_2$ = $\{1,6,8,9,10,12\}$,  $T_2$ = $\{4,5,6,9,12,13\}$, $R_3$ = $\{3,4,5,9,12,13\}$, $S_3$ = $\{2,3,7,9,11,12\}$,  $T_3$ = $\{1,3,8,9,10,12\}$. We consider isomorphism series in their shortest possible forms and also present a digraph of Eulerian isomorphism series in this case in Figure 23. 

From the isomorphism diagram given in Figure 21, we get the following isomorphism series. 	

 \begin{enumerate}
\item [\rm (11.a1)] 	$C_{27}(R_1)$ $\cong$ $C_{27}(S_1)$:  
	
	$(a)$ $C_{27}(R_1)$ $\cong_{T2_{27,3,1}}$ $C_{27}(S_1)$; 
	
	$(b)$ $C_{27}(S_1)$ $\cong_{T2_{27,3,2}}$ $C_{27}(R_1)$.	 
	 
\item [\rm (11.a2)] 	$C_{27}(R_1)$ $\cong$ $C_{27}(T_1)$:  

    $(a)$ $C_{27}(R_1)$ $\cong_{T2_{27,3,2}}$ $C_{27}(T_1)$; 
    
    $(b)$ $C_{27}(T_1)$ $\cong_{T2_{27,3,1}}$ $C_{27}(R_1)$.	 
	
\item [\rm (11.a3)] $C_{27}(S_1)$ $\cong$ $C_{27}(T_1)$:  

    $(a)$ $C_{27}(S_1)$ $\cong_{T2_{27,3,1}}$ $C_{27}(T_1)$; 
    
    $(b)$ $C_{27}(T_1)$ $\cong_{T2_{27,3,2}}$ $C_{27}(S_1)$.	 

\item [\rm (11.a4)] 	$C_{27}(R_2)$ $\cong$ $C_{27}(S_2)$:  

    $(a)$ $C_{27}(R_2)$ $\cong_{T2_{27,3,1}}$ $C_{27}(S_2)$; 
    
    $(b)$ $C_{27}(S_2)$ $\cong_{T2_{27,3,2}}$ $C_{27}(R_2)$.	 

\item [\rm (11.a5)] 	$C_{27}(R_2)$ $\cong$ $C_{27}(T_2)$:  

   $(a)$ $C_{27}(R_2)$ $\cong_{T2_{27,3,2}}$ $C_{27}(T_2)$; 
   
   $(b)$ $C_{27}(T_2)$ $\cong_{T2_{27,3,1}}$ $C_{27}(R_2)$.	 

\item [\rm (11.a6)] 	$C_{27}(S_2)$ $\cong$ $C_{27}(T_2)$:  

  $(a)$ $C_{27}(S_2)$ $\cong_{T2_{27,3,1}}$ $C_{27}(T_2)$; 
  
  $(b)$ $C_{27}(T_2)$ $\cong_{T2_{27,3,2}}$ $C_{27}(S_2)$.
  
\item [\rm (11.a7)] 	$C_{27}(R_3)$ $\cong$ $C_{27}(S_3)$:  

$(a)$ $C_{27}(R_3)$ $\cong_{T2_{27,3,1}}$ $C_{27}(S_3)$; 

$(b)$ $C_{27}(S_3)$ $\cong_{T2_{27,3,2}}$ $C_{27}(R_3)$.	 

\item [\rm (11.a8)] 	$C_{27}(R_3)$ $\cong$ $C_{27}(T_3)$:  

$(a)$ $C_{27}(R_3)$ $\cong_{T2_{27,3,2}}$ $C_{27}(T_3)$; 

$(b)$ $C_{27}(T_3)$ $\cong_{T2_{27,3,1}}$ $C_{27}(R_3)$.	 

\item [\rm (11.a9)] 	$C_{27}(S_3)$ $\cong$ $C_{27}(T_3)$:  

   $(a)$ $C_{27}(S_3)$ $\cong_{T2_{27,3,1}}$ $C_{27}(T_3)$;  
   
   $(b)$ $C_{27}(T_3)$ $\cong_{T2_{27,3,2}}$ $C_{27}(S_3)$.	 

\item [\rm (11.b1)] 	$C_{27}(R_1)$ $\cong$ $C_{27}(R_2)$:  

$(a)$ $C_{27}(R_1)$ $\cong_{T1_{27,2}}$ $C_{27}(R_2)$; 

$(b)$ $C_{27}(R_2)$ $\cong_{T1_{27,4}}$ $C_{27}(R_1)$.	 

\item [\rm (11.b2)] 	$C_{27}(R_1)$ $\cong$ $C_{27}(R_3)$:  

$(a)$ $C_{27}(R_1)$ $\cong_{T1_{27,4}}$ $C_{27}(R_3)$; 

$(b)$ $C_{27}(R_3)$ $\cong_{T1_{27,2}}$ $C_{27}(R_1)$.	 

\item [\rm (11.b3)] 	$C_{27}(S_1)$ $\cong$ $C_{27}(S_2)$:  

$(a)$ $C_{27}(S_1)$ $\cong_{T1_{27,2}}$ $C_{27}(S_2)$; 

$(b)$ $C_{27}(S_2)$ $\cong_{T1_{27,4}}$ $C_{27}(S_1)$.	 

\item [\rm (11.b4)] 	$C_{27}(S_1)$ $\cong$ $C_{27}(S_3)$:  

$(a)$ $C_{27}(S_1)$ $\cong_{T1_{27,4}}$ $C_{27}(S_3)$; 

$(b)$ $C_{27}(S_3)$ $\cong_{T1_{27,2}}$ $C_{27}(S_1)$.	 

\item [\rm (11.b5)] 	$C_{27}(T_1)$ $\cong$ $C_{27}(T_2)$:  

$(a)$ $C_{27}(T_1)$ $\cong_{T1_{27,2}}$ $C_{27}(T_2)$; 

$(b)$ $C_{27}(T_2)$ $\cong_{T1_{27,4}}$ $C_{27}(T_1)$.	 

\item [\rm (11.b6)] 	$C_{27}(T_1)$ $\cong$ $C_{27}(T_3)$:  

$(a)$ $C_{27}(T_1)$ $\cong_{T1_{27,4}}$ $C_{27}(T_3)$; 

$(b)$ $C_{27}(T_3)$ $\cong_{T1_{27,2}}$ $C_{27}(T_1)$.	 

\item [\rm (11.b7)] 	$C_{27}(R_2)$ $\cong$ $C_{27}(R_3)$:  

$(a)$ $C_{27}(R_2)$ $\cong_{T1_{27,2}}$ $C_{27}(R_3)$; 

$(b)$ $C_{27}(R_3)$ $\cong_{T1_{27,4}}$ $C_{27}(R_2)$.	

\item [\rm (11.b8)] 	$C_{27}(S_2)$ $\cong$ $C_{27}(S_3)$:  

$(a)$ $C_{27}(S_2)$ $\cong_{T1_{27,2}}$ $C_{27}(S_3)$; 

$(b)$ $C_{27}(S_3)$ $\cong_{T1_{27,4}}$ $C_{27}(S_2)$.	 

\item [\rm (11.b9)] 	$C_{27}(T_2)$ $\cong$ $C_{27}(T_3)$:  

$(a)$ $C_{27}(T_2)$ $\cong_{T1_{27,2}}$ $C_{27}(T_3)$; 

$(b)$ $C_{27}(T_3)$ $\cong_{T1_{27,4}}$ $C_{27}(T_2)$.	 

\item [\rm (11.ab1)] $(i)$	$C_{27}(R_1)$ $\cong$ $C_{27}(S_1)$ $\cong$ $C_{27}(S_2)$:  

$(a)$ $C_{27}(R_1)$ $\cong_{T2_{27,3,1}}$ $C_{27}(S_1)$ $\cong_{T1_{27,2}}$ $C_{27}(S_2)$; 

$(b)$ $C_{27}(S_2)$ $\cong_{T1_{27,4}}$ $C_{27}(S_1)$ $\cong_{T2_{27,3,2}}$ $C_{27}(R_1)$.	 

\item [\rm (11.ab1)] $(ii)$	$C_{27}(R_1)$ $\cong$ $C_{27}(S_1)$ $\cong$ $C_{27}(S_3)$:  

$(a)$ $C_{27}(R_1)$ $\cong_{T2_{27,3,1}}$ $C_{27}(S_1)$ $\cong_{T1_{27,4}}$ $C_{27}(S_3)$; 

$(b)$ $C_{27}(S_3)$ $\cong_{T1_{27,2}}$ $C_{27}(S_1)$ $\cong_{T2_{27,3,2}}$ $C_{27}(R_1)$.	 

\item [\rm (11.ab2)] $(i)$	$C_{27}(R_1)$ $\cong$ $C_{27}(T_1)$ $\cong$ $C_{27}(T_2)$:  

$(a)$ $C_{27}(R_1)$ $\cong_{T2_{27,3,2}}$ $C_{27}(T_1)$ $\cong_{T1_{27,2}}$ $C_{27}(T_2)$; 

$(b)$ $C_{27}(T_2)$ $\cong_{T1_{27,4}}$ $C_{27}(T_1)$ $\cong_{T2_{27,3,1}}$ $C_{27}(R_1)$.	 

\item [\rm (11.ab2)] $(ii)$	$C_{27}(R_1)$ $\cong$ $C_{27}(T_1)$ $\cong$ $C_{27}(T_3)$:  

$(a)$ $C_{27}(R_1)$ $\cong_{T2_{27,3,2}}$ $C_{27}(T_1)$ $\cong_{T1_{27,4}}$ $C_{27}(T_3)$; 

$(b)$ $C_{27}(T_3)$ $\cong_{T1_{27,2}}$ $C_{27}(T_1)$ $\cong_{T2_{27,3,1}}$ $C_{27}(R_1)$.	 

\item [\rm (11.ab3)] $(i)$	$C_{27}(R_1)$ $\cong$ $C_{27}(R_2)$ $\cong$ $C_{27}(S_2)$:  

$(a)$ $C_{27}(R_1)$ $\cong_{T1_{27,2}}$ $C_{27}(R_2)$ $\cong_{T2_{27,3,1}}$ $C_{27}(S_2)$; 

$(b)$ $C_{27}(S_2)$ $\cong_{T2_{27,3,2}}$ $C_{27}(R_2)$ $\cong_{T1_{27,4}}$ $C_{27}(R_1)$.

\item [\rm (11.ab3)] $(ii)$	$C_{27}(R_1)$ $\cong$ $C_{27}(R_2)$ $\cong$ $C_{27}(T_2)$:  

$(a)$ $C_{27}(R_1)$ $\cong_{T1_{27,2}}$ $C_{27}(R_2)$ $\cong_{T2_{27,3,2}}$ $C_{27}(T_2)$; 

$(b)$ $C_{27}(T_2)$ $\cong_{T2_{27,3,1}}$ $C_{27}(R_2)$ $\cong_{T1_{27,4}}$ $C_{27}(R_1)$.

\item [\rm (11.ab4)] $(i)$	$C_{27}(R_1)$ $\cong$ $C_{27}(R_3)$ $\cong$ $C_{27}(S_3)$:  

$(a)$ $C_{27}(R_1)$ $\cong_{T1_{27,4}}$ $C_{27}(R_3)$ $\cong_{T2_{27,3,1}}$ $C_{27}(S_3)$; 

$(b)$ $C_{27}(S_3)$ $\cong_{T2_{27,3,2}}$ $C_{27}(R_3)$ $\cong_{T1_{27,2}}$ $C_{27}(R_1)$.

\item [\rm (11.ab4] $(ii)$	$C_{27}(R_1)$ $\cong$ $C_{27}(R_3)$ $\cong$ $C_{27}(T_3)$:  

$(a)$ $C_{27}(R_1)$ $\cong_{T1_{27,4}}$ $C_{27}(R_3)$ $\cong_{T2_{27,3,2}}$ $C_{27}(T_3)$; 

$(b)$ $C_{27}(T_3)$ $\cong_{T2_{27,3,1}}$ $C_{27}(R_3)$ $\cong_{T1_{27,2}}$ $C_{27}(R_1)$.

\item [\rm (11.ab5)] $(i)$	$C_{27}(S_1)$ $\cong$ $C_{27}(T_1)$ $\cong$ $C_{27}(T_2)$:  

$(a)$ $C_{27}(S_1)$ $\cong_{T2_{27,3,2}}$ $C_{27}(T_1)$ $\cong_{T1_{27,2}}$ $C_{27}(T_2)$; 

$(b)$ $C_{27}(T_2)$ $\cong_{T1_{27,4}}$ $C_{27}(T_1)$ $\cong_{T2_{27,3,2}}$ $C_{27}(S_1)$.	 

\item [\rm (11.ab5)] $(ii)$	$C_{27}(S_1)$ $\cong$ $C_{27}(T_1)$ $\cong$ $C_{27}(T_3)$:  

$(a)$ $C_{27}(S_1)$ $\cong_{T2_{27,3,1}}$ $C_{27}(T_1)$ $\cong_{T1_{27,4}}$ $C_{27}(T_3)$; 

$(b)$ $C_{27}(T_3)$ $\cong_{T1_{27,4}}$ $C_{27}(T_1)$ $\cong_{T2_{27,3,2}}$ $C_{27}(S_1)$.	 

\item [\rm (11.ab6)] $(i)$	$C_{27}(S_1)$ $\cong$ $C_{27}(R_1)$ $\cong$ $C_{27}(R_2)$:  

$(a)$ $C_{27}(S_1)$ $\cong_{T2_{27,3,2}}$ $C_{27}(R_1)$ $\cong_{T1_{27,2}}$ $C_{27}(R_2)$; 

$(b)$ $C_{27}(R_2)$ $\cong_{T1_{27,4}}$ $C_{27}(R_1)$ $\cong_{T2_{27,3,1}}$ $C_{27}(S_1)$.	 

\item [\rm (11.ab6)] $(ii)$	$C_{27}(S_1)$ $\cong$ $C_{27}(R_1)$ $\cong$ $C_{27}(R_3)$:  

$(a)$ $C_{27}(S_1)$ $\cong_{T2_{27,3,2}}$ $C_{27}(R_1)$ $\cong_{T1_{27,4}}$ $C_{27}(R_3)$; 

$(b)$ $C_{27}(R_3)$ $\cong_{T1_{27,2}}$ $C_{27}(R_1)$ $\cong_{T2_{27,3,1}}$ $C_{27}(S_1)$.	 

\item [\rm (11.ab7)] $(i)$	$C_{27}(S_1)$ $\cong$ $C_{27}(S_2)$ $\cong$ $C_{27}(T_2)$:  

$(a)$ $C_{27}(S_1)$ $\cong_{T1_{27,2}}$ $C_{27}(S_2)$ $\cong_{T2_{27,3,1}}$ $C_{27}(T_2)$; 

$(b)$ $C_{27}(T_2)$ $\cong_{T2_{27,3,2}}$ $C_{27}(S_2)$ $\cong_{T1_{27,4}}$ $C_{27}(S_1)$.

\item [\rm (11.ab7)] $(ii)$	$C_{27}(S_1)$ $\cong$ $C_{27}(S_2)$ $\cong$ $C_{27}(R_2)$:  

$(a)$ $C_{27}(S_1)$ $\cong_{T1_{27,2}}$ $C_{27}(S_2)$ $\cong_{T2_{27,3,2}}$ $C_{27}(R_2)$; 

$(b)$ $C_{27}(R_2)$ $\cong_{T2_{27,3,1}}$ $C_{27}(S_2)$ $\cong_{T1_{27,4}}$ $C_{27}(S_1)$.

\item [\rm (11.ab8)] $(i)$	$C_{27}(S_1)$ $\cong$ $C_{27}(S_3)$ $\cong$ $C_{27}(T_3)$:  

$(a)$ $C_{27}(S_1)$ $\cong_{T1_{27,4}}$ $C_{27}(S_3)$ $\cong_{T2_{27,3,1}}$ $C_{27}(T_3)$; 

$(b)$ $C_{27}(T_3)$ $\cong_{T2_{27,3,2}}$ $C_{27}(S_3)$ $\cong_{T1_{27,2}}$ $C_{27}(S_1)$.

\item [\rm (11.ab8)] $(ii)$	$C_{27}(S_1)$ $\cong$ $C_{27}(S_3)$ $\cong$ $C_{27}(R_3)$:  

$(a)$ $C_{27}(S_1)$ $\cong_{T1_{27,4}}$ $C_{27}(S_3)$ $\cong_{T2_{27,3,2}}$ $C_{27}(R_3)$; 

$(b)$ $C_{27}(R_3)$ $\cong_{T2_{27,3,1}}$ $C_{27}(S_3)$ $\cong_{T1_{27,2}}$ $C_{27}(S_1)$.

\item [\rm (11.ab9)] $(i)$	$C_{27}(T_1)$ $\cong$ $C_{27}(R_1)$ $\cong$ $C_{27}(R_2)$:  

$(a)$ $C_{27}(T_1)$ $\cong_{T2_{27,3,1}}$ $C_{27}(R_1)$ $\cong_{T1_{27,2}}$ $C_{27}(R_2)$; 

$(b)$ $C_{27}(R_2)$ $\cong_{T1_{27,4}}$ $C_{27}(R_1)$ $\cong_{T2_{27,3,2}}$ $C_{27}(T_1)$.	 

\item [\rm (11.ab9)] $(ii)$	$C_{27}(T_1)$ $\cong$ $C_{27}(R_1)$ $\cong$ $C_{27}(R_3)$:  

$(a)$ $C_{27}(T_1)$ $\cong_{T2_{27,3,1}}$ $C_{27}(R_1)$ $\cong_{T1_{27,4}}$ $C_{27}(R_3)$; 

$(b)$ $C_{27}(R_3)$ $\cong_{T1_{27,2}}$ $C_{27}(R_1)$ $\cong_{T2_{27,3,2}}$ $C_{27}(T_1)$.	 

\item [\rm (11.ab10)] $(i)$	$C_{27}(T_1)$ $\cong$ $C_{27}(S_1)$ $\cong$ $C_{27}(S_2)$:  

$(a)$ $C_{27}(T_1)$ $\cong_{T2_{27,3,2}}$ $C_{27}(S_1)$ $\cong_{T1_{27,2}}$ $C_{27}(S_2)$; 

$(b)$ $C_{27}(S_2)$ $\cong_{T1_{27,4}}$ $C_{27}(S_1)$ $\cong_{T2_{27,3,1}}$ $C_{27}(T_1)$.	 

\item [\rm (11.ab10)] $(ii)$	$C_{27}(T_1)$ $\cong$ $C_{27}(S_1)$ $\cong$ $C_{27}(S_3)$:  

$(a)$ $C_{27}(T_1)$ $\cong_{T2_{27,3,2}}$ $C_{27}(S_1)$ $\cong_{T1_{27,4}}$ $C_{27}(S_3)$; 

$(b)$ $C_{27}(S_3)$ $\cong_{T1_{27,2}}$ $C_{27}(S_1)$ $\cong_{T2_{27,3,1}}$ $C_{27}(T_1)$.	 

\item [\rm (11.ab11)] $(i)$	$C_{27}(T_1)$ $\cong$ $C_{27}(T_2)$ $\cong$ $C_{27}(R_2)$:  

$(a)$ $C_{27}(T_1)$ $\cong_{T1_{27,2}}$ $C_{27}(T_2)$ $\cong_{T2_{27,3,1}}$ $C_{27}(R_2)$; 

$(b)$ $C_{27}(R_2)$ $\cong_{T2_{27,3,2}}$ $C_{27}(T_2)$ $\cong_{T1_{27,4}}$ $C_{27}(T_1)$.	 

\item [\rm (11.ab11)] $(ii)$	$C_{27}(T_1)$ $\cong$ $C_{27}(T_2)$ $\cong$ $C_{27}(S_2)$:  

$(a)$ $C_{27}(T_1)$ $\cong_{T1_{27,2}}$ $C_{27}(T_2)$ $\cong_{T2_{27,3,2}}$ $C_{27}(S_2)$; 

$(b)$ $C_{27}(S_2)$ $\cong_{T2_{27,3,1}}$ $C_{27}(T_2)$ $\cong_{T1_{27,4}}$ $C_{27}(T_1)$.	 

\item [\rm (11.ab12)] $(i)$	$C_{27}(T_1)$ $\cong$ $C_{27}(T_3)$ $\cong$ $C_{27}(R_3)$:  

$(a)$ $C_{27}(T_1)$ $\cong_{T1_{27,4}}$ $C_{27}(T_3)$ $\cong_{T2_{27,3,1}}$ $C_{27}(R_3)$; 

$(b)$ $C_{27}(R_3)$ $\cong_{T2_{27,3,2}}$ $C_{27}(T_3)$ $\cong_{T1_{27,2}}$ $C_{27}(T_1)$.

\item [\rm (11.ab12)] $(ii)$	$C_{27}(T_1)$ $\cong$ $C_{27}(T_3)$ $\cong$ $C_{27}(S_3)$:  

$(a)$ $C_{27}(T_1)$ $\cong_{T1_{27,4}}$ $C_{27}(T_3)$ $\cong_{T2_{27,3,2}}$ $C_{27}(S_3)$; 

$(b)$ $C_{27}(S_3)$ $\cong_{T2_{27,3,1}}$ $C_{27}(T_3)$ $\cong_{T1_{27,2}}$ $C_{27}(T_1)$.

\item [\rm (11.ab13)] $(i)$	$C_{27}(R_2)$ $\cong$ $C_{27}(S_2)$ $\cong$ $C_{27}(S_3)$:  

$(a)$ $C_{27}(R_2)$ $\cong_{T2_{27,3,1}}$ $C_{27}(S_2)$ $\cong_{T1_{27,2}}$ $C_{27}(S_3)$; 

$(b)$ $C_{27}(S_3)$ $\cong_{T1_{27,4}}$ $C_{27}(S_2)$ $\cong_{T2_{27,3,2}}$ $C_{27}(R_2)$.	 

\item [\rm (11.ab13)] $(ii)$	$C_{27}(R_2)$ $\cong$ $C_{27}(S_2)$ $\cong$ $C_{27}(S_1)$:  

$(a)$ $C_{27}(R_2)$ $\cong_{T2_{27,3,1}}$ $C_{27}(S_2)$ $\cong_{T1_{27,4}}$ $C_{27}(S_1)$; 

$(b)$ $C_{27}(S_1)$ $\cong_{T1_{27,2}}$ $C_{27}(S_2)$ $\cong_{T2_{27,3,2}}$ $C_{27}(R_2)$.	 

\item [\rm (11.ab14)] $(i)$	$C_{27}(R_2)$ $\cong$ $C_{27}(T_2)$ $\cong$ $C_{27}(T_3)$:  

$(a)$ $C_{27}(R_2)$ $\cong_{T2_{27,3,2}}$ $C_{27}(T_2)$ $\cong_{T1_{27,2}}$ $C_{27}(T_3)$; 

$(b)$ $C_{27}(T_3)$ $\cong_{T1_{27,4}}$ $C_{27}(T_2)$ $\cong_{T2_{27,3,1}}$ $C_{27}(R_2)$.	 

\item [\rm (11.ab14)] $(ii)$	$C_{27}(R_2)$ $\cong$ $C_{27}(T_2)$ $\cong$ $C_{27}(T_1)$:  

$(a)$ $C_{27}(R_2)$ $\cong_{T2_{27,3,2}}$ $C_{27}(T_2)$ $\cong_{T1_{27,4}}$ $C_{27}(T_1)$; 

$(b)$ $C_{27}(T_1)$ $\cong_{T1_{27,2}}$ $C_{27}(T_2)$ $\cong_{T2_{27,3,1}}$ $C_{27}(R_2)$.	 

\item [\rm (11.ab15)] $(i)$	$C_{27}(R_2)$ $\cong$ $C_{27}(R_3)$ $\cong$ $C_{27}(S_3)$:  

$(a)$ $C_{27}(R_2)$ $\cong_{T1_{27,2}}$ $C_{27}(R_3)$ $\cong_{T2_{27,3,1}}$ $C_{27}(S_3)$; 

$(b)$ $C_{27}(S_3)$ $\cong_{T2_{27,3,2}}$ $C_{27}(R_3)$ $\cong_{T1_{27,4}}$ $C_{27}(R_2)$.	 

\item [\rm (11.ab15)] $(ii)$	$C_{27}(R_2)$ $\cong$ $C_{27}(R_3)$ $\cong$ $C_{27}(T_3)$:  

$(a)$ $C_{27}(R_2)$ $\cong_{T1_{27,2}}$ $C_{27}(R_3)$ $\cong_{T2_{27,3,2}}$ $C_{27}(T_3)$; 

$(b)$ $C_{27}(T_3)$ $\cong_{T2_{27,3,1}}$ $C_{27}(R_3)$ $\cong_{T1_{27,4}}$ $C_{27}(R_2)$.

\item [\rm (11.ab16)] $(i)$	$C_{27}(S_2)$ $\cong$ $C_{27}(T_2)$ $\cong$ $C_{27}(T_3)$:  

$(a)$ $C_{27}(S_2)$ $\cong_{T2_{27,3,1}}$ $C_{27}(T_2)$ $\cong_{T1_{27,2}}$ $C_{27}(T_3)$; 

$(b)$ $C_{27}(T_3)$ $\cong_{T1_{27,4}}$ $C_{27}(T_2)$ $\cong_{T2_{27,3,2}}$ $C_{27}(S_2)$.	 

\item [\rm (11.ab16)] $(ii)$	$C_{27}(S_2)$ $\cong$ $C_{27}(T_2)$ $\cong$ $C_{27}(T_1)$:  

$(a)$ $C_{27}(S_2)$ $\cong_{T2_{27,3,1}}$ $C_{27}(T_2)$ $\cong_{T1_{27,4}}$ $C_{27}(T_1)$; 

$(b)$ $C_{27}(T_1)$ $\cong_{T1_{27,2}}$ $C_{27}(T_2)$ $\cong_{T2_{27,3,2}}$ $C_{27}(S_2)$.	 

\item [\rm (11.ab17)] $(i)$	$C_{27}(S_2)$ $\cong$ $C_{27}(R_2)$ $\cong$ $C_{27}(R_3)$:  

$(a)$ $C_{27}(S_2)$ $\cong_{T2_{27,3,2}}$ $C_{27}(R_2)$ $\cong_{T1_{27,2}}$ $C_{27}(R_3)$; 

$(b)$ $C_{27}(R_3)$ $\cong_{T1_{27,4}}$ $C_{27}(R_2)$ $\cong_{T2_{27,3,1}}$ $C_{27}(S_2)$.	 

\item [\rm (11.ab17)] $(ii)$	$C_{27}(S_2)$ $\cong$ $C_{27}(R_2)$ $\cong$ $C_{27}(R_1)$:  

$(a)$ $C_{27}(S_2)$ $\cong_{T2_{27,3,2}}$ $C_{27}(R_2)$ $\cong_{T1_{27,4}}$ $C_{27}(R_1)$; 

$(b)$ $C_{27}(R_1)$ $\cong_{T1_{27,2}}$ $C_{27}(R_2)$ $\cong_{T2_{27,3,1}}$ $C_{27}(S_2)$.	 

\item [\rm (11.ab18)] $(i)$ $C_{27}(S_2)$ $\cong$ $C_{27}(S_3)$ $\cong$ $C_{27}(T_3)$:  

$(a)$ $C_{27}(S_2)$ $\cong_{T1_{27,2}}$ $C_{27}(S_3)$ $\cong_{T2_{27,3,1}}$ $C_{27}(T_3)$; 

$(b)$ $C_{27}(T_3)$ $\cong_{T2_{27,3,2}}$ $C_{27}(S_3)$ $\cong_{T1_{27,4}}$ $C_{27}(S_2)$.

\item [\rm (11.ab18)] $(ii)$ $C_{27}(S_2)$ $\cong$ $C_{27}(S_3)$ $\cong$ $C_{27}(R_3)$:  

$(a)$ $C_{27}(S_2)$ $\cong_{T1_{27,2}}$ $C_{27}(S_3)$ $\cong_{T2_{27,3,2}}$ $C_{27}(R_3)$; 

$(b)$ $C_{27}(R_3)$ $\cong_{T2_{27,3,1}}$ $C_{27}(S_3)$ $\cong_{T1_{27,4}}$ $C_{27}(S_2)$.

\item [\rm (11.ab19)] $(i)$	$C_{27}(T_2)$ $\cong$ $C_{27}(R_2)$ $\cong$ $C_{27}(R_3)$:  

$(a)$ $C_{27}(T_2)$ $\cong_{T2_{27,3,1}}$ $C_{27}(R_2)$ $\cong_{T1_{27,2}}$ $C_{27}(R_3)$; 

$(b)$ $C_{27}(R_3)$ $\cong_{T1_{27,4}}$ $C_{27}(R_2)$ $\cong_{T2_{27,3,2}}$ $C_{27}(T_2)$.	 

\item [\rm (11.ab19)] $(ii)$	$C_{27}(T_2)$ $\cong$ $C_{27}(R_2)$ $\cong$ $C_{27}(R_1)$:  

$(a)$ $C_{27}(T_2)$ $\cong_{T2_{27,3,1}}$ $C_{27}(R_2)$ $\cong_{T1_{27,4}}$ $C_{27}(R_1)$; 

$(b)$ $C_{27}(R_1)$ $\cong_{T1_{27,2}}$ $C_{27}(R_2)$ $\cong_{T2_{27,3,2}}$ $C_{27}(T_2)$.	 

\item [\rm (11.ab20)] $(i)$	$C_{27}(T_2)$ $\cong$ $C_{27}(S_2)$ $\cong$ $C_{27}(S_3)$:  

$(a)$ $C_{27}(T_2)$ $\cong_{T2_{27,3,2}}$ $C_{27}(S_2)$ $\cong_{T1_{27,2}}$ $C_{27}(S_3)$; 

$(b)$ $C_{27}(S_3)$ $\cong_{T1_{27,4}}$ $C_{27}(S_2)$ $\cong_{T2_{27,3,1}}$ $C_{27}(T_2)$.	 

\item [\rm (11.ab20)] $(ii)$	$C_{27}(T_2)$ $\cong$ $C_{27}(S_2)$ $\cong$ $C_{27}(S_1)$:  

$(a)$ $C_{27}(T_2)$ $\cong_{T2_{27,3,2}}$ $C_{27}(S_2)$ $\cong_{T1_{27,4}}$ $C_{27}(S_1)$; 

$(b)$ $C_{27}(S_1)$ $\cong_{T1_{27,2}}$ $C_{27}(S_2)$ $\cong_{T2_{27,3,1}}$ $C_{27}(T_2)$.	 

\item [\rm (11.ab21)] $(i)$ $C_{27}(T_2)$ $\cong$ $C_{27}(T_3)$ $\cong$ $C_{27}(R_3)$:  

$(a)$ $C_{27}(T_2)$ $\cong_{T1_{27,2}}$ $C_{27}(T_3)$ $\cong_{T2_{27,3,1}}$ $C_{27}(R_3)$; 

$(b)$ $C_{27}(R_3)$ $\cong_{T2_{27,3,2}}$ $C_{27}(T_3)$ $\cong_{T1_{27,4}}$ $C_{27}(T_2)$.

\item [\rm (11.ab22)] $(ii)$ $C_{27}(T_2)$ $\cong$ $C_{27}(T_3)$ $\cong$ $C_{27}(S_3)$:  

$(a)$ $C_{27}(T_2)$ $\cong_{T1_{27,2}}$ $C_{27}(T_3)$ $\cong_{T2_{27,3,2}}$ $C_{27}(S_3)$; 

$(b)$ $C_{27}(S_3)$ $\cong_{T2_{27,3,1}}$ $C_{27}(T_3)$ $\cong_{T1_{27,4}}$ $C_{27}(T_2)$.
\end{enumerate}

{\rm 	\begin{center}
		\begin{tikzpicture}  
		[scale=.8,auto=center,every node/.style={draw,circle}]
			
\node (1) at (-10,-0.5) {\tiny{$R_1$}};
\node [scale=.0] (2) at (-9,-.93) [label=90: $>$]{};
\node (3) at (-8,-0.5) {\tiny{$S_1$}};	
\node [scale=.0] (4) at (-7,-0.93) [label=91: $>$]{};
\node (5) at (-6,-0.5)  {\tiny{$S_2$}};

\draw[ line width=0.2mm] [blue](1)[dashed]  to (3);
\draw[ line width=0.2mm] [blue](3) to (5);

\node [scale=.0] (6) at (-5,-.93) [label=90: $>$]{};
\node (7) at (-4,-0.5) {\tiny{$T_2$}};	
\node [scale=.0] (8) at (-3,-0.93) [label=91: $>$]{};
\node (9) at (-2,-0.5)  {\tiny{$T_3$}};

\draw[ line width=0.2mm] [blue](5)[dashed] to (7);
\draw[ line width=0.2mm] [blue](7) to (9);

\node [scale=.0] (10) at (-1,-.93) [label=90: $>$]{};
\node (11) at (0,-0.5) {\tiny{$R_3$}};	
\node [scale=.0] (12) at (1,-0.93) [label=91: $>$]{};
\node (13) at (2,-0.5)  {\tiny{$R_2$}};

\draw[ line width=0.2mm] [blue](9)[dashed]  to (11);
\draw[ line width=0.2mm] [blue](11) to (13);

\node [scale=.0] (14) at (3,-.93) [label=90: $>$]{};
\node (15) at (4,-0.5) {\tiny{$T_2$}};	
\node [scale=.0] (16) at (5,-0.93) [label=91: $>$]{};
\node (17) at (6,-0.5)  {\tiny{$T_1$}};

\draw[ line width=0.2mm] [blue](13)[dashed] to (15);
\draw[ line width=0.2mm] [blue](15) to (17);

\node [scale=.0] (18) at (6,-1.93) [label=91: $\downarrow$]{};
\node (19) at (6,-2.5) {\tiny{$S_1$}};	
\node [scale=.0] (20) at (5,-2.93) [label=90: $<$]{};
\node (21) at (4,-2.5)  {\tiny{$S_3$}};

\draw[ line width=0.2mm] [blue](17)[dashed] to (19);
\draw[ line width=0.2mm] [blue](19) to (21);

\node [scale=.0] (22) at (3,-2.93) [label=90: $<$]{};
\node (23) at (2,-2.5) {\tiny{$R_3$}};	
\node [scale=.0] (24) at (1,-2.93) [label=91: $<$]{};
\node (25) at (0,-2.5)  {\tiny{$R_1$}};

\draw[ line width=0.2mm] [blue](21)[dashed]  to (23);
\draw[ line width=0.2mm] [blue](23) to (25);

\node [scale=.0] (26) at (-1,-2.93) [label=90: $<$]{};
\node (27) at (-2,-2.5) {\tiny{$T_1$}};	
\node [scale=.0] (28) at (-3,-2.93) [label=91: $<$]{};
\node (29) at (-4,-2.5)  {\tiny{$T_3$}};

\draw[ line width=0.2mm] [blue](25)[dashed] to (27);
\draw[ line width=0.2mm] [blue](27) to (29);

\node [scale=.0] (30) at (-5,-2.93) [label=90: $<$]{};
\node (31) at (-6,-2.5) {\tiny{$S_3$}};	
\node [scale=.0] (32) at (-7,-2.93) [label=91: $<$]{};
\node (33) at (-8,-2.5)  {\tiny{$S_2$}};

\draw[ line width=0.2mm] [blue](29)[dashed]  to (31);
\draw[ line width=0.2mm] [blue](31) to (33);

\node [scale=.0] (34) at (-9,-2.93) [label=90: $<$]{};
\node (35) at (-10,-2.5) {\tiny{$R_2$}};	
\node [scale=.0] (36) at (-10,-3.93) [label=91: $\downarrow$]{};
\node (37) at (-10,-4.5)  {\tiny{$R_3$}};

\draw[ line width=0.2mm] [blue](33)[dashed] to (35);
\draw[ line width=0.2mm] [blue](35) to (37);

\node [scale=.0] (38) at (-9,-4.93) [label=90: $>$]{};
\node (39) at (-8,-4.5) {\tiny{$S_3$}};	
\node [scale=.0] (40) at (-7,-4.93) [label=91: $>$]{};
\node (41) at (-6,-4.5)  {\tiny{$S_1$}};

\draw[ line width=0.2mm] [blue](37)[dashed]  to (39);
\draw[ line width=0.2mm] [blue](39) to (41);

\node [scale=.0] (42) at (-5,-4.93) [label=90: $>$]{};
\node (43) at (-4,-4.5) {\tiny{$T_1$}};	
\node [scale=.0] (44) at (-3,-4.93) [label=91: $>$]{};
\node (45) at (-2,-4.5)  {\tiny{$T_2$}};

\draw[ line width=0.2mm] [blue](41)[dashed] to (43);
\draw[ line width=0.2mm] [blue](43) to (45);

\node [scale=.0] (46) at (-1,-4.93) [label=90: $>$]{};
\node (47) at (0,-4.5) {\tiny{$S_2$}};	
\node [scale=.0] (48) at (1,-4.93) [label=91: $>$]{};
\node (49) at (2,-4.5)  {\tiny{$S_1$}};

\draw[ line width=0.2mm] [blue](45)[dashed]  to (47);
\draw[ line width=0.2mm] [blue](47) to (49);

\node [scale=.0] (50) at (3,-4.93) [label=90: $>$]{};
\node (51) at (4,-4.5) {\tiny{$R_1$}};	
\node [scale=.0] (52) at (5,-4.93) [label=91: $>$]{};
\node (53) at (6,-4.5)  {\tiny{$R_2$}};

\draw[ line width=0.2mm] [blue](49)[dashed] to (51);
\draw[ line width=0.2mm] [blue](51) to (53);

\node [scale=.0] (54) at (6,-5.93) [label=91: $\downarrow$]{};
\node (55) at (6,-6.5) {\tiny{$S_2$}};	
\node [scale=.0] (56) at (5,-6.93) [label=90: $<$]{};
\node (57) at (4,-6.5)  {\tiny{$S_3$}};

\draw[ line width=0.2mm] [blue](53)[dashed] to (55);
\draw[ line width=0.2mm] [blue](55) to (57);

\node [scale=.0] (58) at (3,-6.93) [label=90: $<$]{};
\node (59) at (2,-6.5) {\tiny{$T_3$}};	
\node [scale=.0] (60) at (1,-6.93) [label=91: $<$]{};
\node (61) at (0,-6.5)  {\tiny{$T_1$}};

\draw[ line width=0.2mm] [blue](57)[dashed]  to (59);
\draw[ line width=0.2mm] [blue](59) to (61);

\node [scale=.0] (62) at (-1,-6.93) [label=90: $<$]{};
\node (63) at (-2,-6.5) {\tiny{$R_1$}};	
\node [scale=.0] (64) at (-3,-6.93) [label=91: $<$]{};
\node (65) at (-4,-6.5)  {\tiny{$R_3$}};

\draw[ line width=0.2mm] [blue](61)[dashed] to (63);
\draw[ line width=0.2mm] [blue](63) to (65);

\node [scale=.0] (66) at (-5,-6.93) [label=90: $<$]{};
\node (67) at (-6,-6.5) {\tiny{$T_3$}};	
\node [scale=.0] (68) at (-7,-6.93) [label=91: $<$]{};
\node (69) at (-8,-6.5)  {\tiny{$T_2$}};

\draw[ line width=0.2mm] [blue](65)[dashed]  to (67);
\draw[ line width=0.2mm] [blue](67) to (69);

\node [scale=.0] (70) at (-9,-6.93) [label=90: $<$]{};
\node (71) at (-10,-6.5) {\tiny{$R_2$}};	
\node [scale=.0] (72) at (-11,-4.1) [label=91: $\uparrow$]{};

\draw[ line width=0.2mm] [blue](69)[dashed] to (71);
\draw[-, line width=0.2mm] [blue](71) to [out=120,in=240] (1);
	
\end{tikzpicture}

\medskip		
{\small  Fig. 23. Digraph of Eulerian isomorphism series of $C_{27}(1,3,8,9,10)$ with the representation of 
	
	$X_i$ in the place of $C_{54}(X_i)$ where $X_i$ = $R_i, S_i, T_i$, $i$ = 1,2,3, and 
			
$R_1$ = $\{1,3,6,8,9,10\}$,  $S_1$ = $\{3,4,5,6,9,13\}$, ~~$T_1$ = $\{2,3,6,7,9,11\}$,
					
\hspace{.1cm}	    
$R_2$ = $\{2,6,7,9,11,12\}$,  $S_2$ = $\{1,6,8,9,10,12\}$, $T_2$ = $\{4,5,6,9,12,13\}$, 

\hspace{.1cm}	   
$R_3$ = $\{3,4,5,9,12,13\}$, $S_3$ = $\{2,3,7,9,11,12\}$, $T_3$ = $\{1,3,8,9,10,12\}$.   }
\end{center} }

In this case, the isomorphism diagram $\mathcal{D}$ = $\mathcal{I}so\mathcal{D}(C_{27}(1,3,6,8,9,10))$ contains an Eulerian isomorphism series and its digraph of the Eulerian isomorphism series is given in Figure 23 with $C_{27}(R_1)$ = $C_{27}(1,3,6,8,9,10)$. In the figure, digraph of Eulerian isomorphism series of $C_{27}(X_i)$ with the representation of $X_i$ in the place of $C_{27}(X_i)$ is given with $X_i$ = $R_i, S_i, T_i$ for $i$ = 1,2,3 and 
			
$R_1$ = $\{1,3,6,8,9,10\}$,  $S_1$ = $\{3,4,5,6,9,13\}$, $T_1$ = $\{2,3,6,7,9,11\}$, 
					
$R_2$ = $\{2,6,7,9,11,12\}$,  $S_2$ = $\{1,6,8,9,10,12\}$, $T_2$ = $\{4,5,6,9,12,13\}$, 

$R_3$ = $\{3,4,5,9,12,13\}$, $S_3$ = $\{23,,7,9,11,12\}$, $T_3$ = $\{1,3,8,9,10,12\}$, 
\\
the isomorphism set of $C_{27}(1,3,6,8,9,10)$ is 
\\
$Isoset(C_{27}(1,3,6,8,9,10))$ =  $V(\mathcal{D})$ = $V(\mathcal{G})$ = $Iso(C_{27}(2,6,7,9,11,12))$ = $Iso(C_{27}(3,4,5,9,12,13))$ 

\hfill = $\{C_{27}(1,3,6,8,9,10), C_{27}(2,6,7,9,11,12),  C_{27}(3,4,5,9,12,13),$   
 	
 \hfill 	 $C_{27}(3,4,5,6,9,13), C_{27}(1,6,8,9,10,12)$, $C_{27}(2,3,7,9,11,12),$

 	\hfill	$C_{27}(2,3,6,7,9,11), C_{27}(4,5,6,9,12,13), C_{27}(1,3,8,9,10,12)\}$
    
 \hfill   = $Iso(C_{27}(3,4,5,6,9,13))$ = $Iso(C_{27}(1,6,8,9,10,12))$ = $Iso(C_{27}(2,3,7,9,11,12))$ 

\hfill  = $Iso(C_{27}((2,3,6,7,9,11))$ = $Iso(C_{27}(4,5,6,9,12,13))$ = $Iso(C_{27}(1,3,8,9,10,12))$. \\
See case (10) for more details.  \hfill $\Box$

\medskip
In \cite{v2-4}, the author studied Type-2 isomorphic circulant graphs of order 54 and showed that there are 960 triples of Type-2 isomorphic circulant graphs of order 54 and each triple of isomorphic circulant graphs is of Type-2 isomorphic w.r.t. $m$ = 3. In the next problem, we find the isomorphism diagram, the isomorphism graph and isomorphism series corresponding to one triple of isomorphic circulant graphs order 54, namely $C_{54}(1,3,17,19)$, $C_{54}(3,7,11,25)$, $C_{54}(3,5,13,23)$ which are Type-2 isomorphic w.r.t. $m$ = 3. We get similar results in other cases.

 \begin{prm} \quad \label{p7} {\rm Find the isomorphism digraph, the isomorphism graph and Eulerian isomorphism series, if it exists, of $C_{54}(1,3,17,19)$. }
 \end{prm}
 
 \noindent
 {\bf Solution.}\quad We follow the same method used in the previous problem, especially follow its case (10). 

Let $X_1$ = $\{1,3,17,19\}$, $Y_1$ = $\{3,7,11,25\}$, $Z_1$ = $\{3,5,13,23\}$, 

\hspace{.5cm} $X_2$ = $\{2,3,16,20\}$, $Y_2$ = $\{3,4,14,22\}$, $Z_2$ = $\{3,8,10,26\}$, 

\hspace{.5cm} $X_3$ = $\{7,11,21,25\}$, $Y_3$ = $\{5,13,21,23\}$, $Z_3$ = $\{1,17,19,21\}$. 

Here, $n$ = 54 = $2\times 3^3$, $m$ = 3 = $\gcd(54, 3)$ = $r$, $r\in X_i,Y_i,Z_i$, and $i$ = 1,2. 

Using solutions obtained in \cite{v2-4}, we get the followings with $C_{54}(R)$ = $C_{54}(1,3,17,19)$,  

$T2_{54,3 }(C_{54}(1,3,17,19))$ = $\{C_{54}(1,3,17,19)$, $C_{54}(3,7,11,25) = T2_{54,3, 2}(C_{54}(1,3,17,19))$,

\hfill $C_{54}(3,5,13,23) = T2_{54,3, 4}(C_{54}(1,3,17,19))\}$ = $T2_{54,3 }(C_{54}(3,7,11,25))$ 

\hfill = $T2_{54,3 }(C_{54}(3,5,13,23))$,

$T1_{54}(C_{54}(1,3,17,19))$ = $\{C_{54}(1,3,17,19)$, $C_{54}(5,13,15,23) = T1_{27, 5}(C_{54}(1,3,17,19))$, 

\hfill $C_{54}(7,11,21,25) = T1_{27, 7}(C_{54}(1,3,17,19))\}$ = $T1_{54}(C_{54}(5,13,15,23))$ 

\hfill = $T1_{54}(C_{54}(7,11,21,25))$, 

$T2_{54,3 }(C_{54}(5,13,15,23))$ = $\{C_{54}(5,13,15,23)$, $C_{54}(1,15,17,19) = T2_{54,3, 2}(C_{54}(5,13,15,23))$,

\hfill $C_{54}(7,11,15,25) = T2_{54,3, 4}(C_{54}(5,13,15,23))\}$ = $T2_{54,3 }(C_{54}(1,15,17,19))$ 

\hfill = $T2_{54,3 }(C_{54}(7,11,15,25))$,

$T1_{54}(C_{54}(1,15,17,19))$ = $\{C_{54}(1,15,17,19)$,  $C_{54}(5,13,21,23) = T1_{27, 5}(C_{54}(1,15,17,19))$,

\hfill $C_{54}(3,7,11,25) = T1_{27, 7}(C_{54}(1,15,17,19))\}$ = $T1_{54}(C_{54}(5,13,21,23))$ 

\hfill = $T1_{54}(C_{54}(3,7,11,25))$,

$T2_{54,3 }(C_{54}(5,13,21,23))$ = $\{C_{54}(5,13,21,23),$ $C_{54}(1,17,19,21) = T2_{54,3, 2}(C_{54}(5,13,21,23))$, 

\hfill $C_{54}(7,11,21,25) = T2_{54,3, 4}(C_{54}(5,13,21,23))\}$ = $T2_{54,3 }(C_{54}(1,17,19,21))$ 

\hfill = $T2_{54,3 }(C_{54}(7,11,21,25))$,

$T1_{54}(C_{54}(1,17,19,21))$ = $\{C_{54}(1,17,19,21)$, $C_{54}(3,5,13,23) = T1_{27, 5}(C_{54}(1,17,19,21))$,

\hfill $C_{54}(7,11,15,25) = T1_{27, 7}(C_{54}(1,17,19,21))\}$ = $T1_{54}(C_{54}(3,5,13,23))$ 

\hfill = $T1_{54}(C_{54}(7,11,15,25))$ and

 $V(\mathcal{D})$ = $V(\mathcal{G})$ = $Iso(C_{54}(R))$ = $Iso(C_{54}(1,3,17,19))$ = $Iso(C_{54}(3,7,11,25))$ 

\hfill = $Iso(C_{54}(3,5,13,23))$ = $\{C_{54}(1,3,17,19), C_{54}(3,7,11,25), C_{54}(3,5,13,23),$   
 	
 \hfill 	 $C_{54}(5,13,15,23), C_{54}(1,15,17,19), C_{54}(7,11,15,25),$

 	\hfill	$C_{54}(7,11,21,25), C_{54}(5,13,21,23), C_{54}(1,17,19,21)\}$
    
 \hfill   = $Iso(C_{54}(5,13,15,23))$ = $Iso(C_{54}(1,15,17,19))$ = $Iso(C_{54}((7,11,15,25))$ 

\hfill    = $Iso(C_{54}(7,11,21,25))$ = $Iso(C_{54}(5,13,21,23))$ = $Iso(C_{54}(1,17,19,21))$. 
    
  $\Rightarrow$	The isomorphism digraph $\mathcal{D}$ = $\mathcal{ID}_{54, 3}(C_{54}(1,3,17,19))$ = $\mathcal{ID}_{54, 3}(C_{54}(3,7,11,25))$ 

\hfill = $\mathcal{ID}_{54, 3}(C_{54}(3,5,13,23))$ = $\mathcal{ID}_{54, 3}(C_{54}(5,13,15,23))$ = $\mathcal{ID}_{54, 3}(C_{54}(1,15,17,19))$ 

\hfill = $\mathcal{ID}_{54, 3}(C_{54}(7,11,15,25)$ = $\mathcal{ID}_{54, 3}(C_{54}(7,11,21,25))$ = $\mathcal{ID}_{54, 3}(C_{54}(5,13,21,23))$  

\hfill = $\mathcal{ID}_{54, 3}(C_{54}(1,17,19,21))$ and 

the isomorphism graph $\mathcal{G}$ = $\mathcal{I}_{54, 3}(C_{54}(1,3,17,19))$ = $\mathcal{I}_{54, 3}(C_{54}(3,7,11,25))$ 

\hfill = $\mathcal{I}_{54, 3}(C_{54}(3,5,13,23))$ = $\mathcal{I}_{54, 3}(C_{54}(5,13,15,23))$  = $\mathcal{I}_{54, 3}(C_{54}(1,15,17,19))$ 

\hfill = $\mathcal{I}_{54, 3}(C_{54}(7,11,15,25))$ = $\mathcal{I}_{54, 3}(C_{54}(7,11,21,25))$ = $\mathcal{I}_{54, 3}(C_{54}(5,13,21,23))$ 

 \hfill = $\mathcal{I}_{54, 3}(C_{54}(1,17,19,21))$ which are given in Figures 24 and 25.

 {\rm 	\begin{center}
 		\begin{tikzpicture}  
 		[scale=.32,auto=center,every node/.style={draw,circle}] 
 		
 		\node (1) at (0,0) {\tiny{$C_{54}(X_1)$}};
		\node (2) at (9,6) {\tiny{$C_{54}(Z_1)$}};	
		\node (3) at (-9,6){\tiny{$C_{54}(Y_1)$}};

\draw[->, line width=0.2mm] [blue](1)[dashed] to [out=50,in=200] (2);
\draw[->, line width=0.2mm] [blue](2)[dashed] to [out=230,in=20] (1);
\draw[->, line width=0.2mm] [blue](2)[dashed] to [out=175,in=5] (3);
\draw[->, line width=0.2mm] [blue](3)[dashed] to [out=355,in=185](2);
\draw[->, line width=0.2mm] [blue](3)[dashed] to [out=340,in=140](1);
\draw[->, line width=0.2mm] [blue](1)[dashed] to [out=160,in=310](3);
		
\node [scale=.0] (11) at (5,2) [label=90:\tiny{$T2_{54,3,4}$}]{};
\node [scale=.0] (12) at (5,-1) [label=90:\tiny{$T2_{54,3,2}$}]{};
\node [scale=.0] (13) at (-4.5,4.5)[label=90:\tiny{$T2_{54,3,4}$}]{};
\node [scale=.0] (14) at (4.8,3.5)[label=90:\tiny{$T2_{54,3,2}$}]{};
\node [scale=.0](15) at (-6,.5)[label=90:\tiny{$T2_{54,3,2}$}]{};
\node [scale=.0] (16) at (-4,1.5) [label=90:\tiny{$T2_{54,3,4}$}]{};
		\node (4) at (0,10) {\tiny{$C_{54}(X_2)$}};
		\node (5) at (9,16) {\tiny{$C_{54}(Z_2)$}};	
		\node (6) at (-9,16){\tiny{$C_{54}(Y_2)$}};
		
\draw[->, line width=0.2mm] [blue](4)[dashed] to [out=50,in=200] (5);
\draw[->, line width=0.2mm] [blue](5)[dashed] to [out=230,in=20] (4);
\draw[->, line width=0.2mm] [blue](5)[dashed] to [out=175,in=5] (6);
\draw[->, line width=0.2mm] [blue](6)[dashed] to [out=355,in=185] (5);
\draw[->, line width=0.2mm] [blue](6)[dashed] to [out=340,in=140] (4);
\draw[->, line width=0.2mm] [blue](4)[dashed] to [out=160,in=320] (6);
		
\node [scale=.0] (17) at (5,12) [label=90:\tiny{$T2_{54,3,4}$}]{};
\node [scale=.0] (18) at (5.2,9.7) [label=90:\tiny{$T2_{54,3,2}$}]{};
\node [scale=.0](19) at (-3.5,14.7)[label=90:\tiny{$T2_{54,3,4}$}]{};
\node [scale=.0] (20) at (4.1,13.8)[label=90:\tiny{$T2_{54,3,2}$}]{};
\node [scale=.0](21) at (-6,11.5)[label=90:\tiny{$T2_{54,3,2}$}]{};
\node [scale=.0] (22) at (-2.5,10) [label=90:\tiny{$T2_{54,3,4}$}]{};
		\node (7) at (0,20) {\tiny{$C_{54}(X_3)$}};
		\node (8) at (9,26) {\tiny{$C_{54}(Z_3)$}};	
		\node (9) at (-9,26){\tiny{$C_{54}(Y_3)$}};
		
\draw[->, line width=0.2mm] [blue](7)[dashed] to [out=50,in=200] (8);
\draw[->, line width=0.2mm] [blue](8)[dashed] to [out=230,in=20] (7);
\draw[->, line width=0.2mm] [blue](8)[dashed] to [out=175,in=5] (9);
\draw[->, line width=0.2mm] [blue](9)[dashed] to [out=355,in=185] (8);
\draw[->, line width=0.2mm] [blue](9)[dashed] to [out=340,in=140] (7);
\draw[->, line width=0.2mm] [blue](7)[dashed] to [out=160,in=320] (9);
		
\node [scale=.0] (17) at (5,22) [label=90:\tiny{$T2_{54,3,4}$}]{};
\node [scale=.0](18) at (3.5,19.25)[label=90:\tiny{$T2_{54,3,2}$}]{};
\node [scale=.0](19) at (-4,24.7) [label=90:\tiny{$T2_{54,3,4}$}]{};
\node [scale=.0] (20) at (2,23.8) [label=90:\tiny{$T2_{54,3,2}$}]{};
\node [scale=.0](21) at (-5.5,21)[label=90:\tiny{$T2_{54,3,2}$}]{};
\node [scale=.0](22) at (-1.25,20.5)[label=90:\tiny{$T2_{54,3,4}$}]{};
		\draw[->, line width=0.2mm] [blue](1) to [out=110,in=250] (4);
		\draw[->, line width=0.2mm] [blue](4) to [out=290,in=70] (1);
		\draw[->, line width=0.2mm] [blue](1) to [out=120,in=240] (7);
		\draw[->, line width=0.2mm] [blue](7) to [out=300,in=60] (1);
  	\draw[->, line width=0.2mm] [blue](2) to [out=110,in=250] (5);
		\draw[->, line width=0.2mm] [blue](5) to [out=290,in=70] (2);
		\draw[->, line width=0.2mm] [blue](2) to [out=120,in=240] (8);
		\draw[->, line width=0.2mm] [blue](8) to [out=300,in=60] (2);
  	\draw[->, line width=0.2mm] [blue](3) to [out=110,in=250] (6);
		\draw[->, line width=0.2mm] [blue](6) to [out=290,in=70] (3);
		\draw[->, line width=0.2mm] [blue](3) to [out=120,in=240] (9);
		\draw[->, line width=0.2mm] [blue](9) to [out=300,in=60] (3);
 		\draw[->, line width=0.2mm] [blue](4) to [out=110,in=250] (7);
		\draw[->, line width=0.2mm] [blue](7) to [out=290,in=70] (4);
		\draw[->, line width=0.2mm] [blue](5) to [out=110,in=250] (8);
		\draw[->, line width=0.2mm] [blue](8) to [out=290,in=70] (5);
		\draw[->, line width=0.2mm] [blue](6) to [out=110,in=250] (9);
		\draw[->, line width=0.2mm] [blue](9) to [out=290,in=70] (6);
		
\node [scale=.0] (39) at (-1,5.5) [label=90:\tiny{$T1_{54,5}$}]{};
\node [scale=.0] (45) at (1,2.5) [label=90:\tiny{$T1_{54,7}$}]{};
\node [scale=.0] (43) at (-3.5,7.5) [label=90:\tiny{$T1_{54,7}$}]{};
\node [scale=.0] (40) at (4,7.5) [label=90:\tiny{$T1_{54,5}$}]{};
\node [scale=.0] (33) at (10,9) [label=90:\tiny{$T1_{54,7}$}]{};
\node [scale=.0] (44) at (8,11) [label=90:\tiny{$T1_{54,5}$}]{};
\node [scale=.0] (35) at (6.5,19) [label=90:\tiny{$T1_{54,7}$}]{};
\node [scale=.0] (34) at (12.5,13) [label=90:\tiny{$T1_{54,5}$}]{};

\node [scale=.0] (47) at (-10,11) [label=90:\tiny{$T1_{54,5}$}]{};
\node [scale=.0] (42) at (-8,9) [label=90:\tiny{$T1_{54,7}$}]{};
\node [scale=.0] (41) at (-12,19) [label=90:\tiny{$T1_{54,7}$}]{};
\node [scale=.0] (36) at (-6,10) [label=90:\tiny{$T1_{54,5}$}]{};

\node [scale=.0] (46) at (1.5,12) [label=90:\tiny{$T1_{54,7}$}]{};
\node [scale=.0] (37) at (-1,13) [label=90:\tiny{$T1_{54,5}$}]{};
\node [scale=.0] (31) at (8.5,21) [label=90:\tiny{$T1_{54,5}$}]{};
\node [scale=.0] (38) at (10,18) [label=90:\tiny{$T1_{54,7}$}]{};
\node [scale=.0] (31) at (-9.5,21) [label=90:\tiny{$T1_{54,5}$}]{};
\node [scale=.0] (32) at (-8,18) [label=90:\tiny{$T1_{54,7}$}]{};
				
		\node (1) at (26,2) {\tiny{$C_{54}(X_1)$}};
		\node (2) at (33,7) {\tiny{$C_{54}(Z_1)$}};	
		\node (3) at (19,7)  {\tiny{$C_{54}(Y_1)$}};

\draw[line width=0.2mm] [blue](1)[dashed] to (2);
\draw[line width=0.2mm] [blue](2)[dashed] to (3);
\draw[line width=0.2mm] [blue](3)[dashed] to (1);
		
\node [scale=.0] (11) at (31,2) [label=90:\tiny{$T2_{27,3}$}]{};
\node [scale=.0] (13) at (28.5,5.5) [label=90:\tiny{$T2_{27,3}$}]{};
\node [scale=.0] (15) at (22.5,2.5)[label=90:\tiny{$T2_{27,3}$}]{};
		\node (4) at (26,10.5) {\tiny{$C_{54}(X_2)$}};
		\node (5) at (33,15.5) {\tiny{$C_{54}(Z_2)$}};	
		\node (6) at (19,15.5)  {\tiny{$C_{54}(Y_2)$}};
		
\draw[line width=0.2mm] [blue](4)[dashed] to (5);
\draw[line width=0.2mm] [blue](5)[dashed] to (6);
\draw[line width=0.2mm] [blue](6)[dashed] to (4);
		
\node [scale=.0] (17) at (28.5,10.75) [label=90:\tiny{$T2_{27,3}$}]{};
\node [scale=.0] (19) at (28.5,14)[label=90:\tiny{$T2_{27,3}$}]{};
\node [scale=.0](21) at (22.5,11.5)[label=90:\tiny{$T2_{27,3}$}]{};
		\node (7) at (26,19) {\tiny{$C_{54}(X_3)$}};
		\node (8) at (33,24) {\tiny{$C_{54}(Z_3)$}};	
		\node (9) at (19,24)  {\tiny{$C_{54}(Y_3)$}};
		
\draw[line width=0.2mm] [blue](7)[dashed] to (8);
\draw[line width=0.2mm] [blue](8)[dashed] to (9);
\draw[line width=0.2mm] [blue](9)[dashed] to (7);
		
\node [scale=.0] (17) at (23,20) [label=90:\tiny{$T2_{27,3}$}]{};
\node [scale=.0] (19) at (29,20)[label=90:\tiny{$T2_{27,3}$}]{};
\node [scale=.0](21) at (26.5,22.5) [label=90:\tiny{$T2_{27,3}$}]{};
		\draw[line width=0.2mm] [blue](1) to (4);
		\draw[line width=0.2mm] [blue](1) to [out=120,in=240] (7);
		\draw[line width=0.2mm] [blue](2) to (5);
		\draw[line width=0.2mm] [blue](2) to [out=120,in=240] (8);
		\draw[line width=0.2mm] [blue](3) to (6);
		\draw[line width=0.2mm] [blue](3) to [out=120,in=240] (9);
		\draw[line width=0.2mm] [blue](4) to (7);
		\draw[line width=0.2mm] [blue](5) to (8);
		\draw[line width=0.2mm] [blue](6) to (9);
		
\node [scale=.0] (39) at (27.25,4) [label=90:\tiny{$T1_{54}$}]{};
\node [scale=.0] (43) at (22,8) [label=90:\tiny{$T1_{54}$}]{};
\node [scale=.0] (33) at (34.2,9.5) [label=90:\tiny{$T1_{54}$}]{};

\node [scale=.0] (47) at (20.2,10) [label=90:\tiny{$T1_{54}$}]{};
\node [scale=.0] (46) at (27.2,12.5) [label=90:\tiny{$T1_{54}$}]{};
\node [scale=.0] (35) at (34.2,18) [label=90:\tiny{$T1_{54}$}]{};

\node [scale=.0] (41) at (16,19) [label=90:\tiny{$T1_{54}$}]{};
\node [scale=.0] (32) at (20.2,18) [label=90:\tiny{$T1_{54}$}]{};
\node [scale=.0] (31) at (30.5,18) [label=90:\tiny{$T1_{54}$}]{};
 \end{tikzpicture}		
	
	\vspace{.1cm}		
 {\small  \hspace{3cm}		Fig.  24.   Digraph $\mathcal{D}$  \hfill Figure  25. Graph $\mathcal{G}$ \hspace{2cm}		  

 \vspace{.1cm}	
   \hspace{1.1cm} with	$X_1$ = $\{1,3,17,19\}$, ~$Y_1$ = $\{3,7,11,25\}$, ~$Z_1$ = $\{3,5,13,23\}$,   
 			
 		\hspace{2cm}		$X_2$ = $\{5,13,15,23\}$, $Y_2$ = $\{1,15,17,19\}$, $Z_2$ = $\{7,11,15,25\}$,  
 			
 			\hspace{2cm}	$X_3$ = $\{7,11,21,25\}$, $Y_3$ = $\{5,13,21,23\}$, ~$Z_3$ = $\{1,17,19,21\}$.   
}		
\end{center} }

In this case, Eulerian isomorphism series exists and is given below. 

\begin{enumerate}
\item [\rm (a)] $C_{54}(X_1)$ $\cong_{T2_{54,3,2}}$ $C_{54}(Y_1)$ $\cong_{T1_{54,5}}$ $C_{54}(Y_2)$ $\cong_{T2_{54,3,2}}$ $C_{54}(Z_2)$ $\cong_{T1_{54,5}}$ $C_{54}(Z_3)$ 

\hfill $\cong_{T2_{54,3,2}}$ $C_{54}(X_3)$ $\cong_{T1_{54,7}}$ $C_{54}(X_2)$ $\cong_{T2_{54,3,4}}$ $C_{54}(Z_2)$ $\cong_{T1_{54,7}}$ $C_{54}(Z_1)$ $\cong_{T2_{54,3,4}}$ $C_{54}(Y_1)$ 

\hfill $\cong_{T1_{54,7}}$ $C_{54}(Y_3)$ $\cong_{T2_{54,3,4}}$ $C_{54}(X_3)$ $\cong_{T1_{54,5}}$ $C_{54}(X_1)$ $\cong_{T2_{54,3,4}}$ $C_{54}(Z_1)$ $\cong_{T1_{54,7}}$ $C_{54}(Z_3)$ 

\hfill $\cong_{T2_{27,3,4}}$ $C_{54}(Y_3)$ $\cong_{T1_{54,7}}$ $C_{54}(Y_2)$ $\cong_{T2_{54,3,4}}$ $C_{54}(X_2)$ $\cong_{T1_{54,5}}$ $C_{54}(X_3)$ 

\hfill $\cong_{T2_{54,3,2}}$ $C_{54}(Y_3)$ $\cong_{T1_{54,5}}$ $C_{54}(Y_1)$ $\cong_{T2_{54,3,2}}$ $C_{54}(Z_1)$ $\cong_{T1_{54,5}}$ $C_{54}(Z_2)$ 

\hfill $\cong_{T2_{54,3,4}}$ $C_{54}(Y_2)$ $\cong_{T1_{54,7}}$ $C_{54}(Y_1)$ $\cong_{T2_{54,3,4}}$ $C_{54}(X_1)$ $\cong_{T1_{54,5}}$ $C_{54}(X_2)$ 

\hfill $\cong_{T2_{54,3,2}}$ $C_{54}(Y_2)$ $\cong_{T1_{54,5}}$ $C_{54}(Y_3)$ $\cong_{T2_{54,3,2}}$ $C_{54}(Z_3)$
$\cong_{T1_{54,5}}$ $C_{54}(Z_1)$ 

\hfill  $\cong_{T2_{54,3,2}}$ $C_{54}(X_1)$ $\cong_{T1_{54,7}}$ $C_{54}(X_3)$ $\cong_{T2_{54,3,4}}$ $C_{54}(Z_3)$ 

\hfill $\cong_{T1_{54,7}}$ $C_{54}(Z_2)$ $\cong_{T2_{54,3,2}}$ $C_{54}(X_2)$ $\cong_{T1_{54,7}}$ $C_{54}(X_1)$.
\end{enumerate}

Figure 26 represents the digraph of an Eulerian isomorphism series of $C_{54}(R)$ presented above where $R$ = $\{1,3,17,19\}$, and $R_i$ represents $C_{54}(R_i)$, $R_i$ = $X_i, Y_i, Z_i$, and $1 \leq i \leq 3$.

{\rm 	\begin{center}
		\begin{tikzpicture}  
		[scale=.8,auto=center,every node/.style={draw,circle}]
			
\node (1) at (-10,-0.5) {\tiny{$X_1$}};
\node [scale=.0] (2) at (-9,-.93) [label=90: $>$]{};
\node (3) at (-8,-0.5) {\tiny{$Y_1$}};	
\node [scale=.0] (4) at (-7,-0.93) [label=91: $>$]{};
\node (5) at (-6,-0.5)  {\tiny{$Y_2$}};

\draw[ line width=0.2mm] [blue](1)[dashed]  to (3);
\draw[ line width=0.2mm] [blue](3) to (5);

\node [scale=.0] (6) at (-5,-.93) [label=90: $>$]{};
\node (7) at (-4,-0.5) {\tiny{$Z_2$}};	
\node [scale=.0] (8) at (-3,-0.93) [label=91: $>$]{};
\node (9) at (-2,-0.5)  {\tiny{$Z_3$}};

\draw[ line width=0.2mm] [blue](5)[dashed] to (7);
\draw[ line width=0.2mm] [blue](7) to (9);

\node [scale=.0] (10) at (-1,-.93) [label=90: $>$]{};
\node (11) at (0,-0.5) {\tiny{$X_3$}};	
\node [scale=.0] (12) at (1,-0.93) [label=91: $>$]{};
\node (13) at (2,-0.5)  {\tiny{$X_2$}};

\draw[ line width=0.2mm] [blue](9)[dashed]  to (11);
\draw[ line width=0.2mm] [blue](11) to (13);

\node [scale=.0] (14) at (3,-.93) [label=90: $>$]{};
\node (15) at (4,-0.5) {\tiny{$Z_2$}};	
\node [scale=.0] (16) at (5,-0.93) [label=91: $>$]{};
\node (17) at (6,-0.5)  {\tiny{$Z_1$}};

\draw[ line width=0.2mm] [blue](13)[dashed] to (15);
\draw[ line width=0.2mm] [blue](15) to (17);

\node [scale=.0] (18) at (6,-1.93) [label=91: $\downarrow$]{};
\node (19) at (6,-2.5) {\tiny{$Y_1$}};	
\node [scale=.0] (20) at (5,-2.93) [label=90: $<$]{};
\node (21) at (4,-2.5)  {\tiny{$Y_3$}};

\draw[ line width=0.2mm] [blue](17)[dashed] to (19);
\draw[ line width=0.2mm] [blue](19) to (21);

\node [scale=.0] (22) at (3,-2.93) [label=90: $<$]{};
\node (23) at (2,-2.5) {\tiny{$X_3$}};	
\node [scale=.0] (24) at (1,-2.93) [label=91: $<$]{};
\node (25) at (0,-2.5)  {\tiny{$X_1$}};

\draw[ line width=0.2mm] [blue](21)[dashed]  to (23);
\draw[ line width=0.2mm] [blue](23) to (25);

\node [scale=.0] (26) at (-1,-2.93) [label=90: $<$]{};
\node (27) at (-2,-2.5) {\tiny{$Z_1$}};	
\node [scale=.0] (28) at (-3,-2.93) [label=91: $<$]{};
\node (29) at (-4,-2.5)  {\tiny{$Z_3$}};

\draw[ line width=0.2mm] [blue](25)[dashed] to (27);
\draw[ line width=0.2mm] [blue](27) to (29);

\node [scale=.0] (30) at (-5,-2.93) [label=90: $<$]{};
\node (31) at (-6,-2.5) {\tiny{$Y_3$}};	
\node [scale=.0] (32) at (-7,-2.93) [label=91: $<$]{};
\node (33) at (-8,-2.5)  {\tiny{$Y_2$}};

\draw[ line width=0.2mm] [blue](29)[dashed]  to (31);
\draw[ line width=0.2mm] [blue](31) to (33);

\node [scale=.0] (34) at (-9,-2.93) [label=90: $<$]{};
\node (35) at (-10,-2.5) {\tiny{$X_2$}};	
\node [scale=.0] (36) at (-10,-3.93) [label=91: $\downarrow$]{};
\node (37) at (-10,-4.5)  {\tiny{$X_3$}};

\draw[ line width=0.2mm] [blue](33)[dashed] to (35);
\draw[ line width=0.2mm] [blue](35) to (37);

\node [scale=.0] (38) at (-9,-4.93) [label=90: $>$]{};
\node (39) at (-8,-4.5) {\tiny{$Y_3$}};	
\node [scale=.0] (40) at (-7,-4.93) [label=91: $>$]{};
\node (41) at (-6,-4.5)  {\tiny{$Y_1$}};

\draw[ line width=0.2mm] [blue](37)[dashed]  to (39);
\draw[ line width=0.2mm] [blue](39) to (41);

\node [scale=.0] (42) at (-5,-4.93) [label=90: $>$]{};
\node (43) at (-4,-4.5) {\tiny{$Z_1$}};	
\node [scale=.0] (44) at (-3,-4.93) [label=91: $>$]{};
\node (45) at (-2,-4.5)  {\tiny{$Z_2$}};

\draw[ line width=0.2mm] [blue](41)[dashed] to (43);
\draw[ line width=0.2mm] [blue](43) to (45);

\node [scale=.0] (46) at (-1,-4.93) [label=90: $>$]{};
\node (47) at (0,-4.5) {\tiny{$Y_2$}};	
\node [scale=.0] (48) at (1,-4.93) [label=91: $>$]{};
\node (49) at (2,-4.5)  {\tiny{$Y_1$}};

\draw[ line width=0.2mm] [blue](45)[dashed]  to (47);
\draw[ line width=0.2mm] [blue](47) to (49);

\node [scale=.0] (50) at (3,-4.93) [label=90: $>$]{};
\node (51) at (4,-4.5) {\tiny{$X_1$}};	
\node [scale=.0] (52) at (5,-4.93) [label=91: $>$]{};
\node (53) at (6,-4.5)  {\tiny{$X_2$}};

\draw[ line width=0.2mm] [blue](49)[dashed] to (51);
\draw[ line width=0.2mm] [blue](51) to (53);

\node [scale=.0] (54) at (6,-5.93) [label=91: $\downarrow$]{};
\node (55) at (6,-6.5) {\tiny{$Y_2$}};	
\node [scale=.0] (56) at (5,-6.93) [label=90: $<$]{};
\node (57) at (4,-6.5)  {\tiny{$Y_3$}};

\draw[ line width=0.2mm] [blue](53)[dashed] to (55);
\draw[ line width=0.2mm] [blue](55) to (57);

\node [scale=.0] (58) at (3,-6.93) [label=90: $<$]{};
\node (59) at (2,-6.5) {\tiny{$Z_3$}};	
\node [scale=.0] (60) at (1,-6.93) [label=91: $<$]{};
\node (61) at (0,-6.5)  {\tiny{$Z_1$}};

\draw[ line width=0.2mm] [blue](57)[dashed]  to (59);
\draw[ line width=0.2mm] [blue](59) to (61);

\node [scale=.0] (62) at (-1,-6.93) [label=90: $<$]{};
\node (63) at (-2,-6.5) {\tiny{$X_1$}};	
\node [scale=.0] (64) at (-3,-6.93) [label=91: $<$]{};
\node (65) at (-4,-6.5)  {\tiny{$X_3$}};

\draw[ line width=0.2mm] [blue](61)[dashed] to (63);
\draw[ line width=0.2mm] [blue](63) to (65);

\node [scale=.0] (66) at (-5,-6.93) [label=90: $<$]{};
\node (67) at (-6,-6.5) {\tiny{$Z_3$}};	
\node [scale=.0] (68) at (-7,-6.93) [label=91: $<$]{};
\node (69) at (-8,-6.5)  {\tiny{$Z_2$}};

\draw[ line width=0.2mm] [blue](65)[dashed]  to (67);
\draw[ line width=0.2mm] [blue](67) to (69);

\node [scale=.0] (70) at (-9,-6.93) [label=90: $<$]{};
\node (71) at (-10,-6.5) {\tiny{$X_2$}};	
\node [scale=.0] (72) at (-11,-4.1) [label=91: $\uparrow$]{};

\draw[ line width=0.2mm] [blue](69)[dashed] to (71);
\draw[-, line width=0.2mm] [blue](71) to [out=120,in=240] (1);
	
\end{tikzpicture}

\vspace{.2cm}		
{\small  Fig. 26. Digraph of Eulerian isomorphism series of $C_{54}(R_i)$ with the representation of 

$R_i$ in the place of $C_{54}(R_i)$ where $R_i$ = $X_i, Y_i, Z_i$ for $i$ = 1,2,3.    }
\end{center} }

Using the above definition of the isomorphism digraph and remark \ref{r4.5}, we get the following result on circulant graphs.

\begin{theorem} \quad \label{t4.9} {\rm Let $n \geq 5$ and $C_n(R)$ be a  circulant graph having isomorphic  circulant graphs of Type-1 and Type-2. Then the isomorphism digraph of $C_n(R)$ exists and is a non-trivial connected symmetric digraph.  }
\end{theorem}
\begin{proof} \quad Given that circulant graph $C_n(R)$ is having isomorphic circulant graphs of Type-1 and Type-2. This implies, the isomorphism digraph $\mathcal{D}$ = $\mathcal{ID}(C_n(R))$ is a non-trivial connected digraph which is also a symmetric digraph, follows from Lemma \ref{l3.11}.
\end{proof}

\section{Isomorphism digraph and graph of $C_{432}(16, 27, 48, 54, 128, 160, 189)$ }

In this section, we obtain the isomorphism digraph and graph of $C_{432}(R)$ which has isomorphic circulant graphs of Type-2 w.r.t. $m$ = 2 as well as $m$ = 3 where $R$ = $\{16, 27, 48, 54, 128, 160, 189\}$. In problems 1 and 3 in \cite{v2-7}, we obtained all isomorphic circulant graphs of Type-1 and Type-2 w.r.t. $m$ = 2 as well as w.r.t. $m$ = 3 of $C_{432}(R)$ which are presented below. 

\begin{prm} {\rm \cite{v2-7} \label{p8} \quad $(T1_{432}(C_{432}(X_i)), \circ)$ is an Abelian group for $i$ = 1 to 6, $X_i$ = $A_i, B_i, \dots, F_i$, where $T1_{432}(C_{432}(X_i))$ is given as follows.
\begin{enumerate}
\item [\rm (A)] $T1_{432}(C_{432}(A_1))$ = $\{C_{432}(A_i): A_1 = R_1, i = 1,2,\dots,6\}$, 
	
	$A_1$ = $\{16, 27, 48, 54, 128, 160, 189\}$ = $R_1$, 
	
	$A_2$ = $\{64, 80, 81, 135, 162, 192, 208\}$ = $5A_1$, 
	
	$A_3$ = $\{27, 32, 54, 96, 112, 176, 189\}$ = $7A_1$, 
	
	$A_4$ = $\{32, 81, 96, 112, 135, 162, 176\}$ = $11A_1$, 
	
	$A_5$ = $\{16, 48, 81, 128, 135, 160, 162\}$ = $19A_1$,
	
	$A_6$ = $\{27, 54, 64, 80, 189, 192, 208\}$ = $23A_1$ under reflexive modulo 432.
	
\item [\rm (B)] $T1_{432}(C_{432}(B_1))$ = $\{C_{432}(B_i): B_1 = S_1, i = 1,2,\dots,6\}$, 

$B_1$ = $\{27, 32, 48, 54, 112, 176, 189\}$ = $S_1$,

$B_2$ = $\{16, 81, 128, 135, 160, 162, 192\}$ = $5B_1$, 

$B_3$ = $\{27, 54, 64, 80, 96, 189, 208\}$ = $7B_1$,

$B_4$ = $\{64, 80, 81, 96, 135, 162, 208\}$ = $11B_1$, 

$B_5$ = $\{32, 48, 81, 112, 135, 162, 176\}$ = $19B_1$, 

$B_6$ = $\{16, 27, 54, 128, 160, 189, 192\}$ = $23B_1$ under reflexive modulo 432.

\item [\rm (C)] $T1_{432}(C_{432}(C_1))$ = $\{C_{432}(C_i): C_1 = T_1, i = 1,2,\dots,6\}$, 

$C_1$ = $\{27, 48, 54, 64, 80, 189, 208\}$ = $T_1$, 

$C_2$ = $\{32, 81, 112, 135, 162, 176, 192\}$ = $5\times C_1$, 

$C_3$ = $\{16, 27, 54, 96, 128, 160, 189\}$ = $7\times C_1$, 

$C_4$ = $\{16, 81, 96, 128, 135, 160, 162\}$ = $11\times C_1$, 

$C_5$ = $\{48, 64, 80, 81, 135, 162, 208\}$ = $19\times C_1$, 

$C_6$ = $\{27, 32, 54, 112, 176, 189, 192\}$ = $23\times C_1$ under reflexive modulo 432.

\item [\rm (D)] $T1_{432}(C_{432}(D_1))$ = $\{C_{432}(D_i): D_1 = R_2, i = 1,2,\dots,6\}$, 

$D_1$ = $\{16, 48, 54, 81, 128, 135, 160\}$ = $R_2$, 

$D_2$ = $\{27, 64, 80, 162, 189, 192, 208\}$ = $5D_1$, 

$D_3$ = $\{32, 54, 81, 96, 112, 135, 176\}$ = $7D_1$, 

$D_4$ = $\{27, 32, 96, 112, 162, 176, 189\}$ = $11D_1$,  

$D_5$ = $\{16, 27, 48, 128, 160, 162, 189\}$ = $19D_1$, 

$D_6$ = $\{54, 64, 80, 81, 135, 192, 208\}$ = $23D_1$ under reflexive modulo 432.

\item [\rm (E)] $T1_{432}(C_{432}(E_1))$ = $\{C_{432}(E_i): E_1 = S_2, i = 1,2,\dots,6\}$, 

$E_1$ = $\{32, 48, 54, 81, 112, 135, 176\}$ = $S_2$, 

$E_2$ = $\{16, 27, 128, 160, 162, 189, 192\}$ = $5E_1$, 

$E_3$ = $\{54, 64, 80, 81, 96, 135, 208\}$ = $7E_1$, 

$E_4$ = $\{27, 64, 80, 96, 162, 189, 208\}$ = $11E_1$, 

$E_5$ = $\{27, 32, 48, 112, 162, 176, 189\}$ = $19E_1$,

$E_6$ = $\{16, 54, 81, 128, 135, 160, 192\}$ = $23E_1$ under reflexive modulo 432.

\item [\rm (F)] $T1_{432}(C_{432}(F_1))$ = $\{C_{432}(F_i): F_1 = T_2, i = 1,2,\dots,6\}$, 

$F_1$ = $\{48, 54, 64, 80, 81, 135, 208\}$ = $T_2$, 

$F_2$ = $\{27, 32, 112, 162, 176, 189, 192\}$ = $5F_1$, 

$F_3$ = $\{16, 54, 81, 96, 128, 135, 160\}$ = $7F_1$, 

$F_4$ = $\{16, 27, 96, 128, 160, 162, 189\}$ = $11F_1$, 

$F_5$ = $\{27, 48, 64, 80, 162, 189, 208\}$ = $19F_1$, 

$F_6$ = $\{32, 54, 81, 112, 135, 176, 192\}$ = $23F_1$ under reflexive modulo 432. \end{enumerate} }
\end{prm}

\begin{prm} {\rm \cite{v2-7}} \quad \label{p9} {\rm Let $`\circ'$ be as given in definition \ref{d9}. For $i$ = 1, 2 and $j$ = 1 to 6, $R_i, S_i, T_i$, $A_j, B_j, \dots, F_j$ be as given above with $A_1$ = $R_1$, $B_1$ = $S_1$, $C_1$ = $T_1$, $D_1$ = $R_2$, $E_1$ = $S_2$ and $F_1$ = $T_2$. Then, the following statements are true.
		\begin{enumerate}	
			\item [\rm (a)]  Circulant graphs occuring in each of the following cases are either Type-2 isomorphic w.r.t. $m$ = 2 or Type-2 isomorphic w.r.t. $m$ = 3 as given below. 	
			
			\item [\rm (1A1)] $C_{432}(A_1)$ $\cong_{T2_{432,2,54}}$ $C_{432}(D_1)$; 	
			\item [\rm (2A1)] $C_{432}(A_1)$ $\cong_{T2_{432,3,32}}$ $C_{432}(B_1)$; 
			\item [\rm (3A1)] $C_{432}(A_1)$ $\cong_{T2_{432,3,64}}$ $C_{432}(C_1)$; $C_{432}(A_1)$ $\cong_{T2_{432,3,16}}$ $C_{432}(C_1)$;\\
			
			\item [\rm (1A2)] $C_{432}(A_2)$ $\cong_{T2_{432,2,54}}$ $C_{432}(D_2)$; 	
			\item [\rm (2A2)] $C_{432}(A_2)$ $\cong_{T2_{432,3,32}}$ $C_{432}(B_2)$; 
			\item [\rm (3A2)] $C_{432}(A_2)$ $\cong_{T2_{432,3,64}}$ $C_{432}(C_2)$; $C_{432}(A_2)$ $\cong_{T2_{432,3,16}}$ $C_{432}(C_2)$;\\
			
			\item [\rm (1A3)] $C_{432}(A_3)$ $\cong_{T2_{432,2,54}}$ $C_{432}(D_3)$; 	
			\item [\rm (2A3)] $C_{432}(A_3)$ $\cong_{T2_{432,3,32}}$ $C_{432}(B_3)$; 
			\item [\rm (3A3)] $C_{432}(A_3)$ $\cong_{T2_{432,3,64}}$ $C_{432}(C_3)$; $C_{432}(A_3)$ $\cong_{T2_{432,3,16}}$ $C_{432}(C_3)$;\\
			
			\item [\rm (1A4)] $C_{432}(A_4)$ $\cong_{T2_{432,2,54}}$ $C_{432}(D_4)$; 	
			\item [\rm (2A4)] $C_{432}(A_4)$ $\cong_{T2_{432,3,32}}$ $C_{432}(B_4)$; 
			\item [\rm (3A4)] $C_{432}(A_4)$ $\cong_{T2_{432,3,64}}$ $C_{432}(C_4)$; $C_{432}(A_4)$ $\cong_{T2_{432,3,16}}$ $C_{432}(C_4)$;\\
			
			\item [\rm (1A5)]	$C_{432}(A_5)$ $\cong_{T2_{432,2,54}}$ $C_{432}(D_5)$; 	
			\item [\rm (2A5)] $C_{432}(A_5)$ $\cong_{T2_{432,3,32}}$ $C_{432}(B_5)$; 
			\item [\rm (3A5)] $C_{432}(A_5)$ $\cong_{T2_{432,3,64}}$ $C_{432}(C_5)$; $C_{432}(A_5)$ $\cong_{T2_{432,3,16}}$ $C_{432}(C_5)$;\\
			
			\item [\rm (1A6)]	$C_{432}(A_6)$ $\cong_{T2_{432,2,54}}$ $C_{432}(D_6)$; 	
			\item [\rm (2A6)] $C_{432}(A_6)$ $\cong_{T2_{432,3,32}}$ $C_{432}(B_6)$; 
			\item [\rm (3A6)] $C_{432}(A_6)$ $\cong_{T2_{432,3,64}}$ $C_{432}(C_6)$; $C_{432}(A_6)$ $\cong_{T2_{432,3,16}}$ $C_{432}(C_6)$;\\
			
			\item [\rm (1B1)] $C_{432}(B_1)$ $\cong_{T2_{432,2,54}}$ $C_{432}(E_1)$; 	
			\item [\rm (2B1)] $C_{432}(B_1)$ $\cong_{T2_{432,3,32}}$ $C_{432}(C_1)$; 
			\item [\rm (3B1)] $C_{432}(B_1)$ $\cong_{T2_{432,3,64}}$ $C_{432}(A_1)$; $C_{432}(B_1)$ $\cong_{T2_{432,3,16}}$ $C_{432}(A_1)$;\\
			
			\item [\rm (1B2)] $C_{432}(B_2)$ $\cong_{T2_{432,2,54}}$ $C_{432}(E_2)$; 	
			\item [\rm (2B2)] $C_{432}(B_2)$ $\cong_{T2_{432,3,32}}$ $C_{432}(C_2)$; 
			\item [\rm (3B2)] $C_{432}(B_2)$ $\cong_{T2_{432,3,64}}$ $C_{432}(A_2)$; $C_{432}(B_2)$ $\cong_{T2_{432,3,16}}$ $C_{432}(A_2)$;\\
			
			\item [\rm (1B3)] $C_{432}(B_3)$ $\cong_{T2_{432,2,54}}$ $C_{432}(E_3)$; 	
			\item [\rm (2B3)] $C_{432}(B_3)$ $\cong_{T2_{432,3,32}}$ $C_{432}(C_3)$; 
			\item [\rm (3B3)] $C_{432}(B_3)$ $\cong_{T2_{432,3,64}}$ $C_{432}(A_3)$; $C_{432}(B_3)$ $\cong_{T2_{432,3,16}}$ $C_{432}(A_3)$;\\
			
			\item [\rm (1B4)] $C_{432}(B_4)$ $\cong_{T2_{432,2,54}}$ $C_{432}(E_4)$; 	
			\item [\rm (2B4)] $C_{432}(B_4)$ $\cong_{T2_{432,3,32}}$ $C_{432}(C_4)$; 
			\item [\rm (3B4)] $C_{432}(B_4)$ $\cong_{T2_{432,3,64}}$ $C_{432}(A_4)$; $C_{432}(B_4)$ $\cong_{T2_{432,3,16}}$ $C_{432}(A_4)$;\\
			
			\item [\rm (1B5)] $C_{432}(B_5)$ $\cong_{T2_{432,2,54}}$ $C_{432}(E_5)$; 	
			\item [\rm (2B5)] $C_{432}(B_5)$ $\cong_{T2_{432,3,32}}$ $C_{432}(C_5)$; 
			\item [\rm (3B5)] $C_{432}(B_5)$ $\cong_{T2_{432,3,64}}$ $C_{432}(A_5)$; $C_{432}(B_5)$ $\cong_{T2_{432,3,16}}$ $C_{432}(A_5)$;\\
			
			\item [\rm (1B6)] $C_{432}(B_6)$ $\cong_{T2_{432,2,54}}$ $C_{432}(E_6)$; 	
			\item [\rm (2B6)] $C_{432}(B_6)$ $\cong_{T2_{432,3,32}}$ $C_{432}(C_6)$; 
			\item [\rm (3B6)] $C_{432}(B_6)$ $\cong_{T2_{432,3,64}}$ $C_{432}(A_6)$; $C_{432}(B_6)$ $\cong_{T2_{432,3,16}}$ $C_{432}(A_6)$;\\
			
			\item [\rm (1C1)] $C_{432}(C_1)$ $\cong_{T2_{432,2,54}}$ $C_{432}(F_1)$; 	
			\item [\rm (2C1)] $C_{432}(C_1)$ $\cong_{T2_{432,3,32}}$ $C_{432}(A_1)$; 
			\item [\rm (3C1)] $C_{432}(C_1)$ $\cong_{T2_{432,3,64}}$ $C_{432}(B_1)$; $C_{432}(C_1)$ $\cong_{T2_{432,3,16}}$ $C_{432}(B_1)$;\\
			
			\item [\rm (1C2)] $C_{432}(C_2)$ $\cong_{T2_{432,2,54}}$ $C_{432}(F_2)$; 	
			\item [\rm (2C2)] $C_{432}(C_2)$ $\cong_{T2_{432,3,32}}$ $C_{432}(A_2)$; 
			\item [\rm (3C2)] $C_{432}(C_2)$ $\cong_{T2_{432,3,64}}$ $C_{432}(B_2)$; $C_{432}(C_2)$ $\cong_{T2_{432,3,16}}$ $C_{432}(B_2)$;\\
			
			\item [\rm (1C3)] $C_{432}(C_3)$ $\cong_{T2_{432,2,54}}$ $C_{432}(F_3)$; 	
			\item [\rm (2C3)] $C_{432}(C_3)$ $\cong_{T2_{432,3,32}}$ $C_{432}(A_3)$; 
			\item [\rm (3C3)] $C_{432}(C_3)$ $\cong_{T2_{432,3,64}}$ $C_{432}(B_3)$; $C_{432}(C_3)$ $\cong_{T2_{432,3,16}}$ $C_{432}(B_3)$;\\
			
			\item [\rm (1C4)] $C_{432}(C_4)$ $\cong_{T2_{432,2,54}}$ $C_{432}(F_4)$; 	
			\item [\rm (2C4)] $C_{432}(C_4)$ $\cong_{T2_{432,3,32}}$ $C_{432}(A_4)$; 
			\item [\rm (3C4)] $C_{432}(C_4)$ $\cong_{T2_{432,3,64}}$ $C_{432}(B_4)$; $C_{432}(C_4)$ $\cong_{T2_{432,3,16}}$ $C_{432}(B_4)$;\\
			
			\item [\rm (1C5)] $C_{432}(C_5)$ $\cong_{T2_{432,2,54}}$ $C_{432}(F_5)$; 	
			\item [\rm (2C5)] $C_{432}(C_5)$ $\cong_{T2_{432,3,32}}$ $C_{432}(A_5)$; 
			\item [\rm (3C5)] $C_{432}(C_5)$ $\cong_{T2_{432,3,64}}$ $C_{432}(B_5)$; $C_{432}(C_5)$ $\cong_{T2_{432,3,16}}$ $C_{432}(B_5)$;\\
			
			\item [\rm (1C6)] $C_{432}(C_6)$ $\cong_{T2_{432,2,54}}$ $C_{432}(F_6)$; 	
			\item [\rm (2C6)] $C_{432}(C_6)$ $\cong_{T2_{432,3,32}}$ $C_{432}(A_6)$; 
			\item [\rm (3C6)] $C_{432}(C_6)$ $\cong_{T2_{432,3,64}}$ $C_{432}(B_6)$; $C_{432}(C_6)$ $\cong_{T2_{432,3,16}}$ $C_{432}(B_6)$;\\
			
			\item [\rm (1D1)] $C_{432}(D_1)$ $\cong_{T2_{432,2,54}}$ $C_{432}(A_1)$; 	
			\item [\rm (2D1)] $C_{432}(D_1)$ $\cong_{T2_{432,3,32}}$ $C_{432}(E_1)$; 
			\item [\rm (3D1)] $C_{432}(D_1)$ $\cong_{T2_{432,3,64}}$ $C_{432}(F_1)$; $C_{432}(D_1)$ $\cong_{T2_{432,3,16}}$ $C_{432}(F_1)$;\\
			
			\item [\rm (1D2)] $C_{432}(D_2)$ $\cong_{T2_{432,2,54}}$ $C_{432}(A_2)$; 	
			\item [\rm (2D2)] $C_{432}(D_2)$ $\cong_{T2_{432,3,32}}$ $C_{432}(E_2)$; 
			\item [\rm (3D2)] $C_{432}(D_2)$ $\cong_{T2_{432,3,64}}$ $C_{432}(F_2)$; $C_{432}(D_2)$ $\cong_{T2_{432,3,16}}$ $C_{432}(F_2)$;\\
			
			\item [\rm (1D3)] $C_{432}(D_3)$ $\cong_{T2_{432,2,54}}$ $C_{432}(A_3)$; 	
			\item [\rm (2D3)] $C_{432}(D_3)$ $\cong_{T2_{432,3,32}}$ $C_{432}(E_3)$; 
			\item [\rm (3D3)] $C_{432}(D_3)$ $\cong_{T2_{432,3,64}}$ $C_{432}(F_3)$; $C_{432}(D_3)$ $\cong_{T2_{432,3,16}}$ $C_{432}(F_3)$;\\
			
			\item [\rm (1D4)] $C_{432}(D_4)$ $\cong_{T2_{432,2,54}}$ $C_{432}(A_4)$; 	
			\item [\rm (2D4)] $C_{432}(D_4)$ $\cong_{T2_{432,3,32}}$ $C_{432}(E_4)$; 
			\item [\rm (3D4)] $C_{432}(D_4)$ $\cong_{T2_{432,3,64}}$ $C_{432}(F_4)$; $C_{432}(D_4)$ $\cong_{T2_{432,3,16}}$ $C_{432}(F_4)$;\\
			
			\item [\rm (1D5)] $C_{432}(D_5)$ $\cong_{T2_{432,2,54}}$ $C_{432}(A_5)$; 	
			\item [\rm (2D5)] $C_{432}(D_5)$ $\cong_{T2_{432,3,32}}$ $C_{432}(E_5)$; 
			\item [\rm (3D5)] $C_{432}(D_5)$ $\cong_{T2_{432,3,64}}$ $C_{432}(F_5)$; $C_{432}(D_5)$ $\cong_{T2_{432,3,16}}$ $C_{432}(F_5)$;\\
			
			\item [\rm (1D6)] $C_{432}(D_6)$ $\cong_{T2_{432,2,54}}$ $C_{432}(A_6)$; 	
			\item [\rm (2D6)] $C_{432}(D_6)$ $\cong_{T2_{432,3,32}}$ $C_{432}(E_6)$; 
			\item [\rm (3D6)] $C_{432}(D_6)$ $\cong_{T2_{432,3,64}}$ $C_{432}(F_6)$; $C_{432}(D_6)$ $\cong_{T2_{432,3,16}}$ $C_{432}(F_6)$;\\
			
			\item [\rm (1E1)] $C_{432}(E_1)$ $\cong_{T2_{432,2,54}}$ $C_{432}(B_1)$; 	
			\item [\rm (2E1)] $C_{432}(E_1)$ $\cong_{T2_{432,3,32}}$ $C_{432}(F_1)$; 
			\item [\rm (3E1)] $C_{432}(E_1)$ $\cong_{T2_{432,3,64}}$ $C_{432}(D_1)$; $C_{432}(E_1)$ $\cong_{T2_{432,3,16}}$ $C_{432}(D_1)$;\\
			
			\item [\rm (1E2)] $C_{432}(E_2)$ $\cong_{T2_{432,2,54}}$ $C_{432}(B_2)$; 	
			\item [\rm (2E2)] $C_{432}(E_2)$ $\cong_{T2_{432,3,32}}$ $C_{432}(F_2)$; 
			\item [\rm (3E2)] $C_{432}(E_2)$ $\cong_{T2_{432,3,64}}$ $C_{432}(D_2)$; $C_{432}(E_2)$ $\cong_{T2_{432,3,16}}$ $C_{432}(D_2)$;\\
			
			\item [\rm (1E3)] $C_{432}(E_3)$ $\cong_{T2_{432,2,54}}$ $C_{432}(B_3)$; 	
			\item [\rm (2E3)] $C_{432}(E_3)$ $\cong_{T2_{432,3,32}}$ $C_{432}(F_3)$; 
			\item [\rm (3E3)] $C_{432}(E_3)$ $\cong_{T2_{432,3,64}}$ $C_{432}(D_3)$; $C_{432}(E_3)$ $\cong_{T2_{432,3,16}}$ $C_{432}(D_3)$;\\
			
			\item [\rm (1E4)] $C_{432}(E_4)$ $\cong_{T2_{432,2,54}}$ $C_{432}(B_4)$; 	
			\item [\rm (2E4)] $C_{432}(E_4)$ $\cong_{T2_{432,3,32}}$ $C_{432}(F_4)$; 
			\item [\rm (3E4)] $C_{432}(E_4)$ $\cong_{T2_{432,3,64}}$ $C_{432}(D_4)$; $C_{432}(E_4)$ $\cong_{T2_{432,3,16}}$ $C_{432}(D_4)$;\\
			
			\item [\rm (1E5)] $C_{432}(E_5)$ $\cong_{T2_{432,2,54}}$ $C_{432}(B_5)$; 	
			\item [\rm (2E5)] $C_{432}(E_5)$ $\cong_{T2_{432,3,32}}$ $C_{432}(F_5)$; 
			\item [\rm (3E5)] $C_{432}(E_5)$ $\cong_{T2_{432,3,64}}$ $C_{432}(D_5)$; $C_{432}(E_5)$ $\cong_{T2_{432,3,16}}$ $C_{432}(D_5)$;\\
			
			\item [\rm (1E6)] $C_{432}(E_6)$ $\cong_{T2_{432,2,54}}$ $C_{432}(B_6)$; 	
			\item [\rm (2E6)] $C_{432}(E_6)$ $\cong_{T2_{432,3,32}}$ $C_{432}(F_6)$; 
			\item [\rm (3E6)] $C_{432}(E_6)$ $\cong_{T2_{432,3,64}}$ $C_{432}(D_6)$; $C_{432}(E_6)$ $\cong_{T2_{432,3,16}}$ $C_{432}(D_6)$;\\
			
			\item [\rm (1F1)] $C_{432}(F_1)$ $\cong_{T2_{432,2,54}}$ $C_{432}(C_1)$; 	
			\item [\rm (2F1)] $C_{432}(F_1)$ $\cong_{T2_{432,3,32}}$ $C_{432}(D_1)$; 
			\item [\rm (3F1)] $C_{432}(F_1)$ $\cong_{T2_{432,3,64}}$ $C_{432}(E_1)$; $C_{432}(F_1)$ $\cong_{T2_{432,3,16}}$ $C_{432}(E_1)$;\\
			
			\item [\rm (1F2)] $C_{432}(F_2)$ $\cong_{T2_{432,2,54}}$ $C_{432}(C_2)$; 	
			\item [\rm (2F2)] $C_{432}(F_2)$ $\cong_{T2_{432,3,32}}$ $C_{432}(D_2)$; 
			\item [\rm (3F2)] $C_{432}(F_2)$ $\cong_{T2_{432,3,64}}$ $C_{432}(E_2)$; $C_{432}(F_2)$ $\cong_{T2_{432,3,16}}$ $C_{432}(E_2)$;\\
			
			\item [\rm (1F3)] $C_{432}(F_3)$ $\cong_{T2_{432,2,54}}$ $C_{432}(C_3)$; 	
			\item [\rm (2F3)] $C_{432}(F_3)$ $\cong_{T2_{432,3,32}}$ $C_{432}(D_3)$; 
			\item [\rm (3F3)] $C_{432}(F_3)$ $\cong_{T2_{432,3,64}}$ $C_{432}(E_3)$; $C_{432}(F_3)$ $\cong_{T2_{432,3,16}}$ $C_{432}(E_3)$;\\
			
			\item [\rm (1F4)] $C_{432}(F_4)$ $\cong_{T2_{432,2,54}}$ $C_{432}(C_4)$; 	
			\item [\rm (2F4)] $C_{432}(F_4)$ $\cong_{T2_{432,3,32}}$ $C_{432}(D_4)$; 
			\item [\rm (3F4)] $C_{432}(F_4)$ $\cong_{T2_{432,3,64}}$ $C_{432}(E_4)$; $C_{432}(F_4)$ $\cong_{T2_{432,3,16}}$ $C_{432}(E_4)$;\\
			
			\item [\rm (1F5)] $C_{432}(F_5)$ $\cong_{T2_{432,2,54}}$ $C_{432}(C_5)$; 	
			\item [\rm (2F5)] $C_{432}(F_5)$ $\cong_{T2_{432,3,32}}$ $C_{432}(D_5)$; 
			\item [\rm (3F5)] $C_{432}(F_5)$ $\cong_{T2_{432,3,64}}$ $C_{432}(E_5)$; $C_{432}(F_5)$ $\cong_{T2_{432,3,16}}$ $C_{432}(E_5)$;\\
			
			\item [\rm (1F6)] $C_{432}(F_6)$ $\cong_{T2_{432,2,54}}$ $C_{432}(C_6)$; 	
			\item [\rm (2F6)] $C_{432}(F_6)$ $\cong_{T2_{432,3,32}}$ $C_{432}(D_6)$; 
			\item [\rm (3F6)] $C_{432}(F_6)$ $\cong_{T2_{432,3,64}}$ $C_{432}(E_6)$; $C_{432}(F_6)$ $\cong_{T2_{432,3,16}}$ $C_{432}(E_6)$. 
		\end{enumerate}	
		
		\begin{enumerate}	
			\item [\rm (b1)]  $T2_{432,2}(C_{432}(A_i))$ = $\{C_{432}(A_i), C_{432}(D_i)\}$ = $T2_{432,2}(C_{432}(D_i))$ for $i$ = 1 to 6;  
			
			\item [\rm (b2)]  $T2_{432,2}(C_{432}(B_i))$ = $\{C_{432}(B_i), C_{432}(E_i)\}$ = $T2_{432,2}(C_{432}(E_i))$ for $i$ = 1 to 6;
			
			\item [\rm (b3)]  $T2_{432,2}(C_{432}(C_i))$ = $\{C_{432}(C_i), C_{432}(F_i)\}$ = $T2_{432,2}(C_{432}(F_i))$ for $i$ = 1 to 6.
			
			\item [\rm (c1)]  $T2_{432,3}(C_{432}(A_i))$ = $\{C_{432}(A_i), C_{432}(B_i), C_{432}(C_i)\}$ 
			
			\hfill = $T2_{432,2}(C_{432}(B_i))$ = $T2_{432,3}(C_{432}(C_i))$ for $i$ = 1 to 6;  
			
			\item [\rm (c2)] $T2_{432,3}(C_{432}(D_i))$ = $\{C_{432}(D_i), C_{432}(E_i), C_{432}(F_i)\}$ 
			
			\hfill = $T2_{432,2}(C_{432}(E_i))$ = $T2_{432,3}(C_{432}(F_i))$ for $i$ = 1 to 6.  
			
			\item [\rm (d)]  $(T2_{432,2}(C_{432}(X_i)), \circ)$ is an Abelian group for $X_i$ = $A_i, B_i, C_i$ and $i$ = 1 to 6.  
			
			\item [\rm (e)]  $(T2_{432,3}(C_{432}(Y_i)), \circ)$ is an Abelian group for $Y_i$ = $A_i, D_i$ and $i$ = 1 to 6.  \hfill $\Box$			
	\end{enumerate}		}
\end{prm}

 \begin{prm} \quad \label{p10} {\rm Find the isomorphism digraph, and the isomorphism graph of $C_{432}(R)$ where $R$ = $\{16, 27, 48, 54, 128, 160, 189\}$. }
 \end{prm}
 \noindent
 {\bf Solution.}\quad $C_{432}(16, 27, 48, 54, 128, 160, 189)$ has Type-2 isomorphic circulant graphs w.r.t. $m$ = 2 as well as $m$ = 3, see problems 8 and 9. Using them, we present $T1_{432}(C_{432}(R))$ and $T2_{432, m}(C_{432}(R))$ for $m$ = 2, 3 for different isomorphic circulant graphs of $C_{432}(R)$ where $R$ = $\{16, 27, 48, 54, 128, 160, 189\}$.
 
 Let $R$ = $R_1$ = $A_1$ and for $i$ = 1 to 6, $A_i, B_i, C_i, D_i, E_i, F_i$ be as given in the begining of this section. Let $\mathcal{D}$ be the isomorphism digraph and $\mathcal{G}$ be the isomorphism graph of $C_{432}(R)$. Then, 
 
 $Isoset(C_{432}(R))$ = $\{C_{432}(X_i)/~ X_i = A_i,B_i,\dots,F_i~\text{and}~i = 1,2,\dots,6\}$ = $V(\mathcal{D})$ = $V(\mathcal{G})$.
 
Using solution obtained in problem 3 in \cite{v2-7} with $R$ = $A_1$ = $R_1$, $B_1$ = $S_1$, $C_1$ = $T_1$, $D_1$ = $R_2$, $E_1$ = $S_2$, $F_1$ = $T_2$, and for $i$ = 1 to 6, we get,

$T2_{432,2}(C_{432}(A_i))$ = $\{C_{432}(A_i), C_{432}(D_i) = \theta_{432,2,54}(C_{432}(A_i))\}$ = $T2_{432,2}(C_{432}(D_i))$,  

$T2_{432,2}(C_{432}(B_i))$ = $\{C_{432}(B_i), C_{432}(E_i) = \theta_{432,2,54}(C_{432}(B_i))\}$ = $T2_{432,2}(C_{432}(E_i))$,

$T2_{432,2}(C_{432}(C_i))$ = $\{C_{432}(C_i), C_{432}(F_i) = \theta_{432,2,54}(C_{432}(C_i))\}$ = $T2_{432,2}(C_{432}(F_i))$,

$T2_{432,3}(C_{432}(A_i))$ = $\{C_{432}(A_i), C_{432}(B_i) = \theta_{432,3,32}(C_{432}(A_i)), C_{432}(C_i) = \theta_{432,3,64}(C_{432}(C_i))\}$ 

\hfill = $T2_{432,3}(C_{432}(B_i))$ = $T2_{432,3}(C_{432}(C_i))$,  

$T2_{432,3}(C_{432}(D_i))$ = $\{C_{432}(D_i), C_{432}(E_i) = \theta_{432,3,32}(C_{432}(D_i)), C_{432}(F_i) = \theta_{432,3,64}(C_{432}(D_i))\}$ 

\hfill = $T2_{432,3}(C_{432}(E_i))$ = $T2_{432,3}(C_{432}(F_i))$,

$T1_{432}(C_{432}(\alpha_1))$ = $\{C_{432}(\alpha_i)/~i = 1,2,\dots,6\}$, $\alpha_1$ = $A_1, B_1, \dots, F_1$, and 

\medskip
\noindent
$V(\mathcal{D})$ = $V(\mathcal{G})$ = $Isoset(C_{432}(R))$ = $Isoset_{432,2}(C_{432}(R))$ $\cup$ $Isoset_{432,3}(C_{432}(R))$

\hfill  = $\big( \bigcup^6_{i=1} \bigcup_{X_i=A_i,B_i,C_i}  {{Isoset_{432,2}(C_{432}(X_i))}}\big)$ ~$\bigcup$~ $\big(\bigcup^6_{i=1} \bigcup_{X_i=A_i,D_i}  {{Isoset_{432,3}(C_{432}(X_i))}}\big)$

\hspace{1.4cm}  = $\big( \bigcup^6_{i=1} \bigcup_{X_i=A_i,B_i,C_i}  {{Isoset_{432,2}(C_{432}(X_i))}}\big)$ =  $\big(\bigcup^6_{i=1} \bigcup_{X_i=A_i,D_i}  {{Isoset_{432,3}(C_{432}(X_i))}}\big)$

\hspace{1.4cm} = $\{C_{432}(X_i)/~ X_i = A_i,B_i,\dots,F_i~\text{and}~i = 1,2,\dots,6\}$. 

\medskip
The isomorphism digraph $\mathcal{D}$ =  $\mathcal{ID}_{432}(C_{432}(R))$ = $\mathcal{ID}_{432}(C_{432}(X_i))$, and the isomorphism graph 

\medskip
$\mathcal{G}$ = $\mathcal{I}_{432}(C_{432}(R))$ = $\mathcal{I}_{432}(C_{432}(X_i))$ are presented in figures 33 and 34,
$X_i$ = $A_i,B_i,\dots,F_i$ and $i$ = $1,2,\dots,6$. 

\medskip
The isomorphism digraph $\mathcal{D}$ of $C_{432}(R)$ can be written as $\mathcal{D}$ = $\mathcal{D}_1$ $\cup$ $\mathcal{D}_2$, 

\medskip
\hspace{.5cm} $\mathcal{D}_1$ = Directed subgraph of $\mathcal{D}$ without edges which are defined on Type-2 isomorphisms 

\hspace{1cm} = $\bigcup_{X=A,B,\dots,F}{\mathcal{D}^{X}_1}$ 
where 

\medskip
\hspace{1cm} $\mathcal{D}^{X}_1$ = ${\overleftrightarrow{K_{6}}[C_{432}(X_1), C_{432}(X_2), \dots, C_{432}(X_6)]}$, $X \in \{A, B, \dots, F\}$, and

\medskip
\hspace{.5cm} $\mathcal{D}_2$ = Directed subgraph of $\mathcal{D}$ without edges which are defined on Type-1 isomorphisms

\hspace{1cm} = $\bigcup^{6}_{i=1}{\mathcal{D}^{Y_i}_2}$ = $\bigcup^{6}_{i=1}{\big(\mathcal{D}^{Y_i}_{2,2} \cup \mathcal{D}^{Y_i}_{2,3}\big)}$, 

\medskip
\hspace{1cm} $\mathcal{D}^{Y_i}_{2,2}$ = ${\mathcal{D}[C_{432}(A_i), C_{432}(D_i)]}$ $\bigcup$ ${\mathcal{D}[C_{432}(B_i), C_{432}(E_i)]}$ $\bigcup$ ${\mathcal{D}[C_{432}(C_i), C_{432}(F_i)]}$

\hfill  is a directed subgraph of $\mathcal{D}$ without edges which are defined on  

\hfill Type-1 isomorphisms and also on Type-2 isomorphism w.r.t. $m$ = 3, and 

\medskip
\hspace{1cm} $\mathcal{D}^{Y_i}_{2,3}$ = ${\mathcal{D}[C_{432}(A_i), C_{432}(B_i), C_{432}(C_i)]}$ $\bigcup$ ${\mathcal{D}[C_{432}(D_i), C_{432}(E_i), C_{432}(F_i)]}$ 

\hfill is a directed subgraph of $\mathcal{D}$ without edges which are defined on  

\hfill Type-1 isomorphisms and also on Type-2 isomorphism w.r.t. $m$ = 2, 

\hfill and $T2_{,m,t}$ in the place of $T2_{432,m,t}$ in the directed subgraphs,

\hfill  $m$ = 2,3 and $1 \leq i \leq 6$.

\medskip 
The isomorphism graph $\mathcal{G}$ of $C_{432}(R)$ can be written as $\mathcal{G}$ = $\mathcal{G}_1$ $\cup$ $\mathcal{G}_2$, 

\medskip
\hspace{.5cm} $\mathcal{G}_1$ = Subgraph of $\mathcal{G}$ without edges which are defined on Type-2 isomorphisms 

\hspace{1cm} = $\bigcup_{X=A,B,\dots,F}{\mathcal{G}^{X}_1}$ 
where 

\medskip
\hspace{1cm} $\mathcal{G}^{X}_1$ = $K_{6}[C_{432}(X_1), C_{432}(X_2), \dots, C_{432}(X_6)]$, $X \in \{A, B, \dots, F\}$, and

\medskip
\hspace{.5cm} $\mathcal{G}_2$ = Subgraph of $\mathcal{G}$ without edges which are defined on Type-1 isomorphisms

\hspace{1cm} = $\bigcup^{6}_{i=1}{\mathcal{G}^{Y_i}_2}$ = $\bigcup^{6}_{i=1}{\big(\mathcal{G}^{Y_i}_{2,2} \cup \mathcal{G}^{Y_i}_{2,3}\big)}$, 

\medskip
\hspace{1cm} $\mathcal{G}^{Y_i}_{2,2}$ = ${\mathcal{G}[C_{432}(A_i), C_{432}(D_i)]}$ $\bigcup$ ${\mathcal{G}[C_{432}(B_i), C_{432}(E_i)]}$ $\bigcup$ ${\mathcal{G}[C_{432}(C_i), C_{432}(F_i)]}$ 

\hfill  is a subgraph of $\mathcal{G}$ without edges which are defined on  

\hfill Type-1 isomorphisms and also on Type-2 isomorphism w.r.t. $m$ = 3, and 

\medskip
\hspace{1cm} $\mathcal{G}^{Y_i}_{2,3}$ = ${\mathcal{G}[C_{432}(A_i), C_{432}(B_i), C_{432}(C_i)]}$ $\bigcup$ ${\mathcal{G}[C_{432}(D_i), C_{432}(E_i), C_{432}(F_i)]}$ 

\hfill is a subgraph of $\mathcal{G}$ without edges which are defined on Type-1  

\hfill isomorphisms and also on Type-2 isomorphism w.r.t. $m$ = 2, and $T2_{,m,t}$ 

\hfill  in the place of $T2_{432,m,t}$ in the subgraphs, $m$ = 2,3 and $1 \leq i \leq 6$.

\medskip
The following figures related to isomorphism digraph and graph $\mathcal{D}$ and $\mathcal{G}$ of $C_{432}(R)$, respectively, are presented where $R$ = $\{16, 27, 48, 54, 128, 160, 189\}$.

\medskip 
In figure 27, directed subgraph $\mathcal{D}_1$ =  $\bigcup_{X=A,B,\dots,F}{ \mathcal{D}^{X}_1}$ is presented where 
 
 \hfill $\mathcal{D}^{X}_1$ = ${\overleftrightarrow{K_{6}}[Y_1,Y_2,\dots,Y_{6}]}$, $Y_i$ = $C_{432}(X_i)$, $X \in \{A, B, \dots, F\}$, and $1 \leq i \leq 6$. 

 \medskip
 In figure 28, directed subgraph $\mathcal{D}^{Y_i}_2$ = $\mathcal{D}^{Y_i}_{2,2} \cup \mathcal{D}^{Y_i}_{2,3}$ is presented  
 
 \hspace{1cm} with $X$ in the place of $C_{432}(X)$, $X\in \{A, B, \dots, F\}$, and $1 \leq i \leq 6$.
 
\medskip
In figure 29, directed subgraph $\mathcal{D}_1$ $\bigcup$ $\big(\bigcup^6_{i=1}{\mathcal{D}^{Y_i}_{2,2}}\big)$ is presented  

\hspace{1cm} with $X$ in the place of $C_{432}(X)$, $X\in \{A, B, \dots, F\}$.

\medskip
In figure 30, directed subgraph $\mathcal{D}_1$ $\bigcup$ $\big(\bigcup^6_{i=1}{\mathcal{D}^{Y_i}_{2,3}}\big)$ is presented  

\hspace{1cm} with $X$ in the place of $C_{432}(X)$, $X\in \{A, B, \dots, F\}$.

 \medskip
In figure 31, directed subgraph $\mathcal{D}_1$ $\cup$ $\mathcal{D}^{Y_1}_{2,2} \cup \mathcal{D}^{Y_1}_{2,3}$ is presented  

\hspace{1cm} with $X$ in the place of $C_{432}(X)$, $X\in \{A, B, \dots, F\}$.

 \medskip
In figure 32, isomorphism digraph $\mathcal{D}$ = $\mathcal{D}_1$ $\cup$ $\mathcal{D}_2$ of $C_{432}(R)$ is presented  

\hspace{1cm} with $X$ in the place of $C_{432}(X)$, $X\in \{A, B, \dots, F\}$.

 \medskip
In figure 33, subgraph $\mathcal{G}_1$ $\cup$ $\mathcal{G}^{Y_1}_{2,2}$ $\cup$ $\mathcal{G}^{Y_1}_{2,3}$ of $\mathcal{G}$ is presented where 

\hfill  $\mathcal{G}^{Y_i}_{2,2}$ = ${\mathcal{G}[C_{432}(A_i), C_{432}(D_i)]}$ $\bigcup$ ${\mathcal{G}[C_{432}(B_i), C_{432}(E_i)]}$ $\bigcup$ ${\mathcal{G}[C_{432}(C_i), C_{432}(F_i)]}$, 

\hfill  $\mathcal{G}^{Y_i}_{2,3}$ = ${\mathcal{G}[C_{432}(A_i), C_{432}(B_i), C_{432}(C_i)]}$ $\bigcup$ ${\mathcal{G}[C_{432}(D_i), C_{432}(E_i), C_{432}(F_i)]}$ with

\hfill $X$ in the place of $C_{432}(X)$, $X\in \{A,B,\dots,F\}$, and $1 \leq i \leq 6$.

 \medskip
In figure 34, isomorphism graph $\mathcal{G}$ = $\mathcal{G}_1$ $\cup$ $\mathcal{G}_2$ of $C_{432}(R)$ is presented  

\hspace{1cm} with $X$ in the place of $C_{432}(X)$, $X\in \{A, B, \dots, F\}$.

\medskip
From figures 27 to 34, we observe the following related to the isomorphism graph $\mathcal{G}$ of $C_{432}(R)$ where $R$ = $\{16, 27, 48, 54, 128, 160, 189\}$.

\medskip
The isomorphism graph $\mathcal{G}$ = $\mathcal{G}_1$ $\cup$ $\mathcal{G}_2$ = $\big(\bigcup_{X=A,B,\dots,F}{\mathcal{G}^{X}_1}\big)$ 
~$\bigcup$ ~$\big(\bigcup^{6}_{i=1}{\mathcal{G}^{Y_i}_2}\big)$

\medskip
\hspace{1cm} $\cong$ $\big(\bigcup_{X=A,B,\dots,F}{K_{6}[C_{432}(X_1), C_{432}(X_2),\dots,C_{432}(X_{6})]}\big)$ 
~$\bigcup$ ~$\big(\bigcup^{6}_{i=1}{\mathcal{G}^{Y_i}_2}\big)$,

\medskip
\hfill $X \in \{A, B, \dots, F\}$, and $1 \leq i \leq 6$ 

\medskip
\hspace{1cm} = $\big(\bigcup_{X=A,B,\dots,F}{K_{6}[C_{432}(X_1), C_{432}(X_2),\dots,C_{432}(X_{6})]}\big)$ 
~$\bigcup$ ~$\big(\bigcup^{6}_{i=1}{{\mathcal{G}[A_i, B_i,\dots,F_i]}}\big)$,

\medskip
\hfill $X \in \{A, B, C, D, E, F\}$, $X$ in the place  of $C_{432}(X)$, 

\hfill ($T2_{,m,t}$ in the place of $T2_{432,m,t}$ in the figures, $m$ = 2,3,) and $1 \leq i \leq 6$.

\medskip
\hspace{1cm} $\cong$ $K_{6} \Box (C_{2} \Box C_3)$ $\cong$ $K_{6} \Box C_{6}(2, 3)$ $\cong$  $C_{6}(1,2,3) \Box C_{6}(2, 3)$ $\ncong$ $C_{6 \times 6}(S)$ for any $S$.

\medskip
The corresponding isomorphism digraph $\mathcal{D}$ of $C_{432}(R)$ is $\cong$ ${\overleftrightarrow{K_{6}}}$ $\Box$ ${\overleftrightarrow{(C_2 \Box C_3)}}$ $\cong$  ${\overleftrightarrow{K_{6}}}$ $\Box$ ${\overleftrightarrow{C_{6}(2, 3)}}$ $\cong$ ${\overleftrightarrow{C_{6}(1,2,3)}}$ $\Box$ ${\overleftrightarrow{C_{6}(2, 3)}}$ where $R$ = $\{16, 27, 48, 54, 128, 160, 189\}$ and ${\overleftrightarrow{K_{6}}}$ = ${\overleftrightarrow{C_{6}(1,2,3)}}$.

\medskip
Also, from problems 8 and 9, we get, 

$C_{432}(16, 27, 48, 54, 128, 160, 189)$ $\cong$ $C_{16}(1, 2, 7) \Box C_{27}(1, 3, 8, 10)$,

$T1_{16}(C_{16}(1, 2, 7))$ = $\{C_{16}(1, 2, 7)$, $C_{16}(3, 5, 6)\}$,

$T1_{27}(C_{27}(1,3,8,10))$ = $\{C_{27}(1,3,8,10)$, $C_{27}(2,6,7,11)$, $C_{27}(4,5,12,13)\}$,

$T2_{16,2}(C_{16}(1, 2, 7))$ = $\{C_{16}(1, 2, 7)$, $C_{16}(2, 3, 5)\}$, and 

$T2_{27,3}(C_{27}(1,3,8,10))$ = $\{C_{27}(1,3,8,10)$, $C_{27}(3,4,5,13)$, $C_{27}(2,3,7,11)\}$.
\\
$\Rightarrow$  $|T1_{16}(C_{16}(1, 2, 7))|$ = 2, $|T1_{27}(C_{27}(1,3,8,10))|$ = 3,

$|T2_{16,2}(C_{16}(1, 2, 7))|$ = 2, and
$|T2_{27,3}(C_{27}(1,3,8,10))|$ = 3.

\medskip
Based on the above discussion, we get the following result.

\begin{theorem} \label{t31} {\rm Let $R$ = $\{16, 27, 48, 54, 128, 160, 189\}$, the isomorphism digraph of $C_{432}(R)$ be $\mathcal{D}$, and the isomorphism graph of $C_{432}(R)$ be $\mathcal{G}$. Then the following statements are true.
		\begin{enumerate}
			\item [\rm (i)] $\mathcal{D}$ $\cong$ ${\overleftrightarrow{K_{6}}}$ $\Box$ ${\overleftrightarrow{(C_2 \Box C_3)}}$ $\cong$ ${\overleftrightarrow{K_{6}}}$ $\Box$ ${\overleftrightarrow{C_{6}(2, 3)}}$ $\cong$  ${\overleftrightarrow{C_{6}(1,2,3)}}$ $\Box$ ${\overleftrightarrow{C_{6}(2, 3)}}$. 
			
			\item [\rm (ii)] $\mathcal{G}$ $\cong$ $K_{6}$ $\Box$ $(C_2 \Box C_3)$ $\cong$ $K_{6}$ $\Box$ $C_{6}(2, 3)$ $\cong$ $C_{6}(1,2,3)$ $\Box$ $C_{6}(2, 3)$. \hfill $\Box$
			
	\end{enumerate}}
\end{theorem}

Now, it is much easier to check whether the directed graph $\mathcal{D}$ has Eulerian isomorphism series or not from the Cartesian product structure of $\mathcal{D}$ as given in Theorem \ref{t31}.

\begin{oprm} \quad \label{op1} {\rm 
		Find whether the directed graph $\mathcal{D}$ of $C_{432}(R)$ has Eulerian isomorphism series or not where $R$ = $\{16, 27, 48, 54, 128, 160, 189\}$. \hfill $\Box$}
\end{oprm}

{\rm 	\begin{center}
		\begin{tikzpicture}  
			[scale=.25,auto=center,every node/.style={draw,circle}] 
			
			\node (1) at (26,2) {\tiny{$A_1$}};
			\node (2) at (33,7) {\tiny{$C_1$}};	
			\node (3) at (19,7)  {\tiny{$B_1$}};
			
			\node (7) at (26,22) {\tiny{$D_1$}};
			\node (8) at (33,17) {\tiny{$F_1$}};	
			\node (9) at (19,17)  {\tiny{$E_1$}};
			
			\node (51) at (30,0) {\tiny{$A_2$}};
			\node (52) at (30,-3.5) {\tiny{$A_3$}};	
			\node (53) at (26,-6)  {\tiny{$A_4$}};
			\node (54) at (22,-3.5) {\tiny{$A_5$}};
			\node (55) at (22,0) {\tiny{$A_6$}};	
			
			\draw[->, line width=0.2mm] [blue](1) to [out=0,in=120] (51);
			\draw[->, line width=0.2mm] [blue](51) to [out=150,in=330] (1);
			\draw[->, line width=0.2mm] [blue](1) to [out=310,in=120] (52);
			\draw[->, line width=0.2mm] [blue](52) to [out=140,in=295] (1);
			\draw[->, line width=0.2mm] [blue](1) to [out=280,in=80] (53);
			\draw[->, line width=0.2mm] [blue](53) to [out=100,in=260] (1);
			\draw[->, line width=0.2mm] [blue](1) to [out=245,in=45] (54);
			\draw[->, line width=0.2mm] [blue](54) to [out=65,in=225] (1);
			\draw[->, line width=0.2mm] [blue](1) to [out=205,in=35] (55);
			\draw[->, line width=0.2mm] [blue](55) to [out=60,in=185] (1);
			
			\draw[->, line width=0.2mm] [blue](51) to [out=290,in=70] (52);
			\draw[->, line width=0.2mm] [blue](52) to [out=90,in=270] (51);
			\draw[->, line width=0.2mm] [blue](51) to [out=245,in=45] (53);
			\draw[->, line width=0.2mm] [blue](53) to [out=60,in=225] (51);
			\draw[->, line width=0.2mm] [blue](51) to [out=210,in=20] (54);
			\draw[->, line width=0.2mm] [blue](54) to [out=30,in=200] (51);
			\draw[->, line width=0.2mm] [blue](51) to [out=165,in=10] (55);
			\draw[->, line width=0.2mm] [blue](55) to [out=350,in=185] (51);
			
			\draw[->, line width=0.2mm] [blue](52) to [out=270,in=340] (53);
			\draw[->, line width=0.2mm] [blue](53) to [out=5,in=240] (52);
			\draw[->, line width=0.2mm] [blue](52) to [out=180,in=0] (54);
			\draw[->, line width=0.2mm] [blue](54) to [out=340,in=200] (52);
			\draw[->, line width=0.2mm] [blue](52) to [out=150,in=335] (55);
			\draw[->, line width=0.2mm] [blue](55) to [out=320,in=165] (52);
			
			\draw[->, line width=0.2mm] [blue](53) to [out=170,in=310] (54);
			\draw[->, line width=0.2mm] [blue](54) to [out=270,in=210] (53);
			\draw[->, line width=0.2mm] [blue](53) to [out=120,in=310] (55);
			\draw[->, line width=0.2mm] [blue](55) to [out=290,in=140] (53);
			
			\draw[->, line width=0.2mm] [blue](54) to [out=90,in=270] (55);
			\draw[->, line width=0.2mm] [blue](55) to [out=250,in=110] (54);
			
			\node (56) at (22,24.5) {\tiny{$D_2$}};
			\node (57) at (22,28) {\tiny{$D_3$}};	
			\node (58) at (26,30.5)  {\tiny{$D_4$}};
			\node (59) at (30,28) {\tiny{$D_5$}};
			\node (60) at (30,24.5) {\tiny{$D_6$}};	
			
			\draw[->, line width=0.2mm] [blue](58) to [out=0,in=120] (59);
			\draw[->, line width=0.2mm] [blue](59) to [out=150,in=330] (58);
			\draw[->, line width=0.2mm] [blue](58) to [out=310,in=120] (60);
			\draw[->, line width=0.2mm] [blue](60) to [out=140,in=295] (58);
			\draw[->, line width=0.2mm] [blue](58) to [out=280,in=80] (7);
			\draw[->, line width=0.2mm] [blue](7) to [out=100,in=260] (58);
			\draw[->, line width=0.2mm] [blue](58) to [out=245,in=45] (56);
			\draw[->, line width=0.2mm] [blue](56) to [out=65,in=225] (58);
			\draw[->, line width=0.2mm] [blue](58) to [out=205,in=35] (57);
			\draw[->, line width=0.2mm] [blue](57) to [out=60,in=185] (58);
			
			\draw[->, line width=0.2mm] [blue](59) to [out=290,in=70] (60);
			\draw[->, line width=0.2mm] [blue](60) to [out=90,in=270] (59);
			\draw[->, line width=0.2mm] [blue](59) to [out=245,in=45] (7);
			\draw[->, line width=0.2mm] [blue](7) to [out=60,in=225] (59);
			\draw[->, line width=0.2mm] [blue](59) to [out=210,in=20] (56);
			\draw[->, line width=0.2mm] [blue](56) to [out=30,in=200] (59);
			\draw[->, line width=0.2mm] [blue](59) to [out=165,in=10] (57);
			\draw[->, line width=0.2mm] [blue](57) to [out=350,in=185] (59);
			
			\draw[->, line width=0.2mm] [blue](60) to [out=250,in=10] (7);
			\draw[->, line width=0.2mm] [blue](7) to [out=25,in=230] (60);
			\draw[->, line width=0.2mm] [blue](60) to [out=180,in=0] (56);
			\draw[->, line width=0.2mm] [blue](56) to [out=340,in=200] (60);
			\draw[->, line width=0.2mm] [blue](60) to [out=150,in=335] (57);
			\draw[->, line width=0.2mm] [blue](57) to [out=320,in=165] (60);
			
			\draw[->, line width=0.2mm] [blue](7) to [out=155,in=310] (56);
			\draw[->, line width=0.2mm] [blue](56) to [out=270,in=170] (7);
			\draw[->, line width=0.2mm] [blue](7) to [out=120,in=310] (57);
			\draw[->, line width=0.2mm] [blue](57) to [out=290,in=140] (7);
			
			\draw[->, line width=0.2mm] [blue](56) to [out=90,in=270] (57);
			\draw[->, line width=0.2mm] [blue](57) to [out=250,in=110] (56);
			
			\node (61) at (18,2.5) {\tiny{$B_2$}};
			\node (62) at (13.5,0) {\tiny{$B_3$}};	
			\node (63) at (9.5,2.5)  {\tiny{$B_4$}};
			\node (64) at (10.5,7) {\tiny{$B_5$}};
			\node (65) at (15,9) {\tiny{$B_6$}};	
			
			\draw[->, line width=0.2mm] [blue](65) to [out=0,in=135] (3);
			\draw[->, line width=0.2mm] [blue](3) to [out=150,in=330] (65);
			\draw[->, line width=0.2mm] [blue](65) to [out=310,in=100] (61);
			\draw[->, line width=0.2mm] [blue](61) to [out=120,in=295] (65);
			\draw[->, line width=0.2mm] [blue](65) to [out=280,in=80] (62);
			\draw[->, line width=0.2mm] [blue](62) to [out=100,in=260] (65);
			\draw[->, line width=0.2mm] [blue](65) to [out=245,in=45] (63);
			\draw[->, line width=0.2mm] [blue](63) to [out=65,in=225] (65);
			\draw[->, line width=0.2mm] [blue](65) to [out=200,in=35] (64);
			\draw[->, line width=0.2mm] [blue](64) to [out=55,in=180] (65);
			
			\draw[->, line width=0.2mm] [blue](3) to [out=285,in=60] (61);
			\draw[->, line width=0.2mm] [blue](61) to [out=80,in=265] (3);
			\draw[->, line width=0.2mm] [blue](3) to [out=245,in=45] (62);
			\draw[->, line width=0.2mm] [blue](62) to [out=60,in=225] (3);
			\draw[->, line width=0.2mm] [blue](3) to [out=210,in=15] (63);
			\draw[->, line width=0.2mm] [blue](63) to [out=30,in=200] (3);
			\draw[->, line width=0.2mm] [blue](3) to [out=165,in=10] (64);
			\draw[->, line width=0.2mm] [blue](64) to [out=350,in=185] (3);
			
			\draw[->, line width=0.2mm] [blue](61) to [out=270,in=340] (62);
			\draw[->, line width=0.2mm] [blue](62) to [out=5,in=240] (61);
			\draw[->, line width=0.2mm] [blue](61) to [out=185,in=350] (63);
			\draw[->, line width=0.2mm] [blue](63) to [out=330,in=200] (61);
			\draw[->, line width=0.2mm] [blue](61) to [out=140,in=335] (64);
			\draw[->, line width=0.2mm] [blue](64) to [out=320,in=155] (61);
			
			\draw[->, line width=0.2mm] [blue](62) to [out=170,in=310] (63);
			\draw[->, line width=0.2mm] [blue](63) to [out=270,in=210] (62);
			\draw[->, line width=0.2mm] [blue](62) to [out=120,in=310] (64);
			\draw[->, line width=0.2mm] [blue](64) to [out=290,in=140] (62);
			
			\draw[->, line width=0.2mm] [blue](63) to [out=80,in=260] (64);
			\draw[->, line width=0.2mm] [blue](64) to [out=240,in=100] (63);
			
			\node (66) at (34,21.5) {\tiny{$F_2$}};
			\node (67) at (38.5,23.5) {\tiny{$F_3$}};	
			\node (68) at (42.5,21.5)  {\tiny{$F_4$}};
			\node (69) at (42,17) {\tiny{$F_5$}};
			\node (70) at (37,14.5) {\tiny{$F_6$}};	
			
			\draw[->, line width=0.2mm] [blue](67) to [out=10,in=110] (68);
			\draw[->, line width=0.2mm] [blue](68) to [out=135,in=340] (67);
			\draw[->, line width=0.2mm] [blue](67) to [out=310,in=100] (69);
			\draw[->, line width=0.2mm] [blue](69) to [out=120,in=295] (67);
			\draw[->, line width=0.2mm] [blue](67) to [out=280,in=80] (70);
			\draw[->, line width=0.2mm] [blue](70) to [out=100,in=260] (67);
			\draw[->, line width=0.2mm] [blue](67) to [out=245,in=45] (8);
			\draw[->, line width=0.2mm] [blue](8) to [out=65,in=225] (67);
			\draw[->, line width=0.2mm] [blue](67) to [out=180,in=45] (66);
			\draw[->, line width=0.2mm] [blue](66) to [out=65,in=160] (67);
			
			\draw[->, line width=0.2mm] [blue](68) to [out=295,in=50] (69);
			\draw[->, line width=0.2mm] [blue](69) to [out=70,in=275] (68);
			\draw[->, line width=0.2mm] [blue](68) to [out=255,in=40] (70);
			\draw[->, line width=0.2mm] [blue](70) to [out=60,in=235] (68);
			\draw[->, line width=0.2mm] [blue](68) to [out=215,in=15] (8);
			\draw[->, line width=0.2mm] [blue](8) to [out=30,in=200] (68);
			\draw[->, line width=0.2mm] [blue](68) to [out=165,in=10] (66);
			\draw[->, line width=0.2mm] [blue](66) to [out=350,in=185] (68);
			
			\draw[->, line width=0.2mm] [blue](69) to [out=270,in=340] (70);
			\draw[->, line width=0.2mm] [blue](70) to [out=5,in=240] (69);
			\draw[->, line width=0.2mm] [blue](69) to [out=180,in=355] (8);
			\draw[->, line width=0.2mm] [blue](8) to [out=340,in=200] (69);
			\draw[->, line width=0.2mm] [blue](69) to [out=140,in=335] (66);
			\draw[->, line width=0.2mm] [blue](66) to [out=320,in=155] (69);
			
			\draw[->, line width=0.2mm] [blue](70) to [out=160,in=320] (8);
			\draw[->, line width=0.2mm] [blue](8) to [out=280,in=200] (70);
			\draw[->, line width=0.2mm] [blue](70) to [out=120,in=310] (66);
			\draw[->, line width=0.2mm] [blue](66) to [out=290,in=140] (70);
			
			\draw[->, line width=0.2mm] [blue](8) to [out=90,in=250] (66);
			\draw[->, line width=0.2mm] [blue](66) to [out=230,in=110] (8);
			
			\node (71) at (15,14.5) {\tiny{$E_2$}};
			\node (72) at (10.5,17) {\tiny{$E_3$}};	
			\node (73) at (9.5,21.5)  {\tiny{$E_4$}};
			\node (74) at (13.5,23.5) {\tiny{$E_5$}};
			\node (75) at (18,21.5) {\tiny{$E_6$}};	
			
			\draw[->, line width=0.2mm] [blue](74) to [out=0,in=120] (75);
			\draw[->, line width=0.2mm] [blue](75) to [out=150,in=330] (74);
			\draw[->, line width=0.2mm] [blue](74) to [out=310,in=120] (9);
			\draw[->, line width=0.2mm] [blue](9) to [out=140,in=295] (74);
			\draw[->, line width=0.2mm] [blue](74) to [out=280,in=80] (71);
			\draw[->, line width=0.2mm] [blue](71) to [out=100,in=260] (74);
			\draw[->, line width=0.2mm] [blue](74) to [out=245,in=60] (72);
			\draw[->, line width=0.2mm] [blue](72) to [out=80,in=225] (74);
			\draw[->, line width=0.2mm] [blue](74) to [out=205,in=35] (73);
			\draw[->, line width=0.2mm] [blue](73) to [out=60,in=185] (74);
			
			\draw[->, line width=0.2mm] [blue](75) to [out=305,in=80] (9);
			\draw[->, line width=0.2mm] [blue](9) to [out=100,in=280] (75);
			\draw[->, line width=0.2mm] [blue](75) to [out=255,in=45] (71);
			\draw[->, line width=0.2mm] [blue](71) to [out=60,in=235] (75);
			\draw[->, line width=0.2mm] [blue](75) to [out=220,in=20] (72);
			\draw[->, line width=0.2mm] [blue](72) to [out=35,in=205] (75);
			\draw[->, line width=0.2mm] [blue](75) to [out=165,in=10] (73);
			\draw[->, line width=0.2mm] [blue](73) to [out=350,in=185] (75);
			
			\draw[->, line width=0.2mm] [blue](9) to [out=260,in=0] (71);
			\draw[->, line width=0.2mm] [blue](71) to [out=20,in=220] (9);
			\draw[->, line width=0.2mm] [blue](9) to [out=180,in=0] (72);
			\draw[->, line width=0.2mm] [blue](72) to [out=340,in=200] (9);
			\draw[->, line width=0.2mm] [blue](9) to [out=150,in=335] (73);
			\draw[->, line width=0.2mm] [blue](73) to [out=320,in=165] (9);
			
			\draw[->, line width=0.2mm] [blue](71) to [out=160,in=310] (72);
			\draw[->, line width=0.2mm] [blue](72) to [out=280,in=190] (71);
			\draw[->, line width=0.2mm] [blue](71) to [out=120,in=310] (73);
			\draw[->, line width=0.2mm] [blue](73) to [out=290,in=140] (71);
			
			\draw[->, line width=0.2mm] [blue](72) to [out=110,in=270] (73);
			\draw[->, line width=0.2mm] [blue](73) to [out=250,in=130] (72);
			
			\node (76) at (37,9.5) {\tiny{$C_2$}};
			\node (77) at (41.5,7) {\tiny{$C_3$}};	
			\node (78) at (42.5,2.5)  {\tiny{$C_4$}};
			\node (79) at (38.5,-0.5) {\tiny{$C_5$}};
			\node (80) at (33.5,2.5) {\tiny{$C_6$}};	
			
			\draw[->, line width=0.2mm] [blue](76) to [out=0,in=120] (77);
			\draw[->, line width=0.2mm] [blue](77) to [out=150,in=330] (76);
			\draw[->, line width=0.2mm] [blue](76) to [out=310,in=120] (78);
			\draw[->, line width=0.2mm] [blue](78) to [out=140,in=295] (76);
			\draw[->, line width=0.2mm] [blue](76) to [out=280,in=80] (79);
			\draw[->, line width=0.2mm] [blue](79) to [out=100,in=260] (76);
			\draw[->, line width=0.2mm] [blue](76) to [out=245,in=60] (80);
			\draw[->, line width=0.2mm] [blue](80) to [out=85,in=225] (76);
			\draw[->, line width=0.2mm] [blue](76) to [out=205,in=35] (2);
			\draw[->, line width=0.2mm] [blue](2) to [out=90,in=170] (76);
			
			\draw[->, line width=0.2mm] [blue](77) to [out=305,in=80] (78);
			\draw[->, line width=0.2mm] [blue](78) to [out=100,in=280] (77);
			\draw[->, line width=0.2mm] [blue](77) to [out=255,in=45] (79);
			\draw[->, line width=0.2mm] [blue](79) to [out=60,in=235] (77);
			\draw[->, line width=0.2mm] [blue](77) to [out=220,in=15] (80);
			\draw[->, line width=0.2mm] [blue](80) to [out=35,in=200] (77);
			\draw[->, line width=0.2mm] [blue](77) to [out=165,in=10] (2);
			\draw[->, line width=0.2mm] [blue](2) to [out=350,in=185] (77);
			
			\draw[->, line width=0.2mm] [blue](78) to [out=260,in=350] (79);
			\draw[->, line width=0.2mm] [blue](79) to [out=20,in=235] (78);
			\draw[->, line width=0.2mm] [blue](78) to [out=185,in=350] (80);
			\draw[->, line width=0.2mm] [blue](80) to [out=330,in=210] (78);
			\draw[->, line width=0.2mm] [blue](78) to [out=150,in=335] (2);
			\draw[->, line width=0.2mm] [blue](2) to [out=320,in=165] (78);
			
			\draw[->, line width=0.2mm] [blue](79) to [out=160,in=300] (80);
			\draw[->, line width=0.2mm] [blue](80) to [out=265,in=190] (79);
			\draw[->, line width=0.2mm] [blue](79) to [out=120,in=310] (2);
			\draw[->, line width=0.2mm] [blue](2) to [out=290,in=140] (79);
			
			\draw[->, line width=0.2mm] [blue](80) to [out=110,in=270] (2);
			\draw[->, line width=0.2mm] [blue](2) to [out=250,in=130] (80);
			
		\end{tikzpicture}		

	\vspace{.1cm}	
	{\small	Fig.  27. Sub-digraph $\mathcal{D}_1$ =  $\bigcup_{X=A,B,\dots,F}{ \mathcal{D}^{X}_1}$ of $\mathcal{D}$ = $\mathcal{D}_1$ $\cup$ $\mathcal{D}_2$ where    
		
		\hspace{2cm} $\mathcal{D}^{X}_1$ = ${\overleftrightarrow{K_{6}}[Y_1,Y_2,\dots,Y_{6}]}$,  $Y_i$ = $C_{432}(X_i)$, $X \in \{A, B, \dots, F\}$, and $1 \leq i \leq 6$.  }
		
\end{center} }

{\rm 	\begin{center}
		\begin{tikzpicture}  
			[scale=.37,auto=center,every node/.style={draw,circle}] 
			
			\node (1) at (26,2) {\tiny{$A_i$}};
			\node (2) at (33,7) {\tiny{$C_i$}};	
			\node (3) at (19,7)  {\tiny{$B_i$}};
			
			\node (7) at (26,22) {\tiny{$D_i$}};
			\node (8) at (33,17) {\tiny{$F_i$}};	
			\node (9) at (19,17)  {\tiny{$E_i$}};
			
			\node [scale=.0] (11) at (30,1) [label=90:\tiny{$T2_{432,3,32}$}]{};
			\node [scale=.0] (12) at (30,6) [label=90:\tiny{$T2_{432,3,32}$}]{};
			\node [scale=.0] (13) at (19.5,3)[label=90:\tiny{$T2_{432,3,32}$}]{};
			\node [scale=.0] (14) at (23.5,1.75) [label=90:\tiny{$T2_{432,3,64}$}]{};
			\node [scale=.0] (15) at (22.75,4.5) [label=90:\tiny{$T2_{432,3,64}$}]{};
			\node [scale=.0] (16) at (30,3.25)[label=90:\tiny{$T2_{432,3,64}$}]{};
			
			\node [scale=.0] (17) at (30,19.5)[label=90:\tiny{$T2_{432,3,32}$}]{};
			\node [scale=.0] (18) at (30,14)[label=90:\tiny{$T2_{432,3,32}$}]{};
			\node [scale=.0](19) at (20.25,17)[label=90:\tiny{$T2_{432,3,32}$}]{};
			\node [scale=.0] (20) at (29.75,16.5)[label=90:\tiny{$T2_{432,3,64}$}]{};
			\node [scale=.0] (21) at (24,18)[label=90:\tiny{$T2_{432,3,64}$}]{};
			\node [scale=.0](22) at (22.5,15.5)[label=90:\tiny{$T2_{432,3,64}$}]{};
			
			\draw[->, line width=0.3mm] [blue](1)[dashed] to [out=50,in=210] (2);
			\draw[->, line width=0.3mm] [blue](2)[dashed] to [out=235,in=20] (1);
			\draw[<-, line width=0.3mm] [blue](2)[dashed] to [out=165,in=15] (3);
			\draw[<-, line width=0.3mm] [blue](3)[dashed] to [out=345,in=195](2);
			\draw[->, line width=0.3mm] [blue](3)[dashed] to [out=330,in=140](1);
			\draw[->, line width=0.3mm] [blue](1)[dashed] to [out=165,in=300](3);
			
			\draw[<-, line width=0.3mm] [blue](7)[dashed] to [out=350,in=130] (8);
			\draw[<-, line width=0.3mm] [blue](8)[dashed] to [out=150,in=310](7);
			\draw[->, line width=0.3mm] [blue](8)[dashed] to [out=175,in=5] (9);
			\draw[->, line width=0.3mm] [blue](9)[dashed] to [out=345,in=195](8);
			\draw[<-, line width=0.3mm] [blue](9)[dashed] to [out=60,in=190](7);
			\draw[<-, line width=0.3mm] [blue](7)[dashed] to [out=220,in=30](9);
			
			\draw[->, line width=0.3mm] [red](1)[dashed] to [out=110,in=250] (7);
			\draw[->, line width=0.3mm] [red](7)[dashed] to [out=290,in=70] (1);
			\draw[->, line width=0.3mm] [red](2)[dashed] to [out=120,in=240] (8);
			\draw[->, line width=0.3mm] [red](8)[dashed] to [out=300,in=60] (2);
			\draw[->, line width=0.3mm] [red](3)[dashed] to [out=120,in=240] (9);
			\draw[->, line width=0.3mm] [red](9)[dashed] to [out=300,in=60] (3);
			
			\node [scale=.0] (31) at (21,7) [label=90:\tiny{$T2_{432,2,54}$}]{};
			\node [scale=.0] (32) at (35,8.5)[label=90:\tiny{$T2_{432,2,54}$}]{};
			\node [scale=.0] (33) at (28,7)[label=90:\tiny{$T2_{432,2,54}$}]{};
			\node [scale=.0] (34) at (18,12) [label=90:\tiny{$T2_{432,2,54}$}]{};
			\node [scale=.0] (35) at (32,12)[label=90:\tiny{$T2_{432,2,54}$}]{};
			\node [scale=.0] (36) at (25,13)[label=90:\tiny{$T2_{432,2,54}$}]{};
			
		\end{tikzpicture}		
		
		\vspace{.1cm}		
	{\small		Fig.  28. Sub-digraph $\mathcal{D}^{Y_i}_2$ = $\mathcal{D}^{Y_i}_{2,2} \cup \mathcal{D}^{Y_i}_{2,3}$ with $X$ in the place of $C_{432}(X)$,  
		
		\hspace{1cm}  $X\in \{A, B, \dots, F\}\}$, and $1 \leq i \leq 6$. }
						
\end{center} }

\begin{figure}[ht]
	\centerline{\includegraphics[width=5in]{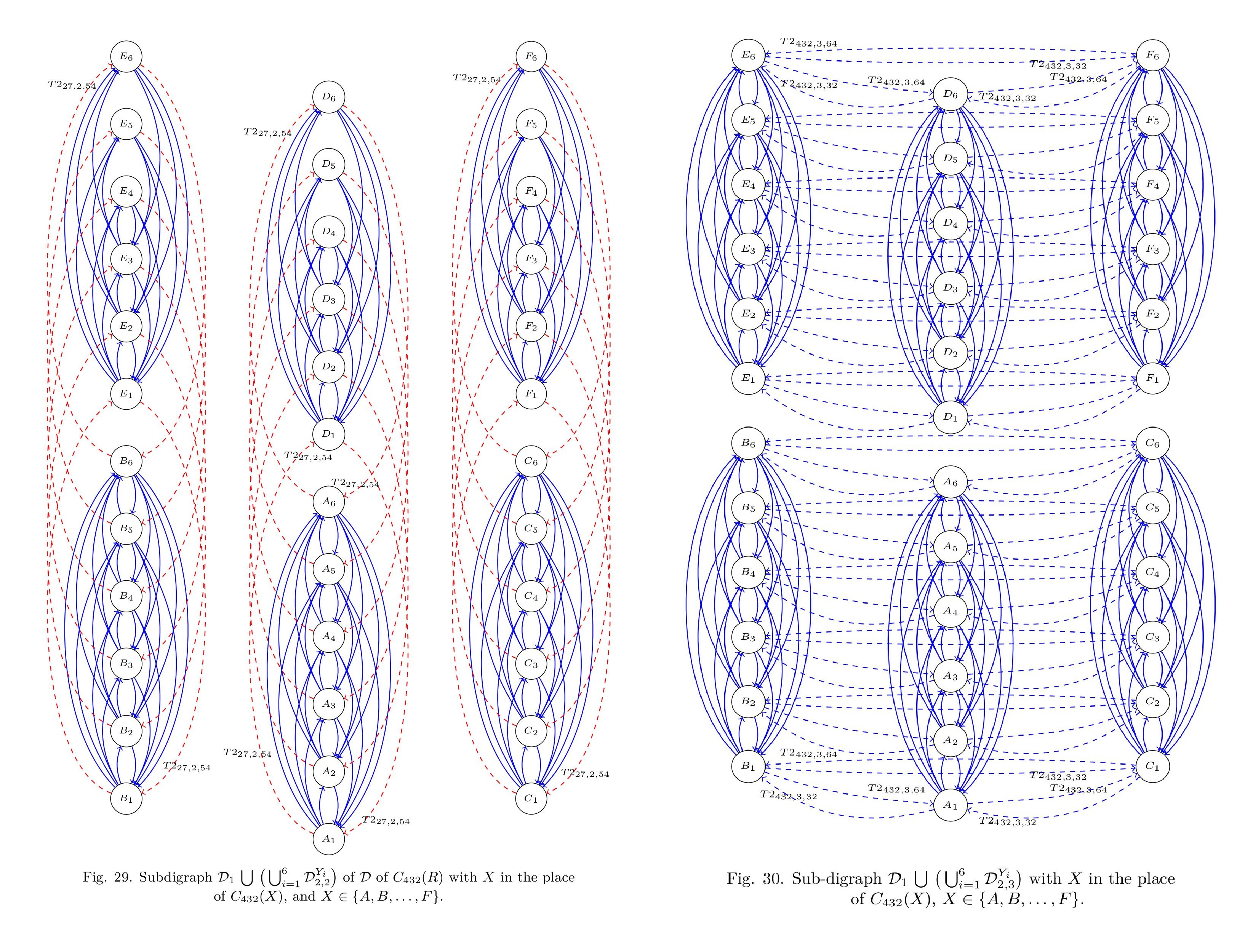}}
\end{figure}

{\rm 	\begin{center}
		\begin{tikzpicture}  
			[scale=.2,auto=center,every node/.style={draw,circle}] 
			
			\node (1) at (26,2) {\tiny{$A_1$}};
			\node (2) at (33,7) {\tiny{$C_1$}};	
			\node (3) at (19,7)  {\tiny{$B_1$}};
			
			\node (7) at (26,22) {\tiny{$D_1$}};
			\node (8) at (33,17) {\tiny{$F_1$}};	
			\node (9) at (19,17)  {\tiny{$E_1$}};
			
			\draw[->, line width=0.3mm] [blue](1)[dashed] to [out=50,in=210] (2);
			\draw[->, line width=0.3mm] [blue](2)[dashed] to [out=235,in=20] (1);
			\draw[<-, line width=0.3mm] [blue](2)[dashed] to [out=165,in=15] (3);
			\draw[<-, line width=0.3mm] [blue](3)[dashed] to [out=345,in=195](2);
			\draw[->, line width=0.3mm] [blue](3)[dashed] to [out=330,in=140](1);
			\draw[->, line width=0.3mm] [blue](1)[dashed] to [out=165,in=300](3);
			
			\draw[<-, line width=0.3mm] [blue](7)[dashed] to [out=350,in=130] (8);
			\draw[<-, line width=0.3mm] [blue](8)[dashed] to [out=150,in=310](7);
			\draw[->, line width=0.3mm] [blue](8)[dashed] to [out=175,in=5] (9);
			\draw[->, line width=0.3mm] [blue](9)[dashed] to [out=345,in=195](8);
			\draw[<-, line width=0.3mm] [blue](9)[dashed] to [out=60,in=190](7);
			\draw[<-, line width=0.3mm] [blue](7)[dashed] to [out=220,in=30](9);
			
			\draw[->, line width=0.3mm] [red](1)[dashed] to [out=110,in=250] (7);
			\draw[->, line width=0.3mm] [red](7)[dashed] to [out=290,in=70] (1);
			\draw[->, line width=0.3mm] [red](2)[dashed] to [out=120,in=240] (8);
			\draw[->, line width=0.3mm] [red](8)[dashed] to [out=300,in=60] (2);
			\draw[->, line width=0.3mm] [red](3)[dashed] to [out=120,in=240] (9);
			\draw[->, line width=0.3mm] [red](9)[dashed] to [out=300,in=60] (3);
			
			\node (51) at (30,0) {\tiny{$A_2$}};
			\node (52) at (30,-3.5) {\tiny{$A_3$}};	
			\node (53) at (26,-6)  {\tiny{$A_4$}};
			\node (54) at (22,-3.5) {\tiny{$A_5$}};
			\node (55) at (22,0) {\tiny{$A_6$}};	
			
			\draw[->, line width=0.2mm] [blue](1) to [out=0,in=120] (51);
			\draw[->, line width=0.2mm] [blue](51) to [out=150,in=330] (1);
			\draw[->, line width=0.2mm] [blue](1) to [out=310,in=120] (52);
			\draw[->, line width=0.2mm] [blue](52) to [out=140,in=295] (1);
			\draw[->, line width=0.2mm] [blue](1) to [out=280,in=80] (53);
			\draw[->, line width=0.2mm] [blue](53) to [out=100,in=260] (1);
			\draw[->, line width=0.2mm] [blue](1) to [out=245,in=45] (54);
			\draw[->, line width=0.2mm] [blue](54) to [out=65,in=225] (1);
			\draw[->, line width=0.2mm] [blue](1) to [out=205,in=35] (55);
			\draw[->, line width=0.2mm] [blue](55) to [out=60,in=185] (1);
			
			\draw[->, line width=0.2mm] [blue](51) to [out=290,in=70] (52);
			\draw[->, line width=0.2mm] [blue](52) to [out=90,in=270] (51);
			\draw[->, line width=0.2mm] [blue](51) to [out=245,in=45] (53);
			\draw[->, line width=0.2mm] [blue](53) to [out=60,in=225] (51);
			\draw[->, line width=0.2mm] [blue](51) to [out=210,in=20] (54);
			\draw[->, line width=0.2mm] [blue](54) to [out=30,in=200] (51);
			\draw[->, line width=0.2mm] [blue](51) to [out=165,in=10] (55);
			\draw[->, line width=0.2mm] [blue](55) to [out=350,in=185] (51);
			
			\draw[->, line width=0.2mm] [blue](52) to [out=270,in=340] (53);
			\draw[->, line width=0.2mm] [blue](53) to [out=5,in=240] (52);
			\draw[->, line width=0.2mm] [blue](52) to [out=180,in=0] (54);
			\draw[->, line width=0.2mm] [blue](54) to [out=340,in=200] (52);
			\draw[->, line width=0.2mm] [blue](52) to [out=150,in=335] (55);
			\draw[->, line width=0.2mm] [blue](55) to [out=320,in=165] (52);
			
			\draw[->, line width=0.2mm] [blue](53) to [out=170,in=310] (54);
			\draw[->, line width=0.2mm] [blue](54) to [out=270,in=210] (53);
			\draw[->, line width=0.2mm] [blue](53) to [out=120,in=310] (55);
			\draw[->, line width=0.2mm] [blue](55) to [out=290,in=140] (53);
			
			\draw[->, line width=0.2mm] [blue](54) to [out=90,in=270] (55);
			\draw[->, line width=0.2mm] [blue](55) to [out=250,in=110] (54);
			
			\node (56) at (22,24.5) {\tiny{$D_2$}};
			\node (57) at (22,28) {\tiny{$D_3$}};	
			\node (58) at (26,30.5)  {\tiny{$D_4$}};
			\node (59) at (30,28) {\tiny{$D_5$}};
			\node (60) at (30,24.5) {\tiny{$D_6$}};	
			
			\draw[->, line width=0.2mm] [blue](58) to [out=0,in=120] (59);
			\draw[->, line width=0.2mm] [blue](59) to [out=150,in=330] (58);
			\draw[->, line width=0.2mm] [blue](58) to [out=310,in=120] (60);
			\draw[->, line width=0.2mm] [blue](60) to [out=140,in=295] (58);
			\draw[->, line width=0.2mm] [blue](58) to [out=280,in=80] (7);
			\draw[->, line width=0.2mm] [blue](7) to [out=100,in=260] (58);
			\draw[->, line width=0.2mm] [blue](58) to [out=245,in=45] (56);
			\draw[->, line width=0.2mm] [blue](56) to [out=65,in=225] (58);
			\draw[->, line width=0.2mm] [blue](58) to [out=205,in=35] (57);
			\draw[->, line width=0.2mm] [blue](57) to [out=60,in=185] (58);
			
			\draw[->, line width=0.2mm] [blue](59) to [out=290,in=70] (60);
			\draw[->, line width=0.2mm] [blue](60) to [out=90,in=270] (59);
			\draw[->, line width=0.2mm] [blue](59) to [out=245,in=45] (7);
			\draw[->, line width=0.2mm] [blue](7) to [out=60,in=225] (59);
			\draw[->, line width=0.2mm] [blue](59) to [out=210,in=20] (56);
			\draw[->, line width=0.2mm] [blue](56) to [out=30,in=200] (59);
			\draw[->, line width=0.2mm] [blue](59) to [out=165,in=10] (57);
			\draw[->, line width=0.2mm] [blue](57) to [out=350,in=185] (59);
			
			\draw[->, line width=0.2mm] [blue](60) to [out=250,in=10] (7);
			\draw[->, line width=0.2mm] [blue](7) to [out=25,in=230] (60);
			\draw[->, line width=0.2mm] [blue](60) to [out=180,in=0] (56);
			\draw[->, line width=0.2mm] [blue](56) to [out=340,in=200] (60);
			\draw[->, line width=0.2mm] [blue](60) to [out=150,in=335] (57);
			\draw[->, line width=0.2mm] [blue](57) to [out=320,in=165] (60);
			
			\draw[->, line width=0.2mm] [blue](7) to [out=155,in=310] (56);
			\draw[->, line width=0.2mm] [blue](56) to [out=270,in=170] (7);
			\draw[->, line width=0.2mm] [blue](7) to [out=120,in=310] (57);
			\draw[->, line width=0.2mm] [blue](57) to [out=290,in=140] (7);
			
			\draw[->, line width=0.2mm] [blue](56) to [out=90,in=270] (57);
			\draw[->, line width=0.2mm] [blue](57) to [out=250,in=110] (56);
			
			\node (61) at (18,2.5) {\tiny{$B_2$}};
			\node (62) at (13.5,0) {\tiny{$B_3$}};	
			\node (63) at (9.5,2.5)  {\tiny{$B_4$}};
			\node (64) at (10.5,7) {\tiny{$B_5$}};
			\node (65) at (15,9) {\tiny{$B_6$}};	
			
			\draw[->, line width=0.2mm] [blue](65) to [out=0,in=135] (3);
			\draw[->, line width=0.2mm] [blue](3) to [out=150,in=330] (65);
			\draw[->, line width=0.2mm] [blue](65) to [out=310,in=100] (61);
			\draw[->, line width=0.2mm] [blue](61) to [out=120,in=295] (65);
			\draw[->, line width=0.2mm] [blue](65) to [out=280,in=80] (62);
			\draw[->, line width=0.2mm] [blue](62) to [out=100,in=260] (65);
			\draw[->, line width=0.2mm] [blue](65) to [out=245,in=45] (63);
			\draw[->, line width=0.2mm] [blue](63) to [out=65,in=225] (65);
			\draw[->, line width=0.2mm] [blue](65) to [out=200,in=35] (64);
			\draw[->, line width=0.2mm] [blue](64) to [out=55,in=180] (65);
			
			\draw[->, line width=0.2mm] [blue](3) to [out=285,in=60] (61);
			\draw[->, line width=0.2mm] [blue](61) to [out=80,in=265] (3);
			\draw[->, line width=0.2mm] [blue](3) to [out=245,in=45] (62);
			\draw[->, line width=0.2mm] [blue](62) to [out=60,in=225] (3);
			\draw[->, line width=0.2mm] [blue](3) to [out=210,in=15] (63);
			\draw[->, line width=0.2mm] [blue](63) to [out=30,in=200] (3);
			\draw[->, line width=0.2mm] [blue](3) to [out=165,in=10] (64);
			\draw[->, line width=0.2mm] [blue](64) to [out=350,in=185] (3);
			
			\draw[->, line width=0.2mm] [blue](61) to [out=270,in=340] (62);
			\draw[->, line width=0.2mm] [blue](62) to [out=5,in=240] (61);
			\draw[->, line width=0.2mm] [blue](61) to [out=185,in=350] (63);
			\draw[->, line width=0.2mm] [blue](63) to [out=330,in=200] (61);
			\draw[->, line width=0.2mm] [blue](61) to [out=140,in=335] (64);
			\draw[->, line width=0.2mm] [blue](64) to [out=320,in=155] (61);
			
			\draw[->, line width=0.2mm] [blue](62) to [out=170,in=310] (63);
			\draw[->, line width=0.2mm] [blue](63) to [out=270,in=210] (62);
			\draw[->, line width=0.2mm] [blue](62) to [out=120,in=310] (64);
			\draw[->, line width=0.2mm] [blue](64) to [out=290,in=140] (62);
			
			\draw[->, line width=0.2mm] [blue](63) to [out=80,in=260] (64);
			\draw[->, line width=0.2mm] [blue](64) to [out=240,in=100] (63);
			
			\node (66) at (34,21.5) {\tiny{$F_2$}};
			\node (67) at (38.5,23.5) {\tiny{$F_3$}};	
			\node (68) at (42.5,21.5)  {\tiny{$F_4$}};
			\node (69) at (42,17) {\tiny{$F_5$}};
			\node (70) at (37,14.5) {\tiny{$F_6$}};	
			
			\draw[->, line width=0.2mm] [blue](67) to [out=10,in=110] (68);
			\draw[->, line width=0.2mm] [blue](68) to [out=135,in=340] (67);
			\draw[->, line width=0.2mm] [blue](67) to [out=310,in=100] (69);
			\draw[->, line width=0.2mm] [blue](69) to [out=120,in=295] (67);
			\draw[->, line width=0.2mm] [blue](67) to [out=280,in=80] (70);
			\draw[->, line width=0.2mm] [blue](70) to [out=100,in=260] (67);
			\draw[->, line width=0.2mm] [blue](67) to [out=245,in=45] (8);
			\draw[->, line width=0.2mm] [blue](8) to [out=65,in=225] (67);
			\draw[->, line width=0.2mm] [blue](67) to [out=180,in=45] (66);
			\draw[->, line width=0.2mm] [blue](66) to [out=65,in=160] (67);
			
			\draw[->, line width=0.2mm] [blue](68) to [out=295,in=50] (69);
			\draw[->, line width=0.2mm] [blue](69) to [out=70,in=275] (68);
			\draw[->, line width=0.2mm] [blue](68) to [out=255,in=40] (70);
			\draw[->, line width=0.2mm] [blue](70) to [out=60,in=235] (68);
			\draw[->, line width=0.2mm] [blue](68) to [out=215,in=15] (8);
			\draw[->, line width=0.2mm] [blue](8) to [out=30,in=200] (68);
			\draw[->, line width=0.2mm] [blue](68) to [out=165,in=10] (66);
			\draw[->, line width=0.2mm] [blue](66) to [out=350,in=185] (68);
			
			\draw[->, line width=0.2mm] [blue](69) to [out=270,in=340] (70);
			\draw[->, line width=0.2mm] [blue](70) to [out=5,in=240] (69);
			\draw[->, line width=0.2mm] [blue](69) to [out=180,in=355] (8);
			\draw[->, line width=0.2mm] [blue](8) to [out=340,in=200] (69);
			\draw[->, line width=0.2mm] [blue](69) to [out=140,in=335] (66);
			\draw[->, line width=0.2mm] [blue](66) to [out=320,in=155] (69);
			
			\draw[->, line width=0.2mm] [blue](70) to [out=160,in=320] (8);
			\draw[->, line width=0.2mm] [blue](8) to [out=280,in=200] (70);
			\draw[->, line width=0.2mm] [blue](70) to [out=120,in=310] (66);
			\draw[->, line width=0.2mm] [blue](66) to [out=290,in=140] (70);
			
			\draw[->, line width=0.2mm] [blue](8) to [out=90,in=250] (66);
			\draw[->, line width=0.2mm] [blue](66) to [out=230,in=110] (8);
			
			\node (71) at (15,14.5) {\tiny{$E_2$}};
			\node (72) at (10.5,17) {\tiny{$E_3$}};	
			\node (73) at (9.5,21.5)  {\tiny{$E_4$}};
			\node (74) at (13.5,23.5) {\tiny{$E_5$}};
			\node (75) at (18,21.5) {\tiny{$E_6$}};	
			
			\draw[->, line width=0.2mm] [blue](74) to [out=0,in=120] (75);
			\draw[->, line width=0.2mm] [blue](75) to [out=150,in=330] (74);
			\draw[->, line width=0.2mm] [blue](74) to [out=310,in=120] (9);
			\draw[->, line width=0.2mm] [blue](9) to [out=140,in=295] (74);
			\draw[->, line width=0.2mm] [blue](74) to [out=280,in=80] (71);
			\draw[->, line width=0.2mm] [blue](71) to [out=100,in=260] (74);
			\draw[->, line width=0.2mm] [blue](74) to [out=245,in=60] (72);
			\draw[->, line width=0.2mm] [blue](72) to [out=80,in=225] (74);
			\draw[->, line width=0.2mm] [blue](74) to [out=205,in=35] (73);
			\draw[->, line width=0.2mm] [blue](73) to [out=60,in=185] (74);
			
			\draw[->, line width=0.2mm] [blue](75) to [out=305,in=80] (9);
			\draw[->, line width=0.2mm] [blue](9) to [out=100,in=280] (75);
			\draw[->, line width=0.2mm] [blue](75) to [out=255,in=45] (71);
			\draw[->, line width=0.2mm] [blue](71) to [out=60,in=235] (75);
			\draw[->, line width=0.2mm] [blue](75) to [out=220,in=20] (72);
			\draw[->, line width=0.2mm] [blue](72) to [out=35,in=205] (75);
			\draw[->, line width=0.2mm] [blue](75) to [out=165,in=10] (73);
			\draw[->, line width=0.2mm] [blue](73) to [out=350,in=185] (75);
			
			\draw[->, line width=0.2mm] [blue](9) to [out=260,in=0] (71);
			\draw[->, line width=0.2mm] [blue](71) to [out=20,in=220] (9);
			\draw[->, line width=0.2mm] [blue](9) to [out=180,in=0] (72);
			\draw[->, line width=0.2mm] [blue](72) to [out=340,in=200] (9);
			\draw[->, line width=0.2mm] [blue](9) to [out=150,in=335] (73);
			\draw[->, line width=0.2mm] [blue](73) to [out=320,in=165] (9);
			
			\draw[->, line width=0.2mm] [blue](71) to [out=160,in=310] (72);
			\draw[->, line width=0.2mm] [blue](72) to [out=280,in=190] (71);
			\draw[->, line width=0.2mm] [blue](71) to [out=120,in=310] (73);
			\draw[->, line width=0.2mm] [blue](73) to [out=290,in=140] (71);
			
			\draw[->, line width=0.2mm] [blue](72) to [out=110,in=270] (73);
			\draw[->, line width=0.2mm] [blue](73) to [out=250,in=130] (72);
			
			\node (76) at (37,9.5) {\tiny{$C_2$}};
			\node (77) at (41.5,7) {\tiny{$C_3$}};	
			\node (78) at (42.5,2.5)  {\tiny{$C_4$}};
			\node (79) at (38.5,-0.5) {\tiny{$C_5$}};
			\node (80) at (33.5,2.5) {\tiny{$C_6$}};	
			
			\draw[->, line width=0.2mm] [blue](76) to [out=0,in=120] (77);
			\draw[->, line width=0.2mm] [blue](77) to [out=150,in=330] (76);
			\draw[->, line width=0.2mm] [blue](76) to [out=310,in=120] (78);
			\draw[->, line width=0.2mm] [blue](78) to [out=140,in=295] (76);
			\draw[->, line width=0.2mm] [blue](76) to [out=280,in=80] (79);
			\draw[->, line width=0.2mm] [blue](79) to [out=100,in=260] (76);
			\draw[->, line width=0.2mm] [blue](76) to [out=245,in=60] (80);
			\draw[->, line width=0.2mm] [blue](80) to [out=85,in=225] (76);
			\draw[->, line width=0.2mm] [blue](76) to [out=205,in=35] (2);
			\draw[->, line width=0.2mm] [blue](2) to [out=90,in=170] (76);
			
			\draw[->, line width=0.2mm] [blue](77) to [out=305,in=80] (78);
			\draw[->, line width=0.2mm] [blue](78) to [out=100,in=280] (77);
			\draw[->, line width=0.2mm] [blue](77) to [out=255,in=45] (79);
			\draw[->, line width=0.2mm] [blue](79) to [out=60,in=235] (77);
			\draw[->, line width=0.2mm] [blue](77) to [out=220,in=15] (80);
			\draw[->, line width=0.2mm] [blue](80) to [out=35,in=200] (77);
			\draw[->, line width=0.2mm] [blue](77) to [out=165,in=10] (2);
			\draw[->, line width=0.2mm] [blue](2) to [out=350,in=185] (77);
			
			\draw[->, line width=0.2mm] [blue](78) to [out=260,in=350] (79);
			\draw[->, line width=0.2mm] [blue](79) to [out=20,in=235] (78);
			\draw[->, line width=0.2mm] [blue](78) to [out=185,in=350] (80);
			\draw[->, line width=0.2mm] [blue](80) to [out=330,in=210] (78);
			\draw[->, line width=0.2mm] [blue](78) to [out=150,in=335] (2);
			\draw[->, line width=0.2mm] [blue](2) to [out=320,in=165] (78);
			
			\draw[->, line width=0.2mm] [blue](79) to [out=160,in=300] (80);
			\draw[->, line width=0.2mm] [blue](80) to [out=265,in=190] (79);
			\draw[->, line width=0.2mm] [blue](79) to [out=120,in=310] (2);
			\draw[->, line width=0.2mm] [blue](2) to [out=290,in=140] (79);
			
			\draw[->, line width=0.2mm] [blue](80) to [out=110,in=270] (2);
			\draw[->, line width=0.2mm] [blue](2) to [out=250,in=130] (80);
			
		\end{tikzpicture}		
		
		\vspace{.2cm}		
	{\small 	Fig.  31. Directed subgraph $\mathcal{D}_1$ $\cup$ $\mathcal{D}^{Y_1}_{2,2}$ $\cup$ $\mathcal{D}^{Y_1}_{2,3}$ of $\mathcal{D}$ with $X$ in the place of $C_{432}(X)$  
		
		\hspace{2cm}  in the figure, $X\in \{A,B,\dots,F\}$, and $1 \leq i \leq 6$.  }
		
\end{center} }

\medskip

\begin{figure}[ht]
	\centerline{\includegraphics[width=5in]{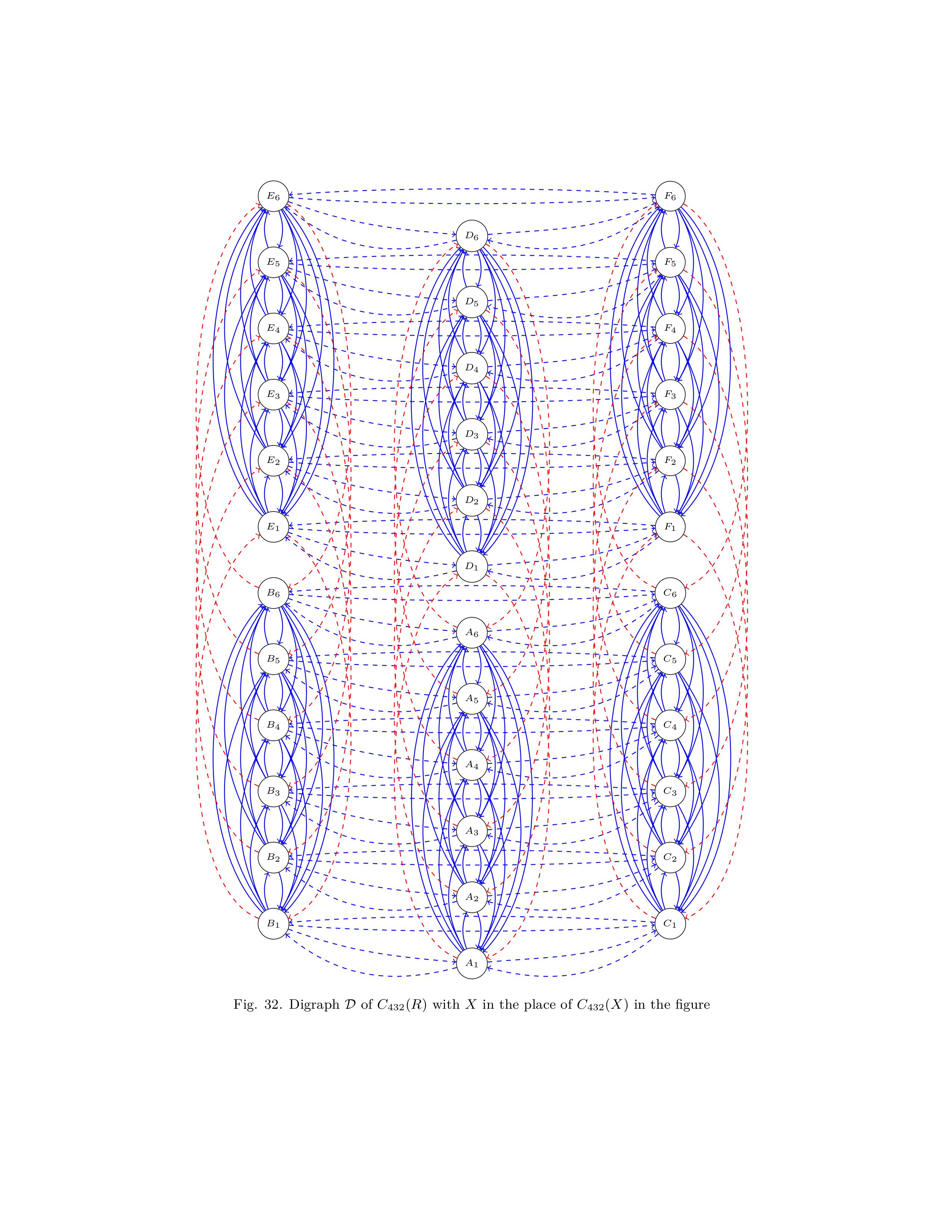}}
\end{figure}

{\rm 	\begin{center}
		\begin{tikzpicture}  
			[scale=.25,auto=center,every node/.style={draw,circle}] 
			
			\node (1) at (26,2) {\tiny{$A_1$}};
			\node (2) at (33,7) {\tiny{$C_1$}};	
			\node (3) at (19,7)  {\tiny{$B_1$}};
			
			\draw[line width=0.3mm] [blue](1)[dashed] to (2);
			\draw[line width=0.4mm] [blue](2)[dashed] to (3);
			\draw[line width=0.3mm] [blue](3)[dashed] to (1);
			
			\node [scale=.0] (11) at (30.5,2.5) [label=90:\tiny{$T2_{432,3}$}]{};
			\node [scale=.0] (12) at (28.5,5.75) [label=90:\tiny{$T2_{432,3}$}]{};
			\node [scale=.0] (13) at (23.5,3)[label=90:\tiny{$T2_{432,3}$}]{};
			
			\node [scale=.0] (14) at (30.5,17.5)[label=90:\tiny{$T2_{432,3}$}]{};
			\node [scale=.0] (15) at (29,15)[label=90:\tiny{$T2_{432,3}$}]{};
			\node [scale=.0](16) at (23.5,17.5)[label=90:\tiny{$T2_{432,3}$}]{};
			
			\node (7) at (26,22) {\tiny{$D_1$}};
			\node (8) at (33,17) {\tiny{$F_1$}};	
			\node (9) at (19,17)  {\tiny{$E_1$}};
			
			\draw[line width=0.3mm] [blue](7)[dashed] to (8);
			\draw[line width=0.3mm] [blue](8)[dashed] to (9);
			\draw[line width=0.3mm] [blue](9)[dashed] to (7);
			
			\draw[line width=0.3mm] [red] (1) [dashed] to (7);
			\draw[line width=0.3mm] [red] (2) [dashed] to (8);
			\draw[line width=0.4mm] [red] (3) [dashed] to (9);
			
			\node [scale=.0] (17) at (20.5,11) [label=90:\tiny{$T2_{432,2}$}]{};
			\node [scale=.0] (18) at (34.5,11)[label=90:\tiny{$T2_{432,2}$}]{};
			\node [scale=.0] (19) at (27.5,11)[label=90:\tiny{$T2_{432,2}$}]{};
			
			\node (51) at (30,0) {\tiny{$A_2$}};
			\node (52) at (30,-3.5) {\tiny{$A_3$}};	
			\node (53) at (26,-6)  {\tiny{$A_4$}};
			\node (54) at (22,-3.5) {\tiny{$A_5$}};
			\node (55) at (22,0) {\tiny{$A_6$}};	
			
			\draw[line width=0.2mm] [blue](1) to (51);
			\draw[line width=0.2mm] [blue](1) to (52);
			\draw[line width=0.2mm] [blue](1) to (53);
			\draw[line width=0.2mm] [blue](1) to (54);
			\draw[line width=0.2mm] [blue](1) to (55);
			
			\draw[line width=0.2mm] [blue](51) to (52);
			\draw[line width=0.2mm] [blue](51) to (53);
			\draw[line width=0.2mm] [blue](51) to (54);
			\draw[line width=0.2mm] [blue](51) to (55);
			
			\draw[line width=0.2mm] [blue](52) to (53);
			\draw[line width=0.2mm] [blue](52) to (54);
			\draw[line width=0.2mm] [blue](52) to (55);
			
			\draw[line width=0.2mm] [blue](53) to (54);
			\draw[line width=0.2mm] [blue](53) to (55);
			
			\draw[line width=0.2mm] [blue](54) to (55);
			
			\node (56) at (22,24.5) {\tiny{$D_2$}};
			\node (57) at (22,28) {\tiny{$D_3$}};	
			\node (58) at (26,30.5)  {\tiny{$D_4$}};
			\node (59) at (30,28) {\tiny{$D_5$}};
			\node (60) at (30,24.5) {\tiny{$D_6$}};	
			
			\draw[line width=0.2mm] [blue](7) to (56);
			\draw[line width=0.2mm] [blue](7) to (57);
			\draw[line width=0.2mm] [blue](7) to (58);
			\draw[line width=0.2mm] [blue](7) to (59);
			\draw[line width=0.2mm] [blue](7) to (60);
			
			\draw[line width=0.2mm] [blue](56) to (57);
			\draw[line width=0.2mm] [blue](56) to (58);
			\draw[line width=0.2mm] [blue](56) to (59);
			\draw[line width=0.2mm] [blue](56) to (60);
			
			\draw[line width=0.2mm] [blue](57) to (58);
			\draw[line width=0.2mm] [blue](57) to (59);
			\draw[line width=0.2mm] [blue](57) to (60);
			
			\draw[line width=0.2mm] [blue](58) to (59);
			\draw[line width=0.2mm] [blue](58) to (60);
			
			\draw[line width=0.2mm] [blue](59) to (60);
			
			\node (61) at (18,2.5) {\tiny{$B_2$}};
			\node (62) at (13.5,.5) {\tiny{$B_3$}};	
			\node (63) at (9.5,2.5)  {\tiny{$B_4$}};
			\node (64) at (10.5,7) {\tiny{$B_5$}};
			\node (65) at (15,9) {\tiny{$B_6$}};	
			
			\draw[line width=0.2mm] [blue](3) to (61);
			\draw[line width=0.2mm] [blue](3) to (62);
			\draw[line width=0.2mm] [blue](3) to (63);
			\draw[line width=0.2mm] [blue](3) to (64);
			\draw[line width=0.2mm] [blue](3) to (65);
			
			\draw[line width=0.2mm] [blue](61) to (62);
			\draw[line width=0.2mm] [blue](61) to (63);
			\draw[line width=0.2mm] [blue](61) to (64);
			\draw[line width=0.2mm] [blue](61) to (65);
			
			\draw[line width=0.2mm] [blue](62) to (63);
			\draw[line width=0.2mm] [blue](62) to (64);
			\draw[line width=0.2mm] [blue](62) to (65);
			
			\draw[line width=0.2mm] [blue](63) to (64);
			\draw[line width=0.2mm] [blue](63) to (65);
			
			\draw[line width=0.2mm] [blue](64) to (65);
			
			\node (66) at (34,21.5) {\tiny{$F_2$}};
			\node (67) at (38.5,23.5) {\tiny{$F_3$}};	
			\node (68) at (42.5,21.5)  {\tiny{$F_4$}};
			\node (69) at (41.5,17) {\tiny{$F_5$}};
			\node (70) at (37,15) {\tiny{$F_6$}};	
			
			\draw[line width=0.2mm] [blue](8) to (66);
			\draw[line width=0.2mm] [blue](8) to (67);
			\draw[line width=0.2mm] [blue](8) to (68);
			\draw[line width=0.2mm] [blue](8) to (69);
			\draw[line width=0.2mm] [blue](8) to (70);
			
			\draw[line width=0.2mm] [blue](66) to (67);
			\draw[line width=0.2mm] [blue](66) to (68);
			\draw[line width=0.2mm] [blue](66) to (69);
			\draw[line width=0.2mm] [blue](66) to (70);
			
			\draw[line width=0.2mm] [blue](67) to (68);
			\draw[line width=0.2mm] [blue](67) to (69);
			\draw[line width=0.2mm] [blue](67) to (70);
			
			\draw[line width=0.2mm] [blue](68) to (69);
			\draw[line width=0.2mm] [blue](68) to (70);
			
			\draw[line width=0.2mm] [blue](69) to (70);
			
			\node (71) at (15,15) {\tiny{$E_2$}};
			\node (72) at (10.5,17) {\tiny{$E_3$}};	
			\node (73) at (9.5,21.5)  {\tiny{$E_4$}};
			\node (74) at (13.5,23.5) {\tiny{$E_5$}};
			\node (75) at (18,21.5) {\tiny{$E_6$}};	
			
			\draw[line width=0.2mm] [blue](9) to (71);
			\draw[line width=0.2mm] [blue](9) to (72);
			\draw[line width=0.2mm] [blue](9) to (73);
			\draw[line width=0.2mm] [blue](9) to (74);
			\draw[line width=0.2mm] [blue](9) to (75);
			
			\draw[line width=0.2mm] [blue](71) to (72);
			\draw[line width=0.2mm] [blue](71) to (73);
			\draw[line width=0.2mm] [blue](71) to (74);
			\draw[line width=0.2mm] [blue](71) to (75);
			
			\draw[line width=0.2mm] [blue](72) to (73);
			\draw[line width=0.2mm] [blue](72) to (74);
			\draw[line width=0.2mm] [blue](72) to (75);
			
			\draw[line width=0.2mm] [blue](73) to (74);
			\draw[line width=0.2mm] [blue](73) to (75);
			
			\draw[line width=0.2mm] [blue](74) to (75);
			
			\node (76) at (37,9) {\tiny{$C_2$}};
			\node (77) at (41.5,7) {\tiny{$C_3$}};	
			\node (78) at (42.5,2.5)  {\tiny{$C_4$}};
			\node (79) at (38.5,.5) {\tiny{$C_5$}};
			\node (80) at (34,2.5) {\tiny{$C_6$}};	
			
			\draw[line width=0.2mm] [blue](2) to (76);
			\draw[line width=0.2mm] [blue](2) to (77);
			\draw[line width=0.2mm] [blue](2) to (78);
			\draw[line width=0.2mm] [blue](2) to (79);
			\draw[line width=0.2mm] [blue](2) to (80);
			
			\draw[line width=0.2mm] [blue](76) to (77);
			\draw[line width=0.2mm] [blue](76) to (78);
			\draw[line width=0.2mm] [blue](76) to (79);
			\draw[line width=0.2mm] [blue](76) to (80);
			
			\draw[line width=0.2mm] [blue](77) to (78);
			\draw[line width=0.2mm] [blue](77) to (79);
			\draw[line width=0.2mm] [blue](77) to (80);
			
			\draw[line width=0.2mm] [blue](78) to (79);
			\draw[line width=0.2mm] [blue](78) to (80);
			
			\draw[line width=0.2mm] [blue](79) to (80);
			
		\end{tikzpicture}		
		
		\medskip		
		Fig.  33. Subgraph $\mathcal{G}_1$ $\cup$ $\mathcal{G}^{Y_1}_{2,2}$ $\cup$ $\mathcal{G}^{Y_1}_{2,3}$ of $\mathcal{G}$ with $X$ in the place
		
	\hspace{1cm}  of $C_{432}(X)$, $X\in \{A,B,\dots,F\}$, and $1 \leq i \leq 6$.
		
\end{center} }

\begin{figure}[ht]
	\centerline{\includegraphics[width=5in]{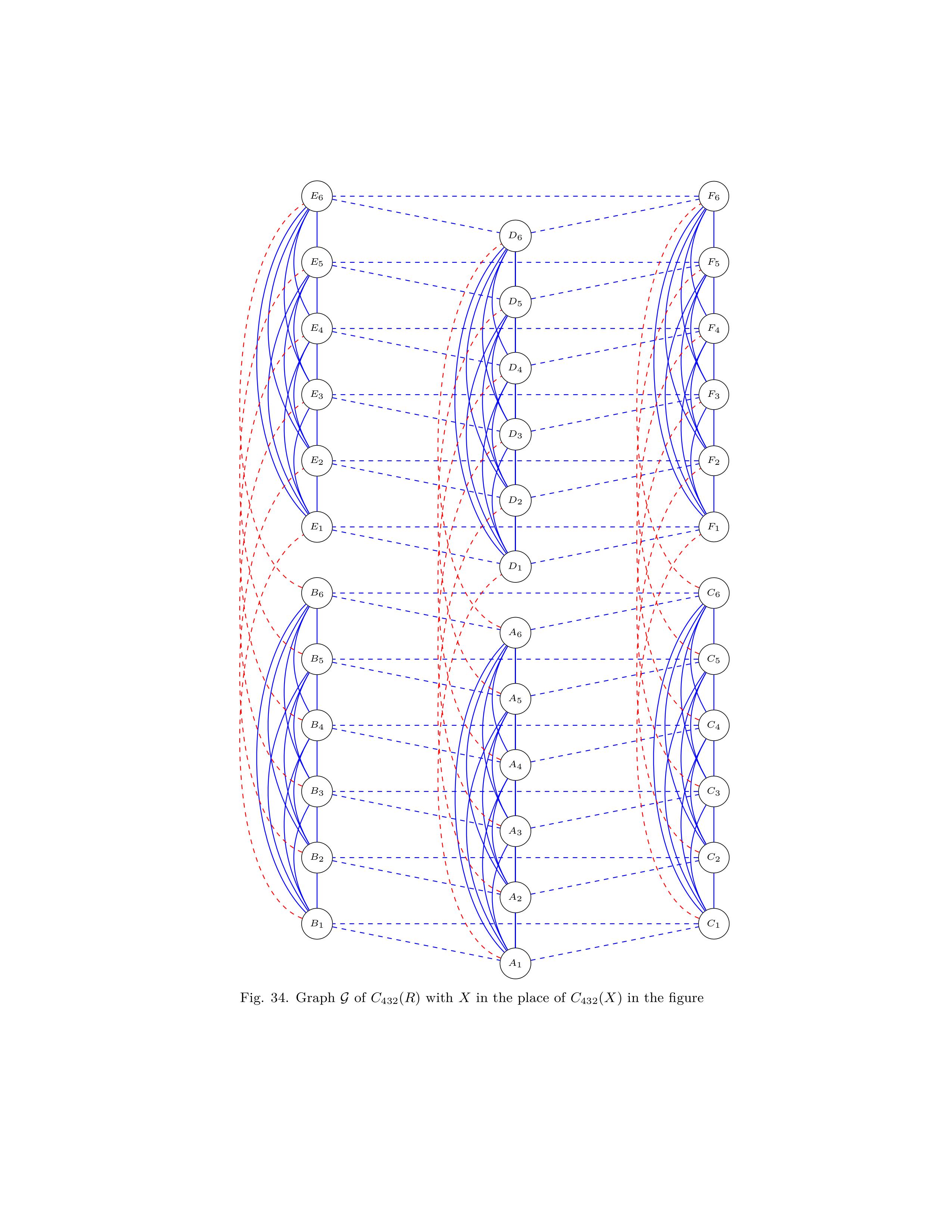}}
\end{figure}

\section{Isomorphism digraph and graph of $C_{6750}(135, 243, 250, 750, 1107, 1593, 2000, 2457, 2500, 2943)$}

In this section, for $R$ = $\{135, 243, 250, 750, 1107, 1593, 2000, 2457, 2500, 2943\}$, we obtain the isomorphism digraph and graph of $C_{6750}(R)$ which has isomorphic circulant graphs of Type-2 w.r.t. $m$ = 3 as well as $m$ = 5. In problems 4 and 6 in \cite{v2-7}, we obtained all isomorphic circulant graphs of Type-1 and Type-2 w.r.t. $m$ = 3 as well as w.r.t. $m$ = 5 of $C_{6750}(R)$ which are presented below.  To find the isomorphism digraph and graph of $C_{6750}(R)$, we use similar method used in obtaining the isomorphism digraph and graph of $C_{432}(16, 27, 48, 54, 128, 160, 189)$ in the previous section.

\begin{prm} {\rm \cite{v2-7} } \quad \label{p11} {\rm Let 
		
		$R_1$ = $\{135, 243, 250, 750, 1107, 1593, 2000, 2457, 2500, 2943\}$ = $A_1$, 
		
		$R_2$ = $\{135, 243, 750, 1000, 1107, 1250, 1593, 2457, 2943, 3250\}$ = $B_1$, 
		
		$R_3$ = $\{135, 243, 500, 750, 1107, 1593, 1750, 2457, 2750, 2943\}$ = $C_1$, 
		
		$S_1$ = $\{27, 135, 250, 750, 1323, 1377, 2000, 2500, 2673, 2727\}$ = $D_1$, 
		
		$S_2$ = $\{27, 135, 750, 1000, 1250, 1323, 1377, 2673, 2727, 3250\}$ = $E_1$,
		
		$S_3$ = $\{27, 135, 500, 750, 1323, 1377, 1750, 2673, 2727, 2750\}$ = $F_1$,
		
		$T_1$ = $\{135, 250, 297, 750, 1053, 1647, 2000, 2403, 2500, 2997\}$ = $G_1$, 
		
		$T_2$ = $\{135, 297, 750, 1000, 1053, 1250, 1647, 2403, 2997, 3250\}$ = $H_1$, 
		
		$T_3$ = $\{135, 297, 500, 750, 1053, 1647, 1750, 2403, 2750, 2997\}$ = $I_1$, 
		
		$U_1$ = $\{135, 250, 567, 750, 783, 1917, 2000, 2133, 2500, 3267\}$ = $J_1$, 
		
		$U_2$ = $\{135, 567, 750, 783, 1000, 1250, 1917, 2133, 3250, 3267\}$ = $K_1$,
		
		$U_3$ = $\{135, 500, 567, 750, 783, 1750, 1917, 2133, 2750, 3267\}$ = $L_1$, 
		
		$V_1$ = $\{135, 250, 513, 750, 837, 1863, 2000, 2187, 2500, 3213\}$ = $M_1$, 
		
		$V_2$ = $\{135, 513, 750, 837, 1000, 1250, 1863, 2187, 3213, 3250\}$ = $N_1$, and
		
		$V_3$ = $\{135, 500, 513, 750, 837, 1750, 1863, 2187, 2750, 3213\}$ = $O_1$.  

Then, for $X_1$ = $A_1, B_1, \dots, O_1$, $(T1_{6750}(C_{6750}(X_1)), \circ')$ is an Abelian group where $`\circ'$ is as given in definition \ref{d2} where   $T1_{6750}(C_{6750}(X_1))$ is given as follows for $X_1$ = $A_1, B_1, \dots, O_1$.		

		\begin{enumerate}
			\item [\rm (1a2a)]  $T1_{6750}(C_{6750}(R_1))$ = $T1_{6750}(C_{6750}(A_1))$ = $\{C_{6750}(A_i): A_1 = R_1 ~\text{and}~i = 1~\text{to}~30\}$ 
			
			= $\{C_{6750}(135, 243, 250, 750, 1107, 1593, 2000, 2457, 2500, 2943)$ = $C_{6750}(A_1)$ =  $C_{6750}(R_1)$,
			
			\hspace{.5cm}  $C_{6750}(250, 405, 621, 729, 750, 1971, 2000, 2079, 2500, 3321)$ = $C_{6750}(A_2)$,
			
			\hspace{.5cm} $C_{6750}(250, 351, 750, 945, 999, 1701, 2000, 2349, 2500, 3051)$ = $C_{6750}(A_3)$,
			
			\hspace{.5cm}  $C_{6750}(250, 513, 750, 837, 1215, 1863, 2000, 2187, 2500, 3213)$  = $C_{6750}(A_4)$,
			
			\hspace{.5cm} $C_{6750}(27, 250, 750, 1323, 1377, 1485, 2000, 2500, 2673, 2727)$ = $C_{6750}(A_5)$,
			
			\hspace{.5cm}  $C_{6750}(250, 459, 750, 891, 1755, 1809, 2000, 2241, 2500, 3159)$ = $C_{6750}(A_6)$, 
			
			\hspace{.5cm} $C_{6750}(81, 250, 750, 1269, 1431, 2000, 2295, 2500, 2619, 2781)$ = $C_{6750}(A_7)$,
			
			\hspace{.5cm}  $C_{6750}(250, 567, 750, 783, 1917, 2000, 2133, 2500, 2565, 3267)$ = $C_{6750}(A_8)$, 
			
			\hspace{.5cm} $C_{6750}(250, 297, 750, 1053, 1647, 2000, 2403, 2500, 2835, 2997)$ = $C_{6750}(A_9)$,
			
			\hspace{.5cm}  $C_{6750}(189, 250, 750, 1161, 1539, 2000, 2500, 2511, 2889, 3105) = C_{6750}(A_{10})$,
			
			\hspace{.5cm}  $C_{6750}(135, 243, 500, 1107, 1500, 1593, 1750, 2457, 2750, 2943) = C_{6750}(A_{11})$,
			
			\hspace{.5cm}  $C_{6750}(405, 500, 621, 729, 1500, 1750, 1971, 2079, 2750, 3321) = C_{6750}(A_{12})$,
			
			\hspace{.5cm} $C_{6750}(351, 500, 945, 999, 1500, 1701, 1750, 2349, 2750, 3051) = C_{6750}(A_{13}),$
			
			\hspace{.5cm}  $C_{6750}(500, 513, 837, 1215, 1500, 1750, 1863, 2187, 2750, 3213) = C_{6750}(A_{14}),$
			
			\hspace{.5cm} $C_{6750}(27, 500, 1323, 1377, 1485, 1500, 1750, 2673, 2727, 2750) = C_{6750}(A_{15}),$
			
			\hspace{.5cm}  $C_{6750}(459, 500, 891, 1500, 1750, 1755, 1809, 2241, 2750, 3159) = C_{6750}(A_{16}),$ 
			
			\hspace{.5cm} $C_{6750}(81, 500, 1269, 1431, 1500, 1750, 2295, 2619, 2750, 2781) = C_{6750}(A_{17}),$
			
			\hspace{.5cm}  $C_{6750}(500, 567, 783, 1500, 1750, 1917, 2133, 2565, 2750, 3267) = C_{6750}(A_{18}),$ 
			
			\hspace{.5cm} $C_{6750}(297, 500, 1053, 1500, 1647, 1750, 2403, 2750, 2835, 2997) = C_{6750}(A_{19}),$
			
			\hspace{.5cm}  $C_{6750}(189, 500, 1161, 1500, 1539, 1750, 2511, 2750, 2889, 3105) = C_{6750}(A_{20}),$
			
			\hspace{.5cm}  $C_{6750}(135, 243, 1000, 1107, 1250, 1593, 2457, 2943, 3000, 3250) = C_{6750}(A_{21})$,
			
			\hspace{.5cm}  $C_{6750}(405, 621, 729, 1000, 1250, 1971, 2079, 3000, 3250, 3321) = C_{6750}(A_{22})$,
			
			\hspace{.5cm} $C_{6750}(351, 945, 999, 1000, 1250, 1701, 2349, 3000, 3051, 3250) = C_{6750}(A_{23}),$
			
			\hspace{.5cm}  $C_{6750}(513, 837, 1000, 1215, 1250, 1863, 2187, 3000, 3213, 3250) = C_{6750}(A_{24}),$
			
			\hspace{.5cm} $C_{6750}(27, 1000, 1250, 1323, 1377, 1485, 2673, 2727, 3000, 3250) = C_{6750}(A_{25}),$
			
			\hspace{.5cm}  $C_{6750}(459, 891, 1000, 1250, 1755, 1809, 2241, 3000, 3159, 3250) = C_{6750}(A_{26}),$ 
			
			\hspace{.5cm} $C_{6750}(81, 1000, 1250, 1269, 1431, 2295, 2619, 2781, 3000, 3250) = C_{6750}(A_{27}),$
			
			\hspace{.5cm}  $C_{6750}(567, 783, 1000, 1250, 1917, 2133, 2565, 3000, 3250, 3267) = C_{6750}(A_{28}),$ 
			
			\hspace{.5cm} $C_{6750}(297, 1000, 1053, 1250, 1647, 2403, 2835, 2997, 3000, 3250) = C_{6750}(A_{29}),$
			
			\hspace{.5cm}  $C_{6750}(189, 1000, 1161, 1250, 1539, 2511, 2889, 3000, 3105, 3250) = C_{6750}(A_{30})\}$.\\
			
	\item [\rm (1a2b)]  $T1_{6750}(C_{6750}(S_1))$ = $T1_{6750}(C_{6750}(B_1))$ = $\{C_{6750}(B_i): B_1 = S_1 ~\text{and}~i = 1~\text{to}~30\}$ 
		
	= $\{C_{6750}(27, 135, 250, 750, 1323, 1377, 2000, 2500, 2673, 2727)$ = $C_{6750}(B_1)$ =  $C_{6750}(S_1)$,
	
	\hspace{.5cm}  $C_{6750}(81, 250, 405, 750, 1269, 1431, 2000, 2500, 2619, 2781)$ = $C_{6750}(B_2)$,
	
	\hspace{.5cm} $C_{6750}(189, 250, 750, 945, 1161, 1539, 2000, 2500, 2511, 2889)$ = $C_{6750}(B_3)$,
	
	\hspace{.5cm}  $C_{6750}(243, 250, 750, 1107, 1215, 1593, 2000, 2457, 2500, 2943)$  = $C_{6750}(B_4)$,
	
	\hspace{.5cm} $C_{6750}(250, 297, 750, 1053, 1485, 1647, 2000, 2403, 2500, 2997)$ = $C_{6750}(B_5)$,
	
	\hspace{.5cm}  $C_{6750}(250, 351, 750, 999, 1701, 1755, 2000, 2349, 2500, 3051)$ = $C_{6750}(B_6)$, 
	
	\hspace{.5cm} $C_{6750}(250, 459, 750, 891, 1809, 2000, 2241, 2295, 2500, 3159)$ = $C_{6750}(B_7)$,
	
	\hspace{.5cm}  $C_{6750}(250, 513, 750, 837, 1863, 2000, 2187, 2500, 2565, 3213)$ = $C_{6750}(B_8)$, 
	
	\hspace{.5cm} $C_{6750}(250, 567, 750, 783, 1917, 2000, 2133, 2500, 2835, 3267)$ = $C_{6750}(B_9)$,
	
	\hspace{.5cm}  $C_{6750}(250, 621, 729, 750, 1971, 2000, 2079, 2500, 3105, 3321) = C_{6750}(B_{10})$,
	
	\hspace{.5cm}  $C_{6750}(27, 135, 500, 1323, 1377, 1500, 1750, 2673, 2727, 2750) = C_{6750}(B_{11})$,
	
	\hspace{.5cm}  $C_{6750}(81, 405, 500, 1269, 1431, 1500, 1750, 2619, 2750, 2781) = C_{6750}(B_{12})$,
	
	\hspace{.5cm} $C_{6750}(189, 500, 945, 1161, 1500, 1539, 1750, 2511, 2750, 2889) = C_{6750}(B_{13}),$
	
	\hspace{.5cm}  $C_{6750}(243, 500, 1107, 1215, 1500, 1593, 1750, 2457, 2750, 2943) = C_{6750}(B_{14}),$
	
	\hspace{.5cm} $C_{6750}(297, 500, 1053, 1485, 1500, 1647, 1750, 2403, 2750, 2997) = C_{6750}(B_{15}),$
	
	\hspace{.5cm}  $C_{6750}(351, 500, 999, 1500, 1701, 1750, 1755, 2349, 2750, 3051) = C_{6750}(B_{16}),$ 
	
	\hspace{.5cm} $C_{6750}(459, 500, 891, 1500, 1750, 1809, 2241, 2295, 2750, 3159) = C_{6750}(B_{17}),$
	
	\hspace{.5cm}  $C_{6750}(500, 513, 837, 1500, 1750, 1863, 2187, 2565, 2750, 3213) = C_{6750}(B_{18}),$ 
	
	\hspace{.5cm} $C_{6750}(500, 567, 783, 1500, 1750, 1917, 2133, 2750, 2835, 3267) = C_{6750}(B_{19}),$
	
	\hspace{.5cm}  $C_{6750}(500, 621, 729, 1500, 1750, 1971, 2079, 2750, 3105, 3321) = C_{6750}(B_{20}),$
	
	\hspace{.5cm}  $C_{6750}(27, 135, 1000, 1250, 1323, 1377, 2673, 2727, 3000, 3250) = C_{6750}(B_{21})$,
	
	\hspace{.5cm}  $C_{6750}(81, 405, 1000, 1250, 1269, 1431, 2619, 2781, 3000, 3250) = C_{6750}(B_{22})$,
	
	\hspace{.5cm} $C_{6750}(189, 945, 1000, 1161, 1250, 1539, 2511, 2889, 3000, 3250) = C_{6750}(B_{23}),$
	
	\hspace{.5cm}  $C_{6750}(243, 1000, 1107, 1215, 1250, 1593, 2457, 2943, 3000, 3250) = C_{6750}(B_{24}),$
	
	\hspace{.5cm} $C_{6750}(297, 1000, 1053, 1250, 1485, 1647, 2403, 2997, 3000, 3250) = C_{6750}(B_{25}),$
	
	\hspace{.5cm}  $C_{6750}(351, 999, 1000, 1250, 1701, 1755, 2349, 3000, 3051, 3250) = C_{6750}(B_{26}),$ 
	
	\hspace{.5cm} $C_{6750}(459, 891, 1000, 1250, 1809, 2241, 2295, 3000, 3159, 3250) = C_{6750}(B_{27}),$
	
	\hspace{.5cm}  $C_{6750}(513, 837, 1000, 1250, 1863, 2187, 2565, 3000, 3213, 3250) = C_{6750}(B_{28}),$ 
	
	\hspace{.5cm} $C_{6750}(567, 783, 1000, 1250, 1917, 2133, 2835, 3000, 3250, 3267) = C_{6750}(B_{29}),$
	
	\hspace{.5cm}  $C_{6750}(621, 729, 1000, 1250, 1971, 2079, 3000, 3105, 3250, 3321) = C_{6750}(B_{30})\}$.		\\
			
\item [\rm (1a2c)]  $T1_{6750}(C_{6750}(T_1))$ = $T1_{6750}(C_{6750}(C_1))$ = $\{C_{6750}(C_i): C_1 = T_1 ~\text{and}~i = 1~\text{to}~30\}$ 

= $\{C_{6750}(135, 250, 297, 750, 1053, 1647, 2000, 2403, 2500, 2997)$ = $C_{6750}(C_1)$ =  $C_{6750}(T_1)$,

\hspace{.5cm}  $C_{6750}(250, 405, 459, 750, 891, 1809, 2000, 2241, 2500, 3159)$ = $C_{6750}(C_2)$,

\hspace{.5cm} $C_{6750}(250, 621, 729, 750, 945, 1971, 2000, 2079, 2500, 3321)$ = $C_{6750}(C_3)$,

\hspace{.5cm}  $C_{6750}(27, 250, 750, 1215, 1323, 1377, 2000, 2500, 2673, 2727)$  = $C_{6750}(C_4)$,

\hspace{.5cm} $C_{6750}(250, 567, 750, 783, 1485, 1917, 2000, 2133, 2500, 3267)$ = $C_{6750}(C_5)$,

\hspace{.5cm}  $C_{6750}(189, 250, 750, 1161, 1539, 1755, 2000, 2500, 2511, 2889)$ = $C_{6750}(C_6)$, 

\hspace{.5cm} $C_{6750}(250, 351, 750, 999, 1701, 2000, 2295, 2349, 2500, 3051)$ = $C_{6750}(C_7)$,

\hspace{.5cm}  $C_{6750}(243, 250, 750, 1107, 1593, 2000, 2457, 2500, 2565, 2943)$ = $C_{6750}(C_8)$, 

\hspace{.5cm} $C_{6750}(250, 513, 750, 837, 1863, 2000, 2187, 2500, 2835, 3213)$ = $C_{6750}(C_9)$,

\hspace{.5cm}  $C_{6750}(81, 250, 750, 1269, 1431, 2000, 2500, 2619, 2781, 3105) = C_{6750}(C_{10})$,

\hspace{.5cm}  $C_{6750}(135, 297, 500, 1053, 1500, 1647, 1750, 2403, 2750, 2997) = C_{6750}(C_{11})$,

\hspace{.5cm}  $C_{6750}(405, 459, 500, 891, 1500, 1750, 1809, 2241, 2750, 3159) = C_{6750}(C_{12})$,

\hspace{.5cm} $C_{6750}(500, 621, 729, 945, 1500, 1750, 1971, 2079, 2750, 3321) = C_{6750}(C_{13}),$

\hspace{.5cm}  $C_{6750}(27, 500, 1215, 1323, 1377, 1500, 1750, 2673, 2727, 2750) = C_{6750}(C_{14}),$

\hspace{.5cm} $C_{6750}(500, 567, 783, 1485, 1500, 1750, 1917, 2133, 2750, 3267) = C_{6750}(C_{15}),$

\hspace{.5cm}  $C_{6750}(189, 500, 1161, 1500, 1539, 1750, 1755, 2511, 2750, 2889) = C_{6750}(C_{16}),$ 

\hspace{.5cm} $C_{6750}(351, 500, 999, 1500, 1701, 1750, 2295, 2349, 2750, 3051) = C_{6750}(C_{17}),$

\hspace{.5cm}  $C_{6750}(243, 500, 1107, 1500, 1593, 1750, 2457, 2565, 2750, 2943) = C_{6750}(C_{18}),$ 

\hspace{.5cm} $C_{6750}(500, 513, 837, 1500, 1750, 1863, 2187, 2750, 2835, 3213) = C_{6750}(C_{19}),$

\hspace{.5cm}  $C_{6750}(81, 500, 1269, 1431, 1500, 1750, 2619, 2750, 2781, 3105) = C_{6750}(C_{20}),$

\hspace{.5cm}  $C_{6750}(135, 297, 1000, 1053, 1250, 1647, 2403, 2997, 3000, 3250) = C_{6750}(C_{21})$,

\hspace{.5cm}  $C_{6750}(405, 459, 891, 1000, 1250, 1809, 2241, 3000, 3159, 3250) = C_{6750}(C_{22})$,

\hspace{.5cm} $C_{6750}(621, 729, 945, 1000, 1250, 1971, 2079, 3000, 3250, 3321) = C_{6750}(C_{23}),$

\hspace{.5cm}  $C_{6750}(27, 1000, 1215, 1250, 1323, 1377, 2673, 2727, 3000, 3250) = C_{6750}(C_{24}),$

\hspace{.5cm} $C_{6750}(567, 783, 1000, 1250, 1485, 1917, 2133, 3000, 3250, 3267) = C_{6750}(C_{25}),$

\hspace{.5cm}  $C_{6750}(189, 1000, 1161, 1250, 1539, 1755, 2511, 2889, 3000, 3250) = C_{6750}(C_{26}),$ 

\hspace{.5cm} $C_{6750}(351, 999, 1000, 1250, 1701, 2295, 2349, 3000, 3051, 3250) = C_{6750}(C_{27}),$

\hspace{.5cm}  $C_{6750}(243, 1000, 1107, 1250, 1593, 2457, 2565, 2943, 3000, 3250) = C_{6750}(C_{28}),$ 

\hspace{.5cm} $C_{6750}(513, 837, 1000, 1250, 1863, 2187, 2835, 3000, 3213, 3250) = C_{6750}(C_{29}),$

\hspace{.5cm}  $C_{6750}(81, 1000, 1250, 1269, 1431, 2619, 2781, 3000, 3105, 3250) = C_{6750}(C_{30})\}$. \\
			
\item [\rm (1a2d)]  $T1_{6750}(C_{6750}(U_1))$ = $T1_{6750}(C_{6750}(D_1))$ = $\{C_{6750}(D_i): D_1 = U_1 ~\text{and}~i = 1~\text{to}~30\}$ 
			
= $\{C_{6750}(135, 250, 567, 750, 783, 1917, 2000, 2133, 2500, 3267)$ = $C_{6750}(D_1)$ = $C_{6750}(U_1)$,

\hspace{.5cm}  $C_{6750}(250, 351, 405, 750, 999, 1701, 2000, 2349, 2500, 3051)$ = $C_{6750}(D_2)$,

\hspace{.5cm} $C_{6750}(81, 250, 750, 945, 1269, 1431, 2000, 2500, 2619, 2781)$ = $C_{6750}(D_3)$,

\hspace{.5cm}  $C_{6750}(250, 297, 750, 1053, 1215, 1647, 2000, 2403, 2500, 2997)$  = $C_{6750}(D_4)$,

\hspace{.5cm} $C_{6750}(250, 513, 750, 837, 1485, 1863, 2000, 2187, 2500, 3213)$ = $C_{6750}(D_5)$,

\hspace{.5cm}  $C_{6750}(250, 621, 729, 750, 1755, 1971, 2000, 2079, 2500, 3321)$ = $C_{6750}(D_6)$, 

\hspace{.5cm} $C_{6750}(189, 250, 750, 1161, 1539, 2000, 2295, 2500, 2511, 2889)$ = $C_{6750}(D_7)$,

\hspace{.5cm}  $C_{6750}(27, 250, 750, 1323, 1377, 2000, 2500, 2565, 2673, 2727)$ = $C_{6750}(D_8)$, 

\hspace{.5cm} $C_{6750}(243, 250, 750, 1107, 1593, 2000, 2457, 2500, 2835, 2943)$ = $C_{6750}(D_9)$,

\hspace{.5cm}  $C_{6750}(250, 459, 750, 891, 1809, 2000, 2241, 2500, 3105, 3159) = C_{6750}(D_{10})$,

\hspace{.5cm}  $C_{6750}(135, 500, 567, 783, 1500, 1750, 1917, 2133, 2750, 3267) = C_{6750}(D_{11})$,

\hspace{.5cm}  $C_{6750}(351, 405, 500, 999, 1500, 1701, 1750, 2349, 2750, 3051) = C_{6750}(D_{12})$,

\hspace{.5cm} $C_{6750}(81, 500, 945, 1269, 1431, 1500, 1750, 2619, 2750, 2781) = C_{6750}(D_{13}),$

\hspace{.5cm}  $C_{6750}(297, 500, 1053, 1215, 1500, 1647, 1750, 2403, 2750, 2997) = C_{6750}(D_{14}),$

\hspace{.5cm} $C_{6750}(500, 513, 837, 1485, 1500, 1750, 1863, 2187, 2750, 3213) = C_{6750}(D_{15}),$

\hspace{.5cm}  $C_{6750}(500, 621, 729, 1500, 1750, 1755, 1971, 2079, 2750, 3321) = C_{6750}(D_{16}),$ 

\hspace{.5cm} $C_{6750}(189, 500, 1161, 1500, 1539, 1750, 2295, 2511, 2750, 2889) = C_{6750}(D_{17}),$

\hspace{.5cm}  $C_{6750}(27, 500, 1323, 1377, 1500, 1750, 2565, 2673, 2727, 2750) = C_{6750}(D_{18}),$ 

\hspace{.5cm} $C_{6750}(243, 500, 1107, 1500, 1593, 1750, 2457, 2750, 2835, 2943) = C_{6750}(D_{19}),$

\hspace{.5cm}  $C_{6750}(459, 500, 891, 1500, 1750, 1809, 2241, 2750, 3105, 3159) = C_{6750}(D_{20}),$

\hspace{.5cm}  $C_{6750}(135, 567, 783, 1000, 1250, 1917, 2133, 3000, 3250, 3267) = C_{6750}(D_{21})$,

\hspace{.5cm}  $C_{6750}(351, 405, 999, 1000, 1250, 1701, 2349, 3000, 3051, 3250) = C_{6750}(D_{22})$,

\hspace{.5cm} $C_{6750}(81, 945, 1000, 1250, 1269, 1431, 2619, 2781, 3000, 3250) = C_{6750}(D_{23}),$

\hspace{.5cm}  $C_{6750}(297, 1000, 1053, 1215, 1250, 1647, 2403, 2997, 3000, 3250) = C_{6750}(D_{24}),$

\hspace{.5cm} $C_{6750}(513, 837, 1000, 1250, 1485, 1863, 2187, 3000, 3213, 3250) = C_{6750}(D_{25}),$

\hspace{.5cm}  $C_{6750}(621, 729, 1000, 1250, 1755, 1971, 2079, 3000, 3250, 3321) = C_{6750}(D_{26}),$ 

\hspace{.5cm} $C_{6750}(189, 1000, 1161, 1250, 1539, 2295, 2511, 2889, 3000, 3250) = C_{6750}(D_{27}),$

\hspace{.5cm}  $C_{6750}(27, 1000, 1250, 1323, 1377, 2565, 2673, 2727, 3000, 3250) = C_{6750}(D_{28}),$ 

\hspace{.5cm} $C_{6750}(243, 1000, 1107, 1250, 1593, 2457, 2835, 2943, 3000, 3250) = C_{6750}(D_{29}),$

\hspace{.5cm}  $C_{6750}(459, 891, 1000, 1250, 1809, 2241, 3000, 3105, 3159, 3250) = C_{6750}(D_{30})\}$.\\

\item [\rm (1a2e)]  $T1_{6750}(C_{6750}(V_1))$ = $T1_{6750}(C_{6750}(E_1))$ = $\{C_{6750}(E_i): E_1 = V_1 ~\text{and}~i = 1~\text{to}~30\}$ 

= $\{C_{6750}(135, 250, 513, 750, 837, 1863, 2000, 2187, 2500, 3213)$ = $C_{6750}(E_1)$ = $C_{6750}(V_1)$,

\hspace{.5cm}  $C_{6750}(189, 250, 405, 750, 1161, 1539, 2000, 2500, 2511, 2889)$ = $C_{6750}(E_2)$,

\hspace{.5cm} $C_{6750}(250, 459, 750, 891, 945, 1809, 2000, 2241, 2500, 3159)$ = $C_{6750}(E_3)$,

\hspace{.5cm}  $C_{6750}(250, 567, 750, 783, 1215, 1917,  2000, 2133, 2500, 3267)$  = $C_{6750}(E_4)$,

\hspace{.5cm} $C_{6750}(243, 250, 750, 1107, 1485, 1593, 2000, 2457, 2500, 2943)$ = $C_{6750}(E_5)$,

\hspace{.5cm}  $C_{6750}(81, 250, 750, 1269, 1431, 1755, 2000, 2500, 2619, 2781)$ = $C_{6750}(E_6)$, 

\hspace{.5cm} $C_{6750}(250, 621, 729, 750, 1971, 2000, 2079, 2295, 2500, 3321)$ = $C_{6750}(E_7)$,

\hspace{.5cm}  $C_{6750}(250, 297, 750, 1053, 1647, 2000, 2403, 2500, 2565, 2997)$ = $C_{6750}(E_8)$, 

\hspace{.5cm} $C_{6750}(27, 250, 750, 1323, 1377, 2000, 2500, 2673, 2727, 2835)$ = $C_{6750}(E_9)$,

\hspace{.5cm}  $C_{6750}(250, 351, 750, 999, 1701, 2000, 2349, 2500, 3051, 3105) = C_{6750}(E_{10})$,

\hspace{.5cm}  $C_{6750}(135, 500, 513, 837, 1500, 1750, 1863, 2187, 2750, 3213) = C_{6750}(E_{11})$,

\hspace{.5cm}  $C_{6750}(189, 405, 500, 1161, 1500, 1539, 1750, 2511, 2750, 2889) = C_{6750}(E_{12})$,

\hspace{.5cm} $C_{6750}(459, 500, 891, 945, 1500, 1750, 1809, 2241, 2750, 3159) = C_{6750}(E_{13}),$

\hspace{.5cm}  $C_{6750}(500, 567, 783, 1215, 1500, 1750, 1917, 2133, 2750, 3267) = C_{6750}(E_{14}),$

\hspace{.5cm} $C_{6750}(243, 500, 1107, 1485, 1500, 1593, 1750, 2457, 2750, 2943) = C_{6750}(E_{15}),$

\hspace{.5cm}  $C_{6750}(81, 500, 1269, 1431, 1500, 1750, 1755, 2619, 2750, 2781) = C_{6750}(E_{16}),$ 

\hspace{.5cm} $C_{6750}(500, 621, 729, 1500, 1750, 1971, 2079, 2295, 2750, 3321) = C_{6750}(E_{17}),$

\hspace{.5cm}  $C_{6750}(297, 500, 1053, 1500, 1647, 1750, 2403, 2565, 2750, 2997) = C_{6750}(E_{18}),$ 

\hspace{.5cm} $C_{6750}(27, 500, 1323, 1377, 1500, 1750, 2673, 2727, 2750, 2835) = C_{6750}(E_{19}),$

\hspace{.5cm}  $C_{6750}(351, 500, 999, 1500, 1701, 1750, 2349, 2750, 3051, 3105) = C_{6750}(E_{20}),$

\hspace{.5cm}  $C_{6750}(135, 513, 837, 1000, 1250, 1863, 2187, 3000, 3213, 3250) = C_{6750}(E_{21})$,

\hspace{.5cm}  $C_{6750}(189, 405, 1000, 1161, 1250, 1539, 2511, 2889, 3000, 3250) = C_{6750}(E_{22})$,

\hspace{.5cm} $C_{6750}(459, 891, 945, 1000, 1250, 1809, 2241, 3000, 3159, 3250) = C_{6750}(E_{23}),$

\hspace{.5cm}  $C_{6750}(567, 783, 1000, 1215, 1250, 1917, 2133, 3000, 3250, 3267) = C_{6750}(E_{24}),$

\hspace{.5cm} $C_{6750}(243, 1000, 1107, 1250, 1485, 1593, 2457, 2943, 3000, 3250) = C_{6750}(E_{25}),$

\hspace{.5cm}  $C_{6750}(81, 1000, 1250, 1269, 1431, 1755, 2619, 2781, 3000, 3250) = C_{6750}(E_{26}),$ 

\hspace{.5cm} $C_{6750}(621, 729, 1000, 1250, 1971, 2079, 2295, 3000, 3250, 3321) = C_{6750}(E_{27}),$

\hspace{.5cm}  $C_{6750}(297, 1000, 1053, 1250, 1647, 2403, 2565, 2997, 3000, 3250) = C_{6750}(E_{28}),$ 

\hspace{.5cm} $C_{6750}(27, 1000, 1250, 1323, 1377, 2673, 2727, 2835, 3000, 3250) = C_{6750}(E_{29}),$

\hspace{.5cm}  $C_{6750}(351, 999, 1000, 1250, 1701, 2349, 3000, 3051, 3105, 3250) = C_{6750}(E_{30})\}$.\\

\item [\rm (1b2a)]  $T1_{6750}(C_{6750}(R_2))$ = $T1_{6750}(C_{6750}(F_1))$ = $\{C_{6750}(F_i): F_1 = R_2 ~\text{and}~i = 1~\text{to}~30\}$ 

= $\{C_{6750}(135, 243, 750, 1000, 1107, 1250, 1593, 2457, 2943, 3250)$ = $C_{6750}(F_1)$ =  $C_{6750}(R_2)$,

\hspace{.5cm}  $C_{6750}(405, 621, 729, 750, 1000, 1250, 1971, 2079, 3250, 3321)$ = $C_{6750}(F_2)$,

\hspace{.5cm} $C_{6750}(351, 750, 945, 999, 1000, 1250, 1701, 2349, 3051, 3250)$ = $C_{6750}(F_3)$,

\hspace{.5cm}  $C_{6750}(513, 750, 837, 1000, 1215, 1250, 1863, 2187, 3213, 3250)$  = $C_{6750}(F_4)$,

\hspace{.5cm} $C_{6750}(27, 750, 1000, 1250, 1323, 1377, 1485, 2673, 2727, 3250)$ = $C_{6750}(F_5)$,

\hspace{.5cm}  $C_{6750}(459, 750, 891, 1000, 1250, 1755, 1809, 2241, 3159, 3250)$ = $C_{6750}(F_6)$, 

\hspace{.5cm} $C_{6750}(81, 750, 1000, 1250, 1269, 1431, 2295, 2619, 2781, 3250)$ = $C_{6750}(F_7)$,

\hspace{.5cm}  $C_{6750}(567, 750, 783, 1000, 1250, 1917, 2133, 2565, 3250, 3267)$ = $C_{6750}(F_8)$, 

\hspace{.5cm} $C_{6750}(297, 750, 1000, 1053, 1250, 1647, 2403, 2835, 2997, 3250)$ = $C_{6750}(F_9)$,

\hspace{.5cm}  $C_{6750}(189, 750, 1000, 1161, 1250, 1539, 2511, 2889, 3105, 3250) = C_{6750}(F_{10})$,

\hspace{.5cm}  $C_{6750}(135, 243, 250, 1107, 1500, 1593, 2000, 2457, 2500, 2943) = C_{6750}(F_{11})$,

\hspace{.5cm}  $C_{6750}(250, 405, 621, 729, 1500, 1971, 2000, 2079, 2500, 3321) = C_{6750}(F_{12})$,

\hspace{.5cm} $C_{6750}(250, 351, 945, 999, 1500, 1701, 2000, 2349, 2500, 3051) = C_{6750}(F_{13}),$

\hspace{.5cm}  $C_{6750}(250, 513, 837, 1215, 1500, 1863, 2000, 2187, 2500, 3213) = C_{6750}(F_{14}),$

\hspace{.5cm} $C_{6750}(27, 250, 1323, 1377, 1485, 1500, 2000, 2500, 2673, 2727) = C_{6750}(F_{15}),$

\hspace{.5cm}  $C_{6750}(250, 459, 891, 1500, 1755, 1809, 2000, 2241, 2500, 3159) = C_{6750}(F_{16}),$ 

\hspace{.5cm} $C_{6750}(81, 250, 1269, 1431, 1500, 2000, 2295, 2500, 2619, 2781) = C_{6750}(F_{17}),$

\hspace{.5cm}  $C_{6750}(250, 567, 783, 1500, 1917, 2000, 2133, 2500, 2565, 3267) = C_{6750}(F_{18}),$ 

\hspace{.5cm} $C_{6750}(250,  297, 1053, 1500, 1647, 2000, 2403, 2500, 2835, 2997) = C_{6750}(F_{19}),$

\hspace{.5cm}  $C_{6750}(189, 250, 1161, 1500, 1539, 2000, 2500, 2511, 2889, 3105) = C_{6750}(F_{20}),$

\hspace{.5cm}  $C_{6750}(135, 243, 500, 1107, 1593, 1750, 2457, 2750, 2943, 3000) = C_{6750}(F_{21})$,

\hspace{.5cm}  $C_{6750}(405, 500, 621, 729, 1750, 1971, 2079, 2750, 3000, 3321) = C_{6750}(F_{22})$,

\hspace{.5cm} $C_{6750}(351, 500, 945, 999, 1701, 1750, 2349, 2750, 3000, 3051) = C_{6750}(F_{23}),$

\hspace{.5cm}  $C_{6750}(500, 513, 837, 1215, 1750, 1863, 2187, 2750, 3000, 3213) = C_{6750}(F_{24}),$

\hspace{.5cm} $C_{6750}(27, 500, 1323, 1377, 1485, 1750, 2673, 2727, 2750, 3000) = C_{6750}(F_{25}),$

\hspace{.5cm}  $C_{6750}(459, 500, 891, 1750, 1755, 1809, 2241, 2750, 3000, 3159) = C_{6750}(F_{26}),$ 

\hspace{.5cm} $C_{6750}(81, 500, 1269, 1431, 1750, 2295, 2619, 2750, 2781, 3000) = C_{6750}(F_{27}),$

\hspace{.5cm}  $C_{6750}(500, 567, 783, 1750, 1917, 2133, 2565, 2750, 3000, 3267) = C_{6750}(F_{28}),$ 

\hspace{.5cm} $C_{6750}(297, 500, 1053, 1647, 1750, 2403, 2750, 2835, 2997, 3000) = C_{6750}(F_{29}),$

\hspace{.5cm}  $C_{6750}(189, 500, 1161, 1539, 1750, 2511, 2750, 2889, 3000, 3105) = C_{6750}(F_{30})\}$.\\

\item [\rm (1b2b)]  $T1_{6750}(C_{6750}(S_2))$ = $T1_{6750}(C_{6750}(G_1))$ = $\{C_{6750}(G_i): G_1 = S_2 ~\text{and}~i = 1~\text{to}~30\}$ \\

= $\{C_{6750}(27, 135, 750, 1000, 1250, 1323, 1377, 2673, 2727, 3250)$ = $C_{6750}(G_1)$ =  $C_{6750}(S_2)$,

\hspace{.5cm}  $C_{6750}(81, 405, 750, 1000, 1250, 1269, 1431, 2619, 2781, 3250)$ = $C_{6750}(G_2)$,

\hspace{.5cm} $C_{6750}(189, 750, 945, 1000, 1161, 1250, 1539, 2511, 2889, 3250)$ = $C_{6750}(G_3)$,

\hspace{.5cm}  $C_{6750}(243, 750, 1000, 1107, 1215, 1250, 1593, 2457, 2943, 3250)$  = $C_{6750}(G_4)$,

\hspace{.5cm} $C_{6750}(297, 750, 1000, 1053, 1250, 1485, 1647, 2403, 2997, 3250)$ = $C_{6750}(G_5)$,

\hspace{.5cm}  $C_{6750}(351, 750, 999, 1000, 1250, 1701, 1755, 2349, 3051, 3250)$ = $C_{6750}(G_6)$, 

\hspace{.5cm} $C_{6750}(459, 750, 891, 1000, 1250, 1809, 2241, 2295, 3159, 3250)$ = $C_{6750}(G_7)$,

\hspace{.5cm}  $C_{6750}(513, 750, 837, 1000, 1250, 1863, 2187, 2565, 3213, 3250)$ = $C_{6750}(G_8)$, 

\hspace{.5cm} $C_{6750}(567, 750, 783, 1000, 1250, 1917, 2133, 2835, 3250, 3267)$ = $C_{6750}(G_9)$,

\hspace{.5cm}  $C_{6750}(621, 729, 750, 1000, 1250, 1971, 2079, 3105, 3250, 3321) = C_{6750}(G_{10})$,

\hspace{.5cm}  $C_{6750}(27, 135, 250, 1323, 1377, 1500, 2000, 2500, 2673, 2727) = C_{6750}(G_{11})$,

\hspace{.5cm}  $C_{6750}(81, 250, 405, 1269, 1431, 1500, 2000, 2500, 2619, 2781) = C_{6750}(G_{12})$,

\hspace{.5cm} $C_{6750}(189, 250, 945, 1161, 1500, 1539, 2000, 2500, 2511, 2889) = C_{6750}(G_{13}),$

\hspace{.5cm}  $C_{6750}(243, 250, 1107, 1215, 1500, 1593, 2000, 2457, 2500, 2943) = C_{6750}(G_{14}),$

\hspace{.5cm} $C_{6750}(250, 297, 1053, 1485, 1500, 1647, 2000, 2403, 2500, 2997) = C_{6750}(G_{15}),$

\hspace{.5cm}  $C_{6750}(250, 351, 999, 1500, 1701, 1755, 2000, 2349, 2500, 3051) = C_{6750}(G_{16}),$ 

\hspace{.5cm} $C_{6750}(250, 459, 891, 1500, 1809, 2000, 2241, 2295, 2500, 3159) = C_{6750}(G_{17}),$

\hspace{.5cm}  $C_{6750}(250, 513, 837, 1500, 1863, 2000, 2187, 2500, 2565, 3213) = C_{6750}(G_{18}),$ 

\hspace{.5cm} $C_{6750}(250, 567, 783, 1500, 1917, 2000, 2133, 2500, 2835, 3267) = C_{6750}(G_{19}),$

\hspace{.5cm}  $C_{6750}(250, 621, 729, 1500, 1971, 2000, 2079, 2500, 3105, 3321) = C_{6750}(G_{20}),$

\hspace{.5cm}  $C_{6750}(27, 135, 500, 1323, 1377, 1750, 2673, 2727, 2750, 3000) = C_{6750}(G_{21})$,

\hspace{.5cm}  $C_{6750}(81, 405, 500, 1269, 1431, 1750, 2619, 2750, 2781, 3000) = C_{6750}(G_{22})$,

\hspace{.5cm} $C_{6750}(189, 500, 945, 1161, 1539, 1750, 2511, 2750, 2889, 3000) = C_{6750}(G_{23}),$

\hspace{.5cm}  $C_{6750}(243, 500, 1107, 1215, 1593, 1750, 2457, 2750, 2943, 3000) = C_{6750}(G_{24}),$

\hspace{.5cm} $C_{6750}(297, 500, 1053, 1485, 1647, 1750, 2403, 2750, 2997, 3000) = C_{6750}(G_{25}),$

\hspace{.5cm}  $C_{6750}(351, 500, 999, 1701, 1750, 1755, 2349, 2750, 3000, 3051) = C_{6750}(G_{26}),$ 

\hspace{.5cm} $C_{6750}(459, 500, 891, 1750, 1809, 2241, 2295, 2750, 3000, 3159) = C_{6750}(G_{27}),$

\hspace{.5cm}  $C_{6750}(500, 513, 837, 1750, 1863, 2187, 2565, 2750, 3000, 3213) = C_{6750}(G_{28}),$ 

\hspace{.5cm}  $C_{6750}(500, 567, 783, 1750, 1917, 2133, 2750, 2835, 3000, 3267) = C_{6750}(G_{29}),$

\hspace{.5cm}  $C_{6750}(500, 621, 729, 1750, 1971, 2079, 2750, 3000, 3105, 3321) = C_{6750}(G_{30})\}$.\\

\item [\rm (1b2c)]  $T1_{6750}(C_{6750}(T_2))$ = $T1_{6750}(C_{6750}(H_1))$ = $\{C_{6750}(H_i): H_1 = T_2 ~\text{and}~i = 1~\text{to}~30\}$ 

= $\{C_{6750}(135, 297, 750, 1000, 1053, 1250, 1647, 2403, 2997, 3250)$ = $C_{6750}(H_1)$ =  $C_{6750}(T_2)$,

\hspace{.5cm}  $C_{6750}(405, 459, 750, 891, 1000, 1250, 1809, 2241, 3159, 3250)$ = $C_{6750}(H_2)$,

\hspace{.5cm} $C_{6750}(621, 729, 750, 945, 1000, 1250, 1971, 2079, 3250, 3321)$ = $C_{6750}(H_3)$,

\hspace{.5cm}  $C_{6750}(27, 750, 1000, 1215, 1250, 1323, 1377, 2673, 2727, 3250)$  = $C_{6750}(H_4)$,

\hspace{.5cm} $C_{6750}(567, 750, 783, 1000, 1250, 1485, 1917, 2133, 3250, 3267)$ = $C_{6750}(H_5)$,

\hspace{.5cm}  $C_{6750}(189, 750, 1000, 1161, 1250, 1539, 1755, 2511, 2889, 3250)$ = $C_{6750}(H_6)$, 

\hspace{.5cm} $C_{6750}(351, 750, 999, 1000, 1250, 1701, 2295, 2349, 3051, 3250)$ = $C_{6750}(H_7)$,

\hspace{.5cm}  $C_{6750}(243, 750, 1000, 1107, 1250, 1593, 2457, 2565, 2943, 3250)$ = $C_{6750}(H_8)$, 

\hspace{.5cm} $C_{6750}(513, 750, 837, 1000, 1250, 1863, 2187, 2835, 3213, 3250)$ = $C_{6750}(H_9)$,

\hspace{.5cm}  $C_{6750}(81, 750, 1000, 1250, 1269, 1431, 2619, 2781, 3105, 3250) = C_{6750}(H_{10})$,

\hspace{.5cm}  $C_{6750}(135, 250, 297, 1053, 1500, 1647, 2000, 2403, 2500, 2997) = C_{6750}(H_{11})$,

\hspace{.5cm}  $C_{6750}(250, 405, 459, 891, 1500, 1809, 2000, 2241, 2500, 3159) = C_{6750}(H_{12})$,

\hspace{.5cm} $C_{6750}(250, 621, 729, 945, 1500, 1971, 2000, 2079, 2500, 3321) = C_{6750}(H_{13}),$

\hspace{.5cm}  $C_{6750}(27, 250, 1215, 1323, 1377, 1500, 2000, 2500, 2673, 2727) = C_{6750}(H_{14}),$

\hspace{.5cm} $C_{6750}(250, 567, 783, 1485, 1500, 1917, 2000, 2133, 2500, 3267) = C_{6750}(H_{15}),$

\hspace{.5cm}  $C_{6750}(189, 250, 1161, 1500, 1539, 1755, 2000, 2500, 2511, 2889) = C_{6750}(H_{16}),$ 

\hspace{.5cm} $C_{6750}(250, 351, 999, 1500, 1701, 2000, 2295, 2349, 2500, 3051) = C_{6750}(H_{17}),$

\hspace{.5cm}  $C_{6750}(243, 250, 1107, 1500, 1593, 2000, 2457, 2500, 2565, 2943) = C_{6750}(H_{18}),$ 

\hspace{.5cm} $C_{6750}(250, 513, 837, 1500, 1863, 2000, 2187, 2500, 2835, 3213) = C_{6750}(H_{19}),$

\hspace{.5cm}  $C_{6750}(81, 250, 1269, 1431, 1500, 2000, 2500, 2619, 2781, 3105) = C_{6750}(H_{20}),$

\hspace{.5cm}  $C_{6750}(135, 297, 500, 1053, 1647, 1750, 2403, 2750, 2997, 3000) = C_{6750}(H_{21})$,

\hspace{.5cm}  $C_{6750}(405, 459, 500, 891, 1750, 1809, 2241, 2750, 3000, 3159) = C_{6750}(H_{22})$,

\hspace{.5cm} $C_{6750}(500, 621, 729, 945, 1750, 1971, 2079, 2750, 3000, 3321) = C_{6750}(H_{23}),$

\hspace{.5cm}  $C_{6750}(27, 500, 1215, 1323, 1377, 1750, 2673, 2727, 2750, 3000) = C_{6750}(H_{24}),$

\hspace{.5cm} $C_{6750}(500, 567, 783, 1485, 1750, 1917, 2133, 2750, 3000, 3267) = C_{6750}(H_{25}),$

\hspace{.5cm}  $C_{6750}(189, 500, 1161, 1539, 1750, 1755, 2511, 2750, 2889, 3000) = C_{6750}(H_{26}),$ 

\hspace{.5cm} $C_{6750}(351, 500, 999, 1701, 1750, 2295, 2349, 2750, 3000, 3051) = C_{6750}(H_{27}),$

\hspace{.5cm}  $C_{6750}(243, 500, 1107, 1593, 1750, 2457, 2565, 2750, 2943, 3000) = C_{6750}(H_{28}),$ 

\hspace{.5cm} $C_{6750}(500, 513, 837, 1750, 1863, 2187, 2750, 2835, 3000, 3213) = C_{6750}(H_{29}),$

\hspace{.5cm}  $C_{6750}(81, 500, 1269, 1431, 1750, 2619, 2750, 2781, 3000, 3105) = C_{6750}(H_{30})\}$.\\

\item [\rm (1b2d)]  $T1_{6750}(C_{6750}(U_2))$ = $T1_{6750}(C_{6750}(I_1))$ = $\{C_{6750}(I_i): I_1 = U_2 ~\text{and}~i = 1~\text{to}~30\}$ 

= $\{C_{6750}(135, 567, 750, 783, 1000, 1250, 1917, 2133, 3250, 3267)$ = $C_{6750}(I_1)$ = $C_{6750}(U_2)$,

\hspace{.5cm}  $C_{6750}(351, 405, 750, 999, 1000, 1250, 1701, 2349, 3051, 3250)$ = $C_{6750}(I_2)$,

\hspace{.5cm} $C_{6750}(81, 750, 945, 1000, 1250, 1269, 1431, 2619, 2781, 3250)$ = $C_{6750}(I_3)$,

\hspace{.5cm}  $C_{6750}(297, 750, 1000, 1053, 1215, 1250, 1647, 2403, 2997, 3250)$  = $C_{6750}(I_4)$,

\hspace{.5cm} $C_{6750}(513, 750, 837, 1000, 1250, 1485, 1863, 2187, 3213, 3250)$ = $C_{6750}(I_5)$,

\hspace{.5cm}  $C_{6750}(621, 729, 750, 1000, 1250, 1755, 1971, 2079, 3250, 3321)$ = $C_{6750}(I_6)$, 

\hspace{.5cm} $C_{6750}(189, 750, 1000, 1161, 1250, 1539, 2295, 2511, 2889, 3250)$ = $C_{6750}(I_7)$,

\hspace{.5cm}  $C_{6750}(27, 750, 1000, 1250, 1323, 1377, 2565, 2673, 2727, 3250)$ = $C_{6750}(I_8)$, 

\hspace{.5cm} $C_{6750}(243, 750, 1000, 1107, 1250, 1593, 2457, 2835, 2943, 3250)$ = $C_{6750}(I_9)$,

\hspace{.5cm}  $C_{6750}(459, 750, 891, 1000, 1250, 1809, 2241, 3105, 3159, 3250) = C_{6750}(I_{10})$,

\hspace{.5cm}  $C_{6750}(135, 250, 567, 783, 1500, 1917, 2000, 2133, 2500, 3267) = C_{6750}(I_{11})$,

\hspace{.5cm}  $C_{6750}(250, 351, 405, 999, 1500, 1701, 2000, 2349, 2500, 3051) = C_{6750}(I_{12})$,

\hspace{.5cm} $C_{6750}(81, 250, 945, 1269, 1431, 1500, 2000, 2500, 2619, 2781) = C_{6750}(I_{13}),$

\hspace{.5cm}  $C_{6750}(250, 297, 1053, 1215, 1500, 1647, 2000, 2403, 2500, 2997) = C_{6750}(I_{14}),$

\hspace{.5cm} $C_{6750}(250, 513, 837, 1485, 1500, 1863, 2000, 2187, 2500, 3213) = C_{6750}(I_{15}),$

\hspace{.5cm}  $C_{6750}(250, 621, 729, 1500, 1755, 1971, 2000, 2079, 2500, 3321) = C_{6750}(I_{16}),$ 

\hspace{.5cm} $C_{6750}(189, 250, 1161, 1500, 1539, 2000, 2295, 2500, 2511, 2889) = C_{6750}(I_{17}),$

\hspace{.5cm}  $C_{6750}(27, 250, 1323, 1377, 1500, 2000, 2500, 2565, 2673, 2727) = C_{6750}(I_{18}),$ 

\hspace{.5cm} $C_{6750}(243, 250, 1107, 1500, 1593, 2000, 2457, 2500, 2835, 2943) = C_{6750}(I_{19}),$

\hspace{.5cm}  $C_{6750}(250, 459, 891, 1500, 1809, 2000, 2241, 2500, 3105, 3159) = C_{6750}(I_{20}),$

\hspace{.5cm}  $C_{6750}(135, 500, 567, 783, 1750, 1917, 2133, 2750, 3000, 3267) = C_{6750}(I_{21})$,

\hspace{.5cm}  $C_{6750}(351, 405, 500, 999, 1701, 1750, 2349, 2750, 3000, 3051) = C_{6750}(I_{22})$,

\hspace{.5cm} $C_{6750}(81, 500, 945, 1269, 1431, 1750, 2619, 2750, 2781, 3000) = C_{6750}(I_{23}),$

\hspace{.5cm}  $C_{6750}(297, 500, 1053, 1215, 1647, 1750, 2403, 2750, 2997, 3000) = C_{6750}(I_{24}),$

\hspace{.5cm} $C_{6750}(500, 513, 837, 1485, 1750, 1863, 2187, 2750, 3000, 3213) = C_{6750}(I_{25}),$

\hspace{.5cm}  $C_{6750}(500, 621, 729, 1750, 1755, 1971, 2079, 2750, 3000, 3321) = C_{6750}(I_{26}),$ 

\hspace{.5cm} $C_{6750}(189, 500, 1161, 1539, 1750, 2295, 2511, 2750, 2889, 3000) = C_{6750}(I_{27}),$

\hspace{.5cm}  $C_{6750}(27, 500, 1323, 1377, 1750, 2565, 2673, 2727, 2750, 3000) = C_{6750}(I_{28}),$ 

\hspace{.5cm} $C_{6750}(243, 500, 1107, 1593, 1750, 2457, 2750, 2835, 2943, 3000) = C_{6750}(I_{29}),$

\hspace{.5cm}  $C_{6750}(459, 500, 891, 1750, 1809, 2241, 2750, 3000, 3105, 3159) = C_{6750}(I_{30})\}$.\\

\item [\rm (1b2e)]  $T1_{6750}(C_{6750}(V_2))$ = $T1_{6750}(C_{6750}(J_1))$ = $\{C_{6750}(J_i): J_1 = V_2 ~\text{and}~i = 1~\text{to}~30\}$ 

= $\{C_{6750}(135, 513, 750, 837, 1000, 1250, 1863, 2187, 3213, 3250)$ = $C_{6750}(J_1)$ = $C_{6750}(V_2)$,

\hspace{.5cm}  $C_{6750}(189, 405, 750, 1000, 1161, 1250, 1539, 2511, 2889, 3250)$ = $C_{6750}(J_2)$,

\hspace{.5cm} $C_{6750}(459, 750, 891, 945, 1000, 1250, 1809, 2241, 3159, 3250)$ = $C_{6750}(J_3)$,

\hspace{.5cm}  $C_{6750}(567, 750, 783, 1000, 1215, 1250, 1917, 2133, 3250, 3267)$  = $C_{6750}(J_4)$,

\hspace{.5cm} $C_{6750}(243, 750, 1000, 1107, 1250, 1485, 1593, 2457, 2943, 3250)$ = $C_{6750}(J_5)$,

\hspace{.5cm}  $C_{6750}(81, 750, 1000, 1250, 1269, 1431, 1755, 2619, 2781, 3250)$ = $C_{6750}(J_6)$, 

\hspace{.5cm} $C_{6750}(621, 729, 750, 1000, 1250, 1971, 2079, 2295, 3250, 3321)$ = $C_{6750}(J_7)$,

\hspace{.5cm}  $C_{6750}(297, 750, 1000, 1053, 1250, 1647, 2403, 2565, 2997, 3250)$ = $C_{6750}(J_8)$, 

\hspace{.5cm} $C_{6750}(27, 750, 1000, 1250, 1323, 1377, 2673, 2727, 2835, 3250)$ = $C_{6750}(J_9)$,

\hspace{.5cm}  $C_{6750}(351, 750, 999, 1000, 1250, 1701, 2349, 3051, 3105, 3250) = C_{6750}(J_{10})$,

\hspace{.5cm}  $C_{6750}(135, 250, 513, 837, 1500, 1863, 2000, 2187, 2500, 3213) = C_{6750}(J_{11})$,

\hspace{.5cm}  $C_{6750}(189, 250, 405, 1161, 1500, 1539, 2000, 2500, 2511, 2889) = C_{6750}(J_{12})$,

\hspace{.5cm} $C_{6750}(250, 459, 891, 945, 1500, 1809, 2000, 2241, 2500, 3159) = C_{6750}(J_{13}),$

\hspace{.5cm}  $C_{6750}(250, 567, 783, 1215, 1500, 1917, 2000, 2133, 2500, 3267) = C_{6750}(J_{14}),$

\hspace{.5cm} $C_{6750}(243, 250, 1107, 1485, 1500, 1593, 2000, 2457, 2500, 2943) = C_{6750}(J_{15}),$

\hspace{.5cm}  $C_{6750}(81, 250, 1269, 1431, 1500, 1755, 2000, 2500, 2619, 2781) = C_{6750}(J_{16}),$ 

\hspace{.5cm} $C_{6750}(250, 621, 729, 1500, 1971, 2000, 2079, 2295, 2500, 3321) = C_{6750}(J_{17}),$

\hspace{.5cm}  $C_{6750}(250, 297, 1053, 1500, 1647, 2000, 2403, 2500, 2565, 2997) = C_{6750}(J_{18}),$ 

\hspace{.5cm} $C_{6750}(27, 250, 1323, 1377, 1500, 2000, 2500, 2673, 2727, 2835) = C_{6750}(J_{19}),$

\hspace{.5cm}  $C_{6750}(250, 351, 999, 1500, 1701, 2000, 2349, 2500, 3051, 3105) = C_{6750}(J_{20}),$

\hspace{.5cm}  $C_{6750}(135, 500, 513, 837, 1750, 1863, 2187, 2750, 3000, 3213) = C_{6750}(J_{21})$,

\hspace{.5cm}  $C_{6750}(189, 405, 500, 1161, 1539, 1750, 2511, 2750, 2889, 3000) = C_{6750}(J_{22})$,

\hspace{.5cm} $C_{6750}(459, 500, 891, 945, 1750, 1809, 2241, 2750, 3000, 3159) = C_{6750}(J_{23}),$

\hspace{.5cm}  $C_{6750}(500, 567, 783, 1215, 1750, 1917, 2133, 2750, 3000, 3267) = C_{6750}(J_{24}),$

\hspace{.5cm} $C_{6750}(243, 500, 1107, 1485, 1593, 1750, 2457, 2750, 2943, 3000) = C_{6750}(J_{25}),$

\hspace{.5cm}  $C_{6750}(81, 500, 1269, 1431, 1750, 1755, 2619, 2750, 2781, 3000) = C_{6750}(J_{26}),$ 

\hspace{.5cm} $C_{6750}(500, 621, 729, 1750, 1971, 2079, 2295, 2750, 3000, 3321) = C_{6750}(J_{27}),$

\hspace{.5cm}  $C_{6750}(297, 500, 1053, 1647, 1750, 2403, 2565, 2750, 2997, 3000) = C_{6750}(J_{28}),$ 

\hspace{.5cm} $C_{6750}(27, 500, 1323, 1377, 1750, 2673, 2727, 2750, 2835, 3000) = C_{6750}(J_{29}),$

\hspace{.5cm}  $C_{6750}(351, 500, 999, 1701, 1750, 2349, 2750, 3000, 3051, 3105) = C_{6750}(J_{30})\}$. \\

\item [\rm (1c2a)]  $T1_{6750}(C_{6750}(R_3))$ = $T1_{6750}(C_{6750}(K_1))$ = $\{C_{6750}(K_i): K_1 = R_3 ~\text{and}~i = 1~\text{to}~30\}$ 

= $\{C_{6750}(135, 243, 500, 750, 1107, 1593, 1750, 2457, 2750, 2943)$ = $C_{6750}(K_1)$ =  $C_{6750}(R_3)$,

\hspace{.5cm}  $C_{6750}(405, 500, 621, 729, 750, 1750, 1971, 2079, 2750, 3321)$ = $C_{6750}(K_2)$,

\hspace{.5cm} $C_{6750}(351, 500, 750, 945, 999, 1701, 1750, 2349, 2750, 3051)$ = $C_{6750}(K_3)$,

\hspace{.5cm}  $C_{6750}(500, 513, 750, 837, 1215, 1750, 1863, 2187, 2750, 3213)$  = $C_{6750}(K_4)$,

\hspace{.5cm} $C_{6750}(27, 500, 750, 1323, 1377, 1485, 1750, 2673, 2727, 2750)$ = $C_{6750}(K_5)$,

\hspace{.5cm}  $C_{6750}(459, 500, 750, 891, 1750, 1755, 1809, 2241, 2750, 3159)$ = $C_{6750}(K_6)$, 

\hspace{.5cm} $C_{6750}(81, 500, 750, 1269, 1431, 1750, 2295, 2619, 2750, 2781)$ = $C_{6750}(K_7)$,

\hspace{.5cm}  $C_{6750}(500, 567, 750, 783, 1750, 1917, 2133, 2565, 2750, 3267)$ = $C_{6750}(K_8)$, 

\hspace{.5cm} $C_{6750}(297, 500, 750, 1053, 1647, 1750, 2403, 2750, 2835, 2997)$ = $C_{6750}(K_9)$,

\hspace{.5cm}  $C_{6750}(189, 500, 750, 1161, 1539, 1750, 2511, 2750, 2889, 3105) = C_{6750}(K_{10})$,

\hspace{.5cm}  $C_{6750}(135, 243, 1000, 1107, 1250, 1500, 1593, 2457, 2943, 3250) = C_{6750}(K_{11})$,

\hspace{.5cm}  $C_{6750}(405, 621, 729, 1000, 1250, 1500, 1971, 2079, 3250, 3321) = C_{6750}(K_{12})$,

\hspace{.5cm} $C_{6750}(351, 945, 999, 1000, 1250, 1500, 1701, 2349, 3051, 3250) = C_{6750}(K_{13}),$

\hspace{.5cm}  $C_{6750}(513, 837, 1000, 1215, 1250, 1500, 1863, 2187, 3213, 3250) = C_{6750}(K_{14}),$

\hspace{.5cm} $C_{6750}(27, 1000, 1250, 1323, 1377, 1485, 1500, 2673, 2727, 3250) = C_{6750}(K_{15}),$

\hspace{.5cm}  $C_{6750}(459, 891, 1000, 1250, 1500, 1755, 1809, 2241, 3159, 3250) = C_{6750}(K_{16}),$ 

\hspace{.5cm} $C_{6750}(81, 1000, 1250, 1269, 1431, 1500, 2295, 2619, 2781, 3250) = C_{6750}(K_{17}),$

\hspace{.5cm}  $C_{6750}(567, 783, 1000, 1250, 1500, 1917, 2133, 2565, 3250, 3267) = C_{6750}(K_{18}),$ 

\hspace{.5cm} $C_{6750}(297, 1000, 1053, 1250, 1500, 1647, 2403, 2835, 2997, 3250) = C_{6750}(K_{19}),$

\hspace{.5cm}  $C_{6750}(189, 1000, 1161, 1250, 1500, 1539, 2511, 2889, 3105, 3250) = C_{6750}(K_{20}),$

\hspace{.5cm}  $C_{6750}(135, 243, 250, 1107, 1593, 2000, 2457, 2500, 2943, 3000) = C_{6750}(K_{21})$,

\hspace{.5cm}  $C_{6750}(250, 405, 621, 729, 1971, 2000, 2079, 2500, 3000, 3321) = C_{6750}(K_{22})$,

\hspace{.5cm} $C_{6750}(250, 351, 945, 999, 1701, 2000, 2349, 2500, 3000, 3051) = C_{6750}(K_{23}),$

\hspace{.5cm}  $C_{6750}(250, 513, 837, 1215, 1863, 2000, 2187, 2500, 3000, 3213) = C_{6750}(K_{24}),$

\hspace{.5cm} $C_{6750}(27, 250, 1323, 1377, 1485, 2000, 2500, 2673, 2727, 3000) = C_{6750}(K_{25}),$

\hspace{.5cm}  $C_{6750}(250, 459, 891, 1755, 1809, 2000, 2241, 2500, 3000, 3159) = C_{6750}(K_{26}),$ 

\hspace{.5cm} $C_{6750}(81, 250, 1269, 1431, 2000, 2295, 2500, 2619, 2781, 3000) = C_{6750}(K_{27}),$

\hspace{.5cm}  $C_{6750}(250, 567, 783, 1917, 2000, 2133, 2500, 2565, 3000, 3267) = C_{6750}(K_{28}),$ 

\hspace{.5cm} $C_{6750}(250, 297, 1053, 1647, 2000, 2403, 2500, 2835, 2997, 3000) = C_{6750}(K_{29}),$

\hspace{.5cm}  $C_{6750}(189, 250, 1161, 1539, 2000, 2500, 2511, 2889, 3000, 3105) = C_{6750}(K_{30})\}$.\\

\item [\rm (1c2b)]  $T1_{6750}(C_{6750}(S_3))$ = $T1_{6750}(C_{6750}(L_1))$ = $\{C_{6750}(L_i): L_1 = S_3 ~\text{and}~i = 1~\text{to}~30\}$ 

= $\{C_{6750}(27, 135, 500, 750, 1323, 1377, 1750, 2673, 2727, 2750)$ = $C_{6750}(L_1)$ =  $C_{6750}(S_3)$,

\hspace{.5cm}  $C_{6750}(81, 405, 500, 750, 1269, 1431, 1750, 2619, 2750, 2781)$ = $C_{6750}(L_2)$,

\hspace{.5cm} $C_{6750}(189, 500, 750, 945, 1161, 1539, 1750, 2511, 2750, 2889)$ = $C_{6750}(L_3)$,

\hspace{.5cm}  $C_{6750}(243, 500, 750, 1107, 1215, 1593, 1750, 2457, 2750, 2943)$  = $C_{6750}(L_4)$,

\hspace{.5cm} $C_{6750}(297, 500, 750, 1053, 1485, 1647, 1750, 2403, 2750, 2997)$ = $C_{6750}(L_5)$,

\hspace{.5cm}  $C_{6750}(351, 500, 750, 999, 1701, 1750, 1755, 2349, 2750, 3051)$ = $C_{6750}(L_6)$, 

\hspace{.5cm} $C_{6750}(459, 500, 750, 891, 1750, 1809, 2241, 2295, 2750, 3159)$ = $C_{6750}(L_7)$,

\hspace{.5cm}  $C_{6750}(500, 513, 750, 837, 1750, 1863, 2187, 2565, 2750, 3213)$ = $C_{6750}(L_8)$, 

\hspace{.5cm} $C_{6750}(500, 567, 750, 783, 1750, 1917, 2133, 2750, 2835, 3267)$ = $C_{6750}(L_9)$,

\hspace{.5cm}  $C_{6750}(500, 621, 729, 750, 1750, 1971, 2079, 2750, 3105, 3321) = C_{6750}(L_{10})$,

\hspace{.5cm}  $C_{6750}(27, 135, 1000, 1250, 1323, 1377, 1500, 2673, 2727, 3250) = C_{6750}(L_{11})$,

\hspace{.5cm}  $C_{6750}(81, 405, 1000, 1250, 1269, 1431, 1500, 2619, 2781, 3250) = C_{6750}(L_{12})$,

\hspace{.5cm} $C_{6750}(189, 945, 1000, 1161, 1250, 1500, 1539, 2511, 2889, 3250) = C_{6750}(L_{13}),$

\hspace{.5cm}  $C_{6750}(243, 1000, 1107, 1215, 1250, 1500, 1593, 2457, 2943, 3250) = C_{6750}(L_{14}),$

\hspace{.5cm} $C_{6750}(297, 1000, 1053, 1250, 1485, 1500, 1647, 2403, 2997, 3250) = C_{6750}(L_{15}),$

\hspace{.5cm}  $C_{6750}(351, 999, 1000, 1250, 1500, 1701, 1755, 2349, 3051, 3250) = C_{6750}(L_{16}),$ 

\hspace{.5cm} $C_{6750}(459, 891, 1000, 1250, 1500, 1809, 2241, 2295, 3159, 3250) = C_{6750}(L_{17}),$

\hspace{.5cm}  $C_{6750}(513, 837, 1000, 1250, 1500, 1863, 2187, 2565, 3213, 3250) = C_{6750}(L_{18}),$ 

\hspace{.5cm} $C_{6750}(567, 783, 1000, 1250, 1500, 1917, 2133, 2835, 3250, 3267) = C_{6750}(L_{19}),$

\hspace{.5cm}  $C_{6750}(621, 729, 1000, 1250, 1500, 1971, 2079, 3105, 3250, 3321) = C_{6750}(L_{20}),$

\hspace{.5cm}  $C_{6750}(27, 135, 250, 1323, 1377, 2000, 2500, 2673, 2727, 3000) = C_{6750}(L_{21})$,

\hspace{.5cm}  $C_{6750}(81, 250, 405, 1269, 1431, 2000, 2500, 2619, 2781, 3000) = C_{6750}(L_{22})$,

\hspace{.5cm} $C_{6750}(189, 250, 945, 1161, 1539, 2000, 2500, 2511, 2889, 3000) = C_{6750}(L_{23}),$

\hspace{.5cm}  $C_{6750}(243, 250, 1107, 1215, 1593, 2000, 2457, 2500, 2943, 3000) = C_{6750}(L_{24}),$

\hspace{.5cm} $C_{6750}(250, 297, 1053, 1485, 1647, 2000, 2403, 2500, 2997, 3000) = C_{6750}(L_{25}),$

\hspace{.5cm}  $C_{6750}(250, 351, 999, 1701, 1755, 2000, 2349, 2500, 3000, 3051) = C_{6750}(L_{26}),$ 

\hspace{.5cm} $C_{6750}(250, 459, 891, 1809, 2000, 2241, 2295, 2500, 3000, 3159) = C_{6750}(L_{27}),$

\hspace{.5cm}  $C_{6750}(250, , 513, 837, 1863, 2000, 2187, 2500, 2565, 3000, 3213) = C_{6750}(L_{28}),$ 

\hspace{.5cm} $C_{6750}(250, 567, 783, 1917, 2000, 2133, 2500, 2835, 3000, 3267) = C_{6750}(L_{29}),$

\hspace{.5cm}  $C_{6750}(250, 621, 729, 1971, 2000, 2079, 2500, 3000, 3105, 3321) = C_{6750}(L_{30})\}$. \\

\item [\rm (1c2c)]  $T1_{6750}(C_{6750}(T_3))$ = $T1_{6750}(C_{6750}(M_1))$ = $\{C_{6750}(M_i): M_1 = T_3 ~\text{and}~i = 1~\text{to}~30\}$ 

= $\{C_{6750}(135, 297, 500, 750, 1053, 1647, 1750, 2403, 2750, 2997)$ = $C_{6750}(M_{1})$ =  $C_{6750}(T_3)$,

\hspace{.5cm}  $C_{6750}(405, 459, 500, 750, 891, 1750, 1809, 2241, 2750, 3159)$ = $C_{6750}(M_{2})$,

\hspace{.5cm} $C_{6750}(500, 621, 729, 750, 945, 1750, 1971, 2079, 2750, 3321)$ = $C_{6750}(M_{3})$,

\hspace{.5cm}  $C_{6750}(27, 500, 750, 1215, 1323, 1377, 1750, 2673, 2727, 2750)$ = $C_{6750}(M_{4})$,

\hspace{.5cm} $C_{6750}(500, 567, 750, 783, 1485, 1750, 1917, 2133, 2750, 3267)$ = $C_{6750}(M_{5})$,

\hspace{.5cm}  $C_{6750}(189, 500, 750, 1161, 1539, 1750, 1755, 2511, 2750, 2889)$ = $C_{6750}(M_{6})$, 

\hspace{.5cm} $C_{6750}(351, 500, 750, 999, 1701, 1750, 2295, 2349, 2750, 3051)$ = $C_{6750}(M_{7})$,

\hspace{.5cm}  $C_{6750}(243, 500, 750, 1107, 1593, 1750, 2457, 2565, 2750, 2943)$ = $C_{6750}(M_{8})$, 

\hspace{.5cm} $C_{6750}(500, 513, 750, 837, 1750, 1863, 2187, 2750, 2835, 3213)$ = $C_{6750}(M_{9})$,

\hspace{.5cm}  $C_{6750}(81, 500, 750, 1269, 1431, 1750, 2619, 2750, 2781, 3105)$ = $C_{6750}(M_{10})$,

\hspace{.5cm}  $C_{6750}(135, 297, 1000, 1053, 1250, 1500, 1647, 2403, 2997, 3250)$ = $C_{6750}(M_{11})$,

\hspace{.5cm}  $C_{6750}(405, 459, 891, 1000, 1250, 1500, 1809, 2241, 3159, 3250)$ = $C_{6750}(M_{12})$,

\hspace{.5cm} $C_{6750}(621, 729, 945, 1000, 1250, 1500, 1971, 2079, 3250, 3321)$ = $C_{6750}(M_{13})$,

\hspace{.5cm}  $C_{6750}(27, 1000, 1215, 1250, 1323, 1377, 1500, 2673, 2727, 3250)$ = $C_{6750}(M_{14})$,

\hspace{.5cm} $C_{6750}(567, 783, 1000, 1250, 1485, 1500, 1917, 2133, 3250, 3267)$ = $C_{6750}(M_{15})$,

\hspace{.5cm}  $C_{6750}(189, 1000, 1161, 1250, 1500, 1539, 1755, 2511, 2889, 3250)$ = $C_{6750}(M_{16})$, 

\hspace{.5cm} $C_{6750}(351, 999, 1000, 1250, 1500, 1701, 2295, 2349, 3051, 3250)$ = $C_{6750}(M_{17})$,

\hspace{.5cm}  $C_{6750}(243, 1000, 1107, 1250, 1500, 1593, 2457, 2565, 2943, 3250)$ = $C_{6750}(M_{18})$, 

\hspace{.5cm} $C_{6750}(513, 837, 1000, 1250, 1500, 1863, 2187, 2835, 3213, 3250)$ = $C_{6750}(M_{19})$,

\hspace{.5cm}  $C_{6750}(81, 1000, 1250, 1269, 1431, 1500, 2619, 2781, 3105, 3250)$ = $C_{6750}(M_{20})$,

\hspace{.5cm}  $C_{6750}(135, 250, 297, 1053, 1647, 2000, 2403, 2500, 2997, 3000)$ = $C_{6750}(M_{21})$,

\hspace{.5cm}  $C_{6750}(250, 405, 459, 891, 1809, 2000, 2241, 2500, 3000, 3159)$ = $C_{6750}(M_{22})$,

\hspace{.5cm} $C_{6750}(250, 621, 729, 945, 1971, 2000, 2079, 2500, 3000, 3321)$ = $C_{6750}(M_{23})$,

\hspace{.5cm}  $C_{6750}(27, 250, 1215, 1323, 1377, 2000, 2500, 2673, 2727, 3000)$ = $C_{6750}(M_{24})$,

\hspace{.5cm} $C_{6750}(250, 567, 783, 1485, 1917, 2000, 2133, 2500, 3000, 3267)$ = $C_{6750}(M_{25})$,

\hspace{.5cm}  $C_{6750}(189, 250, 1161, 1539, 1755, 2000, 2500, 2511, 2889, 3000)$ = $C_{6750}(M_{26})$, 

\hspace{.5cm} $C_{6750}(250, 351, 999, 1701, 2000, 2295, 2349, 2500, 3000, 3051)$ = $C_{6750}(M_{27})$,

\hspace{.5cm}  $C_{6750}(243, 250, 1107, 1593, 2000, 2457, 2500, 2565, 2943, 3000)$ = $C_{6750}(M_{28})$, 

\hspace{.5cm} $C_{6750}(250, 513, 837, 1863, 2000, 2187, 2500, 2835, 3000, 3213)$ = $C_{6750}(M_{29})$,

\hspace{.5cm}  $C_{6750}(81, 250, 1269, 1431, 2000, 2500, 2619, 2781, 3000, 3105)$ = $C_{6750}(M_{30})\}$.\\

\item [\rm (1c2d)]  $T1_{6750}(C_{6750}(U_3))$ = $T1_{6750}(C_{6750}(N_1))$ = $\{C_{6750}(N_i): N_1 = U_3 ~\text{and}~i = 1~\text{to}~30\}$ 

= $\{C_{6750}(135, 500, 567, 750, 783, 1750, 1917, 2133, 2750, 3267)$ = $C_{6750}(N_{1})$ = $C_{6750}(U_{3})$,

\hspace{.5cm}  $C_{6750}(351, 405, 500, 750, 999, 1701, 1750, 2349, 2750, 3051)$ = $C_{6750}(N_{2})$,

\hspace{.5cm} $C_{6750}(81, 500, 750, 945, 1269, 1431, 1750, 2619, 2750, 2781)$ = $C_{6750}(N_{3})$,

\hspace{.5cm}  $C_{6750}(297, 500, 750, 1053, 1215, 1647, 1750, 2403, 2750, 2997)$ = $C_{6750}(N_{4})$,

\hspace{.5cm} $C_{6750}(500, 513, 750, 837, 1485, 1750, 1863, 2187, 2750, 3213)$ = $C_{6750}(N_{5})$,

\hspace{.5cm}  $C_{6750}(500, 621, 729, 750, 1750, 1755, 1971, 2079, 2750, 3321)$ = $C_{6750}(N_{6})$, 

\hspace{.5cm} $C_{6750}(189, 500, 750, 1161, 1539, 1750, 2295, 2511, 2750, 2889)$ = $C_{6750}(N_{7})$,

\hspace{.5cm}  $C_{6750}(27, 500, 750, 1323, 1377, 1750, 2565, 2673, 2727, 2750)$ = $C_{6750}(N_{8})$, 

\hspace{.5cm} $C_{6750}(243, 500, 750, 1107, 1593, 1750, 2457, 2750, 2835, 2943)$ = $C_{6750}(N_{9})$,

\hspace{.5cm}  $C_{6750}(459, 500, 750, 891, 1750, 1809, 2241, 2750, 3105, 3159)$ = $C_{6750}(N_{10})$,

\hspace{.5cm}  $C_{6750}(135, 567, 783, 1000, 1250, 1500, 1917, 2133, 3250, 3267)$ = $C_{6750}(N_{11})$,

\hspace{.5cm}  $C_{6750}(351, 405, 999, 1000, 1250, 1500, 1701, 2349, 3051, 3250)$ = $C_{6750}(N_{12})$,

\hspace{.5cm} $C_{6750}(81, 945, 1000, 1250, 1269, 1431, 1500, 2619, 2781, 3250)$ = $C_{6750}(N_{13})$,

\hspace{.5cm}  $C_{6750}(297, 1000, 1053, 1215, 1250, 1500, 1647, 2403, 2997, 3250)$ = $C_{6750}(N_{14})$,

\hspace{.5cm} $C_{6750}(513, 837, 1000, 1250, 1485, 1500, 1863, 2187, 3213, 3250)$ = $C_{6750}(N_{15})$,

\hspace{.5cm}  $C_{6750}(621, 729, 1000, 1250, 1500, 1755, 1971, 2079, 3250, 3321)$ = $C_{6750}(N_{16})$, 

\hspace{.5cm} $C_{6750}(189, 1000, 1161, 1250, 1500, 1539, 2295, 2511, 2889, 3250)$ = $C_{6750}(N_{17})$,

\hspace{.5cm}  $C_{6750}(27, 1000, 1250, 1323, 1377, 1500, 2565, 2673, 2727, 3250)$ = $C_{6750}(N_{18})$, 

\hspace{.5cm} $C_{6750}(243, 1000, 1107, 1250, 1500, 1593, 2457, 2835, 2943, 3250)$ = $C_{6750}(N_{19})$,

\hspace{.5cm}  $C_{6750}(459, 891, 1000, 1250, 1500, 1809, 2241, 3105, 3159, 3250)$ = $C_{6750}(N_{20})$,

\hspace{.5cm}  $C_{6750}(135, 250, 567, 783, 1917, 2000, 2133, 2500, 3000, 3267)$ = $C_{6750}(N_{21})$,

\hspace{.5cm}  $C_{6750}(250, 351, 405, 999, 1701, 2000, 2349, 2500, 3000, 3051)$ = $C_{6750}(N_{22})$,

\hspace{.5cm} $C_{6750}(81, 250, 945, 1269, 1431, 2000, 2500, 2619, 2781, 3000)$ = $C_{6750}(N_{23})$,

\hspace{.5cm}  $C_{6750}(250, 297, 1053, 1215, 1647, 2000, 2403, 2500, 2997, 3000)$ = $C_{6750}(N_{24})$,

\hspace{.5cm} $C_{6750}(250, 513, 837, 1485, 1863, 2000, 2187, 2500, 3000, 3213)$ = $C_{6750}(N_{25})$,

\hspace{.5cm}  $C_{6750}(250, 621, 729, 1755, 1971, 2000, 2079, 2500, 3000, 3321)$ = $C_{6750}(N_{26})$, 

\hspace{.5cm} $C_{6750}(189, 250, 1161, 1539, 2000, 2295, 2500, 2511, 2889, 3000)$ = $C_{6750}(N_{27})$,

\hspace{.5cm}  $C_{6750}(27, 250, 1323, 1377, 2000, 2500, 2565, 2673, 2727, 3000)$ = $C_{6750}(N_{28})$, 

\hspace{.5cm} $C_{6750}(243, 250, 1107, 1593, 2000, 2457, 2500, 2835, 2943, 3000)$ = $C_{6750}(N_{29})$,

\hspace{.5cm}  $C_{6750}(250, 459, 891, 1809, 2000, 2241, 2500, 3000, 3105, 3159)$  = $C_{6750}(N_{30})\}$.\\

\item [\rm (1c2e)]  $T1_{6750}(C_{6750}(V_3))$ = $T1_{6750}(C_{6750}(O_1))$ = $\{C_{6750}(O_i): O_1 = V_3 ~\text{and}~i = 1~\text{to}~30\}$ 

= $\{C_{6750}(135, 500, 513, 750, 837, 1750, 1863, 2187, 2750, 3213)$ = $C_{6750}(O_1)$ = $C_{6750}(V_3)$,

\hspace{.5cm}  $C_{6750}(189,  405, 500, 750, 1161, 1539, 1750, 2511, 2750, 2889)$ = $C_{6750}(O_{2})$,

\hspace{.5cm} $C_{6750}(459, 500, 750, 891, 945, 1750, 1809, 2241, 2750, 3159)$ = $C_{6750}(O_{3})$,

\hspace{.5cm}  $C_{6750}(500, 567, 750, 783, 1215, 1750, 1917, 2133, 2750, 3267)$ = $C_{6750}(O_{4})$,

\hspace{.5cm} $C_{6750}(243, 500, 750, 1107, 1485, 1593, 1750, 2457, 2750, 2943)$ = $C_{6750}(O_{5})$,

\hspace{.5cm}  $C_{6750}(81, 500, 750, 1269, 1431, 1750, 1755, 2619, 2750, 2781)$ = $C_{6750}(O_{6})$, 

\hspace{.5cm} $C_{6750}(500, 621, 729,  750, 1750, 1971, 2079, 2295, 2750, 3321)$ = $C_{6750}(O_{7})$,

\hspace{.5cm}  $C_{6750}(297, 500, 750, 1053, 1647, 1750, 2403, 2565, 2750, 2997)$ = $C_{6750}(O_{8})$, 

\hspace{.5cm} $C_{6750}(27, 500, 750, 1323, 1377, 1750, 2673, 2727, 2750, 2835)$ = $C_{6750}(O_{9})$,

\hspace{.5cm}  $C_{6750}(351, 500, 750, 999, 1701, 1750, 2349, 2750, 3051, 3105)$ = $C_{6750}(O_{10})$,

\hspace{.5cm}  $C_{6750}(135, 513, 837, 1000, 1250, 1500, 1863, 2187, 3213, 3250)$ = $C_{6750}(O_{11})$,

\hspace{.5cm}  $C_{6750}(189,  405, 1000, 1161, 1250, 1500, 1539, 2511, 2889, 3250)$ = $C_{6750}(O_{12})$,

\hspace{.5cm} $C_{6750}(459, 891, 945, 1000, 1250, 1500, 1809, 2241, 3159, 3250)$ = $C_{6750}(O_{13})$,

\hspace{.5cm}  $C_{6750}(567, 783, 1000, 1215, 1250, 1500, 1917, 2133, 3250, 3267)$ = $C_{6750}(O_{14})$,

\hspace{.5cm} $C_{6750}(243, 1000, 1107,  1250, 1485, 1500, 1593, 2457, 2943, 3250)$ = $C_{6750}(O_{15})$,

\hspace{.5cm}  $C_{6750}(81, 1000, 1250, 1269, 1431, 1500, 1755, 2619, 2781, 3250)$ = $C_{6750}(O_{16})$, 

\hspace{.5cm} $C_{6750}(621, 729,  1000, 1250, 1500, 1971, 2079, 2295, 3250, 3321)$ = $C_{6750}(O_{17})$,

\hspace{.5cm}  $C_{6750}(297, 1000, 1053, 1250, 1500, 1647, 2403, 2565, 2997, 3250)$ = $C_{6750}(O_{18})$, 

\hspace{.5cm} $C_{6750}(27, 1000, 1250, 1323, 1377, 1500, 2673, 2727, 2835, 3250)$ = $C_{6750}(O_{19})$,

\hspace{.5cm}  $C_{6750}(351, 999, 1000, 1250, 1500, 1701, 2349, 3051, 3105, 3250)$ = $C_{6750}(O_{20})$,

\hspace{.5cm}  $C_{6750}(135, 250, 513, 837, 1863, 2000, 2187, 2500, 3000, 3213)$ = $C_{6750}(O_{21})$,

\hspace{.5cm}  $C_{6750}(189,  250, 405, 1161, 1539, 2000, 2500, 2511, 2889, 3000)$ = $C_{6750}(O_{22})$,

\hspace{.5cm} $C_{6750}(250, 459, 891, 945, 1809, 2000, 2241, 2500, 3000, 3159)$ = $C_{6750}(O_{23})$,

\hspace{.5cm}  $C_{6750}(250, 567, 783, 1215, 1917, 2000, 2133, 2500, 3000, 3267)$ = $C_{6750}(O_{24})$,

\hspace{.5cm} $C_{6750}(243, 250, 1107,  1485, 1593, 2000, 2457, 2500, 2943, 3000)$ = $C_{6750}(O_{25})$,

\hspace{.5cm}  $C_{6750}(81, 250, 1269, 1431, 1755, 2000, 2500, 2619, 2781, 3000)$ = $C_{6750}(O_{26})$, 

\hspace{.5cm} $C_{6750}(250, 621, 729,  1971, 2000, 2079, 2295, 2500, 3000, 3321)$ = $C_{6750}(O_{27})$,

\hspace{.5cm}  $C_{6750}(250, 297, 1053, 1647, 2000, 2403, 2500, 2565, 2997, 3000)$ = $C_{6750}(O_{28})$, 

\hspace{.5cm} $C_{6750}(27, 250, 1323, 1377, 2000, 2500, 2673, 2727, 2835, 3000)$ = $C_{6750}(O_{29})$,

\hspace{.5cm}  $C_{6750}(250, 351, 999, 1701, 2000, 2349, 2500, 3000, 3051, 3105) =  C_{6750}(O_{30})\}$.  \hfill $\Box$
		\end{enumerate} 
	}
\end{prm}

\begin{prm} {\rm \cite{v2-7}} \quad \label{p12} {\rm Let $`\circ'$ be as given in definition \ref{d9}. For $i$ = 1 to 3 and $j$ = 1 to 30, $R_i, S_i, T_i$, $A_j, B_j, \dots, O_j$ be as given above with 
		
		$A_1$ = $R_1$, $B_1$ = $S_1$, $C_1$ = $T_1$, $D_1$ = $U_1$, $E_1$ = $V_1$, 
		
		$F_1$ = $R_2$, $G_1$ = $S_2$, $H_1$ = $T_2$, $I_1$ = $U_2$, $J_1$ = $V_2$, 
		
		$K_1$ = $R_3$, $L_1$ = $S_3$, $M_1$ = $T_3$, $N_1$ = $U_3$, $O_1$ = $V_3$. Then, the following statements are true.
		
\begin{enumerate}
			
	\item [\rm (a)]  Circulant graphs occuring in each of the following cases are either Type-2 isomorphic w.r.t. $m$ = 3 or Type-2 isomorphic w.r.t. $m$ = 5 as given below. 	
			
	\item [\rm (1A1)]	 $C_{6750}(A_1)$ $\cong_{T2_{6750,3,500}}$ $C_{6750}(K_1)$; 
			
	\item [\rm (2A1)]	 $C_{6750}(A_1)$ $\cong_{T2_{6750,3,1000}}$ $C_{6750}(F_1)$,
			
	$C_{6750}(A_1)$ $\cong_{T2_{6750,3,250}}$ $C_{6750}(F_1)$, $C_{6750}(A_1)$ $\cong_{T2_{6750,3,750}}$ $C_{6750}(A_1)$; 
			
	\item [\rm (3A1)] $C_{6750}(A_1)$ $\cong_{T2_{6750,5,54}}$ $C_{6750}(C_1)$; 
			
	\item [\rm (4A1)] $C_{6750}(A_1)$ $\cong_{T2_{6750,5,108}}$ $C_{6750}(E_1)$;
			
	\item [\rm (5A1)] $C_{6750}(A_1)$ $\cong_{T2_{6750,5,162}}$ $C_{6750}(B_1)$;
			
	\item [\rm (6A1)] $C_{6750}(A_1)$ $\cong_{T2_{6750,5,216}}$ $C_{6750}(D_1)$;
			
	\item [\rm (7A1)] $C_{6750}(A_1)$ $\cong_{T2_{6750,5,270}}$ $C_{6750}(A_1)$;\\
			
	\item [\rm (1B1)]	 $C_{6750}(B_1)$ $\cong_{T2_{6750,3,500}}$ $C_{6750}(L_1)$; 
			
	\item [\rm (2B1)]	 $C_{6750}(B_1)$ $\cong_{T2_{6750,3,1000}}$ $C_{6750}(G_1)$,
			
	$C_{6750}(B_1)$ $\cong_{T2_{6750,3,250}}$ $C_{6750}(G_1)$, $C_{6750}(B_1)$ $\cong_{T2_{6750,3,750}}$ $C_{6750}(B_1)$; 
			
	\item [\rm (3B1)] $C_{6750}(B_1)$ $\cong_{T2_{6750,5,54}}$ $C_{6750}(D_1)$; 
			
	\item [\rm (4B1)] $C_{6750}(B_1)$ $\cong_{T2_{6750,5,108}}$ $C_{6750}(A_1)$;
			
	\item [\rm (5B1)] $C_{6750}(B_1)$ $\cong_{T2_{6750,5,162}}$ $C_{6750}(C_1)$;
			
	\item [\rm (6B1)] $C_{6750}(B_1)$ $\cong_{T2_{6750,5,216}}$ $C_{6750}(E_1)$;
			
	\item [\rm (7B1)] $C_{6750}(B_1)$ $\cong_{T2_{6750,5,270}}$ $C_{6750}(B_1)$;\\
			
	\item [\rm (1C1)]	 $C_{6750}(C_1)$ $\cong_{T2_{6750,3,500}}$ $C_{6750}(M_1)$; 
			
	\item [\rm (2C1)]	 $C_{6750}(C_1)$ $\cong_{T2_{6750,3,1000}}$ $C_{6750}(H_1)$,
			
	$C_{6750}(C_1)$ $\cong_{T2_{6750,3,250}}$ $C_{6750}(H_1)$, $C_{6750}(C_1)$ $\cong_{T2_{6750,3,750}}$ $C_{6750}(C_1)$; 
			
	\item [\rm (3C1)] $C_{6750}(C_1)$ $\cong_{T2_{6750,5,54}}$ $C_{6750}(E_1)$; 
			
	\item [\rm (4C1)] $C_{6750}(C_1)$ $\cong_{T2_{6750,5,108}}$ $C_{6750}(B_1)$;
			
	\item [\rm (5C1)] $C_{6750}(C_1)$ $\cong_{T2_{6750,5,162}}$ $C_{6750}(D_1)$;
			
	\item [\rm (6C1)] $C_{6750}(C_1)$ $\cong_{T2_{6750,5,216}}$ $C_{6750}(A_1)$;
			
	\item [\rm (7C1)] $C_{6750}(C_1)$ $\cong_{T2_{6750,5,270}}$ $C_{6750}(C_1)$;\\
			
			\item [\rm (1D1)]	 $C_{6750}(D_1)$ $\cong_{T2_{6750,3,500}}$ $C_{6750}(N_1)$; 
			
			\item [\rm (2D1)]	 $C_{6750}(D_1)$ $\cong_{T2_{6750,3,1000}}$ $C_{6750}(I_1)$,
			
			$C_{6750}(D_1)$ $\cong_{T2_{6750,3,250}}$ $C_{6750}(I_1)$, $C_{6750}(D_1)$ $\cong_{T2_{6750,3,750}}$ $C_{6750}(D_1)$; 
			
			\item [\rm (3D1)] $C_{6750}(D_1)$ $\cong_{T2_{6750,5,54}}$ $C_{6750}(A_1)$; 
			
			\item [\rm (4D1)] $C_{6750}(D_1)$ $\cong_{T2_{6750,5,108}}$ $C_{6750}(C_1)$;
			
			\item [\rm (5D1)] $C_{6750}(D_1)$ $\cong_{T2_{6750,5,162}}$ $C_{6750}(E_1)$;
			
			\item [\rm (6D1)] $C_{6750}(D_1)$ $\cong_{T2_{6750,5,216}}$ $C_{6750}(B_1)$;
			
			\item [\rm (7D1)] $C_{6750}(D_1)$ $\cong_{T2_{6750,5,270}}$ $C_{6750}(D_1)$;\\
			
			\item [\rm (1E1)]	 $C_{6750}(E_1)$ $\cong_{T2_{6750,3,500}}$ $C_{6750}(O_1)$; 
			
			\item [\rm (2E1)]	 $C_{6750}(E_1)$ $\cong_{T2_{6750,3,1000}}$ $C_{6750}(J_1)$,
			
			$C_{6750}(E_1)$ $\cong_{T2_{6750,3,250}}$ $C_{6750}(J_1)$, $C_{6750}(E_1)$ $\cong_{T2_{6750,3,750}}$ $C_{6750}(E_1)$; 
			
			\item [\rm (3E1)] $C_{6750}(E_1)$ $\cong_{T2_{6750,5,54}}$ $C_{6750}(B_1)$; 
			
			\item [\rm (4E1)] $C_{6750}(E_1)$ $\cong_{T2_{6750,5,108}}$ $C_{6750}(D_1)$;
			
			\item [\rm (5E1)] $C_{6750}(E_1)$ $\cong_{T2_{6750,5,162}}$ $C_{6750}(A_1)$;
			
			\item [\rm (6E1)] $C_{6750}(E_1)$ $\cong_{T2_{6750,5,216}}$ $C_{6750}(C_1)$;  
			
			\item [\rm (7E1)] $C_{6750}(E_1)$ $\cong_{T2_{6750,5,270}}$ $C_{6750}(E_1)$;\\  
			
			\item [\rm (1F1)]	 $C_{6750}(F_1)$ $\cong_{T2_{6750,3,500}}$ $C_{6750}(A_1)$; 
			
			\item [\rm (2F1)]	 $C_{6750}(F_1)$ $\cong_{T2_{6750,3,1000}}$ $C_{6750}(K_1)$,
			
			$C_{6750}(F_1)$ $\cong_{T2_{6750,3,250}}$ $C_{6750}(K_1)$, $C_{6750}(F_1)$ $\cong_{T2_{6750,3,750}}$ $C_{6750}(F_1)$; 
			
			\item [\rm (3F1)] $C_{6750}(F_1)$ $\cong_{T2_{6750,5,54}}$ $C_{6750}(H_1)$; 
			
			\item [\rm (4F1)] $C_{6750}(F_1)$ $\cong_{T2_{6750,5,108}}$ $C_{6750}(J_1)$;
			
			\item [\rm (5F1)] $C_{6750}(F_1)$ $\cong_{T2_{6750,5,162}}$ $C_{6750}(G_1)$;
			
			\item [\rm (6F1)] $C_{6750}(F_1)$ $\cong_{T2_{6750,5,216}}$ $C_{6750}(I_1)$;
			
			\item [\rm (7F1)] $C_{6750}(F_1)$ $\cong_{T2_{6750,5,270}}$ $C_{6750}(F_1)$;\\
			
			\item [\rm (1G1)]	 $C_{6750}(G_1)$ $\cong_{T2_{6750,3,500}}$ $C_{6750}(B_1)$; 
			
			\item [\rm (2G1)]	 $C_{6750}(G_1)$ $\cong_{T2_{6750,3,1000}}$ $C_{6750}(L_1)$,
			
			$C_{6750}(G_1)$ $\cong_{T2_{6750,3,250}}$ $C_{6750}(L_1)$, $C_{6750}(G_1)$ $\cong_{T2_{6750,3,750}}$ $C_{6750}(G_1)$; 
			
			\item [\rm (3G1)] $C_{6750}(G_1)$ $\cong_{T2_{6750,5,54}}$ $C_{6750}(I_1)$; 
			
			\item [\rm (4G1)] $C_{6750}(G_1)$ $\cong_{T2_{6750,5,108}}$ $C_{6750}(F_1)$;
			
			\item [\rm (5G1)] $C_{6750}(G_1)$ $\cong_{T2_{6750,5,162}}$ $C_{6750}(H_1)$;
			
			\item [\rm (6G1)] $C_{6750}(G_1)$ $\cong_{T2_{6750,5,216}}$ $C_{6750}(J_1)$;
			
			\item [\rm (7G1)] $C_{6750}(G_1)$ $\cong_{T2_{6750,5,270}}$ $C_{6750}(G_1)$;\\
			
			\item [\rm (1H1)]	 $C_{6750}(H_1)$ $\cong_{T2_{6750,3,500}}$ $C_{6750}(C_1)$; 
			
			\item [\rm (2H1)]	 $C_{6750}(H_1)$ $\cong_{T2_{6750,3,1000}}$ $C_{6750}(M_1)$,
			
			$C_{6750}(H_1)$ $\cong_{T2_{6750,3,250}}$ $C_{6750}(M_1)$, $C_{6750}(H_1)$ $\cong_{T2_{6750,3,750}}$ $C_{6750}(H_1)$; 
			
			\item [\rm (3H1)] $C_{6750}(H_1)$ $\cong_{T2_{6750,5,54}}$ $C_{6750}(J_1)$; 
			
			\item [\rm (4H1)] $C_{6750}(H_1)$ $\cong_{T2_{6750,5,108}}$ $C_{6750}(G_1)$;
			
			\item [\rm (5H1)] $C_{6750}(H_1)$ $\cong_{T2_{6750,5,162}}$ $C_{6750}(I_1)$;
			
			\item [\rm (6H1)] $C_{6750}(H_1)$ $\cong_{T2_{6750,5,216}}$ $C_{6750}(F_1)$;
			
			\item [\rm (7H1)] $C_{6750}(H_1)$ $\cong_{T2_{6750,5,270}}$ $C_{6750}(H_1)$;\\
			
			\item [\rm (1I1)]	 $C_{6750}(I_1)$ $\cong_{T2_{6750,3,500}}$ $C_{6750}(D_1)$; 
			
			\item [\rm (2I1)]	 $C_{6750}(I_1)$ $\cong_{T2_{6750,3,1000}}$ $C_{6750}(N_1)$,
			
			$C_{6750}(I_1)$ $\cong_{T2_{6750,3,250}}$ $C_{6750}(N_1)$, $C_{6750}(I_1)$ $\cong_{T2_{6750,3,750}}$ $C_{6750}(I_1)$; 
			
			\item [\rm (3I1)] $C_{6750}(I_1)$ $\cong_{T2_{6750,5,54}}$ $C_{6750}(F_1)$; 
			
			\item [\rm (4I1)] $C_{6750}(I_1)$ $\cong_{T2_{6750,5,108}}$ $C_{6750}(H_1)$;
			
			\item [\rm (5I1)] $C_{6750}(I_1)$ $\cong_{T2_{6750,5,162}}$ $C_{6750}(J_1)$;
			
			\item [\rm (6I1)] $C_{6750}(I_1)$ $\cong_{T2_{6750,5,216}}$ $C_{6750}(G_1)$;
			
			\item [\rm (7I1)] $C_{6750}(I_1)$ $\cong_{T2_{6750,5,270}}$ $C_{6750}(I_1)$;\\
			
			\item [\rm (1J1)] $C_{6750}(J_1)$ $\cong_{T2_{6750,3,500}}$ $C_{6750}(E_1)$; 
			
			\item [\rm (2J1)] $C_{6750}(J_1)$ $\cong_{T2_{6750,3,1000}}$ $C_{6750}(O_1)$,
			
			$C_{6750}(J_1)$ $\cong_{T2_{6750,3,250}}$ $C_{6750}(O_1)$, $C_{6750}(J_1)$ $\cong_{T2_{6750,3,750}}$ $C_{6750}(J_1)$; 
			
			\item [\rm (3J1)] $C_{6750}(J_1)$ $\cong_{T2_{6750,5,54}}$ $C_{6750}(G_1)$; 
			
			\item [\rm (4J1)] $C_{6750}(J_1)$ $\cong_{T2_{6750,5,108}}$ $C_{6750}(I_1)$;
			
			\item [\rm (5J1)] $C_{6750}(J_1)$ $\cong_{T2_{6750,5,162}}$ $C_{6750}(F_1)$;
			
			\item [\rm (6J1)] $C_{6750}(J_1)$ $\cong_{T2_{6750,5,216}}$ $C_{6750}(H_1)$;  
			
			\item [\rm (7J1)] $C_{6750}(J_1)$ $\cong_{T2_{6750,5,270}}$ $C_{6750}(J_1)$;\\  
			
			\item [\rm (1K1)]	 $C_{6750}(K_1)$ $\cong_{T2_{6750,3,500}}$ $C_{6750}(F_1)$; 
			
			\item [\rm (2K1)]	 $C_{6750}(K_1)$ $\cong_{T2_{6750,3,1000}}$ $C_{6750}(A_1)$,
			
			$C_{6750}(K_1)$ $\cong_{T2_{6750,3,250}}$ $C_{6750}(A_1)$, $C_{6750}(K_1)$ $\cong_{T2_{6750,3,750}}$ $C_{6750}(K_1)$; 
			
			\item [\rm (3K1)] $C_{6750}(K_1)$ $\cong_{T2_{6750,5,54}}$ $C_{6750}(M_1)$; 
			
			\item [\rm (4K1)] $C_{6750}(K_1)$ $\cong_{T2_{6750,5,108}}$ $C_{6750}(O_1)$;
			
			\item [\rm (5K1)] $C_{6750}(K_1)$ $\cong_{T2_{6750,5,162}}$ $C_{6750}(L_1)$;
			
			\item [\rm (6K1)] $C_{6750}(K_1)$ $\cong_{T2_{6750,5,216}}$ $C_{6750}(N_1)$;
			
			\item [\rm (7K1)] $C_{6750}(K_1)$ $\cong_{T2_{6750,5,270}}$ $C_{6750}(K_1)$;\\
			
			\item [\rm (1L1)]	 $C_{6750}(L_1)$ $\cong_{T2_{6750,3,500}}$ $C_{6750}(G_1)$; 
			
			\item [\rm (2L1)]	 $C_{6750}(L_1)$ $\cong_{T2_{6750,3,1000}}$ $C_{6750}(B_1)$,
			
			$C_{6750}(L_1)$ $\cong_{T2_{6750,3,250}}$ $C_{6750}(B_1)$, $C_{6750}(L_1)$ $\cong_{T2_{6750,3,750}}$ $C_{6750}(L_1)$; 
			
			\item [\rm (3L1)] $C_{6750}(L_1)$ $\cong_{T2_{6750,5,54}}$ $C_{6750}(N_1)$; 
			
			\item [\rm (4L1)] $C_{6750}(L_1)$ $\cong_{T2_{6750,5,108}}$ $C_{6750}(K_1)$;
			
			\item [\rm (5L1)] $C_{6750}(L_1)$ $\cong_{T2_{6750,5,162}}$ $C_{6750}(M_1)$;
			
			\item [\rm (6L1)] $C_{6750}(L_1)$ $\cong_{T2_{6750,5,216}}$ $C_{6750}(O_1)$;
			
			\item [\rm (7L1)] $C_{6750}(L_1)$ $\cong_{T2_{6750,5,270}}$ $C_{6750}(L_1)$;\\
			
			\item [\rm (1M1)]	 $C_{6750}(M_1)$ $\cong_{T2_{6750,3,500}}$ $C_{6750}(H_1)$; 
			
			\item [\rm (2M1)]	 $C_{6750}(M_1)$ $\cong_{T2_{6750,3,1000}}$ $C_{6750}(C_1)$,
			
			$C_{6750}(M_1)$ $\cong_{T2_{6750,3,250}}$ $C_{6750}(C_1)$, $C_{6750}(M_1)$ $\cong_{T2_{6750,3,750}}$ $C_{6750}(M_1)$; 
			
			\item [\rm (3M1)] $C_{6750}(M_1)$ $\cong_{T2_{6750,5,54}}$ $C_{6750}(O_1)$; 
			
			\item [\rm (4M1)] $C_{6750}(M_1)$ $\cong_{T2_{6750,5,108}}$ $C_{6750}(L_1)$;
			
			\item [\rm (5M1)] $C_{6750}(M_1)$ $\cong_{T2_{6750,5,162}}$ $C_{6750}(N_1)$;
			
			\item [\rm (6M1)] $C_{6750}(M_1)$ $\cong_{T2_{6750,5,216}}$ $C_{6750}(K_1)$;
			
			\item [\rm (7M1)] $C_{6750}(M_1)$ $\cong_{T2_{6750,5,270}}$ $C_{6750}(M_1)$;\\
			
			\item [\rm (1N1)] $C_{6750}(N_1)$ $\cong_{T2_{6750,3,500}}$ $C_{6750}(I_1)$; 
			
			\item [\rm (2N1)] $C_{6750}(N_1)$ $\cong_{T2_{6750,3,1000}}$ $C_{6750}(D_1)$,
			
			$C_{6750}(N_1)$ $\cong_{T2_{6750,3,250}}$ $C_{6750}(D_1)$, $C_{6750}(N_1)$ $\cong_{T2_{6750,3,750}}$ $C_{6750}(N_1)$; 
			
			\item [\rm (3N1)] $C_{6750}(N_1)$ $\cong_{T2_{6750,5,54}}$ $C_{6750}(K_1)$; 
			
			\item [\rm (4N1)] $C_{6750}(N_1)$ $\cong_{T2_{6750,5,108}}$ $C_{6750}(M_1)$;
			
			\item [\rm (5N1)] $C_{6750}(N_1)$ $\cong_{T2_{6750,5,162}}$ $C_{6750}(O_1)$;
			
			\item [\rm (6N1)] $C_{6750}(N_1)$ $\cong_{T2_{6750,5,216}}$ $C_{6750}(L_1)$;
			
			\item [\rm (7N1)] $C_{6750}(N_1)$ $\cong_{T2_{6750,5,270}}$ $C_{6750}(N_1)$;\\
			
			\item [\rm (1O1)] $C_{6750}(O_1)$ $\cong_{T2_{6750,3,500}}$ $C_{6750}(J_1)$; 
			
			\item [\rm (2O1)] $C_{6750}(O_1)$ $\cong_{T2_{6750,3,1000}}$ $C_{6750}(E_1)$,
			
			$C_{6750}(O_1)$ $\cong_{T2_{6750,3,250}}$ $C_{6750}(E_1)$, $C_{6750}(O_1)$ $\cong_{T2_{6750,3,750}}$ $C_{6750}(O_1)$; 
			
			\item [\rm (3O1)] $C_{6750}(O_1)$ $\cong_{T2_{6750,5,54}}$ $C_{6750}(L_1)$; 
			
			\item [\rm (4O1)] $C_{6750}(O_1)$ $\cong_{T2_{6750,5,108}}$ $C_{6750}(N_1)$;
			
			\item [\rm (5O1)] $C_{6750}(O_1)$ $\cong_{T2_{6750,5,162}}$ $C_{6750}(K_1)$;
			
			\item [\rm (6O1)] $C_{6750}(O_1)$ $\cong_{T2_{6750,5,216}}$ $C_{6750}(M_1)$;
			
			\item [\rm (7O1)] $C_{6750}(O_1)$ $\cong_{T2_{6750,5,270}}$ $C_{6750}(O_1)$;\\
			
			\item [\rm (1A2)]	 $C_{6750}(A_2)$ $\cong_{T2_{6750,3,500}}$ $C_{6750}(K_2)$; 
			
			\item [\rm (2A2)]	 $C_{6750}(A_2)$ $\cong_{T2_{6750,3,1000}}$ $C_{6750}(F_2)$,
			
			$C_{6750}(A_2)$ $\cong_{T2_{6750,3,250}}$ $C_{6750}(F_2)$, $C_{6750}(A_2)$ $\cong_{T2_{6750,3,750}}$ $C_{6750}(A_2)$; 
			
			\item [\rm (3A2)] $C_{6750}(A_2)$ $\cong_{T2_{6750,5,54}}$ $C_{6750}(C_2)$; 
			
			\item [\rm (4A2)] $C_{6750}(A_2)$ $\cong_{T2_{6750,5,108}}$ $C_{6750}(E_2)$;
			
			\item [\rm (5A2)] $C_{6750}(A_2)$ $\cong_{T2_{6750,5,162}}$ $C_{6750}(B_2)$;
			
			\item [\rm (6A2)] $C_{6750}(A_2)$ $\cong_{T2_{6750,5,216}}$ $C_{6750}(D_2)$;
			
			\item [\rm (7A2)] $C_{6750}(A_2)$ $\cong_{T2_{6750,5,270}}$ $C_{6750}(A_2)$;\\
			
			\item [\rm (1B2)]	 $C_{6750}(B_2)$ $\cong_{T2_{6750,3,500}}$ $C_{6750}(L_2)$; 
			
			\item [\rm (2B2)]	 $C_{6750}(B_2)$ $\cong_{T2_{6750,3,1000}}$ $C_{6750}(G_2)$,
			
			$C_{6750}(B_2)$ $\cong_{T2_{6750,3,250}}$ $C_{6750}(G_2)$, $C_{6750}(B_2)$ $\cong_{T2_{6750,3,750}}$ $C_{6750}(B_2)$; 
			
			\item [\rm (3B2)] $C_{6750}(B_2)$ $\cong_{T2_{6750,5,54}}$ $C_{6750}(D_2)$; 
			
			\item [\rm (4B2)] $C_{6750}(B_2)$ $\cong_{T2_{6750,5,108}}$ $C_{6750}(A_2)$;
			
			\item [\rm (5B2)] $C_{6750}(B_2)$ $\cong_{T2_{6750,5,162}}$ $C_{6750}(C_2)$;
			
			\item [\rm (6B2)] $C_{6750}(B_2)$ $\cong_{T2_{6750,5,216}}$ $C_{6750}(E_2)$;
			
			\item [\rm (7B2)] $C_{6750}(B_2)$ $\cong_{T2_{6750,5,270}}$ $C_{6750}(B_2)$;\\
			
			\item [\rm (1C2)]	 $C_{6750}(C_2)$ $\cong_{T2_{6750,3,500}}$ $C_{6750}(M_2)$; 
			
			\item [\rm (2C2)]	 $C_{6750}(C_2)$ $\cong_{T2_{6750,3,1000}}$ $C_{6750}(H_2)$,
			
			$C_{6750}(C_2)$ $\cong_{T2_{6750,3,250}}$ $C_{6750}(H_2)$, $C_{6750}(C_2)$ $\cong_{T2_{6750,3,750}}$ $C_{6750}(C_2)$; 
			
			\item [\rm (3C2)] $C_{6750}(C_2)$ $\cong_{T2_{6750,5,54}}$ $C_{6750}(E_2)$; 
			
			\item [\rm (4C2)] $C_{6750}(C_2)$ $\cong_{T2_{6750,5,108}}$ $C_{6750}(B_2)$;
			
			\item [\rm (5C2)] $C_{6750}(C_2)$ $\cong_{T2_{6750,5,162}}$ $C_{6750}(D_2)$;
			
			\item [\rm (6C2)] $C_{6750}(C_2)$ $\cong_{T2_{6750,5,216}}$ $C_{6750}(A_2)$;
			
			\item [\rm (7C2)] $C_{6750}(C_2)$ $\cong_{T2_{6750,5,270}}$ $C_{6750}(C_2)$;\\
			
			\item [\rm (1D2)]	 $C_{6750}(D_2)$ $\cong_{T2_{6750,3,500}}$ $C_{6750}(N_2)$; 
			
			\item [\rm (2D2)]	 $C_{6750}(D_2)$ $\cong_{T2_{6750,3,1000}}$ $C_{6750}(I_2)$,
			
			$C_{6750}(D_2)$ $\cong_{T2_{6750,3,250}}$ $C_{6750}(I_2)$, $C_{6750}(D_2)$ $\cong_{T2_{6750,3,750}}$ $C_{6750}(D_2)$; 
			
			\item [\rm (3D2)] $C_{6750}(D_2)$ $\cong_{T2_{6750,5,54}}$ $C_{6750}(A_2)$; 
			
			\item [\rm (4D2)] $C_{6750}(D_2)$ $\cong_{T2_{6750,5,108}}$ $C_{6750}(C_2)$;
			
			\item [\rm (5D2)] $C_{6750}(D_2)$ $\cong_{T2_{6750,5,162}}$ $C_{6750}(E_2)$;
			
			\item [\rm (6D2)] $C_{6750}(D_2)$ $\cong_{T2_{6750,5,216}}$ $C_{6750}(B_2)$;
			
			\item [\rm (7D2)] $C_{6750}(D_2)$ $\cong_{T2_{6750,5,270}}$ $C_{6750}(D_2)$;\\
			
			\item [\rm (1E2)]	 $C_{6750}(E_2)$ $\cong_{T2_{6750,3,500}}$ $C_{6750}(O_2)$; 
			
			\item [\rm (2E2)]	 $C_{6750}(E_2)$ $\cong_{T2_{6750,3,1000}}$ $C_{6750}(J_2)$,
			
			$C_{6750}(E_2)$ $\cong_{T2_{6750,3,250}}$ $C_{6750}(J_2)$, $C_{6750}(E_2)$ $\cong_{T2_{6750,3,750}}$ $C_{6750}(E_2)$; 
			
			\item [\rm (3E2)] $C_{6750}(E_2)$ $\cong_{T2_{6750,5,54}}$ $C_{6750}(B_2)$; 
			
			\item [\rm (4E2)] $C_{6750}(E_2)$ $\cong_{T2_{6750,5,108}}$ $C_{6750}(D_2)$;
			
			\item [\rm (5E2)] $C_{6750}(E_2)$ $\cong_{T2_{6750,5,162}}$ $C_{6750}(A_2)$;
			
			\item [\rm (6E2)] $C_{6750}(E_2)$ $\cong_{T2_{6750,5,216}}$ $C_{6750}(C_2)$;  
			
			\item [\rm (7E2)] $C_{6750}(E_2)$ $\cong_{T2_{6750,5,270}}$ $C_{6750}(E_2)$;\\  
			
			\item [\rm (1F2)]	 $C_{6750}(F_2)$ $\cong_{T2_{6750,3,500}}$ $C_{6750}(A_2)$; 
			
			\item [\rm (2F2)]	 $C_{6750}(F_2)$ $\cong_{T2_{6750,3,1000}}$ $C_{6750}(K_2)$,
			
			$C_{6750}(F_2)$ $\cong_{T2_{6750,3,250}}$ $C_{6750}(K_2)$, $C_{6750}(F_2)$ $\cong_{T2_{6750,3,750}}$ $C_{6750}(F_2)$; 
			
			\item [\rm (3F2)] $C_{6750}(F_2)$ $\cong_{T2_{6750,5,54}}$ $C_{6750}(H_2)$; 
			
			\item [\rm (4F2)] $C_{6750}(F_2)$ $\cong_{T2_{6750,5,108}}$ $C_{6750}(J_2)$;
			
			\item [\rm (5F2)] $C_{6750}(F_2)$ $\cong_{T2_{6750,5,162}}$ $C_{6750}(G_2)$;
			
			\item [\rm (6F2)] $C_{6750}(F_2)$ $\cong_{T2_{6750,5,216}}$ $C_{6750}(I_2)$;
			
			\item [\rm (7F2)] $C_{6750}(F_2)$ $\cong_{T2_{6750,5,270}}$ $C_{6750}(F_2)$;\\
			
			\item [\rm (1G2)]	 $C_{6750}(G_2)$ $\cong_{T2_{6750,3,500}}$ $C_{6750}(B_2)$; 
			
			\item [\rm (2G2)]	 $C_{6750}(G_2)$ $\cong_{T2_{6750,3,1000}}$ $C_{6750}(L_2)$,
			
			$C_{6750}(G_2)$ $\cong_{T2_{6750,3,250}}$ $C_{6750}(L_2)$, $C_{6750}(G_2)$ $\cong_{T2_{6750,3,750}}$ $C_{6750}(G_2)$; 
			
			\item [\rm (3G2)] $C_{6750}(G_2)$ $\cong_{T2_{6750,5,54}}$ $C_{6750}(I_2)$; 
			
			\item [\rm (4G2)] $C_{6750}(G_2)$ $\cong_{T2_{6750,5,108}}$ $C_{6750}(F_2)$;
			
			\item [\rm (5G2)] $C_{6750}(G_2)$ $\cong_{T2_{6750,5,162}}$ $C_{6750}(H_2)$;
			
			\item [\rm (6G2)] $C_{6750}(G_2)$ $\cong_{T2_{6750,5,216}}$ $C_{6750}(J_2)$;
			
			\item [\rm (7G2)] $C_{6750}(G_2)$ $\cong_{T2_{6750,5,270}}$ $C_{6750}(G_2)$;\\
			
			\item [\rm (1H2)]	 $C_{6750}(H_2)$ $\cong_{T2_{6750,3,500}}$ $C_{6750}(C_2)$; 
			
			\item [\rm (2H2)]	 $C_{6750}(H_2)$ $\cong_{T2_{6750,3,1000}}$ $C_{6750}(M_2)$,
			
			$C_{6750}(H_2)$ $\cong_{T2_{6750,3,250}}$ $C_{6750}(M_2)$, $C_{6750}(H_2)$ $\cong_{T2_{6750,3,750}}$ $C_{6750}(H_2)$; 
			
			\item [\rm (3H2)] $C_{6750}(H_2)$ $\cong_{T2_{6750,5,54}}$ $C_{6750}(J_2)$; 
			
			\item [\rm (4H2)] $C_{6750}(H_2)$ $\cong_{T2_{6750,5,108}}$ $C_{6750}(G_2)$;
			
			\item [\rm (5H2)] $C_{6750}(H_2)$ $\cong_{T2_{6750,5,162}}$ $C_{6750}(I_2)$;
			
			\item [\rm (6H2)] $C_{6750}(H_2)$ $\cong_{T2_{6750,5,216}}$ $C_{6750}(F_2)$;
			
			\item [\rm (7H2)] $C_{6750}(H_2)$ $\cong_{T2_{6750,5,270}}$ $C_{6750}(H_2)$;\\
			
			\item [\rm (1I2)]	 $C_{6750}(I_2)$ $\cong_{T2_{6750,3,500}}$ $C_{6750}(D_2)$; 
			
			\item [\rm (2I2)]	 $C_{6750}(I_2)$ $\cong_{T2_{6750,3,1000}}$ $C_{6750}(N_2)$,
			
			$C_{6750}(I_2)$ $\cong_{T2_{6750,3,250}}$ $C_{6750}(N_2)$, $C_{6750}(I_2)$ $\cong_{T2_{6750,3,750}}$ $C_{6750}(I_2)$; 
			
			\item [\rm (3I2)] $C_{6750}(I_2)$ $\cong_{T2_{6750,5,54}}$ $C_{6750}(F_2)$; 
			
			\item [\rm (4I2)] $C_{6750}(I_2)$ $\cong_{T2_{6750,5,108}}$ $C_{6750}(H_2)$;
			
			\item [\rm (5I2)] $C_{6750}(I_2)$ $\cong_{T2_{6750,5,162}}$ $C_{6750}(J_2)$;
			
			\item [\rm (6I2)] $C_{6750}(I_2)$ $\cong_{T2_{6750,5,216}}$ $C_{6750}(G_2)$;
			
			\item [\rm (7I2)] $C_{6750}(I_2)$ $\cong_{T2_{6750,5,270}}$ $C_{6750}(I_2)$;\\
			
			\item [\rm (1J2)] $C_{6750}(J_2)$ $\cong_{T2_{6750,3,500}}$ $C_{6750}(E_2)$; 
			
			\item [\rm (2J2)] $C_{6750}(J_2)$ $\cong_{T2_{6750,3,1000}}$ $C_{6750}(O_2)$,
			
			$C_{6750}(J_2)$ $\cong_{T2_{6750,3,250}}$ $C_{6750}(O_2)$, $C_{6750}(J_2)$ $\cong_{T2_{6750,3,750}}$ $C_{6750}(J_2)$; 
			
			\item [\rm (3J2)] $C_{6750}(J_2)$ $\cong_{T2_{6750,5,54}}$ $C_{6750}(G_2)$; 
			
			\item [\rm (4J2)] $C_{6750}(J_2)$ $\cong_{T2_{6750,5,108}}$ $C_{6750}(I_2)$;
			
			\item [\rm (5J2)] $C_{6750}(J_2)$ $\cong_{T2_{6750,5,162}}$ $C_{6750}(F_2)$;
			
			\item [\rm (6J2)] $C_{6750}(J_2)$ $\cong_{T2_{6750,5,216}}$ $C_{6750}(H_2)$;  
			
			\item [\rm (7J2)] $C_{6750}(J_2)$ $\cong_{T2_{6750,5,270}}$ $C_{6750}(J_2)$;\\  
			
			\item [\rm (1K2)]	 $C_{6750}(K_2)$ $\cong_{T2_{6750,3,500}}$ $C_{6750}(F_2)$; 
			
			\item [\rm (2K2)]	 $C_{6750}(K_2)$ $\cong_{T2_{6750,3,1000}}$ $C_{6750}(A_2)$,
			
			$C_{6750}(K_2)$ $\cong_{T2_{6750,3,250}}$ $C_{6750}(A_2)$, $C_{6750}(K_2)$ $\cong_{T2_{6750,3,750}}$ $C_{6750}(K_2)$; 
			
			\item [\rm (3K2)] $C_{6750}(K_2)$ $\cong_{T2_{6750,5,54}}$ $C_{6750}(M_2)$; 
			
			\item [\rm (4K2)] $C_{6750}(K_2)$ $\cong_{T2_{6750,5,108}}$ $C_{6750}(O_2)$;
			
			\item [\rm (5K2)] $C_{6750}(K_2)$ $\cong_{T2_{6750,5,162}}$ $C_{6750}(L_2)$;
			
			\item [\rm (6K2)] $C_{6750}(K_2)$ $\cong_{T2_{6750,5,216}}$ $C_{6750}(N_2)$;
			
			\item [\rm (7K2)] $C_{6750}(K_2)$ $\cong_{T2_{6750,5,270}}$ $C_{6750}(K_2)$;\\
			
			\item [\rm (1L2)]	 $C_{6750}(L_2)$ $\cong_{T2_{6750,3,500}}$ $C_{6750}(G_2)$; 
			
			\item [\rm (2L2)]	 $C_{6750}(L_2)$ $\cong_{T2_{6750,3,1000}}$ $C_{6750}(B_2)$,
			
			$C_{6750}(L_2)$ $\cong_{T2_{6750,3,250}}$ $C_{6750}(B_2)$, $C_{6750}(L_2)$ $\cong_{T2_{6750,3,750}}$ $C_{6750}(L_2)$; 
			
			\item [\rm (3L2)] $C_{6750}(L_2)$ $\cong_{T2_{6750,5,54}}$ $C_{6750}(N_2)$; 
			
			\item [\rm (4L2)] $C_{6750}(L_2)$ $\cong_{T2_{6750,5,108}}$ $C_{6750}(K_2)$;
			
			\item [\rm (5L2)] $C_{6750}(L_2)$ $\cong_{T2_{6750,5,162}}$ $C_{6750}(M_2)$;
			
			\item [\rm (6L2)] $C_{6750}(L_2)$ $\cong_{T2_{6750,5,216}}$ $C_{6750}(O_2)$;
			
			\item [\rm (7L2)] $C_{6750}(L_2)$ $\cong_{T2_{6750,5,270}}$ $C_{6750}(L_2)$;\\
			
			\item [\rm (1M2)]	 $C_{6750}(M_2)$ $\cong_{T2_{6750,3,500}}$ $C_{6750}(H_2)$; 
			
			\item [\rm (2M2)]	 $C_{6750}(M_2)$ $\cong_{T2_{6750,3,1000}}$ $C_{6750}(C_2)$,
			
			$C_{6750}(M_2)$ $\cong_{T2_{6750,3,250}}$ $C_{6750}(C_2)$, $C_{6750}(M_2)$ $\cong_{T2_{6750,3,750}}$ $C_{6750}(M_2)$; 
			
			\item [\rm (3M2)] $C_{6750}(M_2)$ $\cong_{T2_{6750,5,54}}$ $C_{6750}(O_2)$; 
			
			\item [\rm (4M2)] $C_{6750}(M_2)$ $\cong_{T2_{6750,5,108}}$ $C_{6750}(L_2)$;
			
			\item [\rm (5M2)] $C_{6750}(M_2)$ $\cong_{T2_{6750,5,162}}$ $C_{6750}(N_2)$;
			
			\item [\rm (6M2)] $C_{6750}(M_2)$ $\cong_{T2_{6750,5,216}}$ $C_{6750}(K_2)$;
			
			\item [\rm (7M2)] $C_{6750}(M_2)$ $\cong_{T2_{6750,5,270}}$ $C_{6750}(M_2)$;\\
			
			\item [\rm (1N2)] $C_{6750}(N_2)$ $\cong_{T2_{6750,3,500}}$ $C_{6750}(I_2)$; 
			
			\item [\rm (2N2)] $C_{6750}(N_2)$ $\cong_{T2_{6750,3,1000}}$ $C_{6750}(D_2)$,
			
			$C_{6750}(N_2)$ $\cong_{T2_{6750,3,250}}$ $C_{6750}(D_2)$, $C_{6750}(N_2)$ $\cong_{T2_{6750,3,750}}$ $C_{6750}(N_2)$; 
			
			\item [\rm (3N2)] $C_{6750}(N_2)$ $\cong_{T2_{6750,5,54}}$ $C_{6750}(K_2)$; 
			
			\item [\rm (4N2)] $C_{6750}(N_2)$ $\cong_{T2_{6750,5,108}}$ $C_{6750}(M_2)$;
			
			\item [\rm (5N2)] $C_{6750}(N_2)$ $\cong_{T2_{6750,5,162}}$ $C_{6750}(O_2)$;
			
			\item [\rm (6N2)] $C_{6750}(N_2)$ $\cong_{T2_{6750,5,216}}$ $C_{6750}(L_2)$;
			
			\item [\rm (7N2)] $C_{6750}(N_2)$ $\cong_{T2_{6750,5,270}}$ $C_{6750}(N_2)$;\\
			
			\item [\rm (1O2)] $C_{6750}(O_2)$ $\cong_{T2_{6750,3,500}}$ $C_{6750}(J_2)$; 
			
			\item [\rm (2O2)] $C_{6750}(O_2)$ $\cong_{T2_{6750,3,1000}}$ $C_{6750}(E_2)$,
			
			$C_{6750}(O_2)$ $\cong_{T2_{6750,3,250}}$ $C_{6750}(E_2)$, $C_{6750}(O_2)$ $\cong_{T2_{6750,3,750}}$ $C_{6750}(O_2)$; 
			
			\item [\rm (3O2)] $C_{6750}(O_2)$ $\cong_{T2_{6750,5,54}}$ $C_{6750}(L_2)$; 
			
			\item [\rm (4O2)] $C_{6750}(O_2)$ $\cong_{T2_{6750,5,108}}$ $C_{6750}(N_2)$;
			
			\item [\rm (5O2)] $C_{6750}(O_2)$ $\cong_{T2_{6750,5,162}}$ $C_{6750}(K_2)$;
			
			\item [\rm (6O2)] $C_{6750}(O_2)$ $\cong_{T2_{6750,5,216}}$ $C_{6750}(M_2)$;
			
			\item [\rm (7O2)] $C_{6750}(O_2)$ $\cong_{T2_{6750,5,270}}$ $C_{6750}(O_2)$;\\
			
			\item [\rm (1A3)]	 $C_{6750}(A_3)$ $\cong_{T2_{6750,3,500}}$ $C_{6750}(K_3)$; 
			
			\item [\rm (2A3)]	 $C_{6750}(A_3)$ $\cong_{T2_{6750,3,1000}}$ $C_{6750}(F_3)$,
			
			$C_{6750}(A_3)$ $\cong_{T2_{6750,3,250}}$ $C_{6750}(F_3)$, $C_{6750}(A_3)$ $\cong_{T2_{6750,3,750}}$ $C_{6750}(A_3)$; 
			
			\item [\rm (3A3)] $C_{6750}(A_3)$ $\cong_{T2_{6750,5,54}}$ $C_{6750}(C_3)$; 
			
			\item [\rm (4A3)] $C_{6750}(A_3)$ $\cong_{T2_{6750,5,108}}$ $C_{6750}(E_3)$;
			
			\item [\rm (5A3)] $C_{6750}(A_3)$ $\cong_{T2_{6750,5,162}}$ $C_{6750}(B_3)$;
			
			\item [\rm (6A3)] $C_{6750}(A_3)$ $\cong_{T2_{6750,5,216}}$ $C_{6750}(D_3)$;
			
			\item [\rm (7A3)] $C_{6750}(A_3)$ $\cong_{T2_{6750,5,270}}$ $C_{6750}(A_3)$;\\
			
			\item [\rm (1B3)]	 $C_{6750}(B_3)$ $\cong_{T2_{6750,3,500}}$ $C_{6750}(L_3)$; 
			
			\item [\rm (2B3)]	 $C_{6750}(B_3)$ $\cong_{T2_{6750,3,1000}}$ $C_{6750}(G_3)$,
			
			$C_{6750}(B_3)$ $\cong_{T2_{6750,3,250}}$ $C_{6750}(G_3)$, $C_{6750}(B_3)$ $\cong_{T2_{6750,3,750}}$ $C_{6750}(B_3)$; 
			
			\item [\rm (3B3)] $C_{6750}(B_3)$ $\cong_{T2_{6750,5,54}}$ $C_{6750}(D_3)$; 
			
			\item [\rm (4B3)] $C_{6750}(B_3)$ $\cong_{T2_{6750,5,108}}$ $C_{6750}(A_3)$;
			
			\item [\rm (5B3)] $C_{6750}(B_3)$ $\cong_{T2_{6750,5,162}}$ $C_{6750}(C_3)$;
			
			\item [\rm (6B3)] $C_{6750}(B_3)$ $\cong_{T2_{6750,5,216}}$ $C_{6750}(E_3)$;
			
			\item [\rm (7B3)] $C_{6750}(B_3)$ $\cong_{T2_{6750,5,270}}$ $C_{6750}(B_3)$;\\
			
			\item [\rm (1C3)]	 $C_{6750}(C_3)$ $\cong_{T2_{6750,3,500}}$ $C_{6750}(M_3)$; 
			
			\item [\rm (2C3)]	 $C_{6750}(C_3)$ $\cong_{T2_{6750,3,1000}}$ $C_{6750}(H_3)$,
			
			$C_{6750}(C_3)$ $\cong_{T2_{6750,3,250}}$ $C_{6750}(H_3)$, $C_{6750}(C_3)$ $\cong_{T2_{6750,3,750}}$ $C_{6750}(C_3)$; 
			
			\item [\rm (3C3)] $C_{6750}(C_3)$ $\cong_{T2_{6750,5,54}}$ $C_{6750}(E_3)$; 
			
			\item [\rm (4C3)] $C_{6750}(C_3)$ $\cong_{T2_{6750,5,108}}$ $C_{6750}(B_3)$;
			
			\item [\rm (5C3)] $C_{6750}(C_3)$ $\cong_{T2_{6750,5,162}}$ $C_{6750}(D_3)$;
			
			\item [\rm (6C3)] $C_{6750}(C_3)$ $\cong_{T2_{6750,5,216}}$ $C_{6750}(A_3)$;
			
			\item [\rm (7C3)] $C_{6750}(C_3)$ $\cong_{T2_{6750,5,270}}$ $C_{6750}(C_3)$;\\
			
			\item [\rm (1D3)]	 $C_{6750}(D_3)$ $\cong_{T2_{6750,3,500}}$ $C_{6750}(N_3)$; 
			
			\item [\rm (2D3)]	 $C_{6750}(D_3)$ $\cong_{T2_{6750,3,1000}}$ $C_{6750}(I_3)$,
			
			$C_{6750}(D_3)$ $\cong_{T2_{6750,3,250}}$ $C_{6750}(I_3)$, $C_{6750}(D_3)$ $\cong_{T2_{6750,3,750}}$ $C_{6750}(D_3)$; 
			
			\item [\rm (3D3)] $C_{6750}(D_3)$ $\cong_{T2_{6750,5,54}}$ $C_{6750}(A_3)$; 
			
			\item [\rm (4D3)] $C_{6750}(D_3)$ $\cong_{T2_{6750,5,108}}$ $C_{6750}(C_3)$;
			
			\item [\rm (5D3)] $C_{6750}(D_3)$ $\cong_{T2_{6750,5,162}}$ $C_{6750}(E_3)$;
			
			\item [\rm (6D3)] $C_{6750}(D_3)$ $\cong_{T2_{6750,5,216}}$ $C_{6750}(B_3)$;
			
			\item [\rm (7D3)] $C_{6750}(D_3)$ $\cong_{T2_{6750,5,270}}$ $C_{6750}(D_3)$;\\
			
			\item [\rm (1E3)]	 $C_{6750}(E_3)$ $\cong_{T2_{6750,3,500}}$ $C_{6750}(O_3)$; 
			
			\item [\rm (2E3)]	 $C_{6750}(E_3)$ $\cong_{T2_{6750,3,1000}}$ $C_{6750}(J_3)$,
			
			$C_{6750}(E_3)$ $\cong_{T2_{6750,3,250}}$ $C_{6750}(J_3)$, $C_{6750}(E_3)$ $\cong_{T2_{6750,3,750}}$ $C_{6750}(E_3)$; 
			
			\item [\rm (3E3)] $C_{6750}(E_3)$ $\cong_{T2_{6750,5,54}}$ $C_{6750}(B_3)$; 
			
			\item [\rm (4E3)] $C_{6750}(E_3)$ $\cong_{T2_{6750,5,108}}$ $C_{6750}(D_3)$;
			
			\item [\rm (5E3)] $C_{6750}(E_3)$ $\cong_{T2_{6750,5,162}}$ $C_{6750}(A_3)$;
			
			\item [\rm (6E3)] $C_{6750}(E_3)$ $\cong_{T2_{6750,5,216}}$ $C_{6750}(C_3)$;  
			
			\item [\rm (7E3)] $C_{6750}(E_3)$ $\cong_{T2_{6750,5,270}}$ $C_{6750}(E_3)$;\\  
			
			\item [\rm (1F3)]	 $C_{6750}(F_3)$ $\cong_{T2_{6750,3,500}}$ $C_{6750}(A_3)$; 
			
			\item [\rm (2F3)]	 $C_{6750}(F_3)$ $\cong_{T2_{6750,3,1000}}$ $C_{6750}(K_3)$,
			
			$C_{6750}(F_3)$ $\cong_{T2_{6750,3,250}}$ $C_{6750}(K_3)$, $C_{6750}(F_3)$ $\cong_{T2_{6750,3,750}}$ $C_{6750}(F_3)$; 
			
			\item [\rm (3F3)] $C_{6750}(F_3)$ $\cong_{T2_{6750,5,54}}$ $C_{6750}(H_3)$; 
			
			\item [\rm (4F3)] $C_{6750}(F_3)$ $\cong_{T2_{6750,5,108}}$ $C_{6750}(J_3)$;
			
			\item [\rm (5F3)] $C_{6750}(F_3)$ $\cong_{T2_{6750,5,162}}$ $C_{6750}(G_3)$;
			
			\item [\rm (6F3)] $C_{6750}(F_3)$ $\cong_{T2_{6750,5,216}}$ $C_{6750}(I_3)$;
			
			\item [\rm (7F3)] $C_{6750}(F_3)$ $\cong_{T2_{6750,5,270}}$ $C_{6750}(F_3)$;\\
			
			\item [\rm (1G3)]	 $C_{6750}(G_3)$ $\cong_{T2_{6750,3,500}}$ $C_{6750}(B_3)$; 
			
			\item [\rm (2G3)]	 $C_{6750}(G_3)$ $\cong_{T2_{6750,3,1000}}$ $C_{6750}(L_3)$,
			
			$C_{6750}(G_3)$ $\cong_{T2_{6750,3,250}}$ $C_{6750}(L_3)$, $C_{6750}(G_3)$ $\cong_{T2_{6750,3,750}}$ $C_{6750}(G_3)$; 
			
			\item [\rm (3G3)] $C_{6750}(G_3)$ $\cong_{T2_{6750,5,54}}$ $C_{6750}(I_3)$; 
			
			\item [\rm (4G3)] $C_{6750}(G_3)$ $\cong_{T2_{6750,5,108}}$ $C_{6750}(F_3)$;
			
			\item [\rm (5G3)] $C_{6750}(G_3)$ $\cong_{T2_{6750,5,162}}$ $C_{6750}(H_3)$;
			
			\item [\rm (6G3)] $C_{6750}(G_3)$ $\cong_{T2_{6750,5,216}}$ $C_{6750}(J_3)$;
			
			\item [\rm (7G3)] $C_{6750}(G_3)$ $\cong_{T2_{6750,5,270}}$ $C_{6750}(G_3)$;\\
			
			\item [\rm (1H3)]	 $C_{6750}(H_3)$ $\cong_{T2_{6750,3,500}}$ $C_{6750}(C_3)$; 
			
			\item [\rm (2H3)]	 $C_{6750}(H_3)$ $\cong_{T2_{6750,3,1000}}$ $C_{6750}(M_3)$,
			
			$C_{6750}(H_3)$ $\cong_{T2_{6750,3,250}}$ $C_{6750}(M_3)$, $C_{6750}(H_3)$ $\cong_{T2_{6750,3,750}}$ $C_{6750}(H_3)$; 
			
			\item [\rm (3H3)] $C_{6750}(H_3)$ $\cong_{T2_{6750,5,54}}$ $C_{6750}(J_3)$; 
			
			\item [\rm (4H3)] $C_{6750}(H_3)$ $\cong_{T2_{6750,5,108}}$ $C_{6750}(G_3)$;
			
			\item [\rm (5H3)] $C_{6750}(H_3)$ $\cong_{T2_{6750,5,162}}$ $C_{6750}(I_3)$;
			
			\item [\rm (6H3)] $C_{6750}(H_3)$ $\cong_{T2_{6750,5,216}}$ $C_{6750}(F_3)$;
			
			\item [\rm (7H3)] $C_{6750}(H_3)$ $\cong_{T2_{6750,5,270}}$ $C_{6750}(H_3)$;\\
			
			\item [\rm (1I3)]	 $C_{6750}(I_3)$ $\cong_{T2_{6750,3,500}}$ $C_{6750}(D_3)$; 
			
			\item [\rm (2I3)]	 $C_{6750}(I_3)$ $\cong_{T2_{6750,3,1000}}$ $C_{6750}(N_3)$,
			
			$C_{6750}(I_3)$ $\cong_{T2_{6750,3,250}}$ $C_{6750}(N_3)$, $C_{6750}(I_3)$ $\cong_{T2_{6750,3,750}}$ $C_{6750}(I_3)$; 
			
			\item [\rm (3I3)] $C_{6750}(I_3)$ $\cong_{T2_{6750,5,54}}$ $C_{6750}(F_3)$; 
			
			\item [\rm (4I3)] $C_{6750}(I_3)$ $\cong_{T2_{6750,5,108}}$ $C_{6750}(H_3)$;
			
			\item [\rm (5I3)] $C_{6750}(I_3)$ $\cong_{T2_{6750,5,162}}$ $C_{6750}(J_3)$;
			
			\item [\rm (6I3)] $C_{6750}(I_3)$ $\cong_{T2_{6750,5,216}}$ $C_{6750}(G_3)$;
			
			\item [\rm (7I3)] $C_{6750}(I_3)$ $\cong_{T2_{6750,5,270}}$ $C_{6750}(I_3)$;\\
			
			\item [\rm (1J3)] $C_{6750}(J_3)$ $\cong_{T2_{6750,3,500}}$ $C_{6750}(E_3)$; 
			
			\item [\rm (2J3)] $C_{6750}(J_3)$ $\cong_{T2_{6750,3,1000}}$ $C_{6750}(O_3)$,
			
			$C_{6750}(J_3)$ $\cong_{T2_{6750,3,250}}$ $C_{6750}(O_3)$, $C_{6750}(J_3)$ $\cong_{T2_{6750,3,750}}$ $C_{6750}(J_3)$; 
			
			\item [\rm (3J3)] $C_{6750}(J_3)$ $\cong_{T2_{6750,5,54}}$ $C_{6750}(G_3)$; 
			
			\item [\rm (4J3)] $C_{6750}(J_3)$ $\cong_{T2_{6750,5,108}}$ $C_{6750}(I_3)$;
			
			\item [\rm (5J3)] $C_{6750}(J_3)$ $\cong_{T2_{6750,5,162}}$ $C_{6750}(F_3)$;
			
			\item [\rm (6J3)] $C_{6750}(J_3)$ $\cong_{T2_{6750,5,216}}$ $C_{6750}(H_3)$;  
			
			\item [\rm (7J3)] $C_{6750}(J_3)$ $\cong_{T2_{6750,5,270}}$ $C_{6750}(J_3)$;\\  
			
			\item [\rm (1K3)]	 $C_{6750}(K_3)$ $\cong_{T2_{6750,3,500}}$ $C_{6750}(F_3)$; 
			
			\item [\rm (2K3)]	 $C_{6750}(K_3)$ $\cong_{T2_{6750,3,1000}}$ $C_{6750}(A_3)$,
			
			$C_{6750}(K_3)$ $\cong_{T2_{6750,3,250}}$ $C_{6750}(A_3)$, $C_{6750}(K_3)$ $\cong_{T2_{6750,3,750}}$ $C_{6750}(K_3)$; 
			
			\item [\rm (3K3)] $C_{6750}(K_3)$ $\cong_{T2_{6750,5,54}}$ $C_{6750}(M_3)$; 
			
			\item [\rm (4K3)] $C_{6750}(K_3)$ $\cong_{T2_{6750,5,108}}$ $C_{6750}(O_3)$;
			
			\item [\rm (5K3)] $C_{6750}(K_3)$ $\cong_{T2_{6750,5,162}}$ $C_{6750}(L_3)$;
			
			\item [\rm (6K3)] $C_{6750}(K_3)$ $\cong_{T2_{6750,5,216}}$ $C_{6750}(N_3)$;
			
			\item [\rm (7K3)] $C_{6750}(K_3)$ $\cong_{T2_{6750,5,270}}$ $C_{6750}(K_3)$;\\
			
			\item [\rm (1L3)]	 $C_{6750}(L_3)$ $\cong_{T2_{6750,3,500}}$ $C_{6750}(G_3)$; 
			
			\item [\rm (2L3)]	 $C_{6750}(L_3)$ $\cong_{T2_{6750,3,1000}}$ $C_{6750}(B_3)$,
			
			$C_{6750}(L_3)$ $\cong_{T2_{6750,3,250}}$ $C_{6750}(B_3)$, $C_{6750}(L_3)$ $\cong_{T2_{6750,3,750}}$ $C_{6750}(L_3)$; 
			
			\item [\rm (3L3)] $C_{6750}(L_3)$ $\cong_{T2_{6750,5,54}}$ $C_{6750}(N_3)$; 
			
			\item [\rm (4L3)] $C_{6750}(L_3)$ $\cong_{T2_{6750,5,108}}$ $C_{6750}(K_3)$;
			
			\item [\rm (5L3)] $C_{6750}(L_3)$ $\cong_{T2_{6750,5,162}}$ $C_{6750}(M_3)$;
			
			\item [\rm (6L3)] $C_{6750}(L_3)$ $\cong_{T2_{6750,5,216}}$ $C_{6750}(O_3)$;
			
			\item [\rm (7L3)] $C_{6750}(L_3)$ $\cong_{T2_{6750,5,270}}$ $C_{6750}(L_3)$;\\
			
			\item [\rm (1M3)]	 $C_{6750}(M_3)$ $\cong_{T2_{6750,3,500}}$ $C_{6750}(H_3)$; 
			
			\item [\rm (2M3)]	 $C_{6750}(M_3)$ $\cong_{T2_{6750,3,1000}}$ $C_{6750}(C_3)$,
			
			$C_{6750}(M_3)$ $\cong_{T2_{6750,3,250}}$ $C_{6750}(C_3)$, $C_{6750}(M_3)$ $\cong_{T2_{6750,3,750}}$ $C_{6750}(M_3)$; 
			
			\item [\rm (3M3)] $C_{6750}(M_3)$ $\cong_{T2_{6750,5,54}}$ $C_{6750}(O_3)$; 
			
			\item [\rm (4M3)] $C_{6750}(M_3)$ $\cong_{T2_{6750,5,108}}$ $C_{6750}(L_3)$;
			
			\item [\rm (5M3)] $C_{6750}(M_3)$ $\cong_{T2_{6750,5,162}}$ $C_{6750}(N_3)$;
			
			\item [\rm (6M3)] $C_{6750}(M_3)$ $\cong_{T2_{6750,5,216}}$ $C_{6750}(K_3)$;
			
			\item [\rm (7M3)] $C_{6750}(M_3)$ $\cong_{T2_{6750,5,270}}$ $C_{6750}(M_3)$;\\
			
			\item [\rm (1N3)] $C_{6750}(N_3)$ $\cong_{T2_{6750,3,500}}$ $C_{6750}(I_3)$; 
			
			\item [\rm (2N3)] $C_{6750}(N_3)$ $\cong_{T2_{6750,3,1000}}$ $C_{6750}(D_3)$,
			
			$C_{6750}(N_3)$ $\cong_{T2_{6750,3,250}}$ $C_{6750}(D_3)$, $C_{6750}(N_3)$ $\cong_{T2_{6750,3,750}}$ $C_{6750}(N_3)$; 
			
			\item [\rm (3N3)] $C_{6750}(N_3)$ $\cong_{T2_{6750,5,54}}$ $C_{6750}(K_3)$; 
			
			\item [\rm (4N3)] $C_{6750}(N_3)$ $\cong_{T2_{6750,5,108}}$ $C_{6750}(M_3)$;
			
			\item [\rm (5N3)] $C_{6750}(N_3)$ $\cong_{T2_{6750,5,162}}$ $C_{6750}(O_3)$;
			
			\item [\rm (6N3)] $C_{6750}(N_3)$ $\cong_{T2_{6750,5,216}}$ $C_{6750}(L_3)$;
			
			\item [\rm (7N3)] $C_{6750}(N_3)$ $\cong_{T2_{6750,5,270}}$ $C_{6750}(N_3)$;\\
			
			\item [\rm (1O3)] $C_{6750}(O_3)$ $\cong_{T2_{6750,3,500}}$ $C_{6750}(J_3)$; 
			
			\item [\rm (2O3)] $C_{6750}(O_3)$ $\cong_{T2_{6750,3,1000}}$ $C_{6750}(E_3)$,
			
			$C_{6750}(O_3)$ $\cong_{T2_{6750,3,250}}$ $C_{6750}(E_3)$, $C_{6750}(O_3)$ $\cong_{T2_{6750,3,750}}$ $C_{6750}(O_3)$; 
			
			\item [\rm (3O3)] $C_{6750}(O_3)$ $\cong_{T2_{6750,5,54}}$ $C_{6750}(L_3)$; 
			
			\item [\rm (4O3)] $C_{6750}(O_3)$ $\cong_{T2_{6750,5,108}}$ $C_{6750}(N_3)$;
			
			\item [\rm (5O3)] $C_{6750}(O_3)$ $\cong_{T2_{6750,5,162}}$ $C_{6750}(K_3)$;
			
			\item [\rm (6O3)] $C_{6750}(O_3)$ $\cong_{T2_{6750,5,216}}$ $C_{6750}(M_3)$;
			
			\item [\rm (7O3)] $C_{6750}(O_2)$ $\cong_{T2_{6750,5,270}}$ $C_{6750}(O_2)$.\\
		\end{enumerate}	
		
		\begin{enumerate}	
			\item [\rm (b1)]  $T2_{6750,3}(C_{6750}(A_i))$ = $\{C_{6750}(A_i), C_{6750}(F_i), C_{6750}(K_i)\}$ 
			
			\hfill = $T2_{6750,3}(C_{6750}(F_i))$ = $T2_{6750,3}(C_{6750}(K_i))$ for $i$ = 1 to 30;  
			
			\item [\rm (b2)]  $T2_{6750,3}(C_{6750}(B_i))$ = $\{C_{6750}(B_i), C_{6750}(G_i), C_{6750}(L_i)\}$ 
			
			\hfill = $T2_{6750,3}(C_{6750}(G_i))$ = $T2_{6750,3}(C_{6750}(L_i))$ for $i$ = 1 to 30;  
			
			\item [\rm (b3)]  $T2_{6750,3}(C_{6750}(C_i))$ = $\{C_{6750}(C_i), C_{6750}(H_i), C_{6750}(M_i)\}$ 
			
			\hfill = $T2_{6750,3}(C_{6750}(H_i))$ = $T2_{6750,3}(C_{6750}(M_i))$ for $i$ = 1 to 30;  
			
			\item [\rm (b4)]  $T2_{6750,3}(C_{6750}(D_i))$ = $\{C_{6750}(D_i), C_{6750}(I_i), C_{6750}(N_i)\}$ 
			
			\hfill = $T2_{6750,3}(C_{6750}(I_i))$ = $T2_{6750,3}(C_{6750}(N_i))$ for $i$ = 1 to 30;  
			
			\item [\rm (b5)]  $T2_{6750,3}(C_{6750}(E_i))$ = $\{C_{6750}(E_i), C_{6750}(J_i), C_{6750}(O_i)\}$ 
			
			\hfill = $T2_{6750,3}(C_{6750}(J_i))$ = $T2_{6750,3}(C_{6750}(O_i))$ for $i$ = 1 to 30.  
			
			\item [\rm (c1)] $T2_{6750,5}(C_{6750}(A_i))$ = $\{C_{6750}(A_i), C_{6750}(B_i), C_{6750}(C_i), C_{6750}(D_i), C_{6750}(E_i)\}$ 
			
			\hfill = $T2_{6750,5}(C_{6750}(B_i))$ = $T2_{6750,5}(C_{6750}(C_i))$ = $T2_{6750,5}(C_{6750}(D_i))$ 
			
			\hfill = $T2_{6750,5}(C_{6750}(E_i))$ for $i$ = 1 to 30;  
			
			\item [\rm (c2)] $T2_{6750,5}(C_{6750}(F_i))$ = $\{C_{6750}(F_i), C_{6750}(G_i), C_{6750}(H_i), C_{6750}(I_i), C_{6750}(J_i)\}$ 
			
			\hfill = $T2_{6750,5}(C_{6750}(G_i))$ = $T2_{6750,5}(C_{6750}(H_i))$ = $T2_{6750,5}(C_{6750}(I_i))$ 
			
			\hfill = $T2_{6750,5}(C_{6750}(J_i))$ for $i$ = 1 to 30;  
			
			\item [\rm (c3)] $T2_{6750,5}(C_{6750}(K_i))$ = $\{C_{6750}(K_i), C_{6750}(L_i), C_{6750}(M_i), C_{6750}(N_i), C_{6750}(O_i)\}$ 
			
			\hfill = $T2_{6750,5}(C_{6750}(L_i))$ = $T2_{6750,5}(C_{6750}(M_i))$ = $T2_{6750,5}(C_{6750}(N_i))$ 
			
			\hfill = $T2_{6750,5}(C_{6750}(O_i))$ for $i$ = 1 to 30.  
			
			\item [\rm (d)]  $(T2_{6750,3}(C_{6750}(X_i)), \circ)$ is an Abelian group for $X_i$ = $A_i,B_i,C_i,D_i,E_i$ and $i$ = 1 to 30. 
			
			\item [\rm (e)]  $(T2_{6750,5}(C_{6750}(Y_i)), \circ)$ is an Abelian group for $Y_i$ = $A_i, F_i, K_i$ and $i$ = 1 to 30. \hfill $\Box$
	\end{enumerate}	}
\end{prm}

\begin{prm} \quad \label{p13} {\rm For $R$ = $\{135, 243, 250, 750, 1107, 1593, 2000, 2457, 2500, 2943\}$, find the isomorphism digraph and graph of $C_{6750}(R)$. }
\end{prm}
\noindent
{\bf Solution.}\quad For $R$ = $\{135, 243, 250, 750, 1107, 1593$, $2000, 2457, 2500, 2943\}$, using problems \ref{p11} and \ref{p12}, we get $T1_{6750}(C_{6750}(R))$ and $T2_{6750, m}(C_{6750}(R))$ for $m$ = 3 and $m$ = 5 for different isomorphic circulant graphs of $C_{6750}(R)$.

Let $R$ = $R_1$ = $A_1$ and for $i$ = 1 to 30, $A_i, B_i, \dots, O_i$ be as given in problem \ref{p11}. Let $\mathcal{D}$ and $\mathcal{G}$ be the isomorphism digraph and graph of $C_{6750}(R)$, respectively. Then, 

$Isoset(C_{6750}(R))$ = $\{C_{6750}(X_i)/~ X_i = A_i,B_i,\dots,O_i~\text{and}~i = 1,2,\dots,30\}$ = $V(\mathcal{D})$ = $V(\mathcal{G})$.
\\
Using solution obtained in problems \ref{p11} and \ref{p12} with 

$R$ = $A_1$ = $R_1$, $B_1$ = $S_1$, $C_1$ = $T_1$, $D_1$ = $U_1$, $E_1$ = $V_1$, 

$F_1$ = $R_2$, $G_1$ = $S_2$, $H_1$ = $T_2$, $I_1$ = $U_2$, $J_1$ = $V_2$, 

$K_1$ = $R_3$, $L_1$ = $S_3$, $M_1$ = $T_3$, $N_1$ = $U_3$, $O_1$ = $V_3$, and for $i$ = 1 to 30, we get,

\medskip
$T2_{6750,3}(C_{6750}(A_i))$ = $\{C_{6750}(A_i), C_{6750}(F_i) = \theta_{6750,3,250}(C_{6750}(A_i))$,  

\hfill $C_{6750}(K_i) = \theta_{6750,3,500}(C_{6750}(A_i))\}$

\hfill = $T2_{6750,3}(C_{6750}(F_i))$ = $T2_{6750,3}(C_{6750}(K_i))$,  

$T2_{6750,3}(C_{6750}(B_i))$ = $\{C_{6750}(B_i), C_{6750}(G_i) = \theta_{6750,3,250}(C_{6750}(B_i))$,  

\hfill $C_{6750}(L_i) = \theta_{6750,3,500}(C_{6750}(B_i))\}$

\hfill = $T2_{6750,3}(C_{6750}(G_i))$ = $T2_{6750,3}(C_{6750}(L_i))$,  

$T2_{6750,3}(C_{6750}(C_i))$ = $\{C_{6750}(C_i), C_{6750}(H_i) = \theta_{6750,3,250}(C_{6750}(C_i))$,  

\hfill $C_{6750}(M_i) = \theta_{6750,3,500}(C_{6750}(C_i))\}$

\hfill = $T2_{6750,3}(C_{6750}(H_i))$ = $T2_{6750,3}(C_{6750}(M_i))$,  

$T2_{6750,3}(C_{6750}(D_i))$ = $\{C_{6750}(D_i), C_{6750}(I_i) = \theta_{6750,3,250}(C_{6750}(D_i))$,  

\hfill $C_{6750}(N_i) = \theta_{6750,3,500}(C_{6750}(D_i))\}$

\hfill = $T2_{6750,3}(C_{6750}(I_i))$ = $T2_{6750,3}(C_{6750}(N_i))$,  

$T2_{6750,3}(C_{6750}(E_i))$ = $\{C_{6750}(E_i), C_{6750}(J_i) = \theta_{6750,3,250}(C_{6750}(E_i))$,  

\hfill $C_{6750}(O_i) = \theta_{6750,3,500}(C_{6750}(E_i))\}$

\hfill = $T2_{6750,3}(C_{6750}(J_i))$ = $T2_{6750,3}(C_{6750}(O_i))$,  

\medskip
$T2_{6750,5}(C_{6750}(A_i))$ = $\{C_{6750}(A_i), C_{6750}(B_i) = \theta_{6750,5,162}(C_{6750}(A_i))$,  

\hfill $C_{6750}(C_i) = \theta_{6750,5,54}(C_{6750}(A_i)),  C_{6750}(D_i) = \theta_{6750,5,216}(C_{6750}(A_i))$, 

\hfill $C_{6750}(E_i) = \theta_{6750,5,108}(C_{6750}(A_i))\}$ = $T2_{6750,5}(C_{6750}(B_i))$ = $T2_{6750,5}(C_{6750}(C_i))$

\hfill = $T2_{6750,5}(C_{6750}(D_i))$ = $T2_{6750,5}(C_{6750}(E_i))$,  

$T2_{6750,5}(C_{6750}(F_i))$ = $\{C_{6750}(F_i), C_{6750}(G_i) = \theta_{6750,5,162}(C_{6750}(F_i))$,  

\hfill $C_{6750}(H_i) = \theta_{6750,5,54}(C_{6750}(F_i)),  C_{6750}(I_i) = \theta_{6750,5,216}(C_{6750}(F_i))$,

\hfill $C_{6750}(J_i) = \theta_{6750,5,108}(C_{6750}(F_i))\}$ = $T2_{6750,5}(C_{6750}(G_i))$ = $T2_{6750,5}(C_{6750}(H_i))$

\hfill = $T2_{6750,5}(C_{6750}(I_i))$ = $T2_{6750,5}(C_{6750}(J_i))$,  

$T2_{6750,5}(C_{6750}(K_i))$ = $\{C_{6750}(K_i), C_{6750}(L_i) = \theta_{6750,5,162}(C_{6750}(K_i))$,  

\hfill $C_{6750}(M_i) = \theta_{6750,5,54}(C_{6750}(K_i)),  C_{6750}(N_i) = \theta_{6750,5,216}(C_{6750}(K_i))$,

\hfill $C_{6750}(O_i) = \theta_{6750,5,108}(C_{6750}(K_i))\}$ = $T2_{6750,5}(C_{6750}(L_i))$ = $T2_{6750,5}(C_{6750}(M_i))$

\hfill = $T2_{6750,5}(C_{6750}(N_i))$ = $T2_{6750,5}(C_{6750}(O_i))$,
  
\medskip
$T1_{6750}(C_{6750}(\alpha_1))$ = $\{C_{6750}(\alpha_i)/~i = 1,2,\dots,30\}$, $\alpha_1$ = $A_1, B_1, \dots, O_1$, and also

\medskip
\noindent
$V(\mathcal{D})$ = $V(\mathcal{G})$ = $Isoset(C_{6750}(R))$ = $Isoset_{6750,3}(C_{6750}(R))$ $\cup$ $Isoset_{6750,5}(C_{6750}(R))$

\medskip
 \hspace{1.7cm} = $\big(\bigcup^{30}_{i=1} \big(\bigcup_{X_i=A_i,B_i,\dots,E_i}  {{Isoset_{6750,3}(C_{6750}(X_i))}}\big)\big)$ 
 
 \hfill $\bigcup$ $\big(\bigcup^{30}_{i=1} \big(\bigcup_{X_i=A_i,F_i,K_i}  {{Isoset_{6750,5}(C_{6750}(X_i))}}\big)\big)$.

\medskip
The isomorphism digraph $\mathcal{D}$ =  $\mathcal{ID}_{6750}(C_{6750}(R))$ = $\mathcal{ID}_{6750}(C_{6750}(X_i))$, and the isomorphism graph 

$\mathcal{G}$ = $\mathcal{I}_{6750}(C_{6750}(R))$ = $\mathcal{I}_{6750}(C_{6750}(X_i))$, $X_i$ = $A_i,B_i,\dots,O_i$, and $i$ = $1,2,\dots,30$. 

\medskip
The isomorphism digraph $\mathcal{D}$ of $C_{6750}(R)$ can also be written as $\mathcal{D}$ = $\mathcal{D}_1$ $\cup$ $\mathcal{D}_2$, 

\medskip
\hspace{.5cm} $\mathcal{D}_1$ = Directed subgraph of $\mathcal{D}$ without edges which are defined on Type-2 isomorphisms 

\hspace{1cm} = $\bigcup_{X=A,B,\dots,O}{\mathcal{D}^{X}_1}$ 
where 

\medskip
\hspace{1cm} $\mathcal{D}^{X}_1$ = ${\overleftrightarrow{K_{30}}[C_{6750}(X_1), C_{6750}(X_2), \dots, C_{6750}(X_{30})]}$, $X \in \{A, B, \dots, O\}$, and

\medskip
\hspace{.5cm} $\mathcal{D}_2$ = Directed subgraph of $\mathcal{D}$ without edges which are defined on Type-1 isomorphisms

\hspace{1cm} = $\bigcup^{30}_{i=1}{\mathcal{D}^{Y_i}_2}$ = $\bigcup^{30}_{i=1}{\big(\mathcal{D}^{Y_i}_{2,3} \cup \mathcal{D}^{Y_i}_{2,5}\big)}$, 

\medskip
\hspace{1cm} $\mathcal{D}^{Y_i}_{2,3}$ = ${\mathcal{D}[C_{6750}(A_i), C_{6750}(F_i), C_{6750}(K_i)]}$  

\hfill $\bigcup$ ${\mathcal{D}[C_{6750}(B_i), C_{6750}(G_i), C_{6750}(L_i)]}$ $\dots$ $\bigcup$ ${\mathcal{D}[C_{6750}(E_i), C_{6750}(J_i), C_{6750}(O_i)]}$ 

\hfill is a directed subgraph of $\mathcal{D}$ without edges which are defined on Type-1 

\hfill isomorphisms and also on Type-2 isomorphism w.r.t. $m$ = 5, and 

\medskip
\hspace{1cm} $\mathcal{D}^{A_i}_{2,5}$ = ${\mathcal{D}[C_{6750}(A_i), C_{6750}(B_i), \dots,  C_{6750}(E_i)]}$ 

\hfill $\bigcup$ ${\mathcal{D}[C_{6750}(F_i), C_{6750}(G_i), \dots,  C_{6750}(J_i)]}$ $\bigcup$ ${\mathcal{D}[C_{6750}(K_i), C_{6750}(L_i), \dots,  C_{6750}(O_i)]}$ 

\hfill is a directed subgraph of $\mathcal{D}$ without edges which are defined on Type-1 

\hfill isomorphisms and also on Type-2 isomorphism w.r.t. $m$ = 3, $1 \leq i \leq 30$.

\medskip 
The isomorphism graph $\mathcal{G}$ of $C_{6750}(R)$ can also be written as $\mathcal{G}$ = $\mathcal{G}_1$ $\cup$ $\mathcal{G}_2$, 

\medskip
\hspace{.5cm} $\mathcal{G}_1$ = Subgraph of $\mathcal{G}$ without edges which are defined on Type-2 isomorphisms 

\hspace{1cm} = $\bigcup_{X=A,B,\dots,O}{\mathcal{G}^{X}_1}$ 
where 

\medskip
\hspace{1cm} $\mathcal{G}^{X}_1$ = $K_{30}[C_{6750}(X_1), C_{6750}(X_2), \dots, C_{6750}(X_{30})]$, $X \in \{A, B, \dots, O\}$, and

\medskip
\hspace{.5cm} $\mathcal{G}_2$ = Subgraph of $\mathcal{G}$ without edges which are defined on Type-1 isomorphisms

\hspace{1cm} = $\bigcup^{30}_{i=1}{\mathcal{G}^{Y_i}_2}$ = $\bigcup^{30}_{i=1}{\big(\mathcal{G}^{Y_i}_{2,3} \cup \mathcal{G}^{Y_i}_{2,5}\big)}$, 

\medskip
\hspace{1cm} $\mathcal{G}^{Y_i}_{2,3}$ = ${\mathcal{G}[C_{6750}(A_i), C_{6750}(F_i), C_{6750}(K_i)]}$  

\hfill $\bigcup$ ${\mathcal{G}[C_{6750}(B_i), C_{6750}(G_i), C_{6750}(L_i)]}$ $\dots$ $\bigcup$ ${\mathcal{G}[C_{6750}(E_i), C_{6750}(J_i), C_{6750}(O_i)]}$ 

\hfill is a subgraph of $\mathcal{G}$ without edges which are defined on Type-1 

\hfill isomorphisms and also on Type-2 isomorphism w.r.t. $m$ = 5, and 

\medskip
\hspace{1cm} $\mathcal{G}^{A_i}_{2,5}$ = ${\mathcal{G}[C_{6750}(A_i), C_{6750}(B_i), \dots,  C_{6750}(E_i)]}$ 

\hfill $\bigcup$ ${\mathcal{G}[C_{6750}(F_i), C_{6750}(G_i), \dots,  C_{6750}(J_i)]}$ $\bigcup$ ${\mathcal{G}[C_{6750}(K_i), C_{6750}(L_i), \dots,  C_{6750}(O_i)]}$ 

\hfill is a subgraph of $\mathcal{G}$ without edges which are defined on Type-1 

\hfill isomorphisms and also on Type-2 isomorphism w.r.t. $m$ = 3, $1 \leq i \leq 30$.

\medskip
The following figures related to the directed graph $\mathcal{D}$ are presented.

\medskip 
In figure 35, directed subgraph $\mathcal{D}^{\alpha}_1$ = ${\overleftrightarrow{K_{30}}[C_{6750}(\alpha_1), C_{6750}(\alpha_2), \dots, C_{6750}(\alpha_{30})]}$ of 

\hfill  $\mathcal{D}_1$ =  $\bigcup_{\alpha=A,B,\dots,O}{ \mathcal{D}^{\alpha}_1}$ is presented where $\alpha\in \{A, B, \dots, O\}$. 

\medskip 
In figure 36, subgraph $\mathcal{G}^{\alpha}_1$ = $K_{30}[C_{6750}(\alpha_1), C_{6750}(\alpha_2), \dots, C_{6750}(\alpha_{30})]$ of 

\hfill  $\mathcal{G}_1$ =  $\bigcup_{\alpha=A,B,\dots,O}{ \mathcal{G}^{\alpha}_1}$ is presented where $\alpha\in \{A, B, \dots, O\}$. 

\medskip
In figure 37, directed subgraph $\mathcal{D}^{Y_i}_2$ = $\mathcal{D}^{Y_i}_{2,3} \cup \mathcal{D}^{Y_i}_{2,5}$ is presented  

\hspace{1cm} with $X$ in the place of $C_{6750}(X)$, $X\in \{A, B, \dots, O\}$, and $1 \leq i \leq 30$.

\medskip
In figure 38, subgraph $\mathcal{G}^{Y_i}_3$ $\cong$ $\mathcal{G}^{Y_i}_2$ = $\mathcal{G}^{Y_i}_{2,3} \cup \mathcal{G}^{Y_i}_{2,5}$ is presented  

\hspace{1cm} with $X$ in the place of $C_{6750}(X)$, $X\in \{A, B, \dots, O\}$, and $1 \leq i \leq 30$.

\medskip
In figure 39, subgraph $\mathcal{G}^{Y_i}_4$ $\cong$ $\mathcal{G}^{Y_i}_2$ is presented  

\hspace{1cm} with $X$ in the place of $C_{6750}(X)$, $X\in \{A, B, \dots, O\}$, and $1 \leq i \leq 30$.

\medskip
In figure 40, subgraph $\mathcal{G}^{Y_i}_5$ $\cong$ $\mathcal{G}^{Y_i}_2$ $\cong$ $C_{15}(3, 5)$ is presented  

\hspace{1cm} with $X$ in the place of $C_{6750}(X)$, $X\in \{A, B, \dots, O\}$, and $1 \leq i \leq 30$.

\medskip
In figure 41, subgraph $\mathcal{G}^{Y_i}_6$ $\cong$ $\mathcal{G}^{Y_i}_2$ $\cong$ $C_{15}(5, 6)$ is presented  

\hspace{1cm} with $S$ in the place of $C_{6750}(S)$, $S\in \{A, B, \dots, O\}$.

\begin{figure}[htbp]
	\centering
	\begin{tikzpicture}[scale=1.5, transform shape]
		\def\n{30}
		\def\r{3.5cm}
		
		[auto=center,every node/.style={draw,circle}] 
		\foreach \i in {1,...,\n} {
			\node[circle, draw, inner sep=.05pt, minimum size=4pt] (n\i) at ({270-(360/\n * (\i - 1))}:\r) {\tiny{$\alpha_{\i}$}};
		}

		\foreach \i in {1,...,\n} {
			\foreach \j in {1,...,\n} {
				\ifnum\i=\j\relax
				\else
				\draw[->, >=stealth, thin, blue!40!black] (n\i) edge (n\j);
				\fi
			}
		}
	\end{tikzpicture}
	
	\medskip		
	Fig.  35. Directed subgraph $\mathcal{D}^{\alpha}_1$ = ${\overleftrightarrow{K_{30}}[\alpha_1,\alpha_2,\dots,\alpha_{30}]}$ of $\mathcal{D}$ with $\alpha_i$ in the place  
	
	~ ~ ~  of $C_{6750}(\alpha_i)$ in $\mathcal{D}^{\alpha}_1$, $\alpha\in \{A, B, \dots, O\}$, and $1 \leq i \leq 30$.
	
\end{figure}

\begin{figure}
	\centering
	\begin{tikzpicture}[scale=1.5, transform shape]
		\def\n{30}
		\def\r{3.5cm}
		
		[auto=center,every node/.style={draw,circle}] 
		\foreach \i in {1,...,\n} {
			\node[circle, draw, inner sep=.05pt, minimum size=4pt] (n\i) at ({270-(360/\n * (\i - 1))}:\r) {\tiny{$\alpha_{\i}$}};
		}

		\foreach \i in {1,...,\n} {
			\foreach \j in {1,...,\n} {
				\ifnum\i=\j\relax
				\else
				\draw[>=stealth, thin, blue!40!black] (n\i) edge (n\j);
				\fi
			}
		}
	\end{tikzpicture}
	
	\medskip		
	Fig.  36. Subgraph $\mathcal{G}^{\alpha}_1$ = $K_{30}[\alpha_1,\alpha_2,\dots,\alpha_{30}]$ of $\mathcal{G}$ with $\alpha_i$ in the place  
	
	~ ~ ~  of $C_{6750}(\alpha_i)$ in $\mathcal{G}^{\alpha}_1$, $\alpha\in \{A, B, \dots, O\}$, and $1 \leq i \leq 30$.
		
\end{figure}

\begin{figure}[ht]
	\centerline{\includegraphics[width=5in]{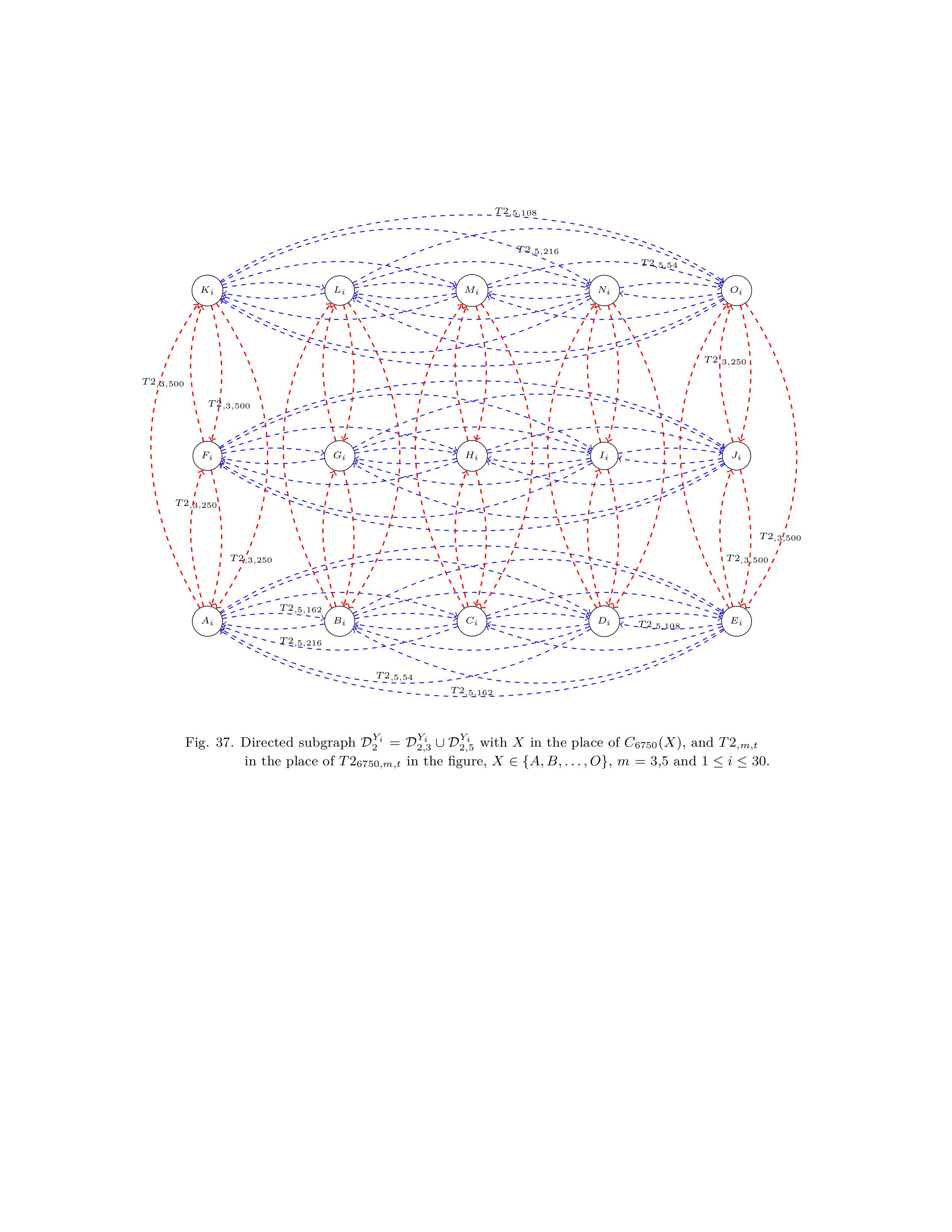}}
\end{figure}

\begin{figure}
	\begin{center}
		\begin{tikzpicture} [scale=.18,auto=center,every node/.style={draw,circle}] 
			
			\node (1) at (-12,0){\tiny{$A_i$}};
			\node (2) at (0,-6) {\tiny{$B_i$}};
			\node (3) at (12,6) {\tiny{$C_i$}};	
			\node (4) at (24,-6) {\tiny{$D_i$}};
			\node (5) at (36,0) {\tiny{$E_i$}};	
			\node (6) at (-9,15){\tiny{$F_i$}};
			\node (7) at (3,9) {\tiny{$G_i$}};
			\node (8) at (15,21) {\tiny{$H_i$}};	
			\node (9) at (27,9){\tiny{$I_i$}};
			\node (10) at (39,15) {\tiny{$J_i$}};
			\node (11) at (-12,30) {\tiny{$K_i$}};
			\node (12) at (0,24) {\tiny{$L_i$}};	
			\node (13) at (12,36){\tiny{$M_i$}};
			\node (14) at (24,24) {\tiny{$N_i$}};
			\node (15) at (36,30) {\tiny{$O_i$}};	
			
			\draw[line width=0.2mm] [blue](1)[dashed] to  (2);
			\draw[line width=0.2mm] [blue](2)[dashed] to  (3);
			\draw[line width=0.2mm] [blue](3)[dashed] to  (4);
			\draw[line width=0.2mm] [blue](4)[dashed] to  (5);
			\draw[line width=0.2mm] [blue](1)[dashed] to  (3);
			\draw[line width=0.2mm] [blue](2)[dashed] to  (4);
			\draw[line width=0.2mm] [blue](3)[dashed] to  (5);		
			\draw[line width=0.2mm] [blue](1)[dashed] to  (4);
			\draw[line width=0.2mm] [blue](2)[dashed] to  (5);
			\draw[line width=0.2mm] [blue](1)[dashed] to  (5);	
			
			\draw[line width=0.2mm] [blue](6)[dashed] to  (7);
			\draw[line width=0.2mm] [blue](7)[dashed] to  (8);
			\draw[line width=0.2mm] [blue](8)[dashed] to  (9);
			\draw[line width=0.2mm] [blue](9)[dashed] to  (10);
			\draw[line width=0.2mm] [blue](6)[dashed] to  (8);
			\draw[line width=0.2mm] [blue](7)[dashed] to  (9);
			\draw[line width=0.2mm] [blue](8)[dashed] to  (10);		
			\draw[line width=0.2mm] [blue](6)[dashed] to  (9);
			\draw[line width=0.2mm] [blue](7)[dashed] to  (10);
			\draw[line width=0.2mm] [blue](6)[dashed] to  (10);
			
			\draw[line width=0.2mm] [blue](11)[dashed] to  (12);
			\draw[line width=0.2mm] [blue](12)[dashed] to  (13);
			\draw[line width=0.2mm] [blue](13)[dashed] to  (14);
			\draw[line width=0.2mm] [blue](14)[dashed] to  (15);
			\draw[line width=0.2mm] [blue](11)[dashed] to  (13);
			\draw[line width=0.2mm] [blue](12)[dashed] to  (14);
			\draw[line width=0.2mm] [blue](13)[dashed] to  (15);		
			\draw[line width=0.2mm] [blue](11)[dashed] to  (14);
			\draw[line width=0.2mm] [blue](12)[dashed] to  (15);
			\draw[line width=0.2mm] [blue](11)[dashed] to  (15);
			
			\draw[line width=0.2mm] [red, thick](1)[dashed]  to  (6);
			\draw[line width=0.2mm] [red, thick](6)[dashed]  to  (11);
			\draw[line width=0.2mm] [red, thick](1)[dashed]  to  (11);		
			\draw[line width=0.2mm] [red, thick](2)[dashed]  to  (7);
			\draw[line width=0.2mm] [red, thick](7)[dashed]  to  (12);
			\draw[line width=0.2mm] [red, thick](2)[dashed]  to  (12);		
			\draw[line width=0.2mm] [red, thick](3)[dashed] to  (8);
			\draw[line width=0.2mm] [red, thick](8)[dashed] to  (13);
			\draw[line width=0.2mm] [red, thick](3)[dashed] to  (13);
			\draw[line width=0.2mm] [red, thick](4)[dashed]  to  (9);
			\draw[line width=0.2mm] [red, thick](9)[dashed] to  (14);
			\draw[line width=0.2mm] [red, thick](4)[dashed] to  (14);
			\draw[line width=0.2mm] [red, thick](5)[dashed] to  (10);
			\draw[line width=0.2mm] [red, thick](10)[dashed] to (15);
			\draw[line width=0.2mm] [red, thick](5)[dashed]  to (15);
			
			\node [scale=.0] (18) at (-8,17) [label=90:\tiny{$T2_{,3,500}$}]{};
			\node [scale=.0] (19) at (-13,8) [label=90:\tiny{$T2_{,3,250}$}]{};
			\node [scale=.0] (20) at (-8,3) [label=90:\tiny{$T2_{,3,250}$}]{};
			\node [scale=.0] (21) at (39.5,6) [label=90:\tiny{$T2_{,3,500}$}]{};
			\node [scale=.0] (22) at (40,17) [label=90:\tiny{$T2_{,3,250}$}]{};
			\node [scale=.0] (23) at (35,15) [label=90:\tiny{$T2_{,3,500}$}]{};
			
			\node [scale=.0] (24) at (-3.5,-1.5) [label=90:\tiny{$T2_{,5,162}$}]{};
			\node [scale=.0] (25) at (29,-3) [label=90:\tiny{$T2_{,5,108}$}]{};
			\node [scale=.0] (26) at (-3.5,-4.5) [label=90:\tiny{$T2_{,5,216}$}]{};
			\node [scale=.0] (27) at (29,29.5) [label=90:\tiny{$T2_{,5,54}$}]{};
			\node [scale=.0] (28) at (12,-9) [label=90:\tiny{$T2_{,5,162}$}]{};
			\node [scale=.0] (29) at (1,31) [label=90:\tiny{$T2_{,5,108}$}]{};
			\node [scale=.0] (30) at (17,29) [label=90:\tiny{$T2_{,5,216}$}]{};
			\node [scale=.0] (31) at (10,-1.5) [label=90:\tiny{$T2_{,5,54}$}]{};
			
		\end{tikzpicture}		
		
		\medskip		
		{\small 
			Fig.  38. Subgraph $\mathcal{G}^{Y_i}_3$ $\cong$ $\mathcal{G}^{Y_i}_2$ = $\mathcal{G}^{Y_i}_{2,3} \cup \mathcal{G}^{Y_i}_{2,5}$ with $X$ in the place  of $C_{6750}(X)$, and $T2_{,m,t}$ 
			
			~ ~ 	in the place of $T2_{6750,m,t}$ in the figure, $X\in \{A,B,\dots,O\}$, $m$ = 3,5 and $1 \leq i \leq 30$.  
		}		
	\end{center} 
\end{figure}

{\rm 	\begin{center}
		\begin{tikzpicture}  
			[scale=.2,auto=center,every node/.style={draw,circle}] 
			
			\node (1) at (26,-0.25) {\tiny{$A_i$}};
			\node (2) at (12.5,9) {\tiny{$B_i$}};
			\node (3) at (16.75,24.25) {\tiny{$C_i$}};
			\node (4) at (35.25,24.25) {\tiny{$D_i$}};
			\node (5) at (39.5,9) {\tiny{$E_i$}};
			
			\draw[line width=0.3mm] [blue](1)[dashed] to (3); 		
			\draw[line width=0.3mm] [blue](3)[dashed] to (5);
			\draw[line width=0.3mm] [blue](5)[dashed] to (2);
			\draw[line width=0.3mm] [blue](2)[dashed] to (4);
			\draw[line width=0.3mm] [blue](4)[dashed] to (1);
			
			\node (6) at (20,2) {\tiny{$F_i$}};
			\node (7) at (16,5) {\tiny{$L_i$}};
			\node (8) at (13,14) {\tiny{$G_i$}};
			\node (9) at (14,19.5) {\tiny{$M_i$}};
			\node (10) at (22,26) {\tiny{$H_i$}};
			
			\node (11) at (30,26) {\tiny{$N_i$}};
			\node (12) at (38,19.5) {\tiny{$I_i$}};
			\node (13) at (39,14) {\tiny{$O_i$}};
			\node (14) at (36,5) {\tiny{$L_i$}};
			\node (15) at (32,2) {\tiny{$K_i$}};
			
			\draw[line width=0.3mm] [blue](6)[dashed] to (10); 		
			\draw[line width=0.3mm] [blue](10)[dashed] to (14);
			\draw[line width=0.3mm] [blue](14)[dashed] to (8);
			\draw[line width=0.3mm] [blue](8)[dashed] to (12);
			\draw[line width=0.3mm] [blue](12)[dashed] to (6);
			
			\draw[line width=0.3mm] [blue](7)[dashed] to (11); 		
			\draw[line width=0.3mm] [blue](11)[dashed] to (15);
			\draw[line width=0.3mm] [blue](15)[dashed] to (9);
			\draw[line width=0.3mm] [blue](9)[dashed] to (13);
			\draw[line width=0.3mm] [blue](13)[dashed] to (7);
			
			\draw[line width=0.3mm] [red] (1) [dashed] to (6);
			\draw[line width=0.3mm] [red] (6) [dashed] to (15);
			\draw[line width=0.3mm] [red] (15) [dashed] to (1);
			
			\draw[line width=0.3mm] [red] (2) [dashed] to (7);
			\draw[line width=0.3mm] [red] (7) [dashed] to (8);
			\draw[line width=0.3mm] [red] (8) [dashed] to (2);
			
			\draw[line width=0.3mm] [red] (3) [dashed] to (9);
			\draw[line width=0.3mm] [red] (9) [dashed] to (10);
			\draw[line width=0.3mm] [red] (10) [dashed] to (3);
			
			\draw[line width=0.3mm] [red] (4) [dashed] to (11);
			\draw[line width=0.3mm] [red] (11) [dashed] to (12);
			\draw[line width=0.3mm] [red] (12) [dashed] to (4);
			
			\draw[line width=0.3mm] [red] (5) [dashed] to (13);
			\draw[line width=0.3mm] [red] (13) [dashed] to (14);
			\draw[line width=0.3mm] [red] (14) [dashed] to (5);
			
		\end{tikzpicture}		
		
		\medskip
	{\small	Fig.  39. Subgraph $\mathcal{G}^{Y_i}_4$ $\cong$ $\mathcal{G}^{Y_i}_2$ with $X$ in the place  of $C_{6750}(X)$, 
		
	\hspace{4cm}	$X\in \{A,B,\dots,O\}$, and $1 \leq i \leq 30$. }
		
\end{center} }

\begin{figure}
	\centering
	\begin{tikzpicture}  
		[scale=.28,auto=center,every node/.style={draw,circle}] 
		
		\node (1) at (26,0) {\tiny{$A_i$}};
		\node (2) at (14,9) {\tiny{$B_i$}};
		\node (3) at (18,24) {\tiny{$C_i$}};
		\node (4) at (34,24) {\tiny{$D_i$}};
		\node (5) at (38,9) {\tiny{$E_i$}};
		
	\draw[line width=0.3mm] [blue](1)[dashed] to (2); 		
		\draw[line width=0.3mm] [blue](2)[dashed] to (3);
		\draw[line width=0.3mm] [blue](3)[dashed] to (4);
		\draw[line width=0.3mm] [blue](4)[dashed] to (5);
		\draw[line width=0.3mm] [blue](5)[dashed] to (1);
		
		\node (6) at (20,2) {\tiny{$M_i$}};
		\node (7) at (16,5) {\tiny{$J_i$}};
		\node (8) at (13,14) {\tiny{$N_i$}};
		\node (9) at (14,19.5) {\tiny{$F_i$}};
		\node (10) at (22,26) {\tiny{$O_i$}};
		
		\node (11) at (30,26) {\tiny{$G_i$}};
		\node (12) at (38,19.5) {\tiny{$K_i$}};
		\node (13) at (39,14) {\tiny{$H_i$}};
		\node (14) at (36,5) {\tiny{$L_i$}};
		\node (15) at (32,2) {\tiny{$I_i$}};
		
\draw[line width=0.3mm] [blue](12)[dashed] to (14); 		
		\draw[line width=0.3mm] [blue](14)[dashed] to (6);
		\draw[line width=0.3mm] [blue](6)[dashed] to (8);
		\draw[line width=0.3mm] [blue](8)[dashed] to (10);
		\draw[line width=0.3mm] [blue](10)[dashed] to (12);
		
	\draw[line width=0.3mm] [blue](7)[dashed] to (9); 		
		\draw[line width=0.3mm] [blue](9)[dashed] to (11);
		\draw[line width=0.3mm] [blue](11)[dashed] to (13);
		\draw[line width=0.3mm] [blue](13)[dashed] to (15);
		\draw[line width=0.3mm] [blue](15)[dashed] to (7);
		
		\draw[line width=0.3mm] [red] (1) [dashed] to (9);
		\draw[line width=0.3mm] [red] (9) [dashed] to (12);
		\draw[line width=0.3mm] [red] (12) [dashed] to (1);
		
		\draw[line width=0.3mm] [red] (6) [dashed] to (3);
		\draw[line width=0.3mm] [red] (3) [dashed] to (13);
		\draw[line width=0.3mm] [red] (13) [dashed] to (6);
		
		\draw[line width=0.3mm] [red] (7) [dashed] to (10);
		\draw[line width=0.3mm] [red] (10) [dashed] to (5);
		\draw[line width=0.3mm] [red] (5) [dashed] to (7);
		
		\draw[line width=0.3mm] [red] (2) [dashed] to (11);
		\draw[line width=0.3mm] [red] (11) [dashed] to (14);
		\draw[line width=0.3mm] [red] (14) [dashed] to (2);
		
		\draw[line width=0.3mm] [red] (8) [dashed] to (4);
		\draw[line width=0.3mm] [red] (4) [dashed] to (15);
		\draw[line width=0.3mm] [red] (15) [dashed] to (8);
		
	\end{tikzpicture}		
	
	\medskip		
	{\small	Fig.  40. Subgraph $\mathcal{G}^{X_i}_5$ $\cong$ $\mathcal{G}^{X_i}_2$ $\cong$ $C_{15}(3,5)$ with $X$ in the place  of $C_{6750}(X)$, 
	
	\hspace{4cm}	$X\in \{A,B,\dots,O\}$, and $1 \leq i \leq 30$. }
\end{figure} 			

{\rm 	\begin{center}
		\begin{tikzpicture}  
			[scale=.21,auto=center,every node/.style={draw,circle}] 
			
			\node (1) at (26,0) {\tiny{$A_i$}};
			\node (2) at (14,9) {\tiny{$B_i$}};
			\node (3) at (18,24) {\tiny{$C_i$}};
			\node (4) at (34,24) {\tiny{$D_i$}};
			\node (5) at (38,9) {\tiny{$E_i$}};
			
			\draw[line width=0.3mm] [blue](1)[dashed] to (3); 		
			\draw[line width=0.3mm] [blue](3)[dashed] to (5);
			\draw[line width=0.3mm] [blue](5)[dashed] to (2);
			\draw[line width=0.3mm] [blue](2)[dashed] to (4);
			\draw[line width=0.3mm] [blue](4)[dashed] to (1);
			
			\node (6) at (20,2) {\tiny{$M_i$}};
			\node (7) at (16,5) {\tiny{$J_i$}};
			\node (8) at (13,14) {\tiny{$N_i$}};
			\node (9) at (14,19.5) {\tiny{$F_i$}};
			\node (10) at (22,26) {\tiny{$O_i$}};
			
			\node (11) at (30,26) {\tiny{$G_i$}};
			\node (12) at (38,19.5) {\tiny{$K_i$}};
			\node (13) at (39,14) {\tiny{$H_i$}};
			\node (14) at (36,5) {\tiny{$L_i$}};
			\node (15) at (32,2) {\tiny{$I_i$}};
			
			\draw[line width=0.3mm] [blue](6)[dashed] to (10); 		
			\draw[line width=0.3mm] [blue](10)[dashed] to (14);
			\draw[line width=0.3mm] [blue](14)[dashed] to (8);
			\draw[line width=0.3mm] [blue](8)[dashed] to (12);
			\draw[line width=0.3mm] [blue](12)[dashed] to (6);
			
			\draw[line width=0.3mm] [blue](7)[dashed] to (11); 		
			\draw[line width=0.3mm] [blue](11)[dashed] to (15);
			\draw[line width=0.3mm] [blue](15)[dashed] to (9);
			\draw[line width=0.3mm] [blue](9)[dashed] to (13);
			\draw[line width=0.3mm] [blue](13)[dashed] to (7);
			
			\draw[line width=0.3mm] [red] (1) [dashed] to (9);
			\draw[line width=0.3mm] [red] (9) [dashed] to (12);
			\draw[line width=0.3mm] [red] (12) [dashed] to (1);
			
			\draw[line width=0.3mm] [red] (6) [dashed] to (3);
			\draw[line width=0.3mm] [red] (3) [dashed] to (13);
			\draw[line width=0.3mm] [red] (13) [dashed] to (6);
			
			\draw[line width=0.3mm] [red] (7) [dashed] to (10);
			\draw[line width=0.3mm] [red] (10) [dashed] to (5);
			\draw[line width=0.3mm] [red] (5) [dashed] to (7);
			
			\draw[line width=0.3mm] [red] (2) [dashed] to (11);
			\draw[line width=0.3mm] [red] (11) [dashed] to (14);
			\draw[line width=0.3mm] [red] (14) [dashed] to (2);
			
			\draw[line width=0.3mm] [red] (8) [dashed] to (4);
			\draw[line width=0.3mm] [red] (4) [dashed] to (15);
			\draw[line width=0.3mm] [red] (15) [dashed] to (8);
			
		\end{tikzpicture}		
		
		\medskip		
		{\small	Fig.  41. Subgraph $\mathcal{G}^{X_i}_6$ $\cong$ $\mathcal{G}^{X_i}_2$ $\cong$ $C_{15}(5,6)$ with $X$ in the place  of $C_{6750}(X)$, 
		
		\hspace{4cm}	$X\in \{A,B,\dots,O\}$, and $1 \leq i \leq 30$. }
		
\end{center} }

In figure 37, edges defined on Type-2 isomorphisms w.r.t. $m$ = 3 and w.r.t. $m$ = 5 are given colors red and blue, respectively, and only a few directed edges are marked with their Type-2 isomorphism w.r.t. $m$ = 3 and w.r.t. $m$ = 5 as samples. For more clarity of subgraph $\mathcal{G}^{X_i}_2$ given in figure 38, we present its isomorphic graphs $\mathcal{G}^{X_i}_3$, $\mathcal{G}^{X_i}_4$, and $\mathcal{G}^{X_i}_5$ in figures 39, 40, and 41. In figure 39, $\mathcal{G}^{X_i}_3$ is given and $\mathcal{G}^{X_i}_3$ $\cong$ $C_{15}(3,5)$ $\cong$ $\mathcal{G}^{X_i}_2$; in figure 40, $\mathcal{G}^{X_i}_4$ is given and $\mathcal{G}^{X_i}_4$ $\cong$ $C_{15}(5,6)$ $\cong$ $\mathcal{G}^{X_i}_2$; and subgraph $\mathcal{G}^{X_i}_5$ given in figure 41 and is $\cong$ $\mathcal{G}^{X_i}_2$. Also, we have $C_{15}(3,5)$ $\cong$ $C_{15}(5,6)$ $\cong$ $C_3 \Box C_5$.

\medskip
From figures 35 to 41, we observe the following related to the isomorphism graph $\mathcal{G}$ of $C_{6750}(R)$ where $R$ = $\{135, 243, 250, 750, 1107, 1593, 2000, 2457, 2500, 2943\}$.

\medskip
The isomorphism graph $\mathcal{G}$ = $\mathcal{G}_1$ $\cup$ $\mathcal{G}_2$ = $\bigcup_{X=A,B,\dots,O}{\mathcal{G}^{X}_1}$ 
~$\bigcup$ ~$\bigcup^{30}_{i=1}{\mathcal{G}^{A_i}_2}$

\medskip
\hspace{1cm} $\cong$ $\bigcup_{X\in\{A,B,\dots,O\}}{K_{30}[Y_1,Y_2,\dots,Y_{30}]}$ 
~$\bigcup$ ~$\bigcup^{30}_{i=1}{\mathcal{G}^{A_i}_2}$,

\medskip
\hfill $Y_i$ = $C_{6750}(X_i)$,
$X \in \{A, B, \dots, O\}$, and $1 \leq i \leq 30$ 

\medskip
\hspace{1cm} = $\bigcup_{X\in\{A,B,\dots,O\}}{K_{30}[Y_1,Y_2,\dots,Y_{30}]}$ 
~$\bigcup$ ~$\bigcup^{30}_{i=1}{{\mathcal{G}[A_i, B_i,\dots, O_i]}}$,

\medskip
\hfill $Y_i$ = $C_{6750}(X_i)$,
$X \in \{A, B, \dots, O\}$, $X$ in the place  of $C_{6750}(X)$, 

 \hfill ($T2_{,m,t}$ in the place of $T2_{6750,m,t}$ in the figures, $m$ = 3,5,) and $1 \leq i \leq 30$.

\medskip
\hspace{1cm} $\cong$ $K_{30} \Box (C_{3} \Box C_5)$ $\cong$ $K_{30} \Box C_{15}(3, 5)$ $\cong$  $K_{30} \Box C_{15}(5, 6)$ $\ncong$ $C_{30 \times 15}(S)$ for any $S$.

The corresponding isomorphism digraph $\mathcal{D}$ of $C_{6750}(R)$ is $\cong$ ${\overleftrightarrow{K_{30}}}$ $\Box$ ${\overleftrightarrow{(C_3 \Box C_5)}}$ $\cong$  ${\overleftrightarrow{K_{30}}}$ $\Box$ ${\overleftrightarrow{C_{15}(3, 5)}}$ $\cong$ ${\overleftrightarrow{K_{30}}}$ $\Box$ ${\overleftrightarrow{C_{15}(5, 6)}}$ where $R$ = $\{135, 243, 250, 750, 1107, 1593, 2000, 2457, 2500, 2943\}$ and ${\overleftrightarrow{K_{30}}}$ = ${\overleftrightarrow{C_{30}(1,2,\dots,15)}}$.

Also, from problems 11 and 12, we get, 

$C_{6750}(135, 243, 250, 750, 1107, 1593, 2000, 2457, 2500, 2943)$ $\cong$ $C_{27}(1, 3, 8, 10) \Box C_{250}(5, 9, 41, 59, 91, 109)$,

$T1_{27}(C_{27}(1,3,8,10))$ = $\{C_{27}(1,3,8,10)$, $C_{27}(2,6,7,11)$, $C_{27}(4,5,12,13)\}$,

$T1_{250}(C_{250}(5,9,41,59,91,109))$ = $\{C_{250}(5,9,41,59,91,109)$, $C_{250}(15,23,27,73,77,123)$, 

\hfill $C_{250}(13,35,37,63,87,113)$, $C_{250}(19,31,45,69,81,119)$, $C_{250}(1,49,51,55,99,101)$,

\hfill $C_{250}(17,33,65,67,83,117)$, $C_{250}(3,47,53,85,97,103)$, $C_{250}(21,29,71,79,95,121)$,

\hfill $C_{250}(11,39,61,89,105,111)$, $C_{250}(7,43,57,93,107,115)\}$, 

$T2_{27,3}(C_{27}(1,3,8,10))$ = $\{C_{27}(1,3,8,10)$, $C_{27}(3,4,5,13)$, $C_{27}(2,3,7,11)\}$, and 

$T2_{250,5}(C_{250}(5,9,41,59,91,109))$ = $\{C_{250}(5,9,41,59,91,109)$, $C_{250}(1,5,49,51,99,101)$, 

\hfill $C_{250}(5,11,39,61,89,111)$, $C_{250}(5,21,29,71,79,121)$, $C_{250}(5,19,31,69,81,119)\}$.
\\
$\Rightarrow$  $|T1_{27}(C_{27}(1,3,8,10))|$ = 3, $|T1_{250}(C_{250}(5,9,41,59,91,109))|$ = 10,

$|T2_{27,3}(C_{27}(1,3,8,10))|$ = 3, and $|T2_{250,5}(C_{250}(5,9,41,59,91,109))|$ = 5.

\medskip
Based on the above discussion, we get the following result.

\begin{theorem} \label{t32} {\rm Let $R$ = $\{135, 243, 250, 750, 1107, 1593, 2000, 2457, 2500, 2943\}$, the isomorphism digraph of $C_{6750}(R)$ be $\mathcal{D}$, and the isomorphism graph of $C_{6750}(R)$ be $\mathcal{G}$. Then the following statements are true.
\begin{enumerate}
	\item [\rm (i)] $\mathcal{D}$ $\cong$ ${\overleftrightarrow{K_{30}}}$ $\Box$ ${\overleftrightarrow{(C_3 \Box C_5)}}$ $\cong$ ${\overleftrightarrow{K_{30}}}$ $\Box$ ${\overleftrightarrow{C_{15}(3, 5)}}$ $\cong$ ${\overleftrightarrow{K_{30}}}$ $\Box$ ${\overleftrightarrow{C_{15}(5, 6)}}$ $\cong$ ${\overleftrightarrow{C_{30}(1,2,\dots,15)}}$ $\Box$ ${\overleftrightarrow{C_{15}(5, 6)}}$. 

	\item [\rm (ii)] $\mathcal{G}$ $\cong$ $K_{30}$ $\Box$ $(C_3 \Box C_5)$ $\cong$ $K_{30}$ $\Box$ $C_{15}(3, 5)$ $\cong$ $K_{30}$ $\Box$ $C_{15}(5, 6)$ $\cong$ $C_{30}(1,2,\dots,15)$ $\Box$ $C_{15}(5, 6)$. \hfill $\Box$

\end{enumerate}}
\end{theorem}

Now, it is much easier to check whether the directed graph $\mathcal{D}$ has Eulerian isomorphism series or not from the Cartesian product structure of $\mathcal{D}$ as given in Theorem \ref{t32}.

\begin{oprm} \quad \label{op2} {\rm 
	Find whether the directed graph $\mathcal{D}$ of $C_{6750}(R)$ has Eulerian isomorphism series or not where $R$ = $\{135, 243, 250, 750, 1107, 1593, 2000, 2457, 2500, 2943\}$. \hfill $\Box$}
\end{oprm}

\section{Diameter of isomorphic set of $C_n(R)$ and isomorphic distance of $C_n(S)$ and $C_n(T)$} 

In sections 3 to 5, we could notice that various circulant graphs have different isomorphism digraphs and their diameters mostly depend on order of the circulant graph. Motivated by the above, in this section, we define {\em diameter of isomorphic set} of a circulant graph, and {\em isomorphic distance} between $C_n(R)$ and $C_n(S)$. We obtain diameter of isomorphic set of $C_{16}(1, 2, 6, 7, 8)$, $C_{16}(1, 4, 6, 7)$, $C_{24}(2, 4, 5, 6, 7, 10)$, $C_{24}(4, 5, 7, 10)$, $C_{27}(3, 4, 5, 6, 12, 13)$, $C_{27}(3, 4, 5, 9, 13)$, $C_{54}(1, 3, 17, 19)$, $C_{432}(T)$ which has isomorphic circulant graphs of Type-2 w.r.t. $m$ = 2,3 and $C_{6750}(U)$ which has isomorphic circulant graphs of Type-2 w.r.t. $m$ = 3,5 where $T$ = $\{16, 27, 48, 54, 128, 160, 189\}$ and $U$ = $\{135, 243, 250, 750, 1107$, $1593, 2000, 2457, 2500, 2943\}$. Using the isomorphism diagram, we find that the order of isomorphic sets of $C_{432}(T)$ and $C_{6750}(U)$ are 36 and 450 and their diameters are 3 and 3, respectively.

\begin{definition} \quad \label{d33} Let $C_n(R)$ and $C_n(S)$ be given circulant graphs. We define {\em isomorphic distance} between $C_n(R)$ and $C_n(S)$ as the length of a shortest path between $C_n(R)$ and $C_n(S)$ in their isomorphism graph $\mathcal{I}_n(C_n(R), C_n(S))$ and it is denoted by $d_{\cong}(C_n(R), C_n(S))$ or $d_{Iso}(C_n(R), C_n(S))$. Also, we consider 
	
	$d_{Iso}(C_n(R), C_n(S))$ = 0 if $R$ = $S$,
	
	$d_{Iso}(C_n(R), C_n(S))$ = 1 if $R \neq S$, and 
	
	\hfill $C_n(S)\in T1_n(C_n(R))$ or $C_n(S)\in T2_{n,m}(C_n(R))$ for some $m$, and 
	
	$d_{Iso}(C_m(R), C_n(S))$ = $\infty$ if $C_n(R)$ $\ncong$ $C_n(S)$, $m,n\in\mathbb{N}$.
\end{definition}

\begin{definition} \quad \label{d34} Let $C_n(R)$ be a given circulant graph. We define {\em diameter of} isomorphic set of $C_n(R)$ as the diameter of $\mathcal{I}_{n,m_1,m_2,\dots,m_c}(C_n(R))$ and we denote it by $D_{Iso}(C_n(R))$. $i.e.$, $D_{Iso}(C_n(R))$ = $D(Isoset(C_n(R)))$ = $D(\mathcal{I}_{n,m_1,m_2,\dots,m_c}(C_n(R)))$ = $D(\mathcal{ID}_{n,m_1,m_2,\dots,m_c}(C_n(R)))$ (since 
	\\	
	$D(\mathcal{ID}_{n,m_1,m_2,\dots,m_c}(C_n(R)))$ is a symmetric digraph). We consider 
	
	$D_{Iso}(C_n(R))$ = 0 if $|Isoset(C_n(R))|$ = 1,
	
	$D_{Iso}(C_n(R))$ = 1 if $|Isoset(C_n(R))| > 1$, and 
	$C_n(S)\in T1_n(C_n(R))$, $\forall$ $C_n(S)\in Isoset(C_n(R))$ or
	
\hfill $C_n(S)\in T2_{n, m}(C_n(R))$, $\forall$ $C_n(S)\in Isoset(C_n(R))$, and

	$D_{Iso}(C_m(R), C_n(S))$ = $\infty$ if $C_m(R)$ and $C_n(S)$ are non-isomorphic, $m,n\in\mathbb{N}$.
\end{definition}

Using results obtained in sections 3, 4, 5 and 6, we obtain the following on the diameter of isomorphic sets of circulant graphs of different orders.
\begin{enumerate} 
	\item [\rm 1.] $D_{Iso}(C_{16}(1,2,7))$ = 2 = $D_{Iso}(C_{16}(2,3,5))$ = $D_{Iso}(C_{16}(3,5,6))$ = $D_{Iso}(C_{16}(1,6,7))$
	
	\hfill = $D_{Iso}(C_{16}(1,2,4,7))$ = $D_{Iso}(C_{16}(2,3,4,5))$ = $D_{Iso}(C_{16}(3,4,5,6))$ = $D_{Iso}(C_{16}(1,4,6,7))$
	
	\hfill = $D_{Iso}(C_{16}(1,2,7,8))$ = $D_{Iso}(C_{16}(2,3,5,8))$ = $D_{Iso}(C_{16}(3,5,6,8))$ = $D_{Iso}(C_{16}(1,6,7,8))$
	
	\hfill = $D_{Iso}(C_{16}(1,4,6,7))$ = $D_{Iso}(C_{16}(3,4,5,6))$ = $D_{Iso}(C_{16}(1,2,4,7))$ = $D_{Iso}(C_{16}(2,3,4,5))$
	
	\hfill = $D_{Iso}(C_{16}(1,6,7,8))$ = $D_{Iso}(C_{16}(3,5,6,8))$ = $D_{Iso}(C_{16}(1,2,7,8))$ = $D_{Iso}(C_{16}(2,3,5,8))$
	
	\hspace{.2cm} = $D_{Iso}(C_{16}(1,2,4,7,8))$ = $D_{Iso}(C_{16}(2,3,4,5,8))$= $D_{Iso}(C_{16}(3,4,5,6,8))$
	
	\hspace{.2cm}  = $D_{Iso}(C_{16}(1,4,6,7,8))$ = $D_{Iso}(C_{16}(1,4,6,7,8))$ = $D_{Iso}(C_{16}(3,4,5,6,8))$
	
	\hspace{.2cm} = $D_{Iso}(C_{16}(1,2,4,7,8))$ = $D_{Iso}(C_{16}(2,3,4,5,8))$.
	
	See problem \ref{p5}  and figures 7 to 14 for more details.  
	
	\item [\rm 2.] $D_{Iso}(C_{27}(1,3,8,10))$ = 2 = $D_{Iso}(C_{27}(3,4,5,13))$ = $D_{Iso}(C_{27}(2,3,7,11))$ 
	
	\hspace{.2cm} = $D_{Iso}(C_{27}(2,6,7,11))$ = $D_{Iso}(C_{27}(1,6,8,10))$ = $D_{Iso}(C_{27}(4,5,6,13))$ 
	
	\hspace{.2cm} = $D_{Iso}(C_{27}(4,5,12,13))$ = $D_{Iso}(C_{27}(2,7,11,12))$ = $D_{Iso}(C_{27}(1,8,10,12))$ 
	
	\hspace{.2cm} = $D_{Iso}(C_{27}(1,3,6,8,10))$ = $D_{Iso}(C_{27}(3,4,5,6,13))$ = $D_{Iso}(C_{27}(2,3,6,7,11))$ 
	
	\hspace{.2cm} = $D_{Iso}(C_{27}(2,6,7,11,12))$ = $D_{Iso}(C_{27}(1,6,8,10,12))$ = $D_{Iso}(C_{27}(4,5,6,12,13))$ 
	
	\hspace{.2cm} = $D_{Iso}(C_{27}(3,4,5,12,13))$ = $D_{Iso}(C_{27}(2,3,7,11,12))$ = $D_{Iso}(C_{27}(1,3,8,10,12))$ 
	
	\hspace{.2cm} = $D_{Iso}(C_{27}(1,3,8,9,10))$ = $D_{Iso}(C_{27}(3,4,5,9,13))$ = $D_{Iso}(C_{27}(2,3,7,9,11))$ 
	
	\hspace{.2cm} = $D_{Iso}(C_{27}(2,6,7,9,11))$ = $D_{Iso}(C_{27}(1,6,8,9,10))$ = $D_{Iso}(C_{27}(4,5,6,9,13))$ 
	
	\hspace{.2cm} = $D_{Iso}(C_{27}(4,5,9,12,13))$ = $D_{Iso}(C_{27}(2,7,9,11,12))$ = $D_{Iso}(C_{27}(1,8,9,10,12))$ 
	
	\hspace{.2cm} = $D_{Iso}(C_{27}(1,3,6,8,9,10))$ = $D_{Iso}(C_{27}(3,4,5,6,9,13))$ = $D_{Iso}(C_{27}(2,3,6,7,9,11))$ 
	
	\hspace{.2cm} = $D_{Iso}(C_{27}(2,6,7,9,11,12))$ = $D_{Iso}(C_{27}(1,6,8,9,10,12))$ = $D_{Iso}(C_{27}(4,5,6,9,12,13))$ 
	
	\hspace{.2cm} = $D_{Iso}(C_{27}(3,4,5,9,12,13))$ = $D_{Iso}(C_{27}(2,3,7,9,11,12))$ = $D_{Iso}(C_{27}(1,3,8,9,10,12))$.
	
	See problem \ref{p6} and figures 15 to 22 for more details.  
	
	\item [\rm 3.] $D_{Iso}(C_{54}(1,3,17,19))$ = 2 = $D_{Iso}(C_{54}(3,7,11,25))$ = $D_{Iso}(C_{54}(3,5,13,23))$ 
	
	\hspace{.2cm} = $D_{Iso}(C_{54}(5,13,15,23))$ = $D_{Iso}(C_{54}(1,15,17,19))$ = $D_{Iso}(C_{54}(7,11,15,25))$ 
	
	\hspace{.2cm} = $D_{Iso}(C_{54}(7,11,21,25))$ = $D_{Iso}(C_{54}(5,13,21,23))$ = $D_{Iso}(C_{54}(1,17,19,21))$. 
	
	See problem \ref{p7} and figures 24 and 25 for more details.  
	
	\item [\rm 4.] $D_{Iso}(C_{432}(16, 27, 48, 54, 128, 160, 189))$ = 3 = $D_{Iso}(C_{432}(27, 32, 48, 54, 112, 176, 189))$ 
	
	\hfill = $D_{Iso}(C_{432}(27, 48, 54, 64, 80, 189, 208))$ 
	
	\hspace{.2cm}  = $D_{Iso}(C_{432}(16, 48, 54, 81, 128, 135, 160))$ = $D_{Iso}(C_{432}(32, 48, 54, 81, 112, 135, 176))$ 
	
	\hfill = $D_{Iso}(C_{432}(48, 54, 64, 80, 81, 135, 208))$ 
	
	\hspace{.2cm} = $D_{Iso}(C_{432}(64, 80, 81, 135, 162, 192, 208))$ = $D_{Iso}(C_{432}(27, 32, 54, 96, 112, 176, 189))$ 
	
	\hspace{.2cm} = $D_{Iso}(C_{432}(32, 81, 96, 112, 135, 162, 176))$ = $D_{Iso}(C_{432}(16, 48, 81, 128, 135, 160, 162))$ 
	
	\hfill = $D_{Iso}(C_{432}(27, 54, 64, 80, 189, 192, 208))$ 
	
	\hspace{.2cm} = $D_{Iso}(C_{432}(16, 81, 128, 135, 160, 162, 192))$ = $D_{Iso}(C_{432}(27, 54, 64, 80, 96, 189, 208))$ 
	
	\hspace{.2cm} = $D_{Iso}(C_{432}(64, 80, 81, 96, 135, 162, 208))$ = $D_{Iso}(C_{432}(32, 48, 81, 112, 135, 162, 176))$ 
	
	\hfill = $D_{Iso}(C_{432}(16, 27, 54, 128, 160, 189, 192))$ 
	
	\hspace{.2cm} = $D_{Iso}(C_{432}(32, 81, 112, 135, 162, 176, 192))$ = $D_{Iso}(C_{432}(16, 27, 54, 96, 128, 160, 189))$ 
	
	\hspace{.2cm} = $D_{Iso}(C_{432}(16, 81, 96, 128, 135, 160, 162))$ = $D_{Iso}(C_{432}(48, 64, 80, 81, 135, 162, 208))$ 
	
	\hfill = $D_{Iso}(C_{432}(27, 32, 54, 112, 176, 189, 192))$ 
	
	\hspace{.2cm} = $D_{Iso}(C_{432}(27, 64, 80, 162, 189, 192, 208))$ = $D_{Iso}(C_{432}(32, 54, 81, 96, 112, 135, 176))$ 
	
	\hspace{.2cm} = $D_{Iso}(C_{432}(27, 32, 96, 112, 162, 176, 189))$ = $D_{Iso}(C_{432}(16, 27, 48, 128, 160, 162, 189))$ 
	
	\hfill = $D_{Iso}(C_{432}(54, 64, 80, 81, 135, 192, 208))$ 
	
	\hspace{.2cm} = $D_{Iso}(C_{432}(16, 27, 128, 160, 162, 189, 192))$ = $D_{Iso}(C_{432}(54, 64, 80, 81, 96, 135, 208))$ 
	
	\hspace{.2cm} = $D_{Iso}(C_{432}(27, 64, 80, 96, 162, 189, 208))$ = $D_{Iso}(C_{432}(27, 32, 48, 112, 162, 176, 189))$ 
	
	\hfill = $D_{Iso}(C_{432}(16, 54, 81, 128, 135, 160, 192))$ 
	
	\hspace{.2cm} = $D_{Iso}(C_{432}(27, 32, 112, 162, 176, 189, 192))$ = $D_{Iso}(C_{432}(16, 54, 81, 96, 128, 135, 160))$ 
	
	\hspace{.2cm} = $D_{Iso}(C_{432}(16, 27, 96, 128, 160, 162, 189))$ = $D_{Iso}(C_{432}(27, 48, 64, 80, 162, 189, 208))$ 
	
	\hspace{.2cm} = $D_{Iso}(C_{432}(32, 54, 81, 112, 135, 176, 192))$. 
	
	\item [\rm 5.] 	The order of isomorphic set of $C_{432}(16, 27, 96, 128, 160, 162, 189)$ is  
	36 and its diameter is 3. See figures 29 and 32.
	
	\item [\rm 6.] 	The order of isomorphic set of $C_{6750}(135, 243, 250, 750, 1107, 1593, 2000, 2457, 2500, 2943)$ is 450 and its diameter is 3. See figures 35 and 39.
\end{enumerate}

	See figures 35 to 41 in problems \ref{p11} and \ref{p12} for more details.

\section{Conclusion} 

In problem \ref{p4} in section 3, it is shown that $C_{54}(1,3,17,19)$ and $C_{54}(7,11,15,25)$ are isomorphic but they are neither
of Type-1 nor of Type-2 w.r.t. $m$ = 3. Such circulant graphs are having isomorphic distance $\geq 2$. One can easily identify such circulant graphs from their isomorphism digraph/graph. See figures 7-22, 24,25, 27,28 to find more such circulant graphs. 

In sections 4 and 5, we noticed that isomorphism digraphs of $C_{432}(16,27,48,54,128,160,189)$ and of $C_{16}(R)$ do not contain Eulerian isomorphism series whereas Eulerian isomorphism series and its directed subgraph exist in isomorphism digraphs of $C_{27}(R)$, and also of $C_{54}(1,3,17,19)$. Related to the above, we propose the following open problem.

Now, it is much easier to check whether the directed graph $\mathcal{D}$ has Eulerian isomorphism series or not from $\mathcal{D}$ as given in Theorem \ref{t31}.

\begin{oprm} {\rm\label{op3}  Find families of circulant graphs $C_n(R)$ such that the isomorphism digraph 
\\
$I\mathcal{D}_n(C_n(R))$ contains an Eulerian  isomorphism series.} \hfill $\Box$
\end{oprm}

In section 5, we obtained isomorphism digraph and isomorphism graph of circulant graphs $C_{432}(R)$ such that each $C_{432}(R)$ has isomorphic circulant graphs of Type-2 w.r.t. $m$ = 2 as well as $m$ = 3. Related to the above, we propose the following open problem.

\begin{oprm} {\rm\label{op2}  Find circulant graphs $C_{n}(R)$ such that each $C_{n}(R)$ has isomorphic circulant graphs of Type-2 w.r.t. $m$ = $m_1$ as well as $m$ = $m_2$.  

In general, find circulant graphs $C_{n}(R)$ such that each $C_{n}(R)$ has isomorphic circulant graphs of Type-2 w.r.t. $m$ = $m_1, m_2, \dots, m_c$, $c \geq 2$ and $c\in\mathbb{N}$.     } \hfill $\Box$
\end{oprm}

In section 6, we obtained the value of diameter of isomorphic sets of a few circulant graphs. Related to this, we propose the following open problem.

\begin{oprm} {\rm\label{op3}  Find the relationship/formula, if it exists, for the diameter of isomorphic sets of any circulant graph.    } \hfill $\Box$
\end{oprm}

The following result related to isometric distance of isomorphic set on circulant graphs is interesting and its proof is obvious.

 \begin{theorem} {\rm \label{t35}  Let $D_{Iso}(C_n(R))$ = $k \geq 2$. Then, for every $j$ $\ni$ $1 \leq j \leq k$, there exists circulant graphs $C_n(S)$ and $C_n(T)$ whose isometric distance is $j$. } \hfill $\Box$
 \end{theorem}
 
 The author proposes following conjecture related to the diameter of isomorphic set on circulant graphs.
 
 \begin{con} \label{c1}  {\rm\label{t6.5}  Corresponding to every positive integer $k$, there exists circulant graph $C_n(R)$  whose isometric diameter is $k$. } \hfill $\Box$
 \end{con}

The author feels that this area of research work is going to spread rapidly and will attract more researchers. 

\vspace{.1cm}
\noindent
\textbf{Declaration of competing interest}\quad 
The author declares that he has no conflict of interest.


\end{document}